\UseRawInputEncoding
\documentclass[11pt,reqno]{amsart}
\RequirePackage[OT1]{fontenc}
\usepackage{palatino}
\usepackage{graphicx}   
\usepackage{multirow}  
\usepackage{bm}
\usepackage{ulem}
\RequirePackage{amsthm,amsmath,amsfonts,amssymb,color,stmaryrd,dsfont}
\RequirePackage[numbers]{natbib}
\RequirePackage{hyperref} 
\usepackage{amssymb}
\usepackage{anysize}
\usepackage{extarrows}
\usepackage{multirow}
\usepackage{natbib}
\usepackage{stmaryrd}
\usepackage{color}
\usepackage{amsmath}
\usepackage{enumitem}
\usepackage{tikz}
\usetikzlibrary{positioning, shapes.arrows}
\usepackage{mathtools}

\numberwithin{equation}{section}

\theoremstyle{plain} 
\newtheorem{theorem}{Theorem}[section]
\newtheorem{lemma}[theorem]{Lemma}
\newtheorem{corollary}[theorem]{Corollary} 
\newtheorem{proposition}[theorem]{Proposition} 

\newtheorem{definition}[theorem]{Definition}
\newtheorem{assumption}[theorem]{Assumption}
\newtheorem*{theorem*}{Theorem}
\newtheorem*{lemma*}{Lemma}
\newtheorem*{corollary*}{Corollary}
\newtheorem*{proposition*}{Proposition}
\newtheorem*{claim*}{Claim}
\newtheorem*{definition*}{Definition}

\theoremstyle{remark}
\newtheorem{remark}{Remark}
\newtheorem*{remark*}{Remark}

\renewcommand{\Re}{\mathrm{Re}\,}
\renewcommand{\Im}{\mathrm{Im}\,}

\newcommand{\im}{\mathrm{Im}\,}

\newcommand{\E}{{\mathbb E }}
\newcommand{\R}{{\mathbb R }}

\renewcommand{\P}{{\mathbb P}}

\newcommand{\ii}{\mathrm{i}}

\newcommand{\Tr}{{\mathrm{Tr}}}

\newcommand{\bs}{\boldsymbol}

\renewcommand{\mathbf}[1]{\bs{#1}}

\newcommand{\prob}{\mathbb{P}}

\newcommand{\eps}{\epsilon}
\newcommand{\indic}{\mathds{1}}

\newcommand{\paren}[1]{\left( #1 \right)}

\newcommand{\imag}{\textnormal{Im}}

\newcommand{\vareps}{\varepsilon}

\begin{document}

\begin{center}
\Large\bf
Phase transition for the smallest eigenvalue of covariance matrices
\end{center}

\vspace{0.5cm}
\renewcommand{\thefootnote}{\fnsymbol{footnote}}
\hspace{5ex}	
\begin{center}
 \begin{minipage}[t]{0.3\textwidth}
\begin{center}
Zhigang Bao\footnotemark[1]  \\
\footnotesize {HKUST}\\
{\it mazgbao@ust.hk}
\end{center}
\end{minipage}
\begin{minipage}[t]{0.3\textwidth}
\begin{center}
Jaehun Lee\footnotemark[2]  \\ 
\footnotesize {HKUST}\\
{\it jaehun.lee@ust.hk}
\end{center}
\end{minipage}
\begin{minipage}[t]{0.3\textwidth}
\begin{center}
Xiaocong Xu\footnotemark[3]  \\
\footnotesize {HKUST}\\
{\it xxuay@connect.ust.hk}
\end{center}
\end{minipage}
\end{center}

\footnotetext[1]{Supported by  Hong Kong RGC grant GRF 16303922  and NSFC22SC01}
\footnotetext[2]{Supported by  Hong Kong RGC grant GRF 16305421}
\footnotetext[3]{Supported by  Hong Kong RGC grant  GRF 16305421}

\vspace{2ex}
\begin{center}
 \begin{minipage}{0.9\textwidth}\footnotesize{
In this paper, we study the smallest non-zero eigenvalue of the  sample covariance matrices $\mathcal{S}(Y)=YY^*$, where $Y=(y_{ij})$ is an $M\times N$ matrix with iid mean $0$ variance $N^{-1}$ entries. We consider the regime $M=M(N)$ and $M/N\to c_\infty\in \mathbb{R}\setminus \{1\}$ as $N\to \infty$. It is known that for the extreme eigenvalues of Wigner matrices and the largest eigenvalue of $\mathcal{S}(Y)$,  a weak 4th moment condition is necessary and sufficient for the Tracy-Widom law \cite{LY14, DY1}. In this paper, we show that the Tracy-Widom law is more robust for the smallest eigenvalue of $\mathcal{S}(Y)$,  by discovering a phase transition induced by the fatness of the tail of $y_{ij}$'s. More specifically, we assume that $y_{ij}$ is symmetrically distributed with tail probability $\mathbb{P}(|\sqrt{N}y_{ij}|\geq x)\sim x^{-\alpha}$ when $x\to \infty$, for some $\alpha\in (2,4)$.  We show the following conclusions: (i). When $\alpha>\frac83$, the smallest eigenvalue follows the Tracy-Widom law on scale $N^{-\frac23}$; (ii). When $2<\alpha<\frac83$, the smallest eigenvalue follows the Gaussian law on scale $N^{-\frac{\alpha}{4}}$; (iii). When $\alpha=\frac83$, the distribution is given by an interpolation between Tracy-Widom and Gaussian; (iv). In case $\alpha\leq \frac{10}{3}$, in addition to the left edge of the MP law, a deterministic shift of order   $N^{1-\frac{\alpha}{2}}$ shall be subtracted from the smallest eigenvalue, in both the Tracy-Widom law and the Gaussian law. Overall speaking, our proof strategy is inspired by \cite{ALY} which is originally done for the bulk regime of the L\'{e}vy Wigner matrices. In addition to various technical complications arising from the bulk-to-edge extension, two ingredients are needed for our derivation: an intermediate left edge local law  based on a simple but effective matrix minor argument, and a mesoscopic CLT for the linear spectral statistic with asymptotic expansion for its expectation.  
}
\end{minipage}
\end{center}

\thispagestyle{headings}

\section{Introduction}\label{sec1}

\subsection{Main results}
As one of the most classic models in random matrix theory, the  sample covariance matrices have been widely studied. When considering the high-dimensional setting  it is well-known that the empirical spectral distribution converges to Marchenko-Pastur law (MP law). Inspired by problems such as PCA, the  extreme eigenvalue has also been extensively studied. Among the most well-known results in this direction are probably the Bai-Yin law \cite{BY} on the first order limit and the Tracy-Widom law \cite{Johansson, Johnstone} on the second order fluctuation of the extreme eigenvalues.
More specifically, let $Y=(y_{ij})\in \mathbb{R}^{M\times N}$ be a random matrix with i.i.d.~mean 0 and variance $N^{-1}$ entries, and assume that $\sqrt{N}y_{ij}$'s are i.i.d.~ copies of an random variable $\Theta$ which is independent of $N$. The covariance matrix with the data matrix $Y$ is defined as $\mathcal{S}(Y)=YY^*$. Let $\lambda_1(\mathcal{S}(Y))\geq\ldots\geq \lambda_M(\mathcal{S}(Y))$ be the ordered eigenvalues of $\mathcal{S}(Y)$. We denote by $\mu_N=\frac{1}{M}\sum_{i=1}^M \delta_{\lambda_i}$ the empirical spectral distribution. In the regime $M=M(N)$, $c_N\coloneqq M/N\to c_\infty\in (0,\infty)$ as $N\to \infty$,  it is well known since \cite{MP} that $\mu_N$ is weakly approximated  by the MP law 
\begin{align}
\rho^{\mathsf{mp}}({\rm d}x)=\frac{1}{2\pi c_N x}\sqrt{[(\lambda_+^{\mathsf{mp}}-x)(x-\lambda_-^{\mathsf{mp}})]_+} {\rm d} x+(1-\frac{1}{c_N})_+\delta_0 (x), \qquad \lambda_\pm^{\mathsf{mp}}=(1\pm \sqrt{c_N})^2.  
\label{eq:MPdensity}\end{align}
The Stieltjes transform of $\rho^{\mathsf{mp}}$ is denoted as $\mathsf{m}_{\mathsf{mp}}(z)$, which satisfies the following equation:
\begin{align}
	zc_N\mathsf{m}_{\mathsf{mp}}^2(z) +\big( z-(1-c_N) \big)\mathsf{m}_{\mathsf{mp}}(z) + 1 = 0.
\label{eq:MPSTfeq}\end{align} 
Equivalently, 
\begin{align}
\mathsf{m}_{\mathsf{mp}}(z)=\frac{1-c_N-z+\ii \sqrt{(\lambda_{+}^{\mathsf{mp}}-z)(z-\lambda_{-}^{\mathsf{mp}})}}{2zc_N}, \label{081420}
\end{align}
where the square root is taken with a branch cut on the negative real axis.

Throughout the paper, we will be interested in the regime $c_\infty\neq 1$. In this case, both $\lambda_\pm^{\mathsf{mp}}$ are called soft edges of the spectrum. Regarding the extreme eigenvalues, Bai-Yin law \cite{BY} states that 
\begin{align*}
\lambda_1(\mathcal{S}(Y))-\lambda_+^{\mathsf{mp}}\stackrel{a.s.}\longrightarrow 0, \qquad \lambda_{M\wedge N}(\mathcal{S}(Y))-\lambda_-^{\mathsf{mp}} \stackrel{a.s} \longrightarrow  0,
\end{align*}
as long as $\mathbb{E}|\sqrt{N}y_{ij}|^4<\infty$ is additionally assumed. It is also shown in \cite{BY} that $\mathbb{E}|\sqrt{N}y_{ij}|^4<\infty$ is necessary and sufficient for the convergence of $\lambda_1(\mathcal{S}(Y))$ to $\lambda_+^{\mathsf{mp}}$. It had been widely believed that the convergence of the smallest eigenvalue $\lambda_{M\wedge N}(\mathcal{S}(Y))$ to $\lambda_-^{\mathsf{mp}}$ requires a weaker moment condition, and indeed it was shown in \cite{Tikhomirov} that the condition of mean 0 and variance 1 for $\sqrt{N}y_{ij}$'s is already sufficient. On the level of the second order fluctuation, as an extension of the seminal work on Wigner matrix \cite{LY14},  it was shown in \cite{DY1} that the sufficient and necessary condition for the Tracy-Widom law of $\lambda_1(\mathcal{S}(Y))$ is the existence of a weak 4-th moment 
\begin{align}
\lim_{s\to \infty} s^4 \mathbb{P}(|\sqrt{N}y_{11}|\geq s)=0.  \label{230601}
\end{align}
Similarly to the first order result in \cite{Tikhomirov}, it has been believed that the Tracy-Widom law shall hold for the smallest eigenvalue $\lambda_{M\wedge N}(\mathcal{S}(Y))$ under a weaker condition. In this work, we are going to show that the smallest eigenvalue counterpart of (\ref{230601}) is 
\begin{align*}
\lim_{s\to \infty} s^{\frac{8}{3}} \mathbb{P}(|\sqrt{N}y_{11}|\geq s)=0, 
\end{align*}
under Assumption \ref{main assum} below. 
Moreover, when the tail $\mathbb{P}(|\sqrt{N}y_{11}|\geq s)$ becomes heavier, the distribution of $\lambda_{M\wedge N}(\mathcal{S}(Y))$ exhibits a phase transition from Tracy-Widom to Gaussian. For technical reason,  we make the following assumptions on $\mathcal{S}(Y)$. 

\begin{assumption} \label{main assum} We make the following assumptions on the covariance matrix $\mathcal{S}(Y)$.

(i). (\it On matrix entries) We suppose that $\sqrt{N}y_{ij}$'s are all iid copies of a random variable $\Theta$ which is independent of $N$. Suppose that $\mathbb{E}\Theta=0$ and $\mathbb{E}\Theta^2=1$. We further  assume that $\Theta$ is symmetrically distributed, absolutely continuous with a positive density at 0 and  as $s\to \infty$, 
\begin{align*}
\left|\mathbb{P}(\Theta>s) + \frac{\mathsf{c}}{\Gamma(1-\alpha/2)}s^{-\alpha}\right| 
\lesssim  s^{-(\alpha + \varrho)}
\end{align*}
for some $\alpha\in (2,4)$, some constant $\mathsf{c}>0$ and some small $\varrho > 0$,

(ii). {\it (On dimension)} We assume that $M\coloneqq M(N)$ and as $N\to \infty$
\begin{align*}
c_N\coloneqq \frac{M}{N}\to c_\infty\in (0, \infty)\setminus \{1\}. 
\end{align*}
\end{assumption}

Our results are collected in the following main theorem.
For brevity, we assume $M<N$ throughout this paper.
Analogous results can be easily obtained by switching the role of $M$ and $N$ when $M>N$.

\begin{theorem}\label{main.thm}
Suppose that Assumption \ref{main assum} holds. There exists a random variable $\mathcal{X}_\alpha$, such that the following statements hold when $N\to \infty$. 

\begin{itemize}
	\item[(i):] $$
\frac{-M^{\frac23}}{\sqrt{c_N}(1-\sqrt{c_N})^{4/3}}\big( \lambda_{M\wedge N} (\mathcal{S}(Y))-\lambda_{-}^{\mathsf{mp}}-\mathcal{X}_\alpha\big)\Rightarrow \mathrm{TW}_1.
$$
	\item[(ii):] \begin{align*}
\frac{N^{\frac{\alpha}{4}}\big(\mathcal{X}_\alpha-\mathbb{E}\mathcal{X}_\alpha\big)}{\sigma_\alpha}\Rightarrow N(0, 1), \qquad
 \sigma_\alpha^2 = \frac{\mathsf{c}c_N^{(4-\alpha)/4}(1-\sqrt{c_N})^4(\alpha-2)}{2} \Gamma\Big(\frac{\alpha}{2} +1 \Big).
\end{align*}
	\item[(iii):] \begin{align*}
	\mathbb{E}\mathcal{X}_\alpha=-N^{1-\frac{\alpha}{2}} \frac{\mathsf{c}(1-\sqrt{c_N})^2}{c_N^{(\alpha-2)/4}} \Gamma\Big(\frac{\alpha}{2} +1\Big)+ \mathfrak{o}(N^{1-\frac{\alpha}{2}}),
	\end{align*}
\item[(iv):] In case  $\alpha = 8/3$, the following convergence holds:
\begin{align*}
\frac{-M^{\frac{2}{3}}}{\sqrt{c_N}(1-\sqrt{c_N})^{\frac{4}{3}}}\left(\lambda_{M\wedge N} (\mathcal{S}(Y))-\lambda_{-}^{\mathsf{mp}}-\mathbb{E}\mathcal{X}_\alpha\right) \Rightarrow \mathrm{TW}_1 + \mathcal{N}(0, \tilde{\sigma}^2), \quad
 \tilde{\sigma}^2 = \frac{\mathsf{c}c_\infty^{\frac{2}{3}}(1-\sqrt{c_{\infty}})^{\frac{4}{3}}}{3}\Gamma\left(\frac{7}{3}\right).
\end{align*}
where $\mathrm{TW}_1$ and $\mathcal{N}(0,\tilde{\sigma})$ in the RHS of the above convergence are independent.
\end{itemize}
\end{theorem}

\vspace{2ex}
\begin{remark} From the above theorem, we can see that a phase transition occurs at $\alpha=8/3$. When $\alpha> 8/3$, the fluctuation of $ \lambda_{M\wedge N} (\mathcal{S}(Y))$ is governed by $\mathrm{TW}_1$ on scale $N^{-2/3}$. When $2<\alpha<8/3$, the fluctuation is dominated by that of $\mathcal{X}_\alpha$, and thus it is Gaussian on scale $N^{-\alpha/4}$. In the case $\alpha=8/3$, the limiting distribution is given by the convolution of a Tracy-Widom and Gaussian. When $\alpha\leq 10/3$, a shift of order $N^{1-{\alpha}/{2}}$ is created by $\mathbb{E}\mathcal{X}_\alpha$. We remark here that a natural further direction is to exploit the expansion of $\mathbb{E}\mathcal{X}_\alpha$ up to an order smaller than the fluctuation. But due to technical reason, we do not pursue this direction in the current paper. 
\end{remark}

\subsection{Related References} 

The Tracy-Widom distribution in random matrices was first obtained for GOE and GUE in \cite{TW1, TW2} and was later extended to Wishart matrices in \cite{Johansson} and \cite{Johnstone}. In the past few decades, the universality of the Tracy-Widom law has been extensively studied. The extreme eigenvalues of many random matrices with general distributions and structures have been proven to follow the Tracy-Widom distribution. We refer to the following literature \cite{Soshnikov, TV, FS, Peche, PY, Sodin, EYY12, LY, KY17, BPZ, LY14, LS16, LS15, AEK20, SX22, DY} for related developments. Although the Tracy-Widom distribution is very robust, some phase transitions may occur when considering heavy-tailed matrices or sparse matrices. For example, for sparse Erd\H{o}s-R\'{e}nyi graphs $G(N,p)$, it is known from \cite{HLY} that a phase transition from Tracy-Widom to Gaussian will occur when $p$ crosses $N^{-2/3}$. We also refer to \cite{EKYY, LS18, HY22, HK21, Lee} for related study.   For heavy-tailed Wigner matrices or sample covariance matrices, as we mentioned, according to \cite{LY14} and \cite{DY1}, the largest eigenvalue follows the Tracy Widom distribution if and only if a weak $4$-th moment condition is satisfied. From \cite{Soshnikov04, ABP, Diaconu}, we also know the distribution of the largest eigenvalue when the matrix entries have heavier tail. We would also like to mention the recent research on the mobility edge of L\'{e}vy matrix with $\alpha<1$ in \cite{ABL}. On the other hand, if we focus on bulk statistics, universality will be very robust. For any $\alpha>0$, it is  proved in  \cite{ALY, Amol} that the bulk universality is valid. An extension of \cite{ALY} to the hard edge of the covariance matrix in case $M=N$ is considered in \cite{Louvaris}.
 In our current work, we focus on the regime $\alpha\in (2,4)$ for the left edge of the covariance matrices. According to \cite{BDG}, even the global law will no longer be MP law in case $\alpha<2$, and thus we expect a significantly different analysis is needed in this regime. Regarding other works on the behaviour of the spectrum for heavy-tailed matrices, we refer to \cite{AG, BG, BG13, BG17, BGC, HY, HM2018, JungL} for instance.

\subsection{Proof strategy}
Our starting point is a decomposition of $Y$, or more precisely a resampling of $Y$,  from the work \cite{ALY}. 
Consider the Bernoulli $0-1$ random variables $\psi_{ij}$ and $\chi_{ij}$ defined by
\begin{align}
	\P [\psi_{ij} = 1] = \P[|y_{ij}| \ge N^{-\epsilon_b} ], \quad  \P [\chi_{ij} = 1] =  \frac{\P [|y_{ij}| \in [N^{-1/2-\epsilon_a}, N^{-\epsilon_b}]]}{\P [|y_{ij}| < N^{-\epsilon_b}]}
	\label{def:definofpsichi}\end{align}
	for some small positive constants $\epsilon_a, \epsilon_b$. In the sequel, we shall first choose $\epsilon_b$ and then choose $\epsilon_a= \epsilon_a(\epsilon_b, \alpha)$ to be sufficiently small. Specifically, throughout the discussion, we can make the following choice
	\begin{align}
	0<\epsilon_b<(\alpha-2)/10\alpha,\qquad 0<\epsilon_a< \min\{\epsilon_b, 4-\alpha\}/10000.  \label{081401}
	\end{align}
Let $a, b$, and $c$ be random variables such that
\begin{align*}
	&\P [a_{ij} \in I] = \frac{\P [ y_{ij} \in (-N^{-1/2-\epsilon_a}, N^{-1/2-\epsilon_a}) \cap I  ]}{\P [ |y_{ij}| \le  N^{-1/2-\epsilon_a}]},\\
	&\P [b_{ij} \in I] = \frac{\P [ y_{ij} \in \big( (-N^{-\epsilon_b}, -N^{-1/2-\epsilon_a}] \cup  [N^{-1/2-\epsilon_a},N^{-\epsilon_b}) \big) \cap I  ]}{\P [ |y_{ij}| \in   [N^{-1/2-\epsilon_a},N^{-\epsilon_b})]},\\
	&\P [c_{ij} \in I] = \frac{\P [ y_{ij} \in \big( (-\infty, -N^{-\epsilon_b}) \cup  (N^{-\epsilon_b}, \infty) \big) \cap I  ]}{\P [ |y_{ij}| \ge    N^{-\epsilon_b}]}.
\end{align*}
 For each $(i,j) \in [M]\times [N]$, we set
\begin{align*}
	\mathsf{A}_{ij} = (1-\psi_{ij})(1-\chi_{ij})a_{ij}, \quad \mathsf{B}_{ij} = (1-\psi_{ij})\chi_{ij}b_{ij},\quad \mathsf{C}_{ij} = \psi_{ij}c_{ij}
\end{align*}
where $a,b,c,\psi, \chi$-variables are all mutually independent. 
Sample $Y$ and $X$ by setting 
\begin{align}
Y = \mathsf{A} + \mathsf{B}+\mathsf{C}, \qquad X = \mathsf{B} + \mathsf{C}.
\label{def:definofYX}\end{align}
 The dependence among $\mathsf{A}, \mathsf{B}$ and $\mathsf{C}$ is then governed by the $\psi$ and $\chi$ variables. 

The purpose of the above decomposition, especially the separation of part $\mathsf{A}$, is to view our model as a deformed model. We hope that the light-tailed  part  $\mathsf{A}$ can regularize the spectrum of the heavy-tailed  part $X=\mathsf{B}+\mathsf{C}$, leading to the emergence of the edge universality. This idea is rooted in the dynamic approach developed in the last decade. We refer to the monograph \cite{EY} for a detailed introduction of this powerful approach, and also refer to \cite{LY, LY17, ES, BEYY16, LSY, HL19, AH20, EY15, AH23} for instance.  On a more specific level, our proof strategy is inspired by  \cite{ALY} where the authors consider the bulk statistics of  the L\'{e}vy Wigner matrices in the regime $\alpha\in (0,2)$, which we will denote by $H$ in the sequel. In \cite{ALY}, the main idea to prove the bulk universality of the local statistics is to compare the L\'{e}vy Wigner matrix $H=\mathsf{A}_H+\mathsf{B}_H+\mathsf{C}_H$ with the Gaussian divisible model $H_t=\sqrt{t}W_H+\mathsf{B}_H+\mathsf{C}_H$, where $\mathsf{A}_H, \mathsf{B}_H$ and $\mathsf{C}_H$ are defined similarly to $\mathsf{A},\mathsf{B},\mathsf{C}$ above,  and $W_H$ is a  GOE independent of $H$. Here $t$ is chosen in such a way that $\sqrt{t}(W_H)_{ij}$ matches $(\mathsf{A}_H)_{ij}$ up to the third moment, conditioning on $(\psi_{H})_{ij}=0$, where $\psi_{H}$ is defined similarly to $\psi$. Roughly speaking, the proof strategy of \cite{ALY} is as follows. First, one needs to prove that the spectrum of $\mathsf{B}_H + \mathsf{C}_H$ satisfies an intermediate local law, which shows that the spectral density of $\mathsf{B}_H + \mathsf{C}_H$ is bounded below and above at a scale $\eta_{\ast} \leq N^{-\delta} t$. This control of the spectral density is also called $\eta_{\ast}$-regularity. Next, with the $\eta_{\ast}$-regularity established, one can use the results from \cite{LY} to prove that the $\sqrt{t} W_H$ component can improve the spectral regularity to the optimal (bulk) scale $\eta\geq N^{-1+\delta}$, and further obtain the bulk universality of $H_t$.  Finally, one can prove that the bulk local eigenvalue statistics of $H$ and $H_t$ have the same asymptotic distribution by comparing the Green functions of $H$ and $H_t$. However, the main difficulty here is that, unlike in $H_t$, the small part $\mathsf{A}_H$ and the major part $\mathsf{B}_H + \mathsf{C}_H$ in $H$ are not independent. They are coupled by the $\psi$ and $\chi$ variables. Despite this dependence being explicit, great effort has been made to carry out the comparison in \cite{ALY}. 

At a high level, our proof strategy involves adapting the approach from \cite{ALY} for the bulk regime to the left edge of the covariance matrices. However, this adaptation is far from being straightforward. We summarize some major ideas as follows. 

1. ({\it Intermediate local law}) Similar to many previous DBM works, if we want to initiate the analysis, we need an intermediate local law for the $X=\mathsf{B}+\mathsf{C}$ part. More precisely, we require an $\eta_{\ast}$-regularity of the eigenvalue density for $\mathcal{S}(X)=XX^*$ at the left edge of the MP law, for some $\eta_{\ast}\ll 1$. According to \cite{ABP}, such a regularity cannot be true at the right edge of the spectrum. In order to  explain heuristically the difference between the largest and smallest eigenvalues under the heavy-tailed assumption,  we recall the variational definition of the smallest and largest singular values of $X$, which are also the square roots of the corresponding eigenvalues of $\mathcal{S}(X)$,  
\begin{align}
\sigma_M(X)=\inf_{v\in S^{M-1}}\left\|X^* v \right\|_2, \qquad \sigma_1(X)=\sup_{v\in S^{M-1}} \left\|X^* v\right\|_2. \label{070601}
\end{align}
Denote by $v_M$ and $v_1$ the right singular vectors of  $X^*$ corresponding to $\sigma_M(X)$ and $\sigma_1(X)$, respectively. From the variational representation, it is clear that $v_1$ favors the large entry of $X^*$, and thus $\sigma_1({X})$ will be large as long as there is a big entry in ${X}$.  This is indeed the case when the weak $4$-th moment condition is not satisfied. In contrast, in (\ref{070601}), since $v_M$ is the minimizer, it tries to avoid the big entries of ${X}^*$, i.e., it tends to live in the null space of $\mathsf{C}^*$. Hence, heuristically, we can believe that removing the $\mathsf{C}$ entries will not significantly change the smallest singular value, as long as the null space of $\mathsf{C}$ is sufficiently big. This will be true if $\text{rank}(\mathsf{C})=o(N)$, which indeed holds when $\alpha>2$. This simple heuristic explains why the first order behaviour of the smallest singular value of $X$, is more robust under the weak moment condition, in contrast to the largest singular value.  It also indicates the following strategy for obtaining an intermediate local law for ${X}$.  Let $\Psi=(\psi_{ij})$. We define the index sets
\begin{align}
\mathcal{D}_r\coloneqq \mathcal{D}_r(\Psi)\coloneqq \Big\{i\in [M]: \sum_{j=1}^N\psi_{ij}\geq 1\Big\}, \qquad  \mathcal{D}_c\coloneqq\mathcal{D}_c(\Psi)\coloneqq\Big\{j\in [N]: \sum_{i=1}^M \psi_{ij}\geq 1\Big\} \label{0818100}
\end{align}
which are the index set of rows/columns in which one can find at least one nonzero $\psi_{ij}$. For any matrix $A\in \mathbb{C}^{M\times N}$, let $A^{(\mathcal{D}_r)}$ and $A^{[\mathcal{D}_c]}$ be the minors of $A$ with the $\mathcal{D}_r$ rows and $\mathcal{D}_c$ columns removed, respectively, and we also use $\mathcal{S}(\mathcal{B})=\mathcal{B}\mathcal{B}^*$ for any rectangle matrix $\mathcal{B}$ in the sequel. By Cauchy interlacing, we can easily see that 
\begin{align*}
\lambda_{M}(\mathcal{S}({X}^{[\mathcal{D}_c]}))\leq \lambda_M(\mathcal{S}({X}))\leq \lambda_{M-|\mathcal{D}_r|}(\mathcal{S}({X}^{(\mathcal{D}_r)}))
\end{align*}
Further notice that ${X}^{(\mathcal{D}_r)}=\mathsf{B}^{(\mathcal{D}_r)}$ and ${X}^{[\mathcal{D}_c]}=\mathsf{B}^{[\mathcal{D}_c]}$, and thus we have
\begin{align}
\lambda_{M}(\mathcal{S}({\mathsf{B}}^{[\mathcal{D}_c]}))\leq \lambda_M(\mathcal{S}({X}))\leq \lambda_{M-|\mathcal{D}_r|}(\mathcal{S}({\mathsf{B}}^{(\mathcal{D}_r)})).  \label{071001}
\end{align}
Conditioning on the matrix $\Psi$, we notice that both 
$\mathcal{S}({\mathsf{B}}^{[\mathcal{D}_c]})$ and $\mathcal{S}({\mathsf{B}}^{(\mathcal{D}_r)})$ are random matrices with bounded support, since $|b_{ij}|\leq N^{-\epsilon_b}$. For such matrices, one has a local law with precision $N^{-2\epsilon_b}$; see \cite{HLS}. This local law together with (\ref{071001}) will give a rigidity estimate of $\lambda_M(\mathcal{S}({X}))$ on scale $\eta_{\ast}=N^{-\epsilon_b}$ according to our choice in (\ref{081401}). Similarly applying the above row and column minor argument, one can derive an intermediate local law for ${X}$, which implies that ${X}$ satisfies the $\eta_{\ast}$-regularity at the left edge. We remark here that in our regime $\alpha\in(2,4)$, a weak intermediate local law, or alternatively, a weak regularity with $\eta_{\ast}\sim N^{-\varepsilon}$ for some small $\varepsilon>0$ would be sufficient. This is always possible if we choose a suitable $\epsilon_b$.  In contrast, in the work \cite{ALY}, in the regime $\alpha\in (0,2)$, a stronger regularity with a more carefully chosen $\eta_{\ast}$ is actually needed. 

2. ({\it Gaussian divisible ensemble}) We then consider the Gaussian divisible model 
\begin{align}
V_t\coloneqq\sqrt{t} W+\mathsf{B}+\mathsf{C}=\sqrt{t}W+{X}, \qquad \mathcal{S}(V_t)=V_tV_t^*, \label{071801}
\end{align} 
where $W=(w_{ij})\in \mathbb{R}^{M\times N}$ is a Gaussian matrix with iid $N(0,N^{-1})$ entries, and $t=N\mathbb{E}|\mathsf{A}_{ij}|^2$ (slightly different from the choice in \cite{ALY} for convenience). 
With the $\eta_{\ast}$-regularity of $\mathcal{S}(X)$, we then choose $1\gg t\gg \sqrt{\eta_{\ast}}$. Actually, our $t$ would be order $N^{-2\epsilon_a}$. By choosing $\epsilon_a$ sufficiently small in light of (\ref{081401}), our $t$ can be sufficiently close to $1$. By conditioning on the matrix ${X}$, the following edge universality can be achieved for the Gaussian divisible model $\mathcal{S}(V_t)$ by extending the result in \cite{LY} and \cite{DY} to the left edge of the sample covariance matrices
\begin{align}
N^{\frac{2}{3}}\gamma \big((\lambda_M(\mathcal{S}(V_t))-\lambda_{-, t}\big)\Rightarrow \mathrm{TW}_1, \label{071101}
\end{align} 
for some constant $\gamma$, 
where $\lambda_{-,t}$ can be approximated by a mesoscopic statistic of the spectrum of $\mathcal{S}(X)$. Specifically,
\begin{align}
\lambda_{-,t}=(1-c_Ntm_X(\zeta_{-,t}))^2\zeta_{-,t}+(1-c_N)t(1-c_Ntm_X(\zeta_{-,t})), 
\label{eq:defoflambdaminust}\end{align}
where $m_X$ is the Stieltjes transform of the spectral distribution of $\mathcal{S}(X)$, and $\zeta_{-,t}$ is a random parameter defined through (\ref{eq:zeta}). We remark here that even though $\zeta_{-,t}$ is random, it can be proven that with a high probability, $\lambda_M(\mathcal{S}(X))-\zeta_{-,t}\sim t^2$. Hence, regarding the Stieltjes transform $m_X(\zeta_{-,t})$, we are at a (random) mesoscopic energy scale of order $t^2$. From the work  \cite{BM, Malysheva}, one already knows that the global statistic $m_X(z)-\mathbb{E}m_X(z)$ follows a CLT on scale $N^{-\alpha/4}$ for a fixed $z$ with $\Im z>0$. Due to the randomness of our parameter $\zeta_{-,t}$, a further expansion of it around a deterministic parameter $\zeta_{\mathsf{e}}$ will be needed to adapt the argument  in \cite{BM, Malysheva}. Consequently, after the expansion, we will need to control the fluctuations of $m_X^{(k)}(\zeta_{\mathsf{e}})$ for $k=0, \ldots, K$ with a sufficiently large $K$. Studying the fluctuations of these mesoscopic statistics eventually leads to a CLT
\begin{align*}
N^{\frac{\alpha}{4}}(\lambda_{-, t}-\mathbb{E}\lambda_{-,t})\Rightarrow N(0, \sigma_\alpha^2).
\end{align*}
 In addition to the above CLT, we need one more step to study the expansion of $\mathbb{E}\lambda_{-,t}$. It turns out that 
\begin{align*}
\mathbb{E}\lambda_{-, t}=\lambda_{-}^{\mathsf{mp}} -N^{1-\frac{\alpha}{2}}s_\alpha+\mathfrak{o}(N^{1-\frac{\alpha}{2}}).
\end{align*}

3. ({\it Green function comparison}) 

Finally, we shall extend the result (\ref{071101}) from the Gaussian divisible model to our original matrix $\mathcal{S}(Y)$, using a Green function comparison inspired by \cite{ALY}.
 It is now well-understood that one can compare certain functionals of the Green functions of two matrices instead of their eigenvalue distributions. Recall $Y_t$ from (\ref{071801}), and we define the interpolations
 \begin{align}
&Y^\gamma=\gamma \mathsf{A}+t^{1/2}(1-\gamma^2)^{1/2}W+\mathsf{B}+\mathsf{C}, \qquad S^\gamma=Y^\gamma (Y^\gamma)^*, \notag\\
& G^\gamma(z)=(S^\gamma-z)^{-1}, \qquad \mathcal{G}^\gamma(z)=((Y^\gamma)^*Y^\gamma-z)^{-1} \qquad m^\gamma(z)=\frac{1}{M} \text{Tr} G^\gamma(z),  \label{081537}
\end{align} 
In order to extend (\ref{071101}) from $S^0=\mathcal{S}(V_t)$ to $S^1=\mathcal{S}(Y)$, from \cite{PY} for instance, we know that it suffices to establish the following result for some smooth bounded $F:\mathbb{R}\to \mathbb{R}$ with bounded derivatives
\begin{align}
\Big|\mathbb{E} F\Big( N \int_{E_1}^{E_2} {\rm d} E \;\Im m^1(\lambda_{-,t}+E+\mathrm{i}\eta_0)\Big)- \mathbb{E} F\Big( N \int_{E_1}^{E_2} {\rm d} E \;\Im m^0(\lambda_{-,t}+E+\mathrm{i}\eta_0)\Big)\Big|\leq N^{-\delta}, \label{071102}
\end{align}
where $E_1<E_2$, and $|E_i|\leq N^{-\frac23+\varepsilon}$ for $i=1,2$, and $\eta_0=N^{-\frac23-\varepsilon}$, if we have 
the rigidity estimate
\begin{align}
|\lambda_M(S^a)-\lambda_{-,t}|\prec N^{-\frac23}, \qquad a=0, 1\label{071105}
\end{align} 
The estimate is easily available for the case $a=0$ (Gaussian divisible model) by a straightforward extension of  \cite{LY} and  \cite{DY}. This rigidity estimate for case $a=0$ is actually a technical input of getting (\ref{071101}). Hence, before the comparison in (\ref{071102}), we shall first prove (\ref{071105}) for $a=1$,  again by a Green function comparison. We claim that it suffices to show for all $z_{-,t}=\lambda_{-,t}+\kappa+\ii\eta$, with $|\kappa| \in N^{-\epsilon_b/2}$, and $\eta\in [N^{-\frac23-\varepsilon}, N^{-\varepsilon}]$ with some small $\varepsilon>0$, 
\begin{align}
\mathbb{E}\Big|N\eta (\Im m^1 (z_{-,t})-\Im \tilde{m}^0(z_{-,t}))\Big|^{2k}\leq (1+o(1))\mathbb{E}\Big|N\eta (\Im m^0(z_{-,t})-\Im \tilde{m}^0(z_{-,t}))\Big|^{2k}+ N^{-\delta k }.  \label{071110}
\end{align} 
Similar estimate also holds when one replaces $\Im$ to $\Re$. 
Here we introduced a copy of $ m^0(z)$
\begin{align*}
\tilde{m}^0(z)=\frac{1}{M}\text{Tr}(\sqrt{t}\tilde{W}+{X}-z)^{-1},
\end{align*}
 and $\tilde{W}$ is an iid copy of $W$. Actually, for the Gaussian divisible model, conditioning on $X$ and extending the Theorem 3 in \cite{DY2} on the deformed rectangle matrices from the right edge to the left edge, one can actually get the estimate 
 \begin{align}
 \big|\Im m^0(z_{-,t})-\Im m_{t}(z_{-,t})\big|\prec \left\{
 \begin{array}{ll}
 \frac{1}{N\eta}, &\text{ if } \kappa\geq 0, \\\\
 \frac{1}{N(|\kappa|+\eta)}+\frac{1}{(N\eta)^2\sqrt{|\kappa|+\eta}}, &\text{ if } \kappa\leq 0,
 \end{array}
 \right.
 \label{eq:averagelocallawGDM}\end{align}
 where $m_{t}$ is defined in (\ref{eq:fpe}).  Apparently, the above estimates also hold with $m^0$ replaced by $\tilde{m}^0$. Combining these estimates with (\ref{071110})  leads to the bounds $|\Im m^1 (z_{-,t})-\Im m_{t}(z_{-,t})|\prec 1/(N\eta)$ when $\kappa> -N^{-\frac23+\varepsilon}$ and $|\Im m^1 (z_{-,t})-\Im m_{t}(z_{-,t})|\ll 1/(N\eta)$ (w.h.p) when $\kappa\leq -N^{-\frac23+\varepsilon}$. Such estimates together with the real part analogue of the former will finally lead to the rigidity estimate in (\ref{071105}).

The proofs of (\ref{071102}) and (\ref{071110}) are similar.  We can turn to bound 
\begin{align}
{\rm d}\; \mathbb{E} F\Big( N \int_{E_1}^{E_2} \Im m^\gamma(z_{-,t}^0) {\rm d} E\Big)/{\rm d} \gamma  \label{071106}
\end{align}
for $z_{-,t}^0\coloneqq\lambda_{-,t}+E+\mathrm{i}\eta_0$ with $\eta_0=N^{-\frac23-\varepsilon}$, and 
\begin{align}
{\rm d}\; \mathbb{E}|N\eta (\Im m^\gamma (z_{-,t})-\Im \tilde{m}^0(z_{-,t}))|^{2k}/{\rm d} \gamma
\label{071201}
\end{align}
for $ z_{-,t}=\lambda_{-,t}+E+\mathrm{i}\eta$, where  $E\in (-N^{-\epsilon_b/2}, N^{-\frac23+\epsilon})$ and $ \eta= N^{-\frac23}$.
Actually, we shall first condition on $\Psi$, and then first estimate $\mathbb{E}_{\Psi}$ and then use a law of total expectation to estimate the full expectation. When one try to take the derivatives in (\ref{071106})-(\ref{071201}) and estimate the resulting terms, we will  need a priori bounds for the Green function entries 
\begin{align}
G^\gamma_{ij}(z), \quad \mathcal{G}^\gamma_{uv}(z), \quad ((Y^\gamma)^*G^\gamma(z))_{u i}  \label{071111}
\end{align}
in the domain
\begin{align}
\mathsf{D}= \mathsf{D}(\varepsilon_1,\varepsilon_2, \varepsilon_3)\coloneqq\{z=\lambda_{-}^{\mathsf{mp}}+E+\ii\eta: |E|\leq N^{-\varepsilon_1}, \eta\in [N^{-\frac{2}{3}-\varepsilon_2}, \varepsilon_3]\} 
\label{eq:domainD}\end{align}
with appropriately chosen small constants $\varepsilon_1, \varepsilon_2, \varepsilon_3$. 
We shall show that most of these entries are stochastically dominated by $1$ while a small amount of them are stochastically dominated by $1/t^2$. 
These bounds are not even known for the Gaussian divisible case, i.e., $\gamma=0$,  at the edge. The idea is to first prove the desired bounds of the quantities in (\ref{071111}) for $\gamma=0$, and then prove another comparison result for the Green functions
\begin{align}
\Big|\mathbb{E} |G^\gamma_{ij}(z)|^{2k}-\mathbb{E} |G^0_{ij}(z)|^{2k}\Big|\leq N^{-\delta} \label{071112}
\end{align}
for all $z\in \mathsf{D}$. Here we refer to \cite{ALY} and \cite{KY17} for similar strategy of using comparison to prove Green function bounds on local scale. 
Hence, based on the above discussion, the proof route is as  following

\vspace{2ex}
\begin{center}
  \begin{tikzpicture}[node distance=1cm, auto,]
    \node[draw, rectangle, align=center] (box1) {bounds of (\ref{071111}) for $\gamma=0$};
    \node[draw, rectangle, align=center, right=of box1] (box2) {(\ref{071112})};
    \node[draw, rectangle, align=center, right=of box2] (box3) {(\ref{071110})};
    \node[draw, rectangle, align=center, right=of box3] (box4) {(\ref{071102})};

   \draw[->, thick, shorten <=2pt, shorten >=2pt] (box1.east) -- (box2.west);
    \draw[->, thick, shorten <=2pt, shorten >=2pt] (box2.east) -- (box3.west);
    \draw[->, thick, shorten <=2pt, shorten >=2pt] (box3.east) -- (box4.west);
  \end{tikzpicture} 
\end{center}
which requires a three steps of Green function comparison with different observables. In contrast, in  \cite{ALY}, one Green function comparison for the observable $F(\Im G_{a_1b_1}(z), \cdots,\Im G_{a_m b_m}(z))$ (and its real part analogoue) with a deterministic parameter $z$ in the bulk regime would be sufficient. Also notice that our parameter $z_{-,t}$ in (\ref{071102}) is random, which further complicates the comparison. Specifically, when we do expansions of the Green function entries w.r.t. the matrix entries, we shall also keep tracking the derivatives of $\lambda_{-,t}$ w.r.t to these entries. The estimates of these derivatives involve delicate analysis of the subordination equations. 

Regarding the bounds of (\ref{071111}) for $\gamma=0$, here we shall explain the argument for $G_{ij}$ only for simplicity. The other two kinds of entries in (\ref{071111})  can be handled similarly.  For the Gaussian divisible model, conditioning on $X$, by extending the Theorem 3 in \cite{DY2} on the deformed rectangle matrices from the right edge to the left edge, with the $\eta_{\ast}$-regularity of the spectrum of $\mathcal{S}(X)$, we have  for $z\in \mathsf{D}$
 \begin{align*}
 |G^0_{ij}(z)-\Pi_{ij}(z)|\prec \Big(t\Big(\sqrt{\frac{\Im m_{t}(z)}{N\eta}}+\frac{1}{N\eta}\Big)+\frac{t^{1/2}}{N^{1/2}}\Big)\varpi^{-2}(z), 
 \end{align*}
 where  $\varpi$ is of order $t^2+\eta$, and 
 \begin{align*}
 \Pi_{ij}= (1+c_Nt m_{t})\big({X}{X}^*- \zeta_{t}(z))\big)^{-1}_{ij}=: (1+c_Nt m_{t})G_{ij}(X,\zeta_{t}(z)) .
 \end{align*}
 which is simply a multiple of a Green function entries of $\mathcal{S}(X)$, but evaluated at a random parameter $\zeta_t(z)$. 
 By the facts $t\sim N^{-2\epsilon_a}$,  $\eta\gtrsim N^{-\frac23-\varepsilon_2}$ and $|m_{t}(z)|\leq (ct|z|)^{-1/2}$ (cf. Lemma \ref{existuniq} (iv)), one can easily get $|G^0_{ij}(z)-\Pi_{ij}(z)|\prec 1$. Hence, what remains is to bound $\Pi_{ij}$, i.e., to bound $G_{ij}(X,\zeta_{t}(z))$, the Green function entry of the heavy-tailed covariance matrix $\mathcal{S}(X)$,  in the regime $2<\alpha< 4$.  We notice that such a bound has been obtained in \cite{Amol} for the heavy-tailed Wigner matrices in the same regime of $\alpha$, but in the bulk. Extending such a bound to edge could be difficult due to the deterioration of the  stability of self-consistent equation of the Stieltjes transform. However, we notice that with the $\eta_{\ast}$-regularity  of the left edge of  $\mathcal{S}(X)$ spectrum, one can show that the parameter $\zeta_{t}(z)$ is away from the left edge of the $\mathcal{S}(X)$ spectrum by a distance of order $t^2$. Hence, we are away from the edge by a mesoscopic distance, which allow us the conduct the argument similarly to the bulk case in  \cite{Amol} to get the desired bound for $\Pi_{ij}(z)$. 
 
\subsection{ Organization}  The rest of the paper will be organized as follows. In Section \ref{s.gdm}, we will state the main results for the Gaussian divisible model, whose proofs will be stated in Section \ref{s.pgdm}. Section \ref{s.general} is devoted to the statements of the Green function comparisons and prove our main theorem based on the comparisons. In Section \ref{s. proof-general}, we prove these comparison results. Some technical estimates are stated in the appendix.

\subsection{Notation} Throughout this paper, we regard $N$ as our fundamental large parameter. Any quantities that are not explicit constant or fixed may depend on $N$; we almost always omit the argument $N$ from our notation. We use $\|u\|_\alpha$ to denote the $\ell^\alpha$-norm of a vector $u$. We further use   $\|A\|$ for the operator norm of a matrix $A$.  We use $C$ to denote some generic (large) positive constant. The notation $a\sim b$ means $C^{-1}b \leq |a| \leq Cb$ for some positive constant $C$. Similarly, we use $a\lesssim b$ to denote the relation $|a|\leq Cb$ for some positive constant $C$. $\mathcal{O}$ and $\mathfrak{o}$ denote the usual big and small O notation, and $\mathcal{O}_{p}$ and $\mathfrak{o}_{p}$ denote the big and small O notation in probability. When we write $a \ll b$ and $a \gg b$ for possibly $N$-dependent quantities $a= a(N)$ and $b= b(N)$, we mean $|a|/b \to 0$ and $|a|/b \to \infty$ when $N\to \infty$, respectively.  For any positive integer $n$, let $[n]=[1:n]$ denote the set $\{1,\ldots,n\}$. For $a,b \in \R$, $a\vee b = \max\{a,b\}$ and $a\wedge b=\min\{a,b\}$. For a square matrix $A \in \mathbb{R}^{n \times n}$, we let $A_{\mathrm{diag}} =(A_{ij}\delta_{ij})\in \mathbb{R}^{n \times n}$.  We  adopt the following Green function notation for any rectangle matrix $A\in \mathbb{R}^{m \times n}$, $G(A,z)=(AA^{\top}-z)^{-1}$.

\section{Gaussian divisible model} \label{s.gdm}
 In this section, we state the main results for a Gaussian divisible model, and leave the detailed proofs to the next section. 
\subsection{Some definitions} 
Recall that $W=(w_{ij})\in \mathbb{R}^{M\times N}$ is a Gaussian matrix with iid $N(0,N^{-1})$ entries, and $t=N\mathbb{E}|\mathsf{A}_{ij}|^2$.
Consider the standard signal-plus-noise model
\begin{align}
	V_t \coloneqq X + \sqrt{t}W.
\label{def:GDM}\end{align}
In this section, we will establish several spectral properties of $\mathcal{S}(V_t)$ that will be extended to $\mathcal{S}(Y)$ later. For most of the discussion in this part, we will condition on $X$ and regard it as given, and work with the randomness of $W$. In light of this, we introduce the asymptotic eigenvalue density of $\mathcal{S}(V_t)$, denoted by $\rho_t$, through its corresponding Stieltjes transform $m_{t} \coloneqq m_{t}(z)$. For any $t > 0$, $m_{t}$ is known to be the unique solution to the following equation:
\begin{align}\label{eq:fpe}
m_{t} = \frac{1}{M}\sum_{i=1}^M \frac{1+c_Ntm_{t}}{\lambda_i(\mathcal{S}(X)) - \zeta_t},
\end{align}
subject to the condition that $\Im m_{t} > 0$ for any $z \in \mathbb{C}_{+}$. Here
\begin{align}
	\zeta_t \coloneqq \zeta_t(z) \coloneqq (1+c_Ntm_{t}(z))^2z-t(1-c_N)(1+c_Ntm_{t}(z)).
\label{eq:zeta}
\end{align}
In the context of free probability theory, $\rho_t$ corresponds to the rectangular free convolution of the spectral distribution of $\mathcal{S}(X)$ with the MP law on scale $t$, and $\zeta_t$ is the so-called subordination function for the rectangular free convolution. The following lemma provides a precise description of the existence and uniqueness of the asymptotic density. The following result holds for any realization of $X$. 
\begin{lemma}[Existence and uniqueness of asymptotic density, Lemma 2 of \cite{DY2}] \label{existuniq}
For any $t>0$, the following properties hold.
\begin{itemize}
\item[(i)] There exists a unique solution $m_{t}$ to equation (\ref{eq:fpe}) satisfying that $\Im m_{t}(z)> 0$ and $\Im z m_{t}(z)> 0$ if $z\in \mathbb{C}^+ $.

\item[(ii)] For all $E \in \mathbb{R}\setminus\{0\},$ $\lim_{\eta  \downarrow 0} m_{t}(E+\mathrm{i} \eta)$ exits, and we denote it as $m_{t}(E).$ The function $m_{t}$ is continuous on $\mathbb{R}\setminus \{0\}$, 
and $\rho_t(E)\coloneqq\pi^{-1} \Im m_{t}(E)$ is a continuous probability density function on $\mathbb{R}^+\coloneqq\{E\in \mathbb{R}:E>0\}$. Moreover, $m_{t}$ is the Stieltjes transform of $\rho_t$. Finally, $m_{t}(E)$ is a solution to (\ref{eq:fpe}) for $z=E$.  
\item[(iii)] For all $E \in \mathbb{R}\setminus\{0\},$ $\lim_{\eta  \downarrow 0} \zeta_t(E+\mathrm{i} \eta)$ exits, and we denote it as $\zeta_t(E).$ Moreover, we have $\Im \zeta_t(z)>0$ if $z\in \mathbb{C}^+ $. 

\item[(iv)] We have $\Re (1 + c_Ntm_{t}(z))>0$ for all $z\in \mathbb{C}^+$ and $|m_{t}(z)|\le (c_Nt|z|)^{-1/2}$.
\end{itemize}
\end{lemma}

For a realization of $X$, we can check if it satisfies the following  regularity condition on $m_X(z)$. Such a condition is crucial for the edge universality of DBM; see \cite{LY, DY2} for instance. 
\begin{definition}[$\eta_\ast$- regularity] \label{081551}
Let $\eta_{\ast}$ be a parameter satisfying $\eta_{\ast} \coloneqq N^{-\tau_{\ast}}$ for some constant $0 < \tau_{\ast} \le 2/3$.  For an $M\times N$ matrix $H$, we say $\mathcal{S}(H)$ is $\eta_{\ast}$-regular around the left edge $\lambda_-=\lambda_M(\mathcal{S}(H))$ if there exist constants $c_{H} > 0$ and $C_{H} > 1$ such that the following properties hold:
\begin{itemize}
	\item[(i)] For $z = E+\ii\eta$ with $\lambda_- \le E \le \lambda_- + c_{H}$ and $\eta_* + \sqrt{\eta_\ast | E - \lambda_-|} \le \eta \le 10$, we have
	\begin{align*}
		\frac{1}{C_H} \sqrt{| E - \lambda_-| + \eta} \le \Im m_H(E + \mathrm{i}\eta) \le C_H \sqrt{| E - \lambda_-| + \eta}.
	\end{align*}
	For $z = E+\mathrm{i}\eta$ with $\lambda_- - c_H \le E \le \lambda_-$ and $\eta_\ast \le  \eta \le 10$, we have
	\begin{align*}
		\frac{1}{C_H} \frac{\eta}{\sqrt{| E - \lambda_-| + \eta}} \le \Im m_H(E + \mathrm{i}\eta) \le C_H \frac{\eta}{\sqrt{| E -\lambda_-| + \eta}}.
	\end{align*}
	\item[(ii)] We have $c_H/2 \le  \lambda_- \le 2C_H$.
	\item[(iii)] We have $\| \mathcal{S}(H)\| \le N^{C_H}$.
\end{itemize}
\label{def:etastar}\end{definition}

The following lemma is a direct implication of $\eta_\ast$- regularity.
\begin{lemma}[Lemma 6 of \cite{DY2}]
	Suppose (a realization of) $\mathcal{S}(X)$ is $\eta_\ast$-regular in the sense of Definition \ref{def:etastar}. Let $\mu_{X}$ be the measure associated with $m_X(z)$.  For any fixed integer $k \ge 2$, and any $z \in \mathcal{D}$ with 
	\begin{align*}
		\mathcal{D} &\coloneqq  \big\{ z = E + \mathrm{i}\eta : \lambda_{-} - \frac{3}{4}\tilde{c} \le  E \le \lambda_{-}  ,2\eta_\ast \le \eta \le  10 \big\}\\
		&\quad \cup \big\{ z = E + \mathrm{i}\eta : \lambda_{-}  \le  E \le \lambda_{-} + \frac{3}{4}\tilde{c} , \eta_\ast + \sqrt{\eta_\ast(E-\lambda_-)} \le \eta \le  10 \big\} \\
		&\quad \cup \big\{ z = E + \mathrm{i}\eta : \lambda_{-} -\frac{3}{4}\tilde{c}  \le  E \le \lambda_{-} -2\eta_\ast  , 0 \le \eta \le  10 \big\}.
	\end{align*}
	Then we have
	\begin{align*}
		\int \frac{\mathrm{d}\mu_X(x)}{|x-E-\mathrm{i}\eta|^k} \sim \frac{\sqrt{|E-\lambda_-|+\eta}}{\eta^{k-1}}\mathbf{1}_{E \ge \lambda_{-}} +   \frac{1}{(|E-\lambda_-|+\eta)^{k-3/2}}\mathbf{1}_{E < \lambda_{-}}.
	\end{align*}
\label{lem:esthighoderi}\end{lemma}

The following notion of stochastic domination which originated from \cite{erdHos2013averaging} will be used throughout the paper.
\begin{definition}[Stochastic domination] \label{def.sd}
 	Let 
 	$
 		\mathsf{X}=(\mathsf{X}^{(N)}(u): N \in \mathbb{N}, u \in \mathrm{U}^{(N)}), \mathrm{Y}=(\mathsf{Y}^{(N)}(u): N \in \mathbb{N}, u \in \mathsf{U}^{(N)})
 	$
 	be two families of random variables, where $\mathsf{Y}$ is nonnegative, and $\mathsf{U}^{
(N)}$ is a possibly $N$-dependent parameter set. We say that $\mathrm{X}$ is stochastically dominated by $\mathsf{Y}$, uniformly in $u$, if for all small $\epsilon > 0$  and large $D > 0$,
$$
\sup _{u \in \mathsf{U}^{(N)}} \mathbb{P}\left(\left|\mathsf{X}^{(N)}(u)\right|>N^{\varepsilon} \mathsf{Y}^{(N)}(u)\right) \leqslant N^{-D}
$$
for large enough $N > N_0(\epsilon, D)$. If $\mathsf{X}$ is stochastically dominated by $\mathsf{Y}$, uniformly in $u$, we use the notation
$\mathsf{X} \prec \mathsf{Y}$ , or equivalently $\mathsf{X} = O_{\prec}(\mathsf{Y})$. Note that in the special case when $\mathsf{X}$ and $\mathsf{Y}$ are deterministic, $\mathsf{X} \prec \mathsf{Y}$ means that for any given $\epsilon > 0$, $|\mathsf{X}^{(N)}(u)| \le N^{\epsilon}\mathsf{Y}^{(N)}(u)$ uniformly in $u$, for all sufficiently large $N \ge N_0(\epsilon)$.
\end{definition}


\subsection{$\eta_\ast$- regularity of $\mathcal{S}(X)$: A matrix minor argument}
In this subsection, we state that with high probability $\eta^*$-regularity holds for $\mathcal{S}(X)$ with $\eta^*=N^{-\epsilon_b}$.  Recall that $X$ defined in (\ref{def:definofYX}). Let us recall $\Psi=(\psi_{ij}) \in \mathbb{R}^{M \times N}$, a random matrix with entries $\psi_{ij}$ as defined in (\ref{def:definofpsichi}). By setting 
\begin{align}
\epsilon_\alpha = (\alpha-2)/5\alpha, \label{081533}
\end{align} 
we call a $\Psi$  \textit{good} if it has at most $N^{1-\epsilon_\alpha}$ entries equal to $1$. The following lemma indicates that $\Psi$ is, indeed, good with high probability.

\noindent

\begin{lemma}
For any large $D > 0$, we have
	$
		\mathbb{P}(\Omega_\Psi = \{\Psi \text{ is good}  \} ) \ge 1 - N^{-D}.
	$
\label{lem:goodpsi}\end{lemma}
\begin{proof}Observe that
	$
		\mathbb{P} (\Omega_\Psi = \{\Psi \text{ is good}  \} )  = 1 - \mathbb{P}( \#\{ (i,j) : x_{ij} > N^{-\epsilon_b}  \}    >  N^{1- \epsilon_\alpha} ) .
	$
By Assumption \ref{main assum} (i), we have 
	\begin{align*}
		\mathbb{P} ( \#\{ (i,j) : x_{ij} > N^{-\epsilon_b}  \}    >  N^{1- \epsilon_\alpha} ) \lesssim \sum_{j = N^{1-\epsilon_\alpha}}^{N^2} \binom{N^2}{j} N^{-\alpha(1/2 -\epsilon_b)j} \lesssim \sum_{j = N^{1-\epsilon_\alpha}}^{N^2}  N^{ -(\alpha-2)j/2 } \lesssim N^{-D}.
	\end{align*}
	The claim now follows by possibly adjusting the constants.
\end{proof}

Given any $\Psi$ is good, the following proposition shows that $\mathcal{S}(X)$ is $\eta_*$-regular 
with 
$\eta_* = N^{-\tau_{*}}$
for some $\tau_{*}>0$.
Actually, we shall work with a truncation of $X$, $X^{\mathcal{C}}\coloneqq(x_{ij}\mathbf{1}_{|x_{ij}|\leq N^{100}})_{i\in[M],j\in[N]}$, in order to guarantee Definition \ref{081551} (iii).
Apparently,  $\|\mathcal{S}(X^{\mathcal{C}})\|\leq N^{102}$ and $\mathbb{P}(X=X^{\mathcal{C}})=1-\mathfrak{o}(1)$.
\begin{proposition}[$\eta_\ast$- regularity of $\mathcal{S}(X)$]\label{prop:etastarRegular} 
Suppose that $\Psi$ is good. Let $\eta_\ast=N^{-\epsilon_b}$. Then $\mathcal{S}(X)$ is $\eta_{\ast}$-regular around its smallest eigenvalue $\lambda_{M}(\mathcal{S}(X))$ in the sense of Definition \ref{def:etastar} with high probability.
\end{proposition}

The proof of Proposition \ref{prop:etastarRegular} is based on the following two lemmas. For notational simplicity, we define $\mathsf{m}_{\mathsf{mp}}^{(t)}(z)\coloneqq (1-t)^{-1} \mathsf{m}_{\mathsf{mp}}(z/(1-t))$  for any $t > 0$.
\begin{lemma}\label{lem: local law for BB^T}
Fix $C>0$. Let us consider $z \in \{ E+i\eta : C^{-1}\lambda_{-}^{\mathsf{mp}} \le E \le \lambda_{+}^{\mathsf{mp}}+1, 0<\eta<3 \}.$
We have
\begin{equation}\label{eq: local law BB^T}
	|m_{\mathsf{B}}(z)- \mathsf{m}_{\mathsf{mp}}^{(t)}(z)| \prec N^{-\eps_{b}} + (N\eta)^{-1}.
\end{equation}
In addition, 
\begin{equation}\label{eq: left edge rigidity BB^T}
	|\lambda_{M}(\mathcal{S}(\mathsf{B})) - (1-t)\lambda_{-}^{\mathsf{mp}}| \prec N^{-2\eps_{b}} + N^{-2/3}.
\end{equation}
\end{lemma}
\begin{proof}
We further denote by $\tilde{t}:=1 - N\mathbb{E} |\mathsf{B}_{ij}|^2$. It is easy to show that $|\tilde{t}-t|=\mathfrak{o}(N^{-1})$, and thus we have $|\mathsf{m}_{\mathsf{mp}}^{(t)}(z)-\mathsf{m}_{\mathsf{mp}}^{(\tilde{t})}(z)|\leq (N\eta)^{-1}$.
Hence, it suffices to show the following estimates
\begin{align}
	|m_{\mathsf{B}}(z)- \mathsf{m}_{\mathsf{mp}}^{(\tilde{t})}(z)| \prec N^{-\eps_{b}} + (N\eta)^{-1},\qquad 
	|\lambda_{M}(\mathcal{S}(\mathsf{B})) - (1-\tilde{t})\lambda_{-}^{\mathsf{mp}}| \prec N^{-2\eps_{b}} + N^{-2/3}. \label{eq: left edge rigidity BB^T1}
\end{align}
Notice that $\mathsf{B}$ is a so-called random matrix with bounded support. The first estimate in \eqref{eq: left edge rigidity BB^T1} is given by \cite[Theorem 2.7]{HLS}. We can show the second estimate in \eqref{eq: left edge rigidity BB^T1} adapting the proof of \cite[Theorem 2.9]{HLS} from the right edge to the left edge, in a straightforward way, given a crude lower bound of $\lambda_{M}(\mathcal{S}(\mathsf{B}))$ which is guaranteed by \cite{Tik16}. We omit the details.
\end{proof}

\begin{lemma}\label{lem: left edge rigidity XX^T}
Suppose $\Psi$ is good. Then, we have
$
	|\lambda_{M}(\mathcal{S}(X))-(1-t)\lambda_{-}^{\mathsf{mp}}| \prec N^{-2\epsilon_b}.
$
\end{lemma}
\begin{proof}
Denote by $\mathfrak{N}(\mathsf{C})$ the number of nonzero columns of $\mathsf{C}$. Since $\Psi$ is good,
$
	|\mathfrak{N}(\mathsf{C})| \le N^{1-\eps_{\alpha}},
$
with high probability.  By Cauchy interlacing, we can easily see that 
\begin{align*}
\lambda_{M}(\mathcal{S}({X}^{[\mathcal{D}_c]}))\leq \lambda_M(\mathcal{S}({X}))\leq \lambda_{M-|\mathcal{D}_r|}(\mathcal{S}({X}^{(\mathcal{D}_r)}))
\end{align*}
Further notice that ${X}^{(\mathcal{D}_r)}=\mathsf{B}^{(\mathcal{D}_r)}$ and ${X}^{[\mathcal{D}_c]}=\mathsf{B}^{[\mathcal{D}_c]}$, and thus we have
\begin{align*}
\lambda_{M}(\mathcal{S}({\mathsf{B}}^{[\mathcal{D}_c]}))\leq \lambda_M(\mathcal{S}({X}))\leq \lambda_{M-|\mathcal{D}_r|}(\mathcal{S}({\mathsf{B}}^{(\mathcal{D}_r)})).  
\end{align*}
Applying (\ref{eq: left edge rigidity BB^T}) to $\mathcal{S}({\mathsf{B}}^{[\mathcal{D}_c]})$ and $\mathcal{S}({\mathsf{B}}^{(\mathcal{D}_r)})$ with the modified parameter $c_N$, i.e., $M/(N-|\mathcal{D}_c|)$ and $(M-|\mathcal{D}_r|)/N$ respectively, in the definition of $\lambda_{-}^{\mathsf{mp}}$, we can prove  the conclusion with the fact $\epsilon_\alpha>2\epsilon_b$. 
\end{proof}

Now we show the proof of Proposition \ref{prop:etastarRegular}.
\begin{proof}[Proof of Proposition \ref{prop:etastarRegular}]
We shall show three properties (i), (ii) and (iii) (as in Definition \ref{def:etastar}) holds with high probability. Suppose that $\Psi$ is good. 

\noindent
(i). Let $\mu_X$ and $\mu_\mathsf{B}$ be the empirical spectral distributions of $\mathcal{S}(X)$ and $\mathcal{S}(\mathsf{B})$, respectively. By the rank inequality \cite[Theorem A.44]{BS10},
$
	|\mu_{X}-\mu_{\mathsf{B}}|\le 2\text{rank}(\mathsf{C})/N.
$
Then,
\begin{equation*}
	\left| \Im m_{X}(z) - \Im m_{\mathsf{B}}(z) \right|
	\le \int \Big|\frac{\eta}{(\lambda-E)^{2}+\eta^{2}} (\mu_{X}-\mu_{\mathsf{B}})(d\lambda)\Big|.
\end{equation*}
It follows from  $\eta ((\lambda-E)^{2}+\eta^{2})^{-1}  \le \eta^{-1}$ that 
$
	| \Im m_{X}(z) - \Im m_{\mathsf{B}}(z) | \lesssim  \text{rank}(\mathsf{C})/(N\eta)	= N^{-\epsilon_{\alpha}}\eta^{-1},
$
where we use the assumption that $\Psi$ is good. This together with Lemma \ref{lem: local law for BB^T} give
\begin{equation*}
	|\Im m_{X}(z) - \Im \mathsf{m}_{\mathsf{mp}}^{(t)}(z)|
	\prec N^{-\epsilon_{\alpha}}\eta^{-1}+N^{-\epsilon_{b}} + (N\eta)^{-1}.
\end{equation*}

For $E\in[\lambda_{M}(\mathcal{S}(X)),\lambda_{M}(\mathcal{S}(X))+\eta_{\ast}]$, by Lemma \ref{lem: left edge rigidity XX^T}, we have with high probability that,
\begin{equation*}
	|E-(1-t)\lambda_{-}^{\mathsf{mp}}| \le |E-\lambda_{M}(\mathcal{S}(X))| + |\lambda_{M}(\mathcal{S}(X))-(1-t)\lambda_{-}^{\mathsf{mp}}| \le 2\eta_{\ast}.
\end{equation*}
Thus, for $\eta\geq \eta_\ast$, we have
$
	\Im \mathsf{m}_{\mathsf{mp}}^{(t)}(z) \sim \sqrt{\eta},
$
which  implies that
$
	\Im m_{X}(z) \sim \sqrt{|E-\lambda_{M}(\mathcal{S}(X))|+\eta}.
$
Similarly, for $E\in[\lambda_{M}(\mathcal{S}(X))-\eta_{\ast},\lambda_{M}(\mathcal{S}(X))]$ and $\eta\geq \eta_\ast$, we can show that
$
	\Im \,m_{X}(z) \sim {\eta}/{\sqrt{|E-\lambda_{M}(\mathcal{S}(X))|+\eta}}.
$
If $E\ge\lambda_{M}(\mathcal{S}(X))+\eta_{\ast}$, we can use the fact $|\lambda_{M}(\mathcal{S}(X))-(1-t)\lambda_{-}^{\mathsf{mp}}|\ll\eta_{\ast}$ to obtain that $E\ge (1-t)\lambda_{-}^{\mathsf{mp}}$ and
$$
	\sqrt{|E-\lambda_{M}(\mathcal{S}(X))|+\eta} \sim \sqrt{|E-(1-t)\lambda_{-}^{\mathsf{mp}}|+\eta}.
$$
Similarly, if $E\le\lambda_{M}(\mathcal{S}(X))-\eta_{\ast}$, we obatin $E\le (1-t)\lambda_{-}^{\mathsf{mp}}$ and
\begin{equation*}
	\frac{\eta}{\sqrt{|E-\lambda_{M}(\mathcal{S}(X))|+\eta}} \sim \frac{\eta}{\sqrt{|E-(1-t)\lambda_{-}^{\mathsf{mp}}|+\eta}}.
\end{equation*}
\noindent
(ii). It holds with high probability by Lemma \ref{lem: local law for BB^T}.
\noindent
(iii). See Remark \ref{081821} below. Therefore, we conclude the proof. 
\end{proof}

\begin{remark}\label{081821}Rigorously speaking, in order to have the above proposition, we shall work with $X^{\mathcal{C}}$ instead of $X$. Since these two matrices are identical with probability $1-\mathfrak{o}(1)$, any spectral statistics of these two matrices are identical with probability $1-\mathfrak{o}(1)$. For our main theorem, it would be sufficient to work with $X^{\mathcal{C}}$ instead of $X$ in the sequel. However, for convenience, we will still work with $X$ as if the above proposition is also true for $\mathcal{S}(X)$.   In this case, the reader may simply assume that the entries of $X$ are bounded by $N^{100}$ (say). We can anyway recover the result without this additional boundedness assumption by comparing the matrix with its truncated version. 
\end{remark}

\vspace{1ex}
Let $\lambda_{-,t}$ be the left edge of $\rho_t$. The Gaussian part in model (\ref{def:GDM}) can further improve the scale of the square root behavior of $\rho_t$ around $\lambda_{-,t}$ on the event that  $\mathcal{S}(X)$ satisfies certain $\eta_*$- regularity. The following theorem makes this precise.
\begin{theorem}[Lemma 1 of \cite{DY2}]On $\Omega_\Psi$,  we have
 	\begin{align*}
 		\rho_t\sim \sqrt{(E -\lambda_{-,t})_+} \quad \text{for} \quad \lambda_{-,t} - \frac{3}{4}\tilde{c} \le  E  \le \lambda_{-,t} + \frac{3}{4}\tilde{c},
 	\end{align*}
 	and for $z = E + \mathrm{i}\eta \in \mathbb{C}^+$,
 	\begin{equation}\label{eq:Immc}
		\Im m_{t}(z) \sim
		\begin{dcases}
			\sqrt{|E-\lambda_{-,t}|+\eta}, & \lambda_{-,t}  \le  E\leq \lambda_{-,t} + \frac{3}{4}\tilde{c}\\
			\frac{\eta}{\sqrt{|E-\lambda_{-,t}|+\eta}}, &\lambda_{-,t} - \frac{3}{4}\tilde{c}\le  E  \le \lambda_{-,t} 
		\end{dcases}.
	\end{equation} 
\label{thm:srbofmxt}\end{theorem}

Next, we recall the definition in (\ref{eq:domainD}). The following theorem provide bounds on the Green function entries for the Gaussian divisible model. Further recall the notation in (\ref{0818100}), we set $\mathcal{T}_r\coloneqq[M]\setminus \mathcal{D}_r$, $\mathcal{T}_c\coloneqq[N]\setminus \mathcal{D}_c$.

\begin{theorem}\label{thm: resolvent entry size V_t}
Suppose that $\Psi$ is good. Let $z\in \mathsf{D}(\varepsilon_1, \varepsilon_2, \varepsilon_3)$ with $10\epsilon_a \le  \varepsilon_1 \le \epsilon_b/500$ and sufficiently small $\varepsilon_2, \varepsilon_3$. The following estimates hold w.r.t.  the probability measure $\mathbb{P}_{\Psi}$.
\begin{itemize}
	\item[(i)] $$|G_{ij}(V_{t}, z)| \prec \mathbf{1}_{i \in \mathcal{T}_{r} \text{ or } j\in\mathcal{T}_{r}} + t^{-2}(1 -  \mathbf{1}_{i \in \mathcal{T}_{r} \text{ or } j\in\mathcal{T}_{r}}), $$
	\item[(ii)] $$|G_{uv}(V_{t}^\top, z)| \prec \mathbf{1}_{u \in \mathcal{T}_{c} \text{ or } v\in\mathcal{T}_{c}} + t^{-2}(1 -  \mathbf{1}_{u \in \mathcal{T}_{c} \text{ or } v\in\mathcal{T}_{c}}), $$
	\item[(iii)] $$|[G(V_{t}, z)V_{t}]_{iu}| \prec N^{-\epsilon_{b}/2}\mathbf{1}_{i \in \mathcal{T}_{r} \text{ or } u\in\mathcal{T}_{c}} + t^{-2}(1 -  \mathbf{1}_{i \in \mathcal{T}_{r} \text{ or } u\in\mathcal{T}_{c}}).$$
\end{itemize}
\end{theorem}
The proof of Theorem \ref{thm: resolvent entry size V_t} is based on the following results.
\begin{lemma}\label{lem: domain of parameter}
Suppose that the assumptions in Theorem \ref{thm: resolvent entry size V_t} hold.
There exist constants $c,C>0$ such that for the domain $\mathsf{D}_\zeta= \mathsf{D}_{\zeta}(c,C)\subset\mathbb{C}_{+}$ defined by
\begin{equation}\label{eq: domain of zeta}
	\mathsf{D}_\zeta\coloneqq \mathsf{D}_{1}\cup\mathsf{D}_{2},
\end{equation}
where
\begin{align*}
	\mathsf{D}_{1} &\coloneqq \{ \zeta=E+i\eta: E \le (1-t)\lambda_{-}^{\mathsf{mp}} - ct^{2}, \eta\ge c t N^{-2/3-\varepsilon_{2}} \} \\
	\mathsf{D}_{2} &\coloneqq \{ \zeta=E+i\eta: \eta\ge c (\log N)^{-C} t^{2} \}
\end{align*}
we have $\zeta_{t}(z) \in \mathsf{D}_\zeta$ with high probability.
\end{lemma}
\begin{proof}
	The proof relies on the definition of $\zeta_t(z)$ as well as the square root behaviour of $\rho_t$ as stated in Theorem \ref{thm:srbofmxt};  see Appendix \ref{2416} for the detailed proof. 
\end{proof}

\begin{proposition}\label{prop: resolvent entry size for X=B+C}
Let $\mathsf{D}_\zeta$ be as in \eqref{eq: domain of zeta}. Consider $\zeta\in\mathsf{D}_\zeta$.
Suppose that $\Psi$ is good. The following estimates hold w.r.t.  the probability measure $\prob_{\Psi}$. There exists a constant $c=c(\eps_a,\eps_{\alpha},\eps_{b})>0$ such that
\begin{align*}
	&|G_{ij}(X, \zeta)-\delta_{ij}\mathsf{m}_{\mathsf{mp}}^{(t)}(\zeta)| \prec N^{-c}\mathbf{1}_{i,j \in \mathcal{T}_r} + t^{-2}(1-\mathbf{1}_{i,j \in \mathcal{T}_r}), \\
	&|G_{uv}(X^\top, \zeta)-\delta_{uv}\underline{\mathsf{m}}_{\mathsf{mp}}^{(t)}(\zeta)| \prec N^{-c}\mathbf{1}_{u,v \in \mathcal{T}_c} + t^{-2}(1-\mathbf{1}_{u,v \in \mathcal{T}_c}),
\end{align*}
where $\underline{\mathsf{m}}_{\mathsf{mp}}^{(t)}(\zeta)=c_{N}\mathsf{m}_{\mathsf{mp}}^{(t)}(\zeta)-(1-c_{N})/{\zeta}$.
\end{proposition}
\begin{proof}
	The proof of Proposition \ref{prop: resolvent entry size for X=B+C} is similar to the light-tailed case proved in \cite{PY}, but here we shall apply large deviation formula for heavy-tailed random variables; see Appendix \ref{pf: resolvent entry size for X=B+C} for the detailed proof. 
\end{proof}

\begin{proof}[Proof of Theorem \ref{thm: resolvent entry size V_t}] 
Given the previous results, the proof strategy for this theorem is briefly introduced in the last paragraph of the Introduction, Section \ref{sec1}, with the detailed proof found in Appendix \ref{sec: resolvent estimates V_t}.
\end{proof}

The above theorems provide strong evidence supporting the validity of the Tracy-Widom law for  $\lambda_{M}(\mathcal{S}(V_t))$ around $\lambda_{-,t}$. In fact, we are able to establish the following theorem regarding the convergence of the distribution. Before stating the result, we define the function
\begin{align}\label{eq:Phi}
	\Phi_t(\zeta) \coloneqq (1-c_Nt m_X(\zeta))^2\zeta + (1-c_N)t(1-c_Ntm_X(\zeta)),
\end{align} 
and the scaling parameter
	\begin{align}\label{eq:gamma}
		\gamma_N \coloneqq \gamma_N(t)  \coloneqq -\Big( \frac{1}{2}\big[4\lambda_{-,t}\zeta_t(\lambda_{-,t}) + (1-c_N)^2t^2 \big] c_N^2t^2\Phi_t''(\zeta_t(\lambda_{-,t})) \Big)^{-1/3}.
	\end{align}
\begin{theorem} \label{thm:081601}
	 Let $f:\mathbb{R}\to \mathbb{R}$ be a test function satisfying $\|f \|_\infty \le C$ and $\|\nabla f \|_\infty \le C$
	for a constant $C$. Then we have for any $X$ whose corresponding $\Psi$ is good,
	\begin{align}
	\lim_{N\to \infty} \E \big[ f\big(\gamma_N M^{2/3}(\lambda_M(\mathcal{S}(V_t)) - \lambda_{-,t} )\big)|X \big] = \lim_{N\to \infty} \E  \big[ f\big(M^{2/3}(\mu_M^{\mathsf{GOE}} + 2 )\big) \big].  \label{092401}
	\end{align}
	This further implies that if $\Psi$ is good, 
	\begin{align}
		\lim_{N\to \infty}\E_\Psi \big[ f\big(\gamma_N M^{2/3}(\lambda_M(\mathcal{S}(V_t)) - \lambda_{-,t} )\big) \big] = \lim_{N\to \infty}\E  \big[ f\big(M^{2/3}(\mu_M^{\mathsf{GOE}} + 2 )\big) \big],
	\end{align}
where $\mu_M^{\mathsf{GOE}}$ denotes the least eigenvalue of a $M$ by $M$ Gaussian Orthogonal Ensemble (GOE) with $N(0, M^{-1})$ off-diagonal entries. 
\end{theorem}
\begin{remark}
The proof of the above theorem is essentially an adapt of the edge universality for the DBM in \cite{LY} and the analogue for the rectangle DBM in \cite{DY2, DY}. More specifically, we shall extend the analysis in  \cite{DY2, DY} from the right edge of the covariance type matrix to the left edge. Based on the $\eta_*$-regularity, the proof is nearly the same as \cite{DY2, DY}, and thus we do not reproduce the details and only provide some remarks in the Appendix \ref{sec:081601}. 
\end{remark}

\subsection{Distribution of $\lambda_{-,t}$}

\begin{theorem}\label{thm: lambda -,t asymp}
There exists a deterministic quantity $\lambda_{\mathsf{shift}}>0$ depending on $N$ such that the following two properties hold.
\begin{itemize}
	\item[(i)] 
	\begin{align*}
		\frac{N^{\alpha/4}(\lambda_{-,t} - \lambda_{\mathsf{shift}} )}{\sigma_\alpha } \Rightarrow \mathcal{N}(0,1),\quad \sigma_\alpha^2 = \frac{\mathsf{c}c_N^{(4-\alpha)/4}(1-\sqrt{c_N})^4(\alpha-2)}{2} \Gamma\Big(\frac{\alpha}{2} +1 \Big).
	\end{align*}
	\item[(ii)]
	\begin{equation*}
		 \lambda_{\mathsf{shift}} = \lambda_{-}^{\mathsf{mp}} - 
		 \frac{\mathsf{c}N^{1-\alpha/2}(1-\sqrt{c_N})^2}{c_N^{(\alpha-2)/4}} \Gamma\Big(\frac{\alpha}{2} +1\Big)+
		 \mathfrak{o}(N^{1-\alpha/2}).
	\end{equation*}
\end{itemize}
\end{theorem}
\begin{remark}
Note that the leading order of $\lambda_{\mathsf{shift}}$ only depends on $\alpha$. The size of the fluctuation of $\lambda_{-,t}$ is also determined by $\alpha$.
\end{remark}

\vspace{5mm}

The proof of Theorem \ref{thm: lambda -,t asymp} is given in the next section.

\vspace{5mm}



\section{Proofs for Gaussian divisible model} \label{s.pgdm}
\subsection{Preliminary estimates}
 Before providing the preliminary estimates for the expansion of the least eigenvalue of $\mathcal{S}(V_t)$, we first state the following lemma, which characterizes the support of $\rho_t$ and its edges using the local extrema of $\Phi_t(\zeta)$ on $\mathbb{R}$.
\begin{lemma}[Proposition 3 of \cite{VPL2012}]\label{lem:Phicharacterization} 
Fix any $t>0$. The function $\Phi_t(x)$ on $\mathbb{R} \setminus \{0\}$ admits $2q$ positive local extrema counting multiplicities for some integer $q \geq 1$. The preminages of these extrema are denoted by $ 0< \zeta_{1,-}(t)\le \zeta_{1,+}(t) \leq \zeta_{2,-}(t) \leq \zeta_{2,+}(t) \leq \cdots \leq \zeta_{q,-}(t) \leq \zeta_{q,+}(t),$  and they belong to  the set $\{\zeta \in \mathbb{R}: 1-c_Ntm_X(\zeta_t)>0 \}.$ Moreover, $\lambda_{-,t}=\Phi_t(\zeta_{1,-}(t))$,  and $\zeta_{1,-}(t) < \lambda_M(\mathcal{S}(X)) < \zeta_{1,+}(t)$. 
\end{lemma}

\vspace{1ex}
\begin{remark}Here we remark that the model considered in  \cite{VPL2012} is slightly different in the sense that the model therein contains many $0$ eigenvalues, which will force $\zeta_{1,-}(t)$ to be negative. In our case, going through the same analysis as \cite{VPL2012}  will simply give $0< \zeta_{1,-}(t)$. 
\end{remark}

\vspace{1ex}
Next, we shall introduce the deterministic counterpart of $\zeta_{-,t}$ (to be denoted by $\bar{\zeta}_{-,t}$). First, we notice that the MP law holds for both the matrix $V_t$ and $X$, but with slightly different scaling factors. Specifically, we have
$
	m_{V_t}(z) -\mathsf{m}_{\mathsf{mp}}(z) = \mathfrak{o}_p(1)$ and $ m_X(z) -\mathsf{m}_{\mathsf{mp}}^{(t)}(z) = \mathfrak{o}_p(1).
$
Recall the definitions of $\zeta_t(z)$ in (\ref{eq:zeta}) and $\Phi_t(\zeta)$ in (\ref{eq:Phi}). It is important to note that these two quantities are random, and we can also define their deterministic counterparts using the Stieltjes transform of the MP Law. We denote them as follows:
\begin{align}
&\bar{\zeta}_t(z) \coloneqq (1+c_Nt\mathsf{m}_{\mathsf{mp}}(z))^2z-t(1-c_N)(1+c_Nt\mathsf{m}_{\mathsf{mp}}(z)),\label{eq:barzetat} \\
&\bar{\Phi}_t(\zeta) \coloneqq (1-c_Nt \mathsf{m}_{\mathsf{mp}}^{(t)}(\zeta))^2\zeta + (1-c_N)t(1-c_Nt\mathsf{m}_{\mathsf{mp}}^{(t)}(\zeta)) \label{eq:barPHIt}.
\end{align}
To further simplify the notation, we let $\zeta_{-,t} = \zeta_t(\lambda_{-,t})$ and $\bar{\zeta}_{-,t} = \bar{\zeta}_t(\lambda^{\mathsf{mp}}_{-})$. Let $\beta = (\alpha-2)/24$.
\begin{lemma}\label{lem:preEst}
The following preliminary estimates hold:
\begin{itemize}
	\item[(i)] $\zeta_{-,t} - \lambda_M(\mathcal{S}(X)) \le 0$, and $\lambda_M(\mathcal{S}(X)) - \zeta_{-,t}  \sim t^2$ holds on $\Omega_\Psi$.
	\item[(ii)] There exist some sufficiently small constant $\tau > 0$, such that for any $z \in \mathbb{C}^+$ satisfying $|z - \zeta_{-,t}|\le \tau t^2$,  we have on $\Omega_\Psi$ that
	\begin{align*}
		m_X(z) - \mathsf{m}_{\mathsf{mp}}^{(t)}(z) \prec N^{- \beta}, \quad |m_X^{(k)}(\zeta)| \lesssim t^{-2k+1}, \quad m_X^{(k)}(\zeta_{-,t}) \sim t^{-2k+1}, \quad k \ge 1.
	\end{align*}
	\item[(iii)] $\bar{\zeta}_{-,t} - \zeta_{-,t} \prec N^{-\beta}t$.
\end{itemize}
\end{lemma}
\begin{proof}
	See the Appendix \ref{3793}.
\end{proof}

We also compute the following limits.
\begin{lemma}\label{lem:mmpcalcualtion}For any $t = \mathfrak{o}(1)$, we have the following approximations:
\begin{itemize}
	\item[(i)] $
	\mathsf{m}_{\mathsf{mp}}^{(t)}(\bar{\zeta}_{-,t}) = (\sqrt{c_N} - c_N)^{-1} - t c_N^{-1/2}(1-\sqrt{c_N})^{-2}  + \mathcal{O}(t^{3/2})$.
	\item[(ii)] $t (\mathsf{m}_{\mathsf{mp}}^{(t)}(\bar{\zeta}_{-,t}))' = c_N^{-1}(1-\sqrt{c_N})^{-2}/2 + \mathcal{O}(t^{1/2})$.
	\item[(iii)] $t^3 (\mathsf{m}_{\mathsf{mp}}^{(t)}(\bar{\zeta}_{-,t}))'' =c_N^{-3/2}(1-\sqrt{c_N})^{-2} /4 + \mathcal{O}(t^{1/2})$.
	\item[(iv)] $\gamma_N - c_N^{-1/2}(1-\sqrt{c_N} )^{-4/3}  = \mathfrak{o}_p(1)$.
\end{itemize}
\end{lemma}
\begin{proof}
	It is easy to solve
	$
		\bar{\zeta}_{-,t} = (1 - t)\lambda_-^{\mathsf{mp}} - \sqrt{c_N}t^2
	$ from (\ref{eq:barzetat}) and (\ref{081420}). 
	The calculation is then elementary by the explicit formula (\ref{081420}). 
\end{proof}

\subsection{Proof of Theorem \ref{thm: lambda -,t asymp}}
Before giving the proof, we need the following pre-process. First, note that we have the following deterministic upper bound when $\Psi$ is good:
\begin{align*}
\zeta_{-,t}\cdot \mathbf{1}_{\Omega_{\Psi}} \le \lambda_M(\mathcal{S}(X))\cdot \mathbf{1}_{ \Omega_{\Psi}} \le \lambda_{M - |\mathcal{D}_r|}(\mathcal{S}(\mathsf{B}^{(\mathcal{D}r)}))\cdot \mathbf{1}_{\Omega_{\Psi}}\le N^{2-2\epsilon_b}.
\end{align*}
This indicates that $\E ( \zeta_{-,t}\cdot \mathbf{1}_{ \Omega_{\Psi}})$ is well-defined. We define
\begin{equation}\label{081433}
	\zeta_{\mathsf{e}} \coloneqq \E ( \zeta_{-,t}\cdot \mathbf{1}_{ \Omega_{\Psi}}), \quad
    \Delta_\zeta \coloneqq \zeta_{-,t} - \zeta_{\mathsf{e}}.
\end{equation}
We also write for $z \in \mathbb{C}^+$ and an integer $k \ge 0$, $\Delta_m(z) \coloneqq m_X(z) - \E m_X(z)$ and $\Delta^{(k)}_m(z) \coloneqq m^{(k)}_X(z) - \E m^{(k)}_X(z)$,
where we remark that $\Delta_m(z)= \Delta^{(0)}_m(z)$. It is noteworthy that $\E m_X(z)$ is well-defined when $z$ possesses a non-zero imaginary part. To ensure that the expectation of $m_X(\zeta_{\mathsf{e}})$ exist, we add a small imaginary part to $\zeta_{\mathsf{e}}$,
and define for any $K_\zeta > 0$, $\hat{\zeta}_{\mathsf{e}} = \hat{\zeta}_{\mathsf{e}}(K_\zeta)  \coloneqq \zeta_{\mathsf{e}} + \mathrm{i}N^{-100K_\zeta}.$

We will begin by stating some preliminary bounds useful to estimate $\E\lambda_{-,t}$.
\begin{lemma} Recall that $\beta = (\alpha-2)/24$.
	There exists some small $\tau>0$, such that for any $z \in \mathbb{C}^+$ satisfies $|z - \zeta_{\mathsf{e}}|\le \tau t^2$ and $\Im z\geq N^{-100K_\zeta}$, the following a priori high probability bounds:
	\begin{align}
		\Delta^{(k)}_m(z) \prec N^{-\beta}t^{-2k} ,\quad \text{and} \quad \Delta_\zeta \prec N^{-\beta/2}t^{2}  \label{081441}
	\end{align}
	Furthermore, we have the following a priori variance bounds:
	\begin{align}
		\mathrm{Var}(\Delta^{(k)}_m(z)  ) \le N^{-1+\epsilon}t^{-2k-4}, \quad \text{and} \quad \mathrm{Var}(\Delta_\zeta \mathbf{1}_{\Omega_\Psi}  ) \le N^{-1+\epsilon}.
	\end{align}
\label{lem:hpbandvb}\end{lemma}
We postpone the proof of Lemma \ref{lem:hpbandvb} to the end of this subsection. Let us prove Theorem \ref{thm: lambda -,t asymp} equipped with Lemma \ref{lem:hpbandvb}.

\begin{proof}[Proof of Theorem \ref{thm: lambda -,t asymp}]
Recall the expression of $\lambda_{-,t}$ in (\ref{eq:defoflambdaminust}).
We shall switch $\zeta_{-,t}$ and $m_{X}(\zeta_{-,t})$ with $\zeta_{\mathsf{e}}$ and $\E m_{X}(\hat{\zeta}_{\mathsf{e}})$ respectively. First, expanding $m_{X}(\zeta_{-,t})$ around $m_{X}(\zeta_{\mathsf{e}})$, we have for sufficiently large $s > 0$,
\begin{align*}
	\lambda_{-,t} =  \zeta_{-,t} \Big(1  - \sum_{k=0}^s \frac{c_Nt}{k!} m^{(k)}_{X}(\zeta_{\mathsf{e}}) \Delta_\zeta^k \Big)^{2}+ (1-c_{N}) t \Big(1 - \sum_{k=0}^s \frac{c_Nt}{k!} m^{(k)}_{X}(\zeta_{\mathsf{e}}) \Delta_\zeta^k  \Big) + \mathcal{O}_\prec(N^{-\alpha/4-\epsilon}).
\end{align*}
Note that for any integer $k \ge 0$, it can be easily verified that w.h.p.,
$
	|m^{(k)}_X(\zeta_{\mathsf{e}}) - m^{(k)}_X(\hat{\zeta}_{\mathsf{e}})|  \le N^{-50s},
$
by chooinsg $K_{\zeta}>0$ large enough.
This means that we can replace $m^{(k)}_X(\zeta_{\mathsf{e}})$ with $m^{(k)}_X(\hat{\zeta}_{\mathsf{e}})$.
Through an elementary calculation, we have
\begin{multline*}
	\lambda_{-,t} = \lambda_\mathsf{shift} - \paren{ 2c_{N}t \big( 1 - c_{N}t\E m_X(\hat{\zeta}_{\mathsf{e}}) \big) \zeta_{\mathsf{e}} - c_{N}t^{2}(1-c_{N}) }\Delta_{m}(\hat{\zeta}_{\mathsf{e}})
	+ \mathsf{ZOT}_\zeta   \Delta_{\zeta} + \mathcal{P}(\Delta_\zeta,\{\Delta^{(k)}_m(\hat{\zeta}_{\mathsf{e}})\}_{k\ge 0}).
\end{multline*}
where 
$
	\lambda_\mathsf{shift} \coloneqq \big(1-c_Nt \E m_X(\hat{\zeta}_{\mathsf{e}}) \big)^2\zeta_{\mathsf{e}} +(1-c_N)t \big(1-c_Nt \E  m_X(\hat{\zeta}_{\mathsf{e}}) \big)
$, and we denote by $\mathsf{ZOT}_\zeta$ the collection of zero-th order terms, i.e.,
\begin{align}\label{eq: def of ZOT}
	\mathsf{ZOT}_\zeta \coloneqq \big(1-c_N t\E m_X(\hat{\zeta}_{\mathsf{e}})\big)\big(1 - c_Nt\E m_X(\hat{\zeta}_{\mathsf{e}}) -2c_Nt \zeta_{\mathsf{e}} \E m_X'(\hat{\zeta}_{\mathsf{e}})     \big)- c_N(1-c_N)t^2\E m_{X}'(\hat{\zeta}_{\mathsf{e}}),
\end{align}
and $\mathcal{P}(\Delta_\zeta,\{\Delta^{(k)}_m(\hat{\zeta}_{\mathsf{e}})\}_{k\ge 1})$ collects all the high order terms.
We need to bound the last two terms. It can be easily obtained by prior bounds in Lemma \ref{lem:hpbandvb} that $\mathcal{P}(\Delta_\zeta,\{\Delta^{(k)}_m(\hat{\zeta}_{\mathsf{e}})\}_{k\ge 0}) = \mathcal{O}_p( N^{-\alpha/4-(4-\alpha)/8})$. 
Moreover, due to Remark \ref{rmk: bound ZOT} below, we find that $\mathsf{ZOT}_{\zeta}=\mathcal{O}( N^{-\alpha/4-(4-\alpha)/8})$.

The following two propositions complete the proof.

\begin{proposition}\label{prop:lambdaminustdis}
Let $\sigma_{\alpha}$ be as in Theorem \ref{thm: lambda -,t asymp}. We have
\begin{equation*}
	 2c_{N} \big( 1 - c_{N}t\E m_X(\hat{\zeta}_{\mathsf{e}}) \big) \zeta_{\mathsf{e}} \cdot \paren{ \frac{ t \Delta_{m}(\hat{\zeta}_{\mathsf{e}}) }{\sigma_\alpha } } \Rightarrow \mathcal{N}(0,1).
\end{equation*}
\end{proposition}

\begin{proposition}\label{prop:lambdashiftexpans}
We have
\begin{align*}
		\lambda_{\mathsf{shift}} = \lambda_{-}^{\mathsf{mp}} - 
		 \frac{\mathsf{c}N^{1-\alpha/2}(1-\sqrt{c_N})^2}{c_N^{(\alpha-2)/4}} \Gamma\Big(\frac{\alpha}{2} +1\Big) + \mathfrak{o}(N^{1-\alpha/2}).
\end{align*}
\end{proposition}

\noindent
We shall prove the above propositions in the next subsections.
\end{proof}

\begin{proof}[Proof of Lemma \ref{lem:hpbandvb}]
Using Lemma \ref{lem:preEst} (ii), we can obtain that
$
		\Delta_m(z) = m_X(z) - \mathsf{m}_{\mathsf{mp}}^{(t)}(z) + \E (\mathsf{m}_{\mathsf{mp}}^{(t)}(z) - m_X(z))  \prec N^{-\beta}.
$
The bound for $\Delta^{(k)}_m(z)$ follows by a simple application of Cauchy integral formula.
	
In order to bound $\Delta_{\zeta}$, we first observe that
\begin{align}\label{eq:zetaminusbarzeta}
		\zeta_{\mathsf{e}} - \bar{\zeta}_{-,t} = \E \big[ (\zeta_{-,t} - \bar{\zeta}_{-,t} ) \cdot \mathbf{1}_{\Omega_\Psi} \big]  - \bar{\zeta}_{-,t} \cdot \mathbb{P}(\Omega_\Psi^c) \le N^{-\beta/2}t^{2}.
\end{align}
where the last step follows from Lemmas \ref{lem:goodpsi} and  \ref{lem:preEst} (iii). Therefore, by Lemma \ref{lem:preEst} (iii) again, we can get the desired bound for $\Delta_{\zeta}$.

Next we consider $\mathrm{Var}(\Delta_m(z))$. We first let $\mathcal{F}_k$ be the $\sigma$-field generated by the first $k$ columns of $X$. Then we define $D_k^+ := \E \big[M^{-1} (\Tr G(X,z) - \Tr G(X^{(k)},z) )\big| \mathcal{F}_{k} \big]$, $D_k^- := \E \big[M^{-1} (\Tr G(X^{(k)},z) -\Tr G(X,z) )\big| \mathcal{F}_{k-1} \big]$, and $D_k \coloneqq D_k^+ + D_k^{-}$.
By the Efron-Stein inequality, we have
\begin{align*}
	\mathrm{Var}(m_X(z))= \sum_{i = 1}^N \E (|D_i|^2) \le 2 \sum_{i = 1}^N \E \big(|D_i^+|^2\big) +  \E \big(|D_i^-|^2\big) .
\end{align*}
Using the resolvent expansion, we can obtain
\begin{equation*}
	\E \big( |D_k^+|^2\big)
	 {\le} \frac{1}{M^2}\E  \Big[ \Big|\frac{x_k^\top  G^2(X^{(k)},z) x_k}{1 + x_k^\top  G(X^{(k)},z) x_k} \Big|^2  \cdot \mathbf{1}_{|z - \lambda_M(\mathcal{S}(X^{(k)}) )| \ge ct^2}\Big] + N^{-D}
	\lesssim \frac{N^{\epsilon}}{N^2t^4},
\end{equation*}
where in the first step, we used Lemma \ref{lem:preEst} (i)  to derive, with high probability,  that  for $|z - \zeta_{\mathsf{e}}| \le \tau t^2$ with sufficiently small $\tau > 0$, there exists some sufficiently small $c > 0$,
\begin{align}
	|z - \lambda_M(\mathcal{S}(X^{(k)}) )| &\ge |\zeta_{-,t} - \lambda_M(\mathcal{S}(X)) | - |z- \zeta_{\mathsf{e}}| - |\Delta_\zeta| \notag\\
	&\quad- |\lambda_M(\mathcal{S}(X))-(1-t)\lambda_{-}^{\mathsf{mp}} |-|\lambda_M(\mathcal{S}(X^{(k)}))-(1-t)\lambda_{-}^{\mathsf{mp}} | \ge ct^2,
\label{eq:zmlsxk}\end{align}
which gives $\mathbb{P}( |z - \lambda_M(\mathcal{S}(X^{(k)}) )|  \ge ct^2) < N^{-D}$ for arbitrary large $D > 0$, and $z$ has non-zero imaginary part which yields deterministic upper bound for the random variable.
Similarly, we have $\E \big( |D_k^-|^2\big) \lesssim N^{\epsilon} /(N^2t^4)$. This establishes the bound for $\mathrm{Var}(\Delta_m(z))$.

The bound for $\mathrm{Var}(\Delta^{(k)}_m(z))$ follows by an application of Cauchy integral formula. Note that, since the contour of the Cauchy integral will cross real line, the integrand may not be well defined deterministically due to the possible singularity (although with tiny probability) of the Green function. Hence, we will need to cut off the part of the integral when the imaginary part of the variable is small. To elucidate the procedure, we will outline how to do the cutoff for the Cauchy integral representation of $\E (m_X^{(k)}(z))$ only. The one for variance can be done similarly.  
Consider $z$ that satisfies $|z - \zeta_{\mathsf{e}}|\le \tau t^2/2$ and $\Im z\geq N^{-100K_\zeta}$, we first define $\Omega_z \coloneqq \{ |z - \lambda_M(\mathcal{S}(X))| \ge ct^2 \}$. A similar argument as (\ref{eq:zmlsxk}) leads to $\mathbb{P}(\Omega_z^c) \le N^{-D}$ for arbitrary large $D > 0$. Then we may choose a contour $\omega_z \coloneqq \{z': |z' - z| = \tau t^2/10  \}$ with sufficiently small $\tau$, and set $\mathfrak{w} \coloneqq \{z': |\Im z'| \ge N^{-100K_\zeta} \}$. Then we obtain 
\begin{align}
	&\E (m_X^{(k)}(z)) = \E (m_X^{(k)}(z)\cdot \mathbf{1}_{\Omega_z} ) +  \E (m_X^{(k)}(z)\cdot \mathbf{1}_{\Omega_z^c} )= \frac{k!}{2\pi \mathrm{i}}\E\Big[\oint_{\omega} \frac{m_X(a)}{(a - z)^{k+1}} \mathrm{d}a\cdot \mathbf{1}_{\Omega_z}\Big] + N^{-D} \notag\\
	&= \frac{k!}{2\pi \mathrm{i}}\Big( \E\Big[\oint_{\omega\cap \mathfrak{w} }\frac{m_X(a)}{(a - z)^{k+1}} \mathrm{d}a\cdot \mathbf{1}_{\Omega_z}\Big] +\E\Big[ \oint_{\omega\cap \mathfrak{w}^c } \frac{m_X(a)}{(a - z)^{k+1}} \mathrm{d}a \cdot \mathbf{1}_{\Omega_z}\Big] \Big)+ N^{-D} \notag\\
	&=\frac{k!}{2\pi \mathrm{i}}\Big( \E\Big[\oint_{\omega\cap \mathfrak{w} }\frac{m_X(a)}{(a - z)^{k+1}} \mathrm{d}a\Big] +\E\Big[ \oint_{\omega\cap \mathfrak{w}^c } \frac{m_X(a)}{(a - z)^{k+1}} \mathrm{d}a \cdot \mathbf{1}_{\Omega_z}\Big] \Big)+ N^{-D} \notag\\
	&=  \frac{k!}{2\pi \mathrm{i}}\E\Big[\oint_{\omega\cap \mathfrak{w} }\frac{m_X(a)}{(a - z)^{k+1}} \mathrm{d}a\Big]  +  \mathcal{O}(N^{-50K_\zeta})+ N^{-D}. \label{081501}
\end{align}
For the remaining term, the effective imaginary part of $a$ within $\omega \cap \mathfrak{w}$ allows us to interchange $\E$ with the contour integral. Then, the upper bound for $\E (m_X(a))$ can be directly applied to estimate this term. 
Using the same cutoff of the contours, the bound for $\mathrm{Var}(\Delta^{(k)}_m(z))$ is obtained through a double integral representation together with the Cauchy-Schwarz inequality. We omit further details for brevity.

Lastly, we shall bound $\mathrm{Var}(\Delta_\zeta)$.
Since $(\lambda_M(\mathcal{S}(X)) - \zeta_{-,t})\cdot \mathbf{1}_{\Omega_\Psi} \sim t^2\cdot \mathbf{1}_{\Omega_\Psi} $ and $\Delta_\zeta \prec N^{-\beta/2}t^2$, on the event $\Omega_\Psi$,
$\lambda_M(\mathcal{S}(X)) - \zeta_{\mathsf{e}} =  \lambda_M(\mathcal{S}(X)) - \zeta_{-,t} +\Delta_\zeta  \sim t^2$ with high probability.
Using Lemma \ref{lem:esthighoderi}, the bound in the above display also implies that on the event $\Omega_\Psi$,
	\begin{align}
	m^{(k)}_{X}(\zeta_{\mathsf{e}}) \sim t^{-2k+1},\quad k \ge 1.
\label{eq:mkbound}\end{align}
Recall that $\Phi_t'(\zeta_{-,t}) = 0$, which reads
\begin{align}
(1-c_Nt m_{X}(\zeta_{-,t}))^2 - 2 c_Nt m_{X}'(\zeta_{-,t}) \cdot \zeta_{-,t} \left( 1-c_Nt m_{X}(\zeta_{-,t})\right) - c_N(1-c_N)t^2m_{X}'(\zeta_{-,t})=0.
\label{eq:zetaequation}\end{align}
Replacing $\zeta_{-,t}$ and $m_{X}(\zeta_{-,t})$ with $\zeta_{\mathsf{e}}$ and $\E[m_{X}(\hat{\zeta}_{\mathsf{e}})]$, as in the proof of Theorem \ref{thm: lambda -,t asymp}, it follows from \eqref{eq:zetaequation} that
\begin{align}
	\mathsf{ZOT}_\zeta + \mathsf{FOT}_\zeta + \mathcal{P}_\zeta(\Delta_\zeta,\{\Delta^{(k)}_m\}_{k\ge 0}) = 0,
\label{eq:zotfot}\end{align}
where the term $\mathsf{ZOT}_\zeta$ is defined as in \eqref{eq: def of ZOT},
\begin{align*}
&\mathsf{FOT}_\zeta  \coloneqq\big(   2c_N^2t^2\zeta_{\mathsf{e}} \E m'_X(\hat{\zeta}_{\mathsf{e}})  -2\mathsf{f}_m  \big) \Delta_m(\hat{\zeta}_{\mathsf{e}}) -\big( c_N(1-c_N)t^2 +  2\mathsf{f}_m \zeta_{\mathsf{e}} \big) \Delta_m^{(1)}(\hat{\zeta}_{\mathsf{e}})\\
	&\quad\quad\quad  -\big( 4\mathsf{f}_m\E m'_X(\hat{\zeta}_{\mathsf{e}})  + c_N(1-c_N)t^2\E m^{(2)}_{X}(\hat{\zeta}_{\mathsf{e}}) + 2c_N^2t^2\zeta_{\mathsf{e}}(\E m'_X(\hat{\zeta}_{\mathsf{e}}) )^2 +2\mathsf{f}_m\zeta_{\mathsf{e}} \E m^{(2)}_{X}(\hat{\zeta}_{\mathsf{e}})  \big) \Delta_\zeta
\end{align*}
with $\mathsf{f}_m \coloneqq c_Nt\big(1-c_Nt\E m_X(\hat{\zeta}_{\mathsf{e}}) \big)$, and $\mathcal{P}_\zeta(\Delta_\zeta,\Delta^{(k)}_m)$ is the collection of high order terms. Note that $\mathsf{f}_m \sim t$ and $\mathcal{P}_\zeta(\Delta_\zeta,\Delta^{(k)}_m)$ is a polynomial in $\Delta_\zeta$ and $\Delta_m^{(k)}$'s, containing monomials of order no smaller than $2$.

Hence, by Cauchy Schwarz and bounds in (\ref{081441}), one can get the following bounds
\begin{align}
	&\mathrm{Var}\Big(\mathcal{P}(\Delta_\zeta,\Delta^{(k)}_m(\hat{\zeta}_{\mathsf{e}})) \mathbf{1}_{\Omega_\Psi} \Big)
	\lesssim N^{-1/2-\beta/4} \sqrt{\mathrm{Var}(\Delta_\zeta \mathbf{1}_{\Omega_\Psi} )} + N^{-\beta/4}\mathrm{Var}(\Delta_\zeta \mathbf{1}_{\Omega_\Psi} ) +N^{-D}, \label{eq:varm} \\
	&\E \Big(\mathcal{P}(\Delta_\zeta,\Delta^{(k)}_m(\hat{\zeta}_{\mathsf{e}})) \mathbf{1}_{\Omega_\Psi} \Big) 
	 \lesssim  N^{-1/2+\epsilon/2}t^{-3} \sqrt{\mathrm{Var}(\Delta_\zeta \mathbf{1}_{\Omega_\Psi} )}+ t^{-2}\mathrm{Var}(\Delta_\zeta \mathbf{1}_{\Omega_\Psi} )+N^{-D}.\label{eq:expem}
\end{align}
Using (\ref{eq:mkbound}), we can see that the leading order term of the coefficient of $\Delta_\zeta$ in $\mathsf{FOT}_\zeta$  is 
$-2\mathsf{f}_m\zeta_{\mathsf{e}} \E(m^{(2)}_{X}(\hat{\zeta}_{\mathsf{e}}))  \sim t^{-2}.$
Therefore, we can derive from \eqref{eq:zotfot} that
\begin{align}
	C_1(t) \Delta_\zeta &= C_2(t) \Delta_m(\hat{\zeta}_{\mathsf{e}})+ C_3(t) \Delta_m^{(1)}(\hat{\zeta}_{\mathsf{e}}) + \frac{\mathsf{ZOT}_\zeta+\mathcal{P}_\zeta(\Delta_\zeta,\Delta^{(k)}_m(\hat{\zeta}_{\mathsf{e}}))}{2\mathsf{f}_m\zeta_{\mathsf{e}} \E m^{(2)}_{X}(\hat{\zeta}_{\mathsf{e}}) },
\label{eq:dzetaexpan}\end{align}
where $C_i(t), i =1,2,3$ are deterministic quantities satisfying $C_1(t) = 1 + \mathcal{O}(t)$, $C_2(t) = \mathcal{O}(t^3)$, and $C_3(t)= \mathcal{O}(t^3)$.
Multiplying $\mathbf{1}_{\Omega_\Psi}$ at both sides and then compute the variance:
\begin{align}
	\mathrm{Var}( \Delta_\zeta \mathbf{1}_{\Omega_\Psi}) &\lesssim  t^{6}\mathrm{Var}(\Delta_m(\hat{\zeta}_{\mathsf{e}}) ) + t^{6} \mathrm{Var}(\Delta^{(1)}_m(\hat{\zeta}_{\mathsf{e}}) ) + t^{4}\mathrm{Var}\big(\mathcal{P}_\zeta(\Delta_\zeta,\Delta^{(k)}_m(\hat{\zeta}_{\mathsf{e}}))\big) \notag\\
	&\lesssim N^{-1+\epsilon} + N^{-1/2-\beta/4}\sqrt{\mathrm{Var}(\Delta_\zeta \mathbf{1}_{\Omega_\Psi} )},
\label{eq:selfineVarm}\end{align}
Solving the above inequality for $\mathrm{Var}( \Delta_\zeta \mathbf{1}_{\Omega_\Psi})$ gives
$
	\mathrm{Var}( \Delta_\zeta \mathbf{1}_{\Omega_\Psi}) \lesssim N^{-1+\epsilon},
$
which completes the proof of Lemma \ref{lem:hpbandvb}.
\end{proof}

\begin{remark}[Bound $\mathsf{ZOT}_\zeta$]\label{rmk: bound ZOT}
We start with (\ref{eq:dzetaexpan}). 
Multiplying $\mathbf{1}_{\Omega_\Psi}$ at both sides and then taking expectation, we have
$ (\mathsf{ZOT}_\zeta+\E[\mathcal{P}_\zeta(\Delta_\zeta,\Delta^{(k)}_m(\hat{\zeta}_{\mathsf{e}}))\cdot \mathbf{1}_{\Omega_\Psi}])/(2\mathsf{f}_m\zeta_{\mathsf{e}} \E m^{(2)}_{X}(\hat{\zeta}_{\mathsf{e}}) ) + \mathcal{O}(N^{-D}) = 0$.
	Using (\ref{eq:expem}) together with the variance bound for $\Delta_\zeta \mathbf{1}_{\Omega_\Psi}$ in Lemma \ref{lem:hpbandvb}, we can obtain that $\mathsf{ZOT}_\zeta = \mathcal{O}(N^{-1+\eps}t^{-3})$.
	By the fact $t \gg N^{(\alpha - 4)/32}$, it follows that $\mathsf{ZOT}_\zeta = \mathcal{O}(N^{-\alpha/4-(4-\alpha)/8})$.
\end{remark}

\subsection{Proof of Proposition \ref{prop:lambdaminustdis}}

Proposition \ref{prop:lambdaminustdis} follows from Lemma \ref{lem:mmpcalcualtion} and the following theorem together with some simple algebraic calculation. Recall that $\Delta_m(\hat{\zeta}_{\mathsf{e}}) = m_X(\hat{\zeta}_{\mathsf{e}}) - \E (m_X(\hat{\zeta}_{\mathsf{e}}))$.
\begin{theorem}[CLT of the linear eigenvalue statistics of $\mathcal{S}(X)$]\label{thm:flucation}
For any $2 < \alpha < 4$, 
\begin{align*}
		\frac{ N^{\alpha/4}t\Delta_m(\hat{\zeta}_{\mathsf{e}})}{\sigma_m} \Rightarrow\mathcal{N}(0,1),	
\end{align*}
where
\begin{align*}
	\sigma_m^2 &\coloneqq \mathsf{c}t^2 c_N\int_0^\infty \int_0^\infty  \partial_z \partial_{z'} \Big\{  \frac{e^{-s -s' - sc_N\mathsf{m}_{\mathsf{mp}}^{(t)}(z) - s'c_N\mathsf{m}_{\mathsf{mp}}^{(t)}(z') }}{ss'}\\
		&\quad\times \Big(\big(s \mathsf{m}_{\mathsf{mp}}^{(t)}(z)+s' \mathsf{m}_{\mathsf{mp}}^{(t)}(z')\big)^{\alpha/2} - \big(s \mathsf{m}_{\mathsf{mp}}^{(t)}(z)\big)^{\alpha/2}-\big(s' \mathsf{m}_{\mathsf{mp}}^{(t)}(z')\big)^{\alpha/2} \Big) \Big\} \Big|_{z=z' = \hat{\zeta}_{\mathsf{e}}} \mathrm{d}s \mathrm{d}s'.
\end{align*}
\end{theorem}

To prove Theorem \ref{thm:flucation}, we will work on the truncated matrix $\tilde{X} = (\tilde{x}_{ij})$ with $\tilde{x}_{ij} = x_{ij}\mathbf{1}_{\sqrt{N}|x_{ij}| \le N^{\vartheta}}$ and $\vartheta =  1/4+1/\alpha + \epsilon_{\vartheta}$ such that $ N^{-\alpha\epsilon_{\vartheta}}\ll t$ and $\epsilon_{\vartheta}< (3\alpha-5)/(4\alpha) $. It will become clear from the following lemma that the fluctuations of $m_X$ and $m_{\tilde{X}}$ are asymptotically the same.

\vspace{1ex}
\begin{lemma}\label{lem:truncation}
	We have $N^{\alpha/4}t\big(m_X(\hat{\zeta}_{\mathsf{e}}) - m_{\tilde{X}}(\hat{\zeta}_{\mathsf{e}})\big) = \mathfrak{o}_p(1)$.
\end{lemma}
\begin{proof}
This lemma simply follows from the rank inequality and Bennett's inequality together with Lemma \ref{lem:preEst} (i).
\end{proof}

\begin{proof}[Proof of Theorem \ref{thm:flucation}]
By Lemma \ref{lem:truncation}, it is enough to consider the convergence (in distribution) of 
$
	\mathcal{M}_N(\tilde{X}) \coloneqq  N^{\alpha/4}t\big(m_{\tilde{X}}(\hat{\zeta}_{\mathsf{e}}) - \E m_{\tilde{X}}(\hat{\zeta}_{\mathsf{e}}) \big).
$
We will use the Martingale approach. To this end, we define $\mathcal{F}_k$ as the sigma-algebra generated by the first $k$ columns of $\tilde{X}$. Denoting conditional expectation w.r.t. $\mathcal{F}_k$ by $\E_k$, 
we obtain the following martingale difference decomposition of $\mathcal{M}_N(\tilde{X})$
\begin{align*}
\mathcal{M}_N(\tilde{X})= \sum_{k=1}^N N^{\alpha/4}t(\E_k - \E_{k-1}) \big(m_{\tilde{X}}(\hat{\zeta}_{\mathsf{e}}) - m_{\tilde{X}^{(k)}}(\hat{\zeta}_{\mathsf{e}}) \big).
\end{align*}
Our aim is to show that $\mathcal{M}_N(\tilde{X})$ converges in distribution to a Gaussian distribution $\mathcal{N}(0,\sigma_m^2)$ via the martingale CLT.

\begin{theorem}[Martingale CLT, Theorem A.3 of \cite{BM}] \label{Martingale CLT}
	Let $(\mathcal{F}_k)_{k \ge 0}$ be a filtration such that $\mathcal{F}_0 = \{\emptyset, \Omega \}$ and let $(\mathcal{W}_k)_{k\ge 0}$ be a square-integrable complex-valued martingale starting at zero w.r.t.  this filtration. For $k \ge  1$, we define the random variables $Y_k \coloneqq \mathcal{W}_{k} - \mathcal{W}_{k-1}$, $v_k \coloneqq \E_k[|Y_k|^2]$, $\tau_k \coloneqq \E_k [Y_k^2]$, and we also define $v(N) \coloneqq \sum_{k \ge 1}v_k$, $\tau(N) \coloneqq \sum_{k\ge 1} \tau_k$, $\sum_{k \ge1}\E [ |Y_k^2| \mathbf{1}_{|Y_k|\ge \varepsilon}]$.
Suppose that for some constants $v \ge 0$, $\tau \in \mathbb{C}$, and for each $\varepsilon > 0$, $v(N) \overset{\mathbb{P}}{\to} v$, $\tau(N) \overset{\mathbb{P}}{\to} \tau$, $L(\varepsilon,N) \to 0$.
Then, the martingale $\mathcal{W}_{N}$ converges in distribution to a centered complex Gaussian variable $\mathcal{Z}$ such that $\E (|\mathcal{Z}|^2) = v$ and $\E (\mathcal{Z}^2) = \tau$ as $N\to\infty$.
\end{theorem}

We want to apply Theorem \ref{Martingale CLT} with setting $\mathcal{W}_N = \mathcal{M}_N(\tilde{X})$.
Using the resolvent identity,
\begin{align*}
	 \mathcal{M}_N(\tilde{X}) = \sum_{k=1}^N Y_k(\hat{\zeta}_{\mathsf{e}}) \coloneqq  \sum_{k=1}^N \frac{t}{N^{1- \alpha/4}}(\E_k - \E_{k-1}) \frac{\tilde{x}_k^\top (G(\tilde{X}^{(k)}, \hat{\zeta}_{\mathsf{e}}))^2 \tilde{x}_k}{1 + \tilde{x}_k^\top G(\tilde{X}^{(k)}, \hat{\zeta}_{\mathsf{e}})\tilde{x}_k}.
\end{align*}
First note that $|Y_k(\hat{\zeta}_{\mathsf{e}}) | \lesssim N^{-1+\alpha/4}t^{-1}$ with high probability.
We also have the deterministic upper bound for $Y_k(\hat{\zeta}_{\mathsf{e}})$ since $\hat{\zeta}_{\mathsf{e}}$ possesses effective imaginary part. Combining these two facts, we can verify that the $L(\varepsilon,N)$ goes to 0.

In order to conclude the proof via Theorem \ref{Martingale CLT}, we need to check convergences of $v(N)$ and $\tau(N)$. This follows from Propositions \ref{prop:removing off-diagonal} and \ref{prop:computaionofkernel} below.
\end{proof}


\begin{proposition}\label{prop:removing off-diagonal}
	Let 
\begin{align*}
	\tilde{Y}_k(\zeta) \coloneqq \frac{t}{N^{1- \alpha/4}}(\E_k - \E_{k-1}) f_k(\zeta)  \coloneqq \frac{t}{N^{1- \alpha/4}}(\E_k - \E_{k-1})  \frac{\tilde{x}_k^\top (G(\tilde{X}^{(k)}, \zeta))^2_{\mathrm{diag}} \tilde{x}_k}{1 + \tilde{x}_k^\top (G(\tilde{X}^{(k)}, \zeta))_{\mathrm{diag}}\tilde{x}_k}.
\end{align*}
Then there exists some constant $\tau$, such that for any $\zeta, \zeta' \in \Xi(\tau) = \{ \xi \in \mathbb{C} : | \xi - \hat{\zeta}_{\mathsf{e}}| \le \tau t^2, |\Im \xi| \ge N^{-100} \}$, the summation $\sum_{k=1}^N \E_{k-1}[Y_k(\zeta)Y_k(\zeta')] -   \E_{k-1}[\tilde{Y}_k(\zeta)\tilde{Y}_k(\zeta')]$
converges in probability to $0$.
\end{proposition}

\begin{proof} The proof is similar to the counterpart in \cite{BM}; see the Appendix \ref{3891} for details. 
\end{proof}

\begin{proposition}\label{prop:computaionofkernel}
For any $k \in [N]$, 
	there exists some constant $\tau$, such that for any $z, z' \in  \{ \zeta \in \mathbb{C} : | \zeta - \hat{\zeta}_{\mathsf{e}}| \le \tau t^2, |\Im \zeta| \ge N^{-100} \}$,
\begin{align*} 
	\frac{N^{-1 + \alpha/2}t^2\E_{k-1}\big( (\E_k - \E_{k-1})f_k(z)(\E_k - \E_{k-1})f_k(z') \big)}{\mathcal{K}(z,z')} \overset{\mathbb{P}}{\rightarrow} 1,
\end{align*}
as $N\to \infty$. The kernel $\mathcal{K}(z,z')$ is defined as
\begin{align*}
	\mathcal{K}(z,z') &\coloneqq \mathsf{c}N^{1-\alpha/2}t^2 c_N\int_0^\infty \int_0^\infty  \partial_z \partial_{z'} \Big\{  \frac{e^{-s -s' - sc_N\mathsf{m}_{\mathsf{mp}}^{(t)}(z) - s'c_N\mathsf{m}_{\mathsf{mp}}^{(t)}(z') }}{ss'}\\
		&\quad\times \Big(\big(s \mathsf{m}_{\mathsf{mp}}^{(t)}(z)+s' \mathsf{m}_{\mathsf{mp}}^{(t)}(z')\big)^{\alpha/2} - \big(s \mathsf{m}_{\mathsf{mp}}^{(t)}(z)\big)^{\alpha/2}-\big(s' \mathsf{m}_{\mathsf{mp}}^{(t)}(z')\big)^{\alpha/2} \Big) \Big\}  \mathrm{d}s \mathrm{d}s'.
\end{align*}
\end{proposition}

Before giving the proof of Proposition \ref{prop:computaionofkernel}, let us introduce the parameter $\sigma_N \coloneqq \sqrt{ N\E \tilde{x}_{ij}^2 }$ and Lemma \ref{lem:truncationprop} below. Note that
\begin{align}\label{eq:sigma1minust}
	\E \big(N\tilde{x}_{ij}^2\mathbf{1}_{\sqrt{N}x_{ij} > N^{\vartheta}}\big) = \int_{N^{2\vartheta}}^\infty \mathbb{P}(|\sqrt{N}x_{ij}|^2 > x) \mathrm{d}x \sim N^{\vartheta(2-\alpha)},
\end{align}
which gives $\sigma_N^2 - (1-t) = \mathcal{O}(N^{\vartheta(2-\alpha)})$. The following lemma collects some useful properties of $\tilde{x}_{ij}$ and the expansion for the characteristic function of $x_{ij}$.
\begin{lemma}\label{lem:truncationprop}
 Then there exists constant $C > 0$, such that
\begin{itemize}
	\item[$({i})$] $\tilde{x}_{ij}$'s are i.i.d.~centered, with variance $\sigma_N^2 /N$,  third moment bound $N^{3/2} \E [ |\tilde{x}_{ij}|^3] \le CN^{\vartheta(3-\alpha)_+}$, and fourth moment  bound $N^{2} \E [ |\tilde{x}_{ij}|^4] \le CN^{\vartheta(4-\alpha)}$,
	\item[$({ii})$] for any $\lambda \in \mathbb{C}$ such that $\Im \lambda \le 0$, 
	  \begin{align*}
	  	{\phi}_N(\lambda) \coloneqq \E \big(e^{-\mathrm{i}\lambda|{x}_{ij}|^2}\big) = 1-\frac{\mathrm{i}(1-t)\lambda}{N} + c\frac{(\mathrm{i}\lambda)^{\frac{\alpha}{2}}}{N^{\frac{\alpha}{2}}} +	\varepsilon_N(\lambda), \quad \text{and} \quad \varepsilon_N(\lambda) =\mathcal{O}\Big(\frac{|\lambda|^{(\alpha+ \varrho)/ 2}}{N^{(\alpha+ \varrho)/2}}\vee \frac{|\lambda| ^{2}}{N^{2}} \Big).
	  \end{align*}
\end{itemize}
\end{lemma}
\begin{proof}
The proof of $(\mathrm{i})$ is elementary. 
To prove $(\mathrm{ii})$, we observe
\begin{equation*}
	1 - \phi_N(\lambda) = \int_0^\infty \big(\exp(-\mathrm{i}\lambda u/N) - 1 \big)\mathrm{d} F^c(u) = \frac{\mathrm{i}\lambda}{N} \int_0^\infty \exp(-\mathrm{i}\lambda u/N) F^c(u) \mathrm{d}u,
\end{equation*}
where $F$ be the distribution function of $Nx_{ij}^2$ and let $F^c = 1-F$.
Since $\int_0^\infty F^c(u) \mathrm{d}u = 1-t$, we notice
\begin{equation*}
	1 - \phi_N(\lambda) = \frac{\mathrm{i}\lambda(1-t)}{N} + \frac{\mathrm{i}\lambda}{N}\int_0^\infty \big(\exp(-\mathrm{i}\lambda u/N) - 1\big) F^c(u) \mathrm{d}u.
\end{equation*}
The estimate $(\mathrm{ii})$ can be obtained using the tail density assumption on $\sqrt{N}y_{ij}$ (cf. Assumption \ref{main assum} (i)).
\end{proof}

\begin{proof}[Proof of Proposition \ref{prop:computaionofkernel}] 
Let $\hat{f}_k$ be defined as ${f}_k$, but with the matrix $\tilde{X}$ replaced by a matrix $\hat{X}$. The columns $\hat{X}_i$ of $\hat{X}$ are the same as those of $\tilde{X}$ if $i \le k$, but are independent random vectors with the same distribution as the columns of $\tilde{X}$ if $i > k$. It is still valid to use the notation $\E_k$ since $\tilde{X}$ and $\hat{X}$ share the same first $k$ columns.
By the following elementary identity
	\begin{align}
	&\E_{k-1}\big( (\E_k - \E_{k-1})(f_k(z))(\E_k - \E_{k-1})(f_k(z')) \big) \notag\\&= \E_k\big( \E_{\tilde{x}_k}(f_k(z)\hat{f}_k(z') ) \big) - \big(\E_k \E_{\tilde{x}_k} f_k(z)\big)\big( \E_k \E_{\tilde{x}_k}\hat{f}_k(z') \big), \label{lem:martingaleexpansion}
\end{align}
 it suffices to study the approximation for $\E_{\tilde{x}_k}f_k(z)$ and $\E_{\tilde{x}_k}f_k(z)\hat{f}_k(z')$. In the sequel, we write $f_k = f_k(z)$, $\hat{f}_k' = \hat{f}_k(z')$, $G_k = G(\tilde{X}^{(k)}, z)$ and $G'_k = G(\tilde{X}^{(k)}, z')$  for simplicity. By a minor process argument, for any $D>0$, there exists constant $C_k > 0$ such that $|\lambda_{M}(\mathcal{S}(\tilde{X}^{(k)}))- \hat{\zeta}_{\mathsf{e}}| \ge C_kt^2$, 
with probability at least $1 - N^{-D}$.
This implies that there exists some constant $C_k >0$ such that for any arbitrary large $D > 0$,
	$$
		\mathbb{P}( \tilde{\Omega}_k= \{ \lambda_{M}(\mathcal{S}(\tilde{X}^{(k)})) - \bar{\zeta}_{-,t} \ge C_kt^2 \}  ) \ge 1 - N^{-D}.
	$$
	Then it is readily seen that $\Re [G(\tilde{X}^{(k)}, z)]_{jj}\cdot \mathbf{1}_{\tilde{\Omega}_k} \ge 0$ for any $z \in \{|z - \bar{\zeta}_{-,t}| \le C_kt^2/10, |\Im z| \ge N^{-100} \}$. Since $\tilde{\Omega}_k$ is independent of $\tilde{x}_k$, we can write
	$
		\E_{\tilde{x}_k} (f_k) = \E_{\tilde{x}_k} (f_k) \mathbf{1}_{\tilde{\Omega}_k}+\E_{\tilde{x}_k} (f_k) \mathbf{1}_{\tilde{\Omega}_k^c}.
	$
	Using the facts that $|\tilde{x}_{jk}| \le N^{1/\alpha+1/4+\epsilon_\vartheta}$ and $|[G_k]_{jj}| \le |\Im z|^{-1} \le N^{101}$, we have for  some large constant $K > 0$ such that
	$
		|\E_{\tilde{x}_k} (f_k) \mathbf{1}_{\tilde{\Omega}_k^c}| \le N^{K} \mathbf{1}_{\tilde{\Omega}_k^c}.
	$
	
	Next, we will mainly focus on the estimation for $\E_{\tilde{x}_k} (f_k) \mathbf{1}_{\tilde{\Omega}_k}$. In the sequel, we omit the indicate function $\mathbf{1}_{\tilde{\Omega}_k}$ from the display for simplicity, and keep in mind that all the estimates are done on the event $\tilde{\Omega}_k$.
	Using the identity that for $w$ with $\Re w > 0$, $w^{-1} = \int_0^\infty e^{-sw} \mathrm{d}s$ ,
	we have
	\begin{align*}
		\E_{\tilde{x}_k} f_k &= \E_{\tilde{x}_k} \Big(\int_0^\infty \sum_{j} \tilde{x}_{jk}^2 [G_k^2]_{jj} e^{-s\big(1 + \sum_{j} \tilde{x}_{jk}^2[G_k]_{jj}\big)} \mathrm{d}s \Big)  
		 =- \int_0^\infty \frac{e^{-s}}{s}   \partial_z \Big\{  \E_{\tilde{x}_k} \Big( e^{-s\sum_{j} \tilde{x}_{jk}^2[G_k]_{jj}}\Big) \Big\} \mathrm{d}s.
	\end{align*}
 Recall $\tilde{\phi}_N$ and $\phi_N$ in Lemma \ref{lem:truncationprop}. We have
	\begin{align*}
		&\E_{\tilde{x}_k} f_k = - \int_0^\infty \frac{e^{-s}}{s}   \partial_z \Big\{   \prod_{j} {\phi}_N\big( -\mathrm{i}s [G_k]_{jj} \big)\Big\} \mathrm{d}s  + \mathsf{Diff},
	\end{align*}
	where
	$
		\mathsf{Diff} \coloneqq \int_0^\infty \frac{e^{-s}}{s}   \partial_z \Big\{   \prod_{j} {\phi}_N\big( -\mathrm{i}s [G_k]_{jj} \big) -\prod_{j} \tilde{\phi}_N\big( -\mathrm{i}s [G_k]_{jj} \big) \Big\} \mathrm{d}s.
	$
	Note by the definition of $\tilde{x}_{jk}$'s for any $j \in [N]$, the following estimate holds uniformly for all $\lambda$ with $\Im \lambda \le 0$,
	\begin{align*}
		\big| \phi_N(\lambda) - \tilde{\phi}_N(\lambda)\big| &= \Big| \E \Big[ \big(e^{-\mathrm{i}\lambda|x_{ij}|^2} - 1\big) \cdot \mathbf{1}_{\sqrt{N}|x_{ij}| > N^{\vartheta}} \Big] \Big| \le 2\mathbb{P}\big( \sqrt{N}|x_{ij}| > N^{\vartheta} \big)
		\lesssim N^{-\alpha\vartheta}.
	\end{align*}
	Therefore, by a Cauchy integral argument with contour radius equals to $ct^2$ for some sufficiently small $c > 0$, we have for sufficiently large $K$,
	\begin{align*}
		\Big|\int_{N^{-K}}^\infty \frac{e^{-s}}{s}   \partial_z \Big\{   \prod_{j} {\phi}_N\big( -\mathrm{i}s [G_k]_{jj} \big) -\prod_{j} \tilde{\phi}_N\big( -\mathrm{i}s [G_k]_{jj} \big) \Big\} \mathrm{d}s \Big|\lesssim  t^{-2}N^{-\alpha\vartheta} \int_{N^{-K}}^\infty \frac{e^{-s}}{s}  \mathrm{d}s \lesssim N^{1-\alpha/2-\epsilon}.
	\end{align*}
	With the prescribe $K$, we also have
	\begin{align*}
		\Big|\int^{N^{-K}}_0 \frac{e^{-s}}{s}   \partial_z \Big\{   \prod_{j} \tilde{\phi}_N\big( -\mathrm{i}s [G_k]_{jj} \big) \Big\} \mathrm{d}s \Big| = \Big|\E_{\tilde{x}_k} \Big(\int_0^{N^{-K}} \sum_{j} \tilde{x}_{jk}^2 [G_k^2]_{jj} e^{-s\big(1 + \sum_{j} \tilde{x}_{jk}^2[G_k]_{jj}\big)} \mathrm{d}s \Big)\Big|  \lesssim N^{-K/2},
	\end{align*}
	and similar estimate holds if we replace $\tilde{\phi}_N$ by $\phi_N$.  Combining the above two displays, we can obtain that
	\begin{align*}
				\E_{\tilde{x}_k} f_k = - \int_0^\infty \frac{e^{-s}}{s}   \partial_z \Big\{   \prod_{j} \Big(1 + \frac{1-t }{N}u_j(z,s) \Big) \Big\} \mathrm{d}s + \mathcal{O}_\prec(N^{1-\alpha/2-\epsilon}),
	\end{align*}
	where
	\begin{align*}
		u_j(z,s) &= \frac{N}{1-t}\Big( {\phi}\big( -\mathrm{i}s [G_k]_{jj} \big) - 1 \Big)= -s[G_k]_{jj} + \mathsf{c}\frac{(s[G_k]_{jj})^{\frac{\alpha}{2}}}{N^{\frac{\alpha-2}{2}}(1-t)} + \frac{N}{1-t}\varepsilon_N(-\mathrm{i}s[G_k]_{jj}).
	\end{align*}
We introduce the approximation $\mathsf{K}_1(z,s)$ for the integrand as follows:
	\begin{align*}
		\mathsf{K}_1(z,s) := \frac{e^{-s-s(1-t)\Tr G_k/N}}{s}\Big(1 + \frac{\mathsf{c}}{N^{\alpha/2}} \sum_{j=1}^M \big(s [G_k]_{jj} \big)^{\alpha/2} \Big).
	\end{align*}
	Then our goal is to show on the event $\tilde{\Omega}_k$
	\begin{align}
		\int_0^\infty \partial_z \delta(z,s) \mathrm{d}s  \lesssim  N^{1-\alpha/2-\epsilon},
	\label{eq:limiterrorest}\end{align}
	where $
		\delta(z,s) = \frac{e^{-s}}{s}   \prod_{j} \big(1 + \frac{1-t }{N}u_j(z,s) \big)  - \mathsf{K}(z,s),
	$ and $\epsilon > 0$ is a small constant.
	By the Cauchy integral formula, we have $\Big|\int_0^\infty \partial_z \delta(z,s) \mathrm{d}s \Big| \lesssim t^{-2}\int_0^\infty | \delta(z_s,s) |\mathrm{d}s$, 
	where $z_s$ is the  maximizer of $ |\delta(z,s)|$ on the contour $\{z': |z' - z| = C_kt^2/50 \}$. To estimate the RHS of this inequality, we divide it into two parts,
	\begin{align*}
		\frac{1}{t^2}\int_0^\infty | \delta(z_s,s) |\mathrm{d}s = \frac{1}{t^2}\int_0^{N^{\varsigma}} | \delta(z_s,s) |\mathrm{d}s + \frac{1}{t^2}\int_{N^{\varsigma}}^\infty | \delta(z_s,s) |\mathrm{d}s = I_1 + I_2,
	\end{align*}
	with $\varsigma$ being chosen later . Using the fact that $[G_k]_{jj} \lesssim t^{-2}$ on the event $\tilde{\Omega}_k$, we can obtain that 
	\begin{align*}
		I_2 \lesssim \frac{1}{t^{2+\alpha}} \int_{N^{\varsigma}}^\infty s^{\alpha/2 - 1} e^{-s} \mathrm{d}s \le e^{-N^{\varsigma}/3 }.
	\end{align*}
	For $I_1$, we further decompose it into three parts,
	\begin{align*}
		I_1 &= \frac{1}{t^2}\int_0^{N^{\varsigma}} \Big| \frac{e^{-s}}{s}\Big(   \prod_{j} \Big( 1 + \frac{1-t}{N}u_j(z_s,s)\Big)  - e^{\sigma_N^2/N \sum_j u_j(z_s,s)} \Big) \Big|\mathrm{d}s \\
		&\quad +\frac{1}{t^2}\int_0^{N^{\varsigma}} \Big| \frac{e^{-s}}{s}\Big(  e^{(1-t)/N \sum_j u_j(z_s,s)} - e^{-s(1-t) \Tr G_k / N}\Big(1 +\sum_j\mathsf{c}\frac{(s[G_k]_{jj})^{\frac{\alpha}{2}}}{N^{\frac{\alpha}{2}}} + \sum_j\varepsilon_N(-\mathrm{i}s[G_k]_{jj})  \Big) \Big) \Big|\mathrm{d}s\\
		&\quad +\frac{1}{t^2}\int_0^{N^{\varsigma}} \Big| \frac{e^{-s}}{s}\Big(e^{-s(1-t) \Tr G_k / N} \sum_j\varepsilon_N(-\mathrm{i}s[G_k]_{jj})  \Big) \Big|\mathrm{d}s = I_{11} +I_{12} +I_{13} .
	\end{align*}
	Notice that on the event $\tilde{\Omega}_k$,
	$
		M_s = \max_j |u_j(z_s,s)| \sigma_N^2 \lesssim st^{-2},
	$
	and $\Re u_j(z_s,s) = N(\Re \phi_N(-\mathrm{i}s[G_k]_{jj}) - 1) \le 0$.
	Then using \cite[Lemma 4.5]{BM}, we have on the event $\tilde{\Omega}_k$, 
	\begin{align*}
		I_{11} \le \frac{1}{t^2} \int_0^{N^{\varsigma}} \frac{e^{-s}}{s} \cdot \frac{s^2}{Nt^4} e^{s^2/(Nt^4) + \sum_{j}\Re ((1-t)u_j(z_s,s) )/N} \mathrm{d}s \lesssim \frac{1}{Nt^6} \int_0^{N^{\varsigma}} e^{-s}s e^{s^2/(Nt^4)} \mathrm{d}s \lesssim \frac{1}{Nt^6}.
	\end{align*}
By choosing $\varsigma < 1/3$ (say), we can obtain that
$
		I_{11} \lesssim N^{-1}t^{-6}.
$
	Applying the simple inequality that $|e^{x} - (1+x)|\le 2|x|^2 $ for $|x| \le 1/2$, we have
	\begin{align*}
		I_{12} &\lesssim \frac{1}{t^2}\int_0^{N^{\varsigma}} \frac{e^{-s}}{s} \Big|\sum_j\mathsf{c}\frac{(s[G_k]_{jj})^{\frac{\alpha}{2}}}{N^{\frac{\alpha}{2}}} + \sum_j\varepsilon_N(-\mathrm{i}s[G_k]_{jj})\Big|^2 \mathrm{d}s \\
		&\lesssim \frac{1}{N^{\alpha-2}t^{2+2\alpha}} \int_0^{N^{\varsigma}}  s^{\alpha-1} \mathrm{d}s  \lesssim N^{\varsigma \alpha - \alpha+2} t^{-2-2\alpha} \lesssim N^{-3(\alpha-2)/5},
	\end{align*}
	where in the last step, we chose $\varsigma < (\alpha-2)/(4\alpha)$. Finally, for $I_{13}$, we can use Lemma \ref{lem:truncationprop} (ii) to obtain that,
	$
		I_{13} \lesssim  N^{-(\alpha-2)\vartheta}t^{-2-2\alpha} \lesssim N^{-3(\alpha-2)/4}.
	$
	Now we may conclude the proof of (\ref{eq:limiterrorest}) by combining the above estimates and possibly adjusting the constants. This gives
	\begin{align*}
		\E_{\tilde{x}_k} (f_k) = - \int_0^\infty   \partial_z \mathsf{K}_1(z,s)  \mathrm{d}s + O_\prec(N^{1-\alpha/2-\epsilon}) .
	\end{align*}
	Similarly, we can obtain that
	\begin{align*}
		&\E_{\tilde{x}_k}(f_k \hat{f}_k') =\int_0^\infty \int_0^\infty  \partial_z \partial_{z'} \mathsf{K}_2(z,z',s,s') \mathrm{d}s \mathrm{d}s' + O_\prec(N^{1-\alpha/2-\epsilon}),
	\end{align*}
	where
	\begin{align*}
		 \mathsf{K}_2(z,z',s,s') \coloneqq \frac{e^{-s-s'-s(1-t)\Tr G_k/N-s(1-t)\Tr G'_k/N}}{ss'}\Big(1 + \frac{\mathsf{c}}{N^{\alpha/2}} \sum_{j=1}^M \big(s [G_k]_{jj}+s' [G'_k]_{jj} \big)^{\alpha/2} \Big).
	\end{align*}
	Notice that the estimate in Proposition \ref{prop: resolvent entry size for X=B+C} can be obtained for our $G(\tilde{X}^{(k)},z)$ as well in the same manner. Suppose that $\max\{|\mathcal{D}_r|, |\mathcal{D}_c|\}\leq N^{1-\epsilon_d}$ for some $\epsilon_d$. Notice that $\epsilon_d\geq \epsilon_\alpha$ by definition. Hence, the claim now follows by (i) employing (\ref{lem:martingaleexpansion}), then substituting $\sigma_N^2 [G_k]_{jj}(z)$ with ${\mathsf{m}}_{\mathsf{mp}}(z/\sigma_N^2)$ for $j \in \mathcal{T}_r$, and utilizing the bound $[G_k]_{jj} \prec 1/t$ for $j \in \mathcal{D}_r$ with the fact $t \gg N^{-\epsilon_d/4}$; (ii) considering the estimates $\sigma_N^2 - (1-t) =\mathcal{O}(N^{\vartheta(2-\alpha)})$ (refer to Eqn. (\ref{eq:sigma1minust})), $\partial_z {\mathsf{m}}_{\mathsf{mp}}(z/\sigma_N^2) \sim t^{-1}$, and $\partial^2_z {\mathsf{m}}_{\mathsf{mp}}(z/\sigma_N^2) \sim t^{-3}$ for $z$ within the specified domain. This enables us to further replace ${\mathsf{m}}_{\mathsf{mp}}(z/\sigma_N^2)$ and ${\mathsf{m}}_{\mathsf{mp}}(z/\sigma_N^2)$ with $\mathsf{m}_{\mathsf{mp}}^{(t)}(z)$ and $\mathsf{m}_{\mathsf{mp}}^{(t)}(z')$, respectively.
\end{proof}

\subsection{Proof of Proposition \ref{prop:lambdashiftexpans}}
Let us first define
\begin{align*}
	&\mathfrak{p}(z)  := \mathsf{c}N^{1-\alpha/2} c_N \int_0^\infty e^{-s-sc_N  \mathsf{m}_{\mathsf{mp}}^{(t)}(z) }\big( s  \mathsf{m}_{\mathsf{mp}}^{(t)}(z)\big)^{\alpha/2} \mathrm{d}s, \\
	&m_{\mathsf{shift}}(z)  \coloneqq \frac{ \mathrm{i}(\frac{ z}{1-t} -c_N + 1) \mathfrak{p}(z)}{2c_N z \sqrt{(\frac{ z}{1-t} - \lambda_{-}^{\mathsf{mp}})(\lambda_{+}^{\mathsf{mp}} - \frac{ z}{1-t})}   }.
\end{align*}
Then we have the following proposition concerning the expansion of $\E m_X(z)$.
\begin{proposition} 
There exists some sufficiently small constant $\tau > 0$, such that for any $z \in \{\zeta : |\zeta - \bar{\zeta}_{-,t}| \le \tau t^2, |\Im \zeta| \ge N^{-100} \}$, we have $m_{\mathsf{shift}}(z) = \mathcal{O}(t^{-1}N^{1-\alpha/2})$ and
	\begin{align}
		\E m_{X}(z)  = \mathsf{m}_{\mathsf{mp}}^{(t)}(z) + m_{\mathsf{shift}}(z) - \frac{\mathfrak{p}(z)}{2c_Nz}+ \mathcal{O}(N^{1-\alpha/2-\epsilon(4-\alpha)(\alpha-2)/50}).
	\label{eq:c1inprop}\end{align}
	Furthermore, for any $z \in \{\zeta : |\zeta - \bar{\zeta}_{-,t}| \ll t^2, |\Im \zeta| \ge N^{-100} \}$,
	\begin{align}
		&tm_{\mathsf{shift}}(z)   = \frac{ \mathsf{c}N^{1-\alpha/2}  \int_0^\infty e^{-s-sc_N \mathsf{m}_{\mathsf{mp}}(\lambda_{-}^{\mathsf{mp}})  }\big( s \mathsf{m}_{\mathsf{mp}}(\lambda_{-}^{\mathsf{mp}}) \big)^{\alpha/2} \mathrm{d}s}{2\sqrt{c_N}(1-\sqrt{c_N}) } + \mathfrak{o}(N^{1-\alpha/2}).
	\label{eq:c2inprop}\end{align}
\label{prop:expectationexpan}\end{proposition}
\begin{proof}
By the resolvent expansion, we have for any $z \in \{\zeta : |\zeta - \bar{\zeta}_{-,t}| \le \tau t^2, |\Im \zeta| \ge N^{-100} \}$,
\begin{align}
	\big[ G(X^\top, z)\big]_{ii} = -\big(z + zx_{i}^\top G(X^{(i)}, z) x_i \big)^{-1}.
\label{eq:resoexpan}\end{align}
Let $Q = Q_\mathsf{diag} + Q_\mathsf{off}$ with $Q_\mathsf{diag} := \sum_{j=1}^M x_{ji}^2 \big[G(X^{(i)}, z)\big]_{jj}$ and $Q_\mathsf{off} := \sum_{\ell\neq k}  x_{ki} x_{\ell i}  \big[G(X^{(i)}, z)\big]_{k\ell}$.
Then, we can rewrite (\ref{eq:resoexpan}) as:
\begin{align}
	\big[ G(X^\top, z)\big]_{ii} 
	&=-\frac{1}{z(1 + Q_\mathsf{diag})} +  \frac{Q_\mathsf{off}}{z(1 + Q_\mathsf{diag})^2} -  \frac{Q_\mathsf{off}^2}{z(1 + Q_\mathsf{diag})^2(1 + Q )}.
\label{eq:resodecompo}\end{align}
Taking expectation at both sides gives
\begin{align*}
	\E \big[ G(X^\top, z)\big]_{ii} = -\frac{1}{z} \E \Big[\frac{1}{1 + Q_\mathsf{diag}}   \Big] - \frac{1}{z} \E \Big[\frac{Q_\mathsf{off}^2}{(1 + Q_\mathsf{diag})^2(1 + Q )}  \Big] = I_1 + I_2,
\end{align*}
where the second term at the right hand side of (\ref{eq:resodecompo}) vanished due to symmetry. 
Notice that when $\Psi^{(i)}$ is good, we have w.h.p. that
\begin{align*}
		|\lambda_{M}(\mathcal{S}({X}^{(i)}))- z| &= |(1-t)\lambda_{-}^{\mathsf{mp}} - \bar{\zeta}_{-,t}| -|\lambda_{M}(\mathcal{S}({X}^{(i)})) - (1-t)\lambda_{-}^{\mathsf{mp}}| - |\bar{\zeta}_{-,t} - z| \notag\\
		&\ge  \sqrt{c_N}t^2 -  \sqrt{c_N}t^2/4  - \tau t^2 \ge \sqrt{c_N}t^2/2,
	\end{align*} 
where in the last step we used the fact that $|\lambda_{M}(\mathcal{S}({X}^{(i)})) - (1-t)\lambda_{-}^{\mathsf{mp}}| \le N^{-\epsilon_b}$ w.h.p., and we also chose $\tau < \sqrt{c_N}t^2/4$. This together with the fact that $\Psi^{(i)}$ is good w.h.p. gives $\mathbb{P}\big( \Omega_i = \big\{ |\lambda_{M}(\mathcal{S}({X}^{(i)})) - z| \ge \sqrt{c_N}t^2/2 \big\}  \big) \ge 1 - N^{-D}$.
	Notice that $ \Re Q_\mathsf{diag} \ge 0$ and $\Re Q \ge 0$ hold on $\Omega_i$.
Then for $I_2$, with the smallness of $\mathbb{P}(\Omega_i^c)$,  we have $I_2 = \E [ {Q_\mathsf{off}^2\mathbf{1}_{\Omega_i}}/[(1 + Q_\mathsf{diag})^2(1 + Q )] ] + \mathcal{O}(N^{-D})$.
	We then bound $I_2$ as 
	\begin{align*}
		|I_2| \le \E |Q_\mathsf{off}|^2 \mathbf{1}_{\Omega_i} +\mathcal{O}(N^{-D}) 
		 = 2N^{-2} \E  \Big[\Tr  G(X^{(i)}, z) \overline{G(X^{(i)}, z)} \mathbf{1}_{\Omega_i} \Big] +\mathcal{O}(N^{-1}) = \mathcal{O}_\prec(t^{-4}N^{-1}).
	\end{align*}
	Next, we estimate $I_1$. 
	Due to the smallness of $\mathbb{P}(\Omega_i^c)$, we only have to do the estimation on the event $\Omega_i$. Specially, we have $I_1 = - \E [\mathbf{1}_{\Omega_i}/(z + zQ_\mathsf{diag})] + 	\mathcal{O}(N^{-D})$.
Notice that $\Re  Q_\mathsf{diag} \ge 0$ on the event $\Omega_i$. Using the identity that for $w$ with $\Re w > 0$, $w^{-1} = \int_0^\infty e^{-sw} \mathrm{d}s$ and setting $w = 1 +  Q_\mathsf{diag}$, we have
\begin{align*}
	I_1 &= -\frac{1}{z} \E \Big[ \int_0^{\infty} e^{-s(1 +  Q_\mathsf{diag})} \mathrm{d}s  \cdot \mathbf{1}_{\Omega_i}   \Big] + 	\mathcal{O}(N^{-D})
	=-\frac{1}{z} \E \Big(\E_{x_i} \Big[ \int_0^{\infty} e^{-s(1 +  Q_\mathsf{diag})} \mathrm{d}s   \Big]  \cdot \mathbf{1}_{\Omega_i} \Big)+ 	\mathcal{O}(N^{-D})\\
	&=-\frac{1}{z} \E \Big( \int_0^{\infty} e^{-s} \prod_{j} \phi_N\big(-\mathrm{i}s \big[G(X^{(i)}, z)\big]_{jj}\big) \mathrm{d}s    \cdot \mathbf{1}_{\Omega_i} \Big)+ 	\mathcal{O}(N^{-D}).
\end{align*}	
Then we may proceed as the estimation in the proof of Proposition \ref{prop:computaionofkernel} to obtain that
\begin{align*}
	I_1 = -\frac{1}{z}\E \Big[\frac{1}{1 + (1-t)\Tr G(X^{(i)},z) / N } \cdot \mathbf{1}_{\Omega_i}   \Big] -\frac{ \mathfrak{p}(z)}{z} + \mathcal{O}(N^{-3(\alpha-2)/5}).
\end{align*} 
Further using the  $\mathcal{O}_\prec(t^{-4}N^{-1})$ bound for $\mathrm{Var}(M^{-1}\Tr G(X^{(i)},z))$ and the fact $\Tr G(X^{(i)},z) - \Tr G(X,z) \prec t^{-4}$, we arrive at
\begin{align*}
	I_1 = -\frac{1}{z}\Big( \frac{1}{1 + (1-t) \E \Tr G(X,z) / N }\Big)  -\frac{ \mathfrak{p}(z)}{z} + \mathcal{O}(N^{-3(\alpha-2)/5}) + \mathcal{O}_\prec(t^{-4}N^{-1}).
\end{align*}
Collecting the estimates for $I_1$ and $I_2$, and then summing over $i$, we have
\begin{align*}
	N^{-1}\E \Tr  G(X^\top, z) = -\frac{1}{z} \Big(\frac{1}{1 + (1-t) \E \Tr G(X,z) / N } \Big)  -\frac{ \mathfrak{p}(z)}{z} + \mathcal{O}(N^{-3(\alpha-2)/5}). 
\end{align*}
Using the simple equation $\Tr  G(X, z) - \Tr  G(X^\top, z) = (N-M)/z$, the above equation can be rewritten as:
\begin{align}
c_N\E m_X(z)  	= -\frac{1}{z} \Big(\frac{1}{1 + (1-t)c_N\E m_X(z) } \Big)  -\frac{ \mathfrak{p}(z)+1-c_N}{z} +\mathcal{O}(N^{-3(\alpha-2)/5}). 
\label{eq:quadraticequa}\end{align}
Notice that for $z = \bar{\zeta}_{-,t} + \mathrm{i}N^{-100K_\zeta}$, we have
$
	(\frac{ z}{1-t} + c_N - 1)^2-\frac{ 4c_Nz}{1-t} = \frac{c_Nt^2(2-t)^2}{(1-t)^2} + \mathcal{O}(N^{-90K_\zeta}).
$
Then by continuity, we may choose $\tau$ sufficiently small such that for any $z \in \{\zeta : |\zeta - \bar{\zeta}_{-,t}| \le \tau t^2, |\Im \zeta| \ge N^{-100K_\zeta} \}$, we have $(\frac{ z}{1-t} + c_N - 1)^2-\frac{ 4c_Nz}{1-t}  \sim t^2$. 
Having this bound, we may solve the quadric equation (\ref{eq:quadraticequa}) and then compare it with (\ref{eq:MPSTfeq})	 to obtain that
\begin{align*}
	\E m_X(z)=\mathsf{m}_{\mathsf{mp}}^{(t)}(z) + m_{\mathsf{shift}}(z) - \frac{\mathfrak{p}(z)}{2c_Nz}+ \mathcal{O}(N^{-11(\alpha-2)/20}),
\end{align*}
which proves (\ref{eq:c1inprop}).
Using the fact that $\mathsf{m}_{\mathsf{mp}}^{(t)}(\bar{\zeta}_{-,t} + \mathrm{i}N^{-100K_\zeta}) - \mathsf{m}_{\mathsf{mp}}(\lambda_{-}^{\mathsf{mp}}) \le t $, we may further derive that
\begin{align*}
	t m_{\mathsf{shift}}(\bar{\zeta}_{-,t} + \mathrm{i}N^{-100}) =\frac{ \mathsf{c}N^{1-\alpha/2}  \int_0^\infty e^{-s-sc_N \mathsf{m}_{\mathsf{mp}}(\lambda_{-}^{\mathsf{mp}})  }\big( s \mathsf{m}_{\mathsf{mp}}(\lambda_{-}^{\mathsf{mp}}) \big)^{\alpha/2} \mathrm{d}s}{2\sqrt{c_N}(1-\sqrt{c_N}) } +  \mathcal{O}(tN^{1-\alpha/2}).
\end{align*}
This together with the crude bound $\mathsf{m}_{\mathsf{mp}}^{(t)}(\bar{\zeta}_{-,t} + \mathrm{i}N^{-100K_\zeta}) - \mathsf{m}_{\mathsf{mp}}^{(t)}(\bar{\zeta}_{-,t}) = \mathcal{O}(N^{-90K_\zeta})$ proves (\ref{eq:c2inprop}), which completes the proof of Proposition \ref{prop:expectationexpan}.
\end{proof}

The following corollary is a direct consequence of Proposition \ref{prop:expectationexpan}.
\begin{corollary} 
Let $\tau$ be chosen as in Proposition \ref{prop:expectationexpan}. Then for any $z \in \{\zeta : |\zeta - \bar{\zeta}_{-,t}| \le \tau t^2/2, |\Im \zeta| \ge N^{-100} \}$, we have
	$
		\E m^{(k)}_{X}(z) - (\mathsf{m}_{\mathsf{mp}}^{(t)}(z))^{(k)} = \mathcal{O}(t^{-(2k+1)}N^{1-\alpha/2}), 
	$
\label{coro:EmminusmXbound}\end{corollary}
\begin{proof}
	The claim follows from Proposition \ref{prop:expectationexpan} with  Cauchy integral. We omit further details. 
\end{proof}

\begin{proof}[Proof of Proposition \ref{prop:lambdashiftexpans}] 
Replacing $\E[ m_X(\hat{\zeta}_{\mathsf{e}})]$ by $\mathsf{m}_{\mathsf{mp}}^{(t)}(\hat{\zeta}_{\mathsf{e}})$ in the expression of $\lambda_{\mathsf{shift}} $, we can obtain
\begin{align}
		\lambda_{\mathsf{shift}}  &= \bar{\Phi}_t({\zeta}_{\mathsf{e}})  + \big(2c_Nt\lambda_{-}^{\mathsf{mp}} + O(t^2)\big) \cdot \big( \mathsf{m}_{\mathsf{mp}}^{(t)}(\hat{\zeta}_{\mathsf{e}}) - \E[ m_X(\hat{\zeta}_{\mathsf{e}})]  \big) + \mathcal{O}(|\zeta_{\mathsf{e}} - \hat{\zeta}_{\mathsf{e}}|).
\label{eq:lshiftexpan}\end{align}
Expanding $\bar{\Phi}_t({\zeta}_{\mathsf{e}})$ around $\bar{\zeta}_{-,t}$ and using the fact that $\bar{\Phi}'_t(\bar{\zeta}_{-,t}) = 0$, we have that there exists $\tilde{\zeta} \in [\bar{\zeta}_{-,t}, {\zeta}_{\mathsf{e}}]$ such that
$
	\bar{\Phi}_t({\zeta}_{\mathsf{e}}) = \bar{\Phi}_t(\bar{\zeta}_{-,t}) + \bar{\Phi}''_t(\tilde{\zeta})({\zeta}_{\mathsf{e}} - \bar{\zeta}_{-,t})^2 = \lambda_{-}^{\mathsf{mp}}+\bar{\Phi}''_t(\tilde{\zeta})({\zeta}_{\mathsf{e}} - \bar{\zeta}_{-,t})^2.
$
Substituting this expansion back into (\ref{eq:lshiftexpan}), and using the bound in  Corollary \ref{coro:EmminusmXbound}, (\ref{eq:lshiftexpan}) becomes
\begin{align*}
	\lambda_{\mathsf{shift}} = \lambda_{-}^{\mathsf{mp}} + 2c_Nt\lambda_{-}^{\mathsf{mp}}\big( \mathsf{m}_{\mathsf{mp}}^{(t)}(\hat{\zeta}_{\mathsf{e}}) - \E[ m_X(\hat{\zeta}_{\mathsf{e}})]  \big) + \bar{\Phi}''_t(\tilde{\zeta})({\zeta}_{\mathsf{e}} - \bar{\zeta}_{-,t})^2+  \mathfrak{o}(N^{1-\alpha/2}).
\end{align*}
Note by considering that $\tilde{\zeta} - (1-t)\lambda_{-}^{\mathsf{mp}} \sim t^{2}$, it can be easily verified that $\bar{\Phi}''_t(\tilde{\zeta}) \sim t^{-2}$. 

By employing Corollary \ref{coro:EmminusmXbound} along with the variance bounds for $m_X^{(k)}(\hat{\zeta}_{\mathsf{e}})$ in Lemma \ref{lem:hpbandvb}, we can conclude that
\begin{align*}
	\bar{\Delta}^{(k)}_m(\bar{\zeta}_{-,t}) \coloneqq m_X^{(k)}(\bar{\zeta}_{-,t}) - (\mathsf{m}_{\mathsf{mp}}^{(t)}(\bar{\zeta}_{-,t} ))^{(k)}  = \mathcal{O}_p( N^{-1/2+\epsilon/2}t^{-2-k} + N^{1-\alpha/2} t^{-2k-1} ).
\end{align*}
With the above probabilistic bounds in place, we may now proceed to follow the expansion detailed in the proof of  Lemma \ref{lem:hpbandvb}, but this time substitute $\zeta_{\mathsf{e}}$ with $\bar{\zeta}_{-,t}$ and $\E (m_X^{(k)}(\hat{\zeta}_{\mathsf{e}}))$ with $\bar{m}^{(k)}_X(\bar{\zeta}_{-,t})$ (cf. (\ref{eq:zetaequation})-(\ref{eq:selfineVarm})). It becomes evident that the $\mathsf{ZOT}_\zeta$ therein vanishes due to the fact that $\bar{\Phi}_t'(\bar{\zeta}_{-,t}) = 0$. This eventually leads to
	$
		\bar{\Delta}_\zeta \coloneqq \zeta_{-,t} - \bar{\zeta}_{-,t} = \mathcal{O}_p (N^{-1/2+\epsilon/2}  + N^{1-\alpha/2}   ).
	$
	Therefore, with ${\Delta}_\zeta = \mathcal{O}_p(N^{-1/2+\epsilon/2}t^6)$, we have
	$
		\bar{\Phi}''_t(\tilde{\zeta})({\zeta}_\mathsf{e} - \bar{\zeta}_{-,t})^2 \sim t^{-2}( \bar{\Delta}_\zeta - {\Delta}_\zeta)^2 = \mathfrak{o}(N^{1-\alpha/2}).
	$
	Consequently, we arrive at
	\begin{align*}
		\lambda_{\mathsf{shift}} = \lambda_{-}^{\mathsf{mp}} + 2c_Nt\lambda_{-}^{\mathsf{mp}}\big( \mathsf{m}_{\mathsf{mp}}^{(t)}(\hat{\zeta}_{\mathsf{e}}) - \E m_X(\hat{\zeta}_{\mathsf{e}})  \big) +  \mathfrak{o}(N^{1-\alpha/2}).
	\end{align*}
	Recalling from (\ref{eq:zetaminusbarzeta}) that $\bar{\zeta}_{-,t} - {\zeta}_{\mathsf{e}} \prec N^{-\beta/2}t^2$, we can deduce that $\bar{\zeta}_{-,t} - \hat{\zeta}_{\mathsf{e}} \prec N^{-\beta/2}t^2$. The claim now follows by (\ref{eq:c2inprop}) in Proposition \ref{prop:expectationexpan} and the fact $\mathsf{m}_{\mathsf{mp}}(\lambda_{-}^{\mathsf{mp}}) = (\sqrt{c_N}-c_N)^{-1}$.
	\end{proof}


\section{Beyond Gaussian divisible model}\label{s.general}
In this section, we present three Green function function comparison results, as we mentioned in the Section \ref{sec1}. Their proofs will be postponed  to the next section.  Recall the notations in (\ref{081537}).
\subsection{Entry-wise bound}
We first introduce the following shorthand notation: for any $a,b\in [M]$ and $u,v \in [N]$,\\
\begin{equation*}
 \mathfrak{X}_{ab} = \mathfrak{X}_{ab}(\Psi) \coloneqq 
\begin{cases}
1 & \text{if } a \text{  or  } b \in \mathcal{T}_r, 
\\
t^2& \text{if } a \in \mathcal{D}_r, b \in \mathcal{D}_r
\end{cases} ,\quad 
\mathfrak{Y}_{uv} = \mathfrak{Y}_{uv}(\Psi) \coloneqq 
\begin{cases}
1 & \text{if } u  \text{  or  }  v\in \mathcal{T}_c,
\\
t^2& \text{if } u \in \mathcal{D}_c, v \in \mathcal{D}_c 
\end{cases},
\end{equation*}
\begin{equation*}
 \mathfrak{Z}_{au} = \mathfrak{Z}_{au}(\Psi) \coloneqq 
\begin{cases}
1 & \text{if } a \in \mathcal{T}_r \text{ or } u \in \mathcal{T}_c
\\
t^2& \text{if } a \in \mathcal{D}_r, u \in \mathcal{D}_c
\end{cases}.
\end{equation*}

\begin{proposition}[Entry-wise bound]
Recall $\mathsf{D}(\varepsilon_1,\varepsilon_2,\varepsilon_3)$ defined in (\ref{eq:domainD}). Let $\mathsf{D}_{\leq}=\{z=E+\ii\eta\in \mathsf{D}(\varepsilon_1,\varepsilon_2,\varepsilon_3): \eta\leq N^{-\varepsilon}\}$. Set $10\epsilon_a \le  \varepsilon_1 \le \epsilon_b/500$, and set $\varepsilon_2, \varepsilon_3$ sufficiently small, and  $3\epsilon_1<\varepsilon\leq \epsilon_b/100$. Suppose that $\Psi$ is good.  Let $\mathbb{P}_\Psi$ be the probability conditioned on the event that the $(\psi_{ij})$ matrix is a given $\Psi$. Suppose that $\Psi$ is good (cf. (\ref{081533})). Then for each $\delta > 0$ and $D > 0$, there exists a large constant $C >0$ such that  
	\begin{align*}
		\mathbb{P}_\Psi \Big( \sup_{0\le \gamma \le 1} \sup_{z \in \mathsf{D}_{\leq} }\sup_{a, b \in [M] }  &| \mathfrak{X}_{ab}[G^\gamma(z)]_{ab}| \vee \sup_{0\le \gamma \le 1} \sup_{z \in \mathsf{D}_{\leq} }\sup_{u, v \in [N]}  |\mathfrak{Y}_{uv}[\mathcal{G}^\gamma(z)]_{uv}| \\
		&\vee \sup_{0\le \gamma \le 1} \sup_{z \in \mathsf{D}_{\leq} }\sup_{a \in [M],u \in [N]}  | \mathfrak{Z}_{au}[G^\gamma(z)Y^\gamma]_{au}| \ge N^{\delta} \Big) \le CN^{-D},
	\end{align*}
\label{prop:locallawentries}\end{proposition}
The proof of Proposition \ref{prop:locallawentries} follows a similar approach to the one demonstrated in \cite[Proposition 3.17]{ALY}. It relies on the entry-wise bounds for the Green functions of $Y^0$ as provided in Theorem \ref{thm: resolvent entry size V_t}, which serve as an input for the subsequent comparison theorem. We defer the proof to Section \ref{1729}.

\begin{theorem}\label{thm:comparison1}
	 Let $F : \mathbb{R} \to \mathbb{R}$ be a function such that 
	\begin{align*}
		\sup_{0 \le \mu \le d } F^{(\mu)}(x) \le (|x| + 1)^{C_0}, \qquad \sup_{\substack{0 \le \mu \le d \\ |x| \le 2N^2 }} F^{(\mu)}(x) \le N^{C_0},
	\end{align*}
	for some real number $C_0, d > 0$. For any $0-1$ matrix $\Psi$ and complex number $z$, we define for any $a,b\in [M]$ and $u,v \in [N]$,
	\begin{align*}
		&\mathfrak{I}_{0,ab} = \mathfrak{I}_{0,ab} (\Psi,z) \coloneqq \max_{0 \le \mu \le d} \sup_{ 0 \le \gamma \le 1} \E_\Psi \big( \big|  F^{(\mu)}( \mathfrak{X}_{ab}\Im [G^{\gamma}(z)]_{ab}) \big|\big),\\
		&\mathfrak{I}_{1,uv} = \mathfrak{I}_{1,uv} (\Psi,z) \coloneqq \max_{0 \le \mu \le d} \sup_{ 0 \le \gamma \le 1} \E_\Psi \big( \big|  F^{(\mu)}(\mathfrak{Y}_{uv}\Im [\mathcal{G}^{\gamma}(z)]_{uv}) \big|\big),\\
		&\mathfrak{I}_{2,au} = \mathfrak{I}_{2,au} (\Psi,z) \coloneqq\max_{0 \le \mu \le d} \sup_{ 0 \le \gamma \le 1} \E_\Psi \big( \big|  F^{(\mu)}(\mathfrak{Z}_{au}\Im [G^{\gamma}(z)Y^\gamma ]_{au}) \big|\big),
	\end{align*}
	and $\Omega = \Omega_0 \cap \Omega_1 \cap \Omega_2\cap \Omega_w$, $Q_0 = Q_0(\varepsilon,z)\coloneqq 1 -\mathbb{P}_\Psi \big( \Omega \big)$ with 
	\begin{align*}
		&\Omega_0 = \Omega_0(\varepsilon, z) \coloneqq \Big\{\sup_{\substack{a , b \in [M]\\ 0 \le \gamma \le 1}} |\mathfrak{X}_{ab}[G^{\gamma}(z)]_{ab}| \le N^{\varepsilon} \Big\}, \Omega_1= \Omega_1(\varepsilon, z) \coloneqq \Big\{\sup_{\substack{u, v \in [N] \\ 0 \le \gamma \le 1}} |\mathfrak{Y}_{uv}[\mathcal{G}^{\gamma}(z)]_{uv}| \le N^{\varepsilon} \Big\}, \\
		&\Omega_2 = \Omega_2(\varepsilon, z) \coloneqq \Big\{\sup_{\substack{a\in [M],u\in[N] \\ 0 \le \gamma \le 1}} |\mathfrak{Z}_{au}[G^{\gamma}(z)Y^\gamma ]_{au}| \le N^{\varepsilon} \Big\}, \Omega_w = \Omega_w(\varepsilon)  \coloneqq \Big\{ \sup_{i \in [M], j \in [N]} |w_{ij}| \le N^{-1/2+\varepsilon}t \Big\}.
	\end{align*}
	Suppose that $\Psi$ is good. There exist sufficiently small  positive constants $\varepsilon \leq \epsilon_b/100$ and $\omega$,  and a large constant $C >0$ such that for 
	\begin{align*}
		(\#_1,\#_2,\#_3) \in \{&(\mathfrak{X}_{ab}\Im [G^\gamma(z)]_{ab}, \;\mathfrak{X}_{ab}\Im [G^0(z)]_{ab}, \;\mathfrak{I}_{0,ab}   ), \\
		&(\mathfrak{Y}_{uv}\Im [\mathcal{G}^\gamma(z)]_{uv},\;\mathfrak{Y}_{uv}\Im [\mathcal{G}^0(z)]_{uv},\;\mathfrak{I}_{1,uv}) ,\\
		 &(\mathfrak{Z}_{au}\Im [G^{\gamma}(z)Y^\gamma]_{au},\; \mathfrak{Z}_{au}\Im [G^{0}(z)Y^0]_{au}, \;\mathfrak{I}_{2,au})\},
	\end{align*}
	we have
	\begin{align}
		\sup_{0 \le \gamma \le 1} \big| \E_\Psi \big(F(\#_1)  \big)  &- \E_\Psi \big(F(\#_2)  \big)  \big|< CN^{-\omega}(\#_3 + 1) + CQ_0N^{C+C_0},\label{eq:comparisonG}
	\end{align}
	for any $a,b\in [M]$ and $u,v \in [N]$. The same estimates hold if $\Im$'s are replaced by $\Re$'s.
\end{theorem}

\subsection{Average local law}
In this section, we write $m^{\gamma}(z) = m_{Y^{\gamma}}(z)$, $G^{\gamma}(z) = G(Y^\gamma,z)$, and $\bar{G}^{\gamma}(z) = G(Y^\gamma,\bar{z})$ for simplicity. Let $z_t := \lambda_{-,t} + E + \mathrm{i}\eta$. Then we have the following theorem. 
\begin{theorem}\label{thm:comparison2} 
	Suppose that $\Psi$ is good. Let us define $z_t \coloneqq \lambda_{-,t} + E + \mathrm{i}\eta$. We assume that $\eta\in [N^{-\frac23-\epsilon}, N^{-\frac23}]$, $E\in [-N^{-\varepsilon_1}, N^{-\frac23+\epsilon}]$ for a sufficiently small $\epsilon>0$. Then there exists a constant $\delta_0 >0$ such that for all integer $p \ge 3$, 
	\begin{align*}
		&\sup_{0 \le \gamma \le 1}\E_\Psi\big( \big|N\eta \big(\Im m^\gamma(z_t) - \Im \tilde{m}^0(z_t)\big) \big|^{2p} \big) \leq (1+\mathfrak{o}(1)) \E_\Psi\big( \big|N\eta \big(\Im m^0(z_t) - \Im \tilde{m}^0(z_t)\big) \big|^{2p} \big) + N^{-\delta_0 p}, 
	\end{align*}
	where $\tilde{m}^0(z) = m_{X+t^{1/2}\tilde{W}}(z)$. Here $\tilde{W}$ is an i.i.d.~copy of $W$ and it is also independent of $X$. Further, the same estimate holds if $\Im$'s are replaced by $\Re$'s. 
\end{theorem}

The above comparison inequality  directly leads to the following theorem, which is crucial for the rigidity estimate for the $\lambda_M(\mathcal{S}(Y))$, serving as a key component in proving the universality result.

\begin{theorem}[Rigidity estimate]\label{rigidity}
Suppose $\Omega_\Psi$ holds. Then, with high probability,
\begin{equation*}
	|\lambda_{M}( \mathcal{S}(Y)) - \lambda_{-,t} | \le N^{-2/3+\epsilon}.
\end{equation*}
\end{theorem}
\begin{proof}
By Markov's inequality, Theorem \ref{thm:comparison2} and the following local law for $m^0$
	\begin{align}
	|m^0(\lambda_{-,t} +E +\mathrm{i}\eta)-m_{t}(\lambda_{-,t} +E +\mathrm{i}\eta)|\prec \left\{ 
	\begin{array}{ll}
	\frac{1}{N\eta}, &E\geq 0, \\ \\
	\frac{1}{N(|E|+\eta)}+\frac{1}{(N\eta)^2\sqrt{|E|+\eta}}, & E\leq 0,
	\end{array}
	\right. \label{0818200}
	\end{align} 
	we can obtain (\ref{0818200}) with $m^0$ replaced by $m^1$ and further  for the case $E\leq 0$ the following 
\begin{align}
	\Im m^{1}(\lambda_{-,t} +E +\mathrm{i}\eta ) - \Im m_{t}(\lambda_{-,t} +E +\mathrm{i}\eta ) \prec \frac{1}{N(|E|+\eta)}+ \frac{1}{(N\eta)^2\sqrt{|E|+\eta}} + \frac{1}{N^{1+\delta_0/2}\eta}. \label{0818333}
\end{align}
We remark here that the local law in (\ref{0818200}) has been proved in \cite{DY2} around the right edge for the deformed rectangular matrices, under the assumption that the original rectangular matrices satisfy the $\eta_\ast$-regularity. The argument can be adapted to our model, but around the left edge, again with the $\eta_\ast$-regularity as the input. The derivation is almost the same, and thus we do not reproduce it here. 

Further, similarly to Lemma \ref{lem: left edge rigidity XX^T}, we can prove
$
|\lambda_M(\mathcal{S}(V_t))-\lambda_{-}^{\mathsf{mp}}|\prec N^{-2\epsilon_b} $ and $ |\lambda_M(\mathcal{S}(Y))-\lambda_{-}^{\mathsf{mp}}|\prec N^{-2\epsilon_b}. 
$
By (\ref{0818200}), and the crude lower bound on $\lambda_M(\mathcal{S}(V_t))$ implied by \cite{Tik16}, we also have
$
|\lambda_M(\mathcal{S}(V_t))-\lambda_{-,t}|\prec N^{-\frac23+\epsilon}. 
$
Hence, we have 
\begin{align}
|\lambda_M(\mathcal{S}(Y))-\lambda_{-,t}|\prec N^{-2\epsilon_b}.  \label{0818334}
\end{align}

With the aid of the $m^1$ analogue of  (\ref{0818200}),  (\ref{0818333}) and  (\ref{0818334}), the remaining reasoning is routine and thus we omit it; see the proof of Theorem 1.4 in \cite{HLY}, for instance.
\end{proof}

\subsection{Green function comparison for edge universality}

\begin{theorem}[Green function comparison]\label{Green function comparison thm}
	Let $F: \mathbb{R} \to \mathbb{R}$ be a function whose derivatives satisfy
	\begin{align*}
		\max_x |F^{\alpha}(x)|(|x|+1)^{-C_1} \le C_1, \quad  \alpha = 1,\cdots,d
	\end{align*}
	for some constant $C_1 > 0$ and sufficiently large integer $d > 0$. Let $\Psi$ be good. Then there exist $\epsilon_0 > 0$, $N_0 \in \mathbb{N}$ and $\delta_1 > 0$ depending  on $\epsilon_a$ such that for any $\epsilon < \epsilon_0$, $N \ge N_0$ and real numbers $E, E_1$ and $E_2$ satisfying $|E|, \;|E_1|, \;|E_2|\; \le N^{-2/3+\epsilon}$, $\eta_0=N^{-2/3-\epsilon}$, we have
	\begin{align}\label{eq:GFcomparison 2}
		\Big| \E_{\Psi} \Big[ F\Big( N\int_{E_1}^{E_2} \Im m^1(\lambda_{-,t} + y+\mathrm{i}\eta_0) \,\mathrm{d}y \Big) \Big] - \E_{\Psi} \Big[ F\Big( N\int_{E_1}^{E_2} \Im m^0(\lambda_{-,t} + y+\mathrm{i}\eta_0) \,\mathrm{d}y \Big) \Big]  \Big|\le C N^{-\delta_{1}},
	\end{align}
	for some constant $C>0$, and in the case $\alpha=8/3$, (\ref{eq:GFcomparison 2}) holds with $\lambda_{-,t}$ replaced by $\lambda_{\mathsf{shift}}$.
\end{theorem}

Employing the above comparison inequality along with the rigidity estimate int Theorem \ref{rigidity}, we can deduce the following universality result around the random edge $\lambda_{-,t}$ {(and deterministic edge $\lambda_{\mathsf{shift}}$ if $\alpha = 8/3$)}, whose proof will be stated in the Appendix \ref{4216}. 
\begin{corollary}\label{cor.081801}
For all $s \in \mathbb{R}$, we have 
\begin{align}
	\lim_{N\to \infty}\mathbb{P} \Big(N^{2/3}(\lambda_M(\mathcal{S}(V_t)) -\lambda_{-,t} ) \le s \Big) = \lim_{N\to \infty} \mathbb{P} \Big(N^{2/3}(\lambda_M(\mathcal{S}(Y)) -\lambda_{-,t} ) \le s \Big).
\label{eq:convergeindistribution1}\end{align}
Moreover, if $\alpha = 8/3$, we have
\begin{align}
	\lim_{N\to \infty}\mathbb{P} \Big(N^{2/3}(\lambda_M(\mathcal{S}(V_t)) -\lambda_{\mathsf{shift}} ) \le s \Big) =\lim_{N\to \infty}\mathbb{P} \Big(N^{2/3}(\lambda_M(\mathcal{S}(Y)) -\lambda_{\mathsf{shift}} ) \le s \Big).
\label{eq:convergeindistribution2}\end{align}
\end{corollary}

Now we can prove our main theorem: Theorem \ref{main.thm}.
\begin{proof}[Proof of Theorem \ref{main.thm}]
The conclusions (i)-(iii) in Theorem \ref{main.thm} follows from (\ref{eq:convergeindistribution1}) in Corollary \ref{cor.081801}  and Theorem \ref{thm: lambda -,t asymp}. To prove the critical case when $\alpha = 8/3$, i.e., (iv), from (\ref{092401}) in Theorem \ref{thm:081601}, it is easy to show that the distribution of $\lambda_{-,t}$ is asymptotically independent of the fluctuation of $\lambda_{M}\big(\mathcal{S}(V_{t}))-\lambda_{-,t}$ since the former is a function of $X$ only. It can be shown by a standard characteristic function argument that for any $s \in \mathbb{R}$,
\begin{equation*}
\lim_{N\to \infty}\mathbb{P} \Big( \gamma_{N}M^{2/3}(\lambda_{M}\big(\mathcal{S}(V_{t}))-\lambda_{\mathsf{shift}}\big) \le s \Big) = \lim_{N\to \infty} \mathbb{P} \Big(M^{2/3}(\mu_M^{\mathsf{GOE}} + 2 + \gamma_{N}\mathcal{X}_{\alpha} ) \le s \Big). 
\end{equation*}
where in the RHS $\mathcal{X}_\alpha$ is independent of $\mathsf{GOE}$. Then further, 
together with the comparison (\ref{eq:convergeindistribution2}) we conclude (iv).  Hence, we complete the proof of Theorem \ref{main.thm}.
\end{proof}

\section{Proofs for the Green function comparisons}\label{s. proof-general}
In this section, we will mainly prove the Green function comparisons stated in the last section. We will show the details for Theorems \ref{thm:comparison1} and \ref{thm:comparison2} only. The proof of Theorem \ref{Green function comparison thm} is similar to Theorem \ref{thm:comparison2}, and thus will only be discussed briefly here and the details are stated in the Appendix \ref{4271}. 
\subsection{Some further notations}
Let us introduce some additional notations.
We denote by $\mathsf{E}_{(ij)}$ the standard basis for $\mathbb{R}^{M \times N}$, i.e., $[\mathsf{E}_{(ij)}]_{ab} \coloneqq \delta_{ia}\delta_{jb}$.
Replacement matrix notation: For any $A \in \mathbb{R}^{M \times N}$, the replacement matrix $A_{(ij)}^{\lambda}= A_{(ij)}(\lambda)\in \mathbb{R}^{M \times N}$ is defined as,
\begin{equation}
  \big[ A_{(ij)}(\lambda)\big]_{ab} :=
\begin{cases}
\lambda & \text{if } (i,j) = (a,b)
\\
 A_{ab}& \text{if } (i,j) \neq (a,b)
\end{cases},\quad a\in [M], \quad b \in [N].
\label{eq:replacementMatrix}\end{equation}
Let $G_{(ij)}^{\gamma,\lambda}(z) := ( \mathcal{S}(Y_{(ij)}^{\gamma,\lambda}) - z)^{-1}$ be the resolvent of $\mathcal{S}(Y_{(ij)}^{\gamma,\lambda})$ with  $Y_{(ij)}^{\gamma,\lambda}  = (Y^\gamma)_{(ij)}({\lambda})$. We define
\begin{align}
	&d_{ij}(\gamma,w_{ij}) := \gamma (1-\chi_{ij})a_{ij} + \chi_{ij} b_{ij} + (1-\gamma^2)^{1/2}t^{1/2}w_{ji}, \notag\\
	&e_{ij}(\gamma,w_{ij}) := c_{ij} + (1-\gamma^2)^{1/2}t^{1/2}w_{ij}, \qquad i \in [M], \quad j \in [N]. 
\label{eq:dijeij}\end{align}
In the sequel, for brevity, we also write $\sum_{i,j} = \sum_{i=1}^M \sum_{j=1}^N$.

\subsection{Proof of Proposition \ref{prop:locallawentries}}\label{1729}
Let us prove Proposition \ref{prop:locallawentries} assuming that Theorem \ref{thm:comparison1} holds. The proof of Theorem \ref{thm:comparison1} is deferred to the next subsection.
For $\delta > 0$ and $z = E + \mathrm{i}\eta \in \mathsf{D}(\varepsilon_1, \varepsilon_2,\varepsilon_3)$ (cf. (\ref{eq:domainD})),  we define
\begin{align*}
&\mathfrak{P}_{0}(\delta,z,\Psi) \coloneqq\mathbb{P}_{\Psi}\Big( \sup_{0 \le \gamma \le 1}\sup_{a , b \in [M]} |z^{1/2}\mathfrak{X}_{ab}[G^{\gamma}(z)]_{ab}| >  N^{\delta}\Big),\\
	&\mathfrak{P}_{1}(\delta,z,\Psi) \coloneqq \mathbb{P}_{\Psi}\Big( \sup_{0 \le \gamma \le 1}\sup_{u , v \in [N]} |z^{1/2}\mathfrak{Y}_{uv}[\mathcal{G}^{\gamma}(z)]_{uv}| > N^{\delta}\Big), \\
	&\mathfrak{P}_{2}(\delta,z,\Psi) \coloneqq \mathbb{P}_{\Psi}\Big( \sup_{0 \le \gamma \le 1}\sup_{a \in [M], u \in [N]} |\mathfrak{Z}_{au}[G^{\gamma}(z)Y^\gamma ]_{au}|  > N^{\delta}\Big).
	\end{align*}

\noindent
The following monotonicity lemma will be a useful tool.
\begin{lemma}\label{lem:monotonlema}
	Suppose that $\Psi$ is good. Fix $\varepsilon$ and $\omega$ as in Theorem \ref{thm:comparison1}. 
	 For all $z = E + \mathrm{i}\eta \in \mathsf{D}(\varepsilon_1, \varepsilon_2,\varepsilon_3) $, we set $z' = E' + \mathrm{i}\eta'$ by 
		\begin{align}\label{eq:E2eta2}
		E' = E +  \frac{ (1-N^{\varepsilon/3})(\sqrt{E^2+\eta^2}-E)}{2}, \quad \eta' = N^{\varepsilon/6}\eta.
	    \end{align}
	 Then for any $\delta > 0$ and $D > 0$, there exists a large constant $C > 0$ such that
	\begin{align}\label{eq:2862}
		\max_{\substack{k \in \{0,1,2\} }}\mathfrak{P}_{k}(\delta,z,\Psi) \le CN^C \max_{\substack{k \in \{0,1,2\}  }}\mathfrak{P}_{k}(\varepsilon/2,z',\Psi) + CN^{-D}.
	\end{align}
\end{lemma}
\begin{proof}
This is a minor modification of \cite[Lemma 4.3]{ALY}. The proof requires Theorem \ref{thm:comparison1}. For brevity, the detail is provided in the Appendx \ref{4103}.
\end{proof}
With the above lemma, we can prove Proposition \ref{prop:locallawentries}. 

\begin{proof}[Proof of Proposition \ref{prop:locallawentries}]
The proof is similar to the proof of Proposition 3.17 in \cite{ALY}. Let $\varepsilon$ be as in Theorem \ref{thm:comparison1}. It follows from Lemma \ref{lem:monotonlema} that for any $z_0 =\lambda_-^{\mathsf{mp}} + E_0 + \mathrm{i}\eta_0 \in \mathsf{D}(2\varepsilon_1,\varepsilon_2,\varepsilon_3)$ and $\eta_0 \le N^{-\varepsilon}$, we may find $z_1 = \lambda_-^{\mathsf{mp}} + E_1 + \mathrm{i}\eta_1$ defined through  (\ref{eq:E2eta2}) such that for any $\delta > 0$,
\begin{align}
	\max_{\substack{k \in [0:2] }}\mathfrak{P}_{k}(\delta,z_0,\Psi) \le C_1N^{C_1} \max_{\substack{k \in [0:2]  }}\mathfrak{P}_{k}(\varepsilon/2,z_1,\Psi) + C_1N^{-D}.
\label{eq:constep1}\end{align}
Now it suffices to bound $\max_{\substack{k \in [0:2]  }}\mathfrak{P}_{k}(\varepsilon/2,z_1,\Psi)$. Notice that for $\varepsilon > 3\varepsilon_1$
\begin{align*}
			|E_1| \lesssim |E_0| + N^{\varepsilon/3}|\eta_0| \lesssim N^{-2\epsilon_1} + N^{-2/3\varepsilon} \ll N^{-\varepsilon_1}.
\end{align*}
This means that $z_1 \in (\varepsilon_1,\varepsilon_2,\varepsilon_3)$. Applying Lemma \ref{lem:monotonlema} again with $\delta = \varepsilon/2$, we can find $z_2 = \lambda_-^{\mathsf{mp}} +  E_2 + \mathrm{i}\eta_2$ 
\begin{align*}
	\max_{\substack{k \in [0:2] }}\mathfrak{P}_{k}(\varepsilon/2,z_1,\Psi) \le C_2N^{C_2} \max_{\substack{k \in [0:2]  }}\mathfrak{P}_{k}(\varepsilon/2,z_2,\Psi) + C_2N^{-D},
\end{align*}
where $\eta_2 = N^{\varepsilon/6}\eta_1$ and $|E_2| \ll N^{-\varepsilon_1}$. We may now repeat the above procedure until $z_m = \lambda_-^{\mathsf{mp}} + E_m + \mathrm{i}\eta_m$ with $\eta_m \ge KN^{-\varepsilon/2}$ for some sufficiently large $K$. It can be computed that
\begin{align*}
	\eta_m \lesssim N^{-\varepsilon/2}N^{\varepsilon/6} = N^{-\varepsilon/3},\quad \text{and} \quad    |E_m| \lesssim  |E_0| + \sum_{i=1}^{m-1}N^{\varepsilon/3}\eta_i, \quad \eta_i = N^{\varepsilon i/6} \eta_0.
\end{align*}
This implies that $|E_m| \lesssim |E_0| + N^{-\varepsilon/2} \ll N^{-\varepsilon_1}$. Then using the fact that $\max_{\substack{k \in [0:2] }}\mathfrak{P}_{k}(\varepsilon/2,z_m,\Psi) = 0$, we can obtain that
\begin{align*}
	\max_{\substack{k \in [0:2] }}\mathfrak{P}_{k}(\delta,z_0,\Psi) \le C_1N^{C_1} \max_{\substack{k \in [0:2]  }}\mathfrak{P}_{k}(\varepsilon/2,z_1,\Psi) + C_1N^{-D} \le C_mN^{-D}.
\end{align*}
The claim now follows by adjusting constants.
\end{proof}

\subsection{Proof of Theorem \ref{thm:comparison1}}
We need the following elementary resolvent expansion formula.
\begin{lemma}
For any deterministic matrix $A \in \mathbb{R}^{M\times N}$, let its linearisation $\mathcal{L}(A)$ be defined as
	\begin{align}
	\mathcal{L}(A) = \left( 
	\begin{array}{cc}
		0 & A\\
		A^\top & 0 \\
	\end{array}
	\right). \label{081611}
\end{align}
Let $\mathcal{R}(A,z) = (z^{1/2}\mathcal{L}(A)-z)^{-1}$ be the resolvent of $\mathcal{L}(A)$. The Schur complement formula also gives
\begin{align*}
	\mathcal{R}(A,z) =  \left( 
	\begin{array}{cc}
		G(A,z) & z^{-1/2}G(A,z)A\\
		z^{-1/2}A^\top G(A,z) & G(A^\top,z) \\
	\end{array}
	\right).
\end{align*}
Then for any $B = A + \Delta \in \mathbb{R}^{M\times N}$, we have for any integer $s \ge 0$
\begin{align*}
	\mathcal{R}(A,z) = \sum_{j=0}^s \big(\mathcal{R}(B,z)\mathcal{L}(z^{1/2}\Delta) \big)^j\mathcal{R}(B,z) +  \big(\mathcal{R}(B,z)\mathcal{L}(z^{1/2}\Delta)\big)^{s+1}\mathcal{R}(A,z).
\end{align*}
\label{lem:entryperturb}\end{lemma}

\begin{proof}[Proof of Theorem \ref{thm:comparison1}]
	During the proof, we omit the $z$ dependence and write $d_{ij} = d_{ij}(\gamma,w_{ij})$ and $e_{ij} = e_{ij}(\gamma,w_{ij})$ for simplicity. We only show the proof for $(\#_1,\#_2,\#_3) =(\mathfrak{X}_{ab}\Im [G^\gamma(z)]_{ab}, \;\mathfrak{X}_{ab}\Im [G^0(z)]_{ab}, \;\mathfrak{I}_{0,ab}   )$ with  $a \in \mathcal{T}_r$ or $b \in \mathcal{T}_r$, and the others can be proved similarly. 
	Observing that 
	\begin{align*}
		&\frac{\partial\E_{\Psi}\big( F([\Im G^\gamma]_{ab})\big)}{\partial \gamma} = -\sum_{i,j}\E_\Psi\Big[F^{(1)}(\Im[ G^\gamma]_{ab})
		\Im\big([G^\gamma]_{ib}[ G^\gamma Y^\gamma]_{aj}  \big)
		 \Big( \mathsf{A}_{ij} - \frac{\gamma t^{1/2}w_{ij}}{(1-\gamma^2)^{1/2}} \Big) \Big]\\
		 &  -\sum_{i,j}\E_\Psi\Big[F^{(1)}(\Im[ G^\gamma]_{ab})
		\Im\big([G^\gamma]_{ai}[(Y^\gamma)^{\top} G^\gamma]_{jb}  \big)
		 \Big( \mathsf{A}_{ij} - \frac{\gamma t^{1/2}w_{ij}}{(1-\gamma^2)^{1/2}} \Big) \Big] = -\sum_{i,j} \Big[(I)_{ij} + (II)_{ij}\Big],
	\end{align*}
	and therefore it suffices to show that there exists some constant $C$ such that
	\begin{align}
		\sum_{i,j} \Big[ |(I)_{ij}| + |(II)_{ij}|\Big] \le \frac{C}{(1-\gamma^2)^{1/2}}  \big(N^{-\omega} (\mathfrak{I}_{0,ab} + 1) +Q_0N^{C_0+C}\big).
	\label{eq:altercomparisonG}\end{align}
	We will focus on the estimation for $(I)_{ij}$'s , while the estimates for the $(II)_{ij}$'s can be handled in an identical fanshion. 
 To ease the presentation, we further define the shorthand notation 
\begin{align*}
	f_{(ij)}(\lambda) &= U_{(ij)}(\lambda)V_{(ij)}(\lambda),\\ 	
	 U_{(ij)}(\lambda) = F^{(1)}\big(\Im\big[ G_{(ij)}^{\gamma,\lambda}\big]_{ab}\big),& \quad  V_{(ij)}(\lambda) = \Im\big([G_{(ij)}^{\gamma,\lambda}]_{ib}[ G_{(ij)}^{\gamma,\lambda}Y_{(ij)}^{\gamma,\lambda}]_{aj}  \big) .
\end{align*}
We also define $\tilde{V}_{(ij)}(\lambda) = \Im\big([G_{(ij)}^{\gamma,\lambda}]_{ai}[ (Y_{(ij)}^{\gamma,\lambda})^{\top} G_{(ij)}^{\gamma,\lambda}]_{jb}  \big) $.
Then for any $i \in [M], j \in [N]$, $(I)_{ij}$ can be rewritten as
\begin{align*}
	(I)_{ij} 
		 &= \E_\Psi\Big[ f_{(ij)}([Y^\gamma]_{ij})
		 \Big( \mathsf{A}_{ij} - \frac{\gamma t^{1/2}w_{ij}}{(1-\gamma^2)^{1/2}} \Big)\cdot (\mathbf{1}_{\psi_{ij} = 0} + \mathbf{1}_{\psi_{ij} = 1})  \Big]  \\
		 &= \E_\Psi\Big[ f_{(ij)}(d_{ij})
		 \Big( \mathsf{A}_{ij} - \frac{\gamma t^{1/2}w_{ij}}{(1-\gamma^2)^{1/2}} \Big) \Big] \cdot \mathbf{1}_{\psi_{ij} = 0} + \E_\Psi\Big[ f_{(ij)}(e_{ij})
		 \Big( \mathsf{A}_{ij} - \frac{\gamma t^{1/2}w_{ij}}{(1-\gamma^2)^{1/2}} \Big) \Big] \cdot  \mathbf{1}_{\psi_{ij} = 1} \\
		  &\overset{(\ast)}{=} \E_\Psi\Big[ f_{(ij)}(d_{ij})
		 \Big( \mathsf{A}_{ij} - \frac{\gamma t^{1/2}w_{ij}}{(1-\gamma^2)^{1/2}} \Big)  \Big]\cdot \mathbf{1}_{\psi_{ij} = 0} - \frac{\gamma }{(1-\gamma^2)^{1/2}}t^{1/2} \E_\Psi\Big[ w_{ij}f_{(ij)}(e_{ij})
	 \Big]	\cdot \mathbf{1}_{\psi_{ij} = 1}  \\
		 &= (J_1)_{ij} - (J_2)_{ij},
\end{align*}
where in $(\ast)$, we used the fact that $\mathsf{A}_{ij} = 0$ if $\psi_{ij} = 1$.

 Let us consider $(J_2)_{ij}$ first. Applying Gaussian integration by parts on $w_{ij}$, we have
\begin{align*}
	&|(J_2)_{ij}| = \Big| \frac{\gamma t^{1/2}}{(1-\gamma^2)^{1/2}N} \E_\Psi\Big[\partial_{w_{ij}} f_{(ij)}(e_{ij})
		 \Big] \mathbf{1}_{\psi_{ij} = 1} \Big|  \\
		 &\le \frac{\gamma t^{1/2}}{(1-\gamma^2)^{1/2}N} \E_\Psi\Big[ \big| \partial_{w_{ij}} f_{(ij)}(e_{ij}) \mathbf{1}_{\Omega} \big|\Big] \mathbf{1}_{\psi_{ij} = 1}  + \frac{\gamma t^{1/2}}{(1-\gamma^2)^{1/2}N} \E_\Psi\Big[ \big| \partial_{w_{ij}} f_{(ij)}(e_{ij}) \mathbf{1}_{\Omega^c} \big|\Big] \mathbf{1}_{\psi_{ij} = 1}.
\end{align*}
Notice that
\begin{align}
	\partial_{w_{ij}} f_{(ij)}(e_{ij}) = U_{(ij)}(e_{ij}) \cdot  \partial_{w_{ij}}  V_{(ij)}(e_{ij}) + V_{(ij)}(e_{ij}) \cdot  \partial_{w_{ij}}  U_{(ij)}(e_{ij}),
\label{eq:fijderi}\end{align}
and 
\begin{align}
	&\partial_{w_{ij}}  U_{(ij)}(e_{ij}) = -(1-\gamma^2)^{1/2}t^{1/2}F^{(2)}\big(\Im\big[ G_{(ij)}^{\gamma,e_{ij}}\big]_{ab}\big) \big(  V_{(ij)}(e_{ij})  + \tilde{V}_{(ij)}(e_{ij})\big), \notag\\
	&\partial_{w_{ij}}  V_{(ij)}(e_{ij}) = -(1-\gamma^2)^{1/2} t^{1/2}\Im \Big( \big[G_{(ij)}^{\gamma,e_{ij}}\big]_{ii}\big[(Y_{(ij)}^{\gamma,e_{ij}})^{\top}G_{(ij)}^{\gamma,e_{ij}}\big]_{jb}[ G_{(ij)}^{\gamma,e_{ij}}Y_{(ij)}^{\gamma,e_{ij}}]_{aj} \notag\\
	 &\quad  +   \big[G_{(ij)}^{\gamma,e_{ij}}Y_{(ij)}^{\gamma,e_{ij}}\big]_{ij}\big[G_{(ij)}^{\gamma,e_{ij}}\big]_{ib}[ G_{(ij)}^{\gamma,e_{ij}}Y_{(ij)}^{\gamma,e_{ij}}]_{aj}-  [ G_{(ij)}^{\gamma,e_{ij}}]_{ib}[ G_{(ij)}^{\gamma,e_{ij}}]_{ai} \notag\\
	 &\quad + [ G_{(ij)}^{\gamma,e_{ij}}]_{ib}[ G_{(ij)}^{\gamma,e_{ij}}]_{ai}[(Y_{(ij)}^{\gamma,e_{ij}})^\top G_{(ij)}^{\gamma,e_{ij}}Y_{(ij)}^{\gamma,e_{ij}}]_{jj} + [ G_{(ij)}^{\gamma,e_{ij}}]_{ib}[ G_{(ij)}^{\gamma,e_{ij}}Y_{(ij)}^{\gamma,e_{ij}}]_{aj}[ G_{(ij)}^{\gamma,e_{ij}}Y_{(ij)}^{\gamma,e_{ij}}]_{ij} \Big).
\label{eq:UVijderi}\end{align}

When $\psi_{ij} = 1$, we have $i \in \mathcal{D}_r$ and $j \in \mathcal{D}_c$. Then $\mathbf{1}_{\Omega}\mathbf{1}_{\psi_{ij}=1} |  V_{(ij)}(e_{ij}) | \le N^{2\varepsilon}t^{-2}$ and $\mathbf{1}_{\Omega} \mathbf{1}_{\psi_{ij}=1}|\partial_{w_{ij}}  V_{(ij)}(e_{ij}) | \le N^{3\varepsilon}t^{-7/2}$.
Therefore, we may find a large constant $K_1 > 0$ such that 
\begin{align*}
	|(J_2)_{ij}| &\lesssim \frac{\mathbf{1}_{\psi_{ij} = 1} }{N^{1-3\varepsilon}t^3} \E_\Psi\Big[ \big| F^{(1)}\big(\Im\big[ G_{(ij)}^{\gamma,e_{ij}}\big]_{ab}\big) \big|\Big] + \frac{\mathbf{1}_{\psi_{ij} = 1}}{N^{1-4\varepsilon}t^3} \E_\Psi\Big[ \big| F^{(2)}\big(\Im\big[ G_{(ij)}^{\gamma,e_{ij}}\big]_{ab}\big) \big|\Big]  \\
	 &\quad+ \frac{t^{1/2}\mathbf{1}_{\psi_{ij} = 1}}{N}\E_\Psi\Big[ \big| \partial_{w_{ij}} f_{(ij)}(e_{ij}) \mathbf{1}_{\Omega^c} \big|\Big] \\
	 &\lesssim \frac{\mathbf{1}_{\psi_{ij} = 1}}{N^{1-3\varepsilon}t^3} \E_\Psi\Big[ \big| F^{(1)}\big(\Im\big[ G_{(ij)}^{\gamma,e_{ij}}\big]_{ab}\big) \big|\Big] + \frac{ \mathbf{1}_{\psi_{ij} = 1}}{N^{1-4\varepsilon}t^3} \E_\Psi\Big[ \big| F^{(2)}\big(\Im\big[ G_{(ij)}^{\gamma,e_{ij}}\big]_{ab}\big) \big|\Big]  + N^{K_1} Q_0\mathbf{1}_{\psi_{ij} = 1},
\end{align*}
where in the second step, we used the crude bound that $|\partial_{w_{ij}} f_{(ij)}(e_{ij})| \le N^{K_1}$ for some sufficiently large $K_1$, which can be obtained by the fact that $\Im z > N^{-1}$. By the facts that $\sum_{i,j}\mathbf{1}_{\psi_{ij} = 1} \le N^{1-\epsilon_\alpha}$ {and $t \gg N^{-\epsilon_\alpha/4}$} , we can choose $\varepsilon < \epsilon_\alpha/16$ to obtain that
\begin{align*}
	|(J_2)_{ij}| \lesssim \frac{\mathbf{1}_{\psi_{ij} = 1}}{N^{1-\epsilon_\alpha/2}}\E_\Psi\Big[ \big| F^{(1)}\big(\Im\big[ G_{(ij)}^{\gamma,e_{ij}}\big]_{ab}\big) \big| +\big| F^{(2)}\big(\Im\big[ G_{(ij)}^{\gamma,e_{ij}}\big]_{ab}\big) \big|\Big]  + N^{K_1} Q_0\mathbf{1}_{\psi_{ij} = 1}.
\end{align*}

Next, we consider $(J_1)_{ij}$. Recall that $d_{ij} =\gamma (1-\chi_{ij})a_{ij} + \chi_{ij} b_{ij} + (1-\gamma^2)^{1/2}t^{1/2}w_{ji}$. Applying Taylor expansion on $f(d_{ij})$ around $0$, for an $s_1$ to be chosen later, we have
\begin{align*}
	(J_1)_{ij} &= \sum_{k = 0}^{s_1}  \frac{\mathbf{1}_{\psi_{ij} = 0} }{k!} \E_\Psi\Big[ (d_{ij})^kf^{(k)}_{(ij)}(0)		 \Big( \mathsf{A}_{ij} - \frac{\gamma t^{1/2}w_{ij}}{(1-\gamma^2)^{1/2}} \Big)  \Big] \\
	 &\quad +  \frac{\mathbf{1}_{\psi_{ij} = 0}}{(s_1+1)!}\E_\Psi\Big[ (d_{ij})^{s_1+1}  f^{(s_1+1)}_{(ij)}(\tilde{d} _{ij}) \Big( \mathsf{A}_{ij} - \frac{\gamma t^{1/2}w_{ij}}{(1-\gamma^2)^{1/2}} \Big)  \Big] = \sum_{k = 0}^s(J_1)_{ij,k} + \mathsf{Rem}.
\end{align*}
where $\tilde{d}_{ij} \in [0, d_{ij}]$. Before proceeding to the estimation of $(J_1)_{ij,k}$ and $\mathsf{Rem}$, we first establish perturbation bounds for the entries of the resolvents, which are useful for the estimation of $f^{(k)}_{(ij)}(0)$ and $f^{(k)}_{(ij)}(\tilde{d} _{ij})$.
Using Lemma \ref{lem:entryperturb} and the notation therein, we have for any $\mathfrak{u},\mathfrak{v} \in [M+N]$,
\begin{align*}
	\mathbf{1}_{\Omega} \big[ R( Y_{(ij)}^{\gamma,d_{ij}},z) - R( Y_{(ij)}^{\gamma,0},z) \big]_{\mathfrak{u}\mathfrak{v}} &=\sum_{j=0}^s \mathbf{1}_{\Omega} \big[\big(R( Y_{(ij)}^{\gamma,d_{ij}},z)\mathcal{L}(z^{1/2}d_{ij}\mathsf{E}_{ij}) \big)^jR( Y_{(ij)}^{\gamma,d_{ij}},z) \big]_{\mathfrak{u}\mathfrak{v}} \\
	&\quad+ \mathbf{1}_{\Omega} \big[ \big(R( Y_{(ij)}^{\gamma,d_{ij}},z)\mathcal{L}(z^{1/2}d_{ij}\mathsf{E}_{ij})\big)^{s+1} R( Y_{(ij)}^{\gamma,0},z) \big]_{\mathfrak{u}\mathfrak{v}}.
\end{align*}
Further using the fact that $\mathbf{1}_{\Omega}|d_{ij}| \le N^{-\epsilon_b}$, $\mathbf{1}_{\Omega}|\big[R( Y_{(ij)}^{\gamma,d_{ij}},z)\big]_{\mathfrak{u}\mathfrak{v}}| \le N^{\varepsilon}/t^2$, and the crude bound $\| R( Y_{(ij)}^{\gamma,0},z)\| \le N$ when $\Im z \ge N^{-1}$, we may choose $s$ large enough to obtain that
\begin{align}
	\mathbf{1}_{\Omega} \big| \big[ R( Y_{(ij)}^{\gamma,d_{ij}},z) - R( Y_{(ij)}^{\gamma,0},z) \big]_{\mathfrak{u}\mathfrak{v}}\big| \lesssim \sum_{j=1}^s \Big(\frac{N^{2\varepsilon}}{t^4N^{\epsilon_b}}\Big)^j + \Big(\frac{N^{\varepsilon}}{t^2N^{\epsilon_b}}\Big)^{s+1}N \lesssim 1,
\label{eq:perturbbound}\end{align}
which yields directly a control of $ G_{(ij)}^{\gamma,0}$, $ G_{(ij)}^{\gamma,0}Y_{(ij)}^{\gamma,0}$, and $(Y_{(ij)}^{\gamma,0})^\top G_{(ij)}^{\gamma,0}Y_{(ij)}^{\gamma,0}$ on the event $\Omega$. Here we used the fact that $Y^\top GY$ can be written in terms of $\mathcal{G}$, which can be seen easily by singular value decomposition.  Similar estimates hold if  $ Y_{(ij)}^{\gamma,0}$ is replaced by $ Y_{(ij)}^{\gamma,\tilde{d}_{ij}}$, we omit repetitive details. By taking derivatives repeatedly similar to (\ref{eq:fijderi}) and (\ref{eq:UVijderi}), it can be easily seen that for any  integer $k \ge 0$,
\begin{align}
	f^{(k)}_{(ij)}(d_{ij})\cdot \mathbf{1}_{\psi_{ij} = 0}\cdot \mathbf{1}_{\Omega} \lesssim \frac{N^{(C_0+2k+2)\varepsilon}}{t^{2k+2}}.
\label{eq:fkest}\end{align}
Combining the above estimate with the perturbation bounds in (\ref{eq:perturbbound}), we have for any $x \in [0, d_{ij}]$,
\begin{align}
	f^{(k)}_{(ij)}(x) \cdot \mathbf{1}_{\psi_{ij} = 0}\cdot \mathbf{1}_{\Omega} \lesssim \frac{N^{(C_0+2k+2)\varepsilon}}{t^{2k+2}}.
\label{eq:fkest2}\end{align}

Now we may start to estimate $(J_1)_{ij,k}$ and $\mathsf{Rem}$. Using the above perturbation bounds on the event $\Omega$, we have that there exists some large $K_2 > 0$, such that 
\begin{align}
	|\mathsf{Rem}| &\le  \frac{ \mathbf{1}_{\psi_{ij} = 0} }{(s_1+1)!}\E_\Psi\Big[ \Big|(d_{ij})^{s_1+1}  f^{(s_1+1)}_{(ij)}(\tilde{d} _{ij}) \Big( \mathsf{A}_{ij} - \frac{\gamma t^{1/2}w_{ij}}{(1-\gamma^2)^{1/2}} \Big)  \Big| \cdot\mathbf{1}_{\Omega}\Big]\notag\\
	&\quad +\frac{ \mathbf{1}_{\psi_{ij} = 0}}{(s_1+1)!}\E_\Psi\Big[ \Big|(d_{ij})^{s_1+1}  f^{(s_1+1)}_{(ij)}(\tilde{d} _{ij}) \Big( \mathsf{A}_{ij} - \frac{\gamma t^{1/2}w_{ij}}{(1-\gamma^2)^{1/2}} \Big)  \Big| \cdot\mathbf{1}_{\Omega^c}\Big] \notag\\
	&\lesssim \frac{N^{(C_0+2s_1+4)\varepsilon}\mathbf{1}_{\psi_{ij} = 0}}{t^{2s_1+4}N^{1/2+\epsilon_a+(s_1+1)\epsilon_b}} + N^{K_2}Q_0\mathbf{1}_{\psi_{ij} = 0}. 
\label{eq:RemEst}\end{align}
Therefore,  with the fact that $t \gg N^{-\epsilon_b/8 }$, we may choose $\varepsilon < \epsilon_b/8$ and  $s_1 > C_0/4 + 6/\epsilon_b$ to obtain that 
\begin{align*}
	|\mathsf{Rem}| \lesssim N^{-3} \cdot  \mathbf{1}_{\psi_{ij} = 0}+N^{K_2}Q_0\cdot \mathbf{1}_{\psi_{ij} = 0}.
\end{align*}

We estimate $(J_1)_{ij,k}$ for different $k$ separately. For the case when $k$ is even, it follows from the symmetric condition that $(J_1)_{ij,k} = 0$. Thus we mainly focus on the estimation for $k$ is odd.

\textbf{Case 1: $k \ge 5$.} First note by symmetry condition, we can obtain
\begin{align*}
	|(J_1)_{ij,k}| &\lesssim  \sum_{\substack{u_1 + u_2 \ge 1, u_3 \ge 0  \\ u_1 + u_2 + u_3 = (k + 1)/2}} \E_\Psi\Big[ |\mathsf{A}_{ij}|^{2u_1}|t^{1/2}w_{ij}|^{2u_2}|b_{ij}|^{2u_3} |f^{(k)}_{(ij)}(0)|		 \Big] \mathbf{1}_{\psi_{ij} = 0} \lesssim \frac{t\mathbf{1}_{\psi_{ij} = 0}\E_\Psi[ |f^{(k)}_{(ij)}(0)|		 ]}{N^{2+(k-3)\epsilon_b}} ,
\end{align*}
where in the last step we also used the fact that $\E_{\Psi}(b_{ij}^k) \le N^{-\varepsilon_b(k-2)}\E_{\Psi}(x_{ij}^2) \lesssim N^{-1-\varepsilon_b(k-2)}$ for $k\geq 2$. 
 We need to estimate $f^{(k)}_{(ij)}(0)$ again by Taylor expansion. For an $s_2$ to be chosen later, there exists $\hat{d}_{ij} \in [0, d_{ij}]$ such that
 \begin{align}
 	|(J_1)_{ij,k}|&\lesssim \sum_{\ell = 0}^{s_2}\frac{t\mathbf{1}_{\psi_{ij} = 0} \E_\Psi[ |f^{(k+\ell)}_{(ij)}(d_{ij})|		 ]}{N^{2+(k+\ell-3)\epsilon_b}} + \frac{t\mathbf{1}_{\psi_{ij} = 0}\E_\Psi[ |f^{(k+s_2+1)}_{(ij)}(\hat{d}_{ij})|		 ]}{N^{2+(k+s_2-2)\epsilon_b}} .
\label{eq:J1est1} \end{align}
On the event $\Omega^c$, we may estimate the RHS in the above display as in the last step in (\ref{eq:RemEst}), which gives
\begin{align}
	&\Big(\sum_{\ell = 0}^{s_2}\frac{t \E_\Psi[ |f^{(k+\ell)}_{(ij)}(d_{ij})|	 \mathbf{1}_{\Omega^c}		 ]}{N^{2+(k+\ell-3)\epsilon_b}}+  \frac{t \E_\Psi[ |f^{(k+s_2+1)}_{(ij)}(\hat{d}_{ij})|	\mathbf{1}_{\Omega^c}		 ]}{N^{2+(k+s_2-2)\epsilon_b}}\Big)\mathbf{1}_{\psi_{ij} = 0}\lesssim N^{K_3}Q_0 \mathbf{1}_{\psi_{ij}=0 },
\label{eq:J1est2}\end{align} 
for some large $K_3 > 0$. On the event $\Omega$, we may choose $s_2 > C_0 + 30+4/\epsilon_b $ and $\varepsilon < \epsilon_b/8$ to obtain that
\begin{align}
	&\Big(\sum_{\ell = 0}^{s_2}\frac{t}{N^{2+(k+\ell-3)\epsilon_b}} \E_\Psi\Big[ |f^{(k+\ell)}_{(ij)}(d_{ij})|	\mathbf{1}_{\Omega}		 \Big]+  \frac{t}{N^{2+(k+s_2-2)\epsilon_b}} \E_\Psi\Big[ |f^{(k+s_2+1)}_{(ij)}(\hat{d}_{ij})|	 \mathbf{1}_{\Omega}		 \Big]\Big)\cdot \mathbf{1}_{\psi_{ij} = 0}\notag\\
	&\lesssim \sum_{\ell = 0}^{s_2}\frac{t\mathbf{1}_{\psi_{ij} = 0} }{N^{2+(k+\ell-3)\epsilon_b}} \E_\Psi\Big[ |f^{(k+\ell)}_{(ij)}(d_{ij})|	\mathbf{1}_{\Omega}		 \Big]+ N^{-3}\mathbf{1}_{\psi_{ij} = 0} \notag\\
	&\lesssim \sum_{\ell = 0}^{s_2}\frac{N^{2(k+\ell+1)\varepsilon}\mathbf{1}_{\psi_{ij} = 0}}{N^{2+(k+\ell-3)\epsilon_b}t^{2(k+\ell)+1}} \sum_{m=1}^{k+\ell + 1}\E_\Psi\Big[ |F^{(m)}\big(\Im\big[ G_{(ij)}^{\gamma,d_{ij}}\big]_{ab}\big)|\Big] +  N^{-3} \mathbf{1}_{\psi_{ij} = 0} .
\label{eq:J1est3}\end{align}
Collecting the above estimates and choosing $\varepsilon < \epsilon_b / 100$, we have 
\begin{align*}
	|(J_1)_{ij,k}|\lesssim \frac{ \mathbf{1}_{\psi_{ij} = 0} }{N^{2+\epsilon_b/2}}  \sum_{m=1}^{k+s_2+1} \E_\Psi\Big[ |F^{(m)}\big(\Im\big[ G_{(ij)}^{\gamma,d_{ij}}\big]_{ab}\big)|\Big]+\frac{\mathbf{1}_{\psi_{ij} = 0}}{N^{2+\epsilon_b/2}} + N^{K_3}Q_0\mathbf{1}_{\psi_{ij}=0 }.
\end{align*}

\textbf{Case 2: $k = 3$.} By direct calculation, we have
\begin{align*}
	(J_1)_{ij,3} \asymp \E_\Psi\Big[ \Big(\mathsf{A}_{ij}^4 + t\mathsf{A}_{ij}^2w_{ij}^2 + t^2w_{ij}^4 + \mathsf{A}_{ij}^2\mathsf{B}_{ij}^2 + tw_{ij}^2\mathsf{B}_{ij}^2 \Big) f^{(3)}_{(ij)}(0)		 \Big]\mathbf{1}_{\psi_{ij} = 0}. 
\end{align*}
The term $\mathsf{A}_{ij}^2\mathsf{B}_{ij}^2$ becomes null due to the definitions of $\mathsf{A}_{ij}$ and $\mathsf{B}_{ij}$. Concerning the remaining terms, we only show how to estimate the term involving $tw_{ij}^2\mathsf{B}_{ij}^2$ while the others can be handled similarly. Applying Taylor expansion, and then estimating terms on $\Omega^c$ and $\Omega$ separaterly as in (\ref{eq:J1est1})-(\ref{eq:J1est3}), we have for $s_3 > C_0/4 + \epsilon_b/2+4$ ,
\begin{align*}
	\Big|\E_\Psi \Big[ tw_{ij}^2\mathsf{B}_{ij}^2 f^{(3)}_{(ij)}(0)		 \Big] \mathbf{1}_{\psi_{ij} = 0}\Big|  &\lesssim \sum_{\ell=0}^{s_3} \frac{t\mathbf{1}_{\psi_{ij} = 0} \E_\Psi[|f^{(3+\ell)}_{(ij)}(d_{ij})	|  \mathbf{1}_{\Omega}	 ]}{N^{2+\ell\epsilon_b}}  +  N^{-3}\mathbf{1}_{\psi_{ij} = 0} +  N^{K_5}Q_0\mathbf{1}_{\psi_{ij}=0 },
\end{align*}
for some large $K_5 > 0$. Then it remains to estimate the first term of the RHS of the above inequality. For $\ell \ge 1$, the estimate is similar to (\ref{eq:J1est3}), we omit further details. Here we focus on the non-trivial term when $\ell = 0$. It is straightforward to compute that $f^{(3)}_{(ij)}(d_{ij})$ is the products of $F^{(\ell)}\big(\Im\big[ G_{(ij)}^{\gamma,d_{ij}}\big]_{ab}\big), \ell \in [4]$, and the entries of $ G_{(ij)}^{\gamma,d_{ij}}$, $ G_{(ij)}^{\gamma,d_{ij}}Y_{(ij)}^{\gamma,d_{ij}}$, and $(Y_{(ij)}^{\gamma,d_{ij}})^\top G_{(ij)}^{\gamma,d_{ij}}Y_{(ij)}^{\gamma,d_{ij}}$, where the entries' indices can be $(i,i),(j,j),(i,j), (a,i),(a,j)$, $(i,b),(j,b)$. Therefore, 
\begin{align*}
	 \frac{t}{N^2}  \E_\Psi\Big[ |f^{(3)}_{(ij)}(d_{ij})	| \cdot \mathbf{1}_{\Omega}	 \Big]\cdot \mathbf{1}_{\psi_{ij} = 0}  &\le \frac{t N^{6\varepsilon} }{N^2}\sum_{\ell = 1}^4 \E_\Psi\Big[|F^{(\ell)}\big(\Im\big[ G_{(ij)}^{\gamma,d_{ij}}\big]_{ab}\big)	| \Big]\cdot \mathbf{1}_{\psi_{ij} = 0} \cdot \mathbf{1}_{i \in \mathcal{T}_r, j \in \mathcal{T}_c}  \\
	 &\quad + \frac{N^{6\varepsilon} }{N^2t^7}\sum_{\ell = 1}^4 \E_\Psi\Big[|F^{(\ell)}\big(\Im\big[ G_{(ij)}^{\gamma,d_{ij}}\big]_{ab}\big)	| \Big]\cdot \mathbf{1}_{\psi_{ij} = 0} \cdot (1 - \mathbf{1}_{i \in \mathcal{T}_r, j \in \mathcal{T}_c}).
\end{align*} 
This eventually leads to
\begin{align*}
	|(J_1)_{ij,3}| &\lesssim  \frac{t N^{6\varepsilon} \mathbf{1}_{\psi_{ij} = 0} \mathbf{1}_{i \in \mathcal{T}_r, j \in \mathcal{T}_c}}{N^2}\sum_{\ell = 1}^4 \E_\Psi\Big[|F^{(\ell)}\big(\Im\big[ G_{(ij)}^{\gamma,d_{ij}}\big]_{ab}\big)	| \Big]  \\
	 &\quad + \frac{N^{6\varepsilon}\mathbf{1}_{\psi_{ij} = 0} (1 - \mathbf{1}_{i \in \mathcal{T}_r, j \in \mathcal{T}_c}) }{N^2t^7}\sum_{\ell = 1}^4 \E_\Psi\Big[|F^{(\ell)}\big(\Im\big[ G_{(ij)}^{\gamma,d_{ij}}\big]_{ab}\big)	| \Big]\\
	 &\quad +\frac{\mathbf{1}_{\psi_{ij} = 0}}{N^{2+\epsilon_b/2}}  \sum_{\ell=1}^{s_3+4} \E_\Psi\Big[ |F^{(\ell)}\big(\Im\big[ G_{(ij)}^{\gamma,d_{ij}}\big]_{ab}\big)|\Big] + N^{-3}\mathbf{1}_{\psi_{ij} = 0}+  N^{K_5}Q_0\mathbf{1}_{\psi_{ij}=0 }.
\end{align*}
The second term in the above display can be estimated by the fact that $|\mathcal{D}_r| \vee |\mathcal{D}_c| \le N^{1-\epsilon_d}$ for some $\epsilon_d > 0$. By the fact that $t \gg N^{-\epsilon_d / 20} \vee N^{-\epsilon_b/20}$, we then have
\begin{align*}
	|(J_1)_{ij,3}| &\lesssim  \frac{t^{1/2}  }{N^2}\sum_{\ell = 1}^4 \E_\Psi\Big[|F^{(\ell)}\big(\Im\big[ G_{(ij)}^{\gamma,d_{ij}}\big]_{ab}\big)	| \Big]\cdot \mathbf{1}_{\psi_{ij} = 0} \cdot \mathbf{1}_{i \in \mathcal{T}_r, j \in \mathcal{T}_c}  \\
	 &\quad + \frac{ \mathbf{1}_{\psi_{ij} = 0} \cdot (1 - \mathbf{1}_{i \in \mathcal{T}_r, j \in \mathcal{T}_c}) }{N^{2-\epsilon_d/2}}\sum_{\ell = 1}^4 \E_\Psi\Big[|F^{(\ell)}\big(\Im\big[ G_{(ij)}^{\gamma,d_{ij}}\big]_{ab}\big)	| \Big]\\
	 &\quad +\frac{ \mathbf{1}_{\psi_{ij} = 0}}{N^{2+\epsilon_b/2}}  \sum_{\ell=1}^{s_3+4} \E_\Psi\Big[ |F^{(\ell)}\big(\Im\big[ G_{(ij)}^{\gamma,d_{ij}}\big]_{ab}\big)|\Big] + N^{-3}\cdot \mathbf{1}_{\psi_{ij} = 0}+  N^{K_5}Q_0\cdot \mathbf{1}_{\psi_{ij}=0 }.
\end{align*}
\textbf{Case 3: $k = 1$.} In this case, using the fact that $\E_{\Psi}[b_{ij}] = 0$, we may compute
\begin{align*}
	(J_1)_{ij,1} = \gamma\E_\Psi\Big[ \Big((1-\chi_{ij})^2a^2_{ij} -  tw^2_{ij}\Big) f^{(1)}_{(ij)}(0)		\Big]\cdot \mathbf{1}_{\psi_{ij} = 0}. 
\end{align*}
Recall that $t = N\E (\mathsf{A}_{ij}^2) = N\E\big((1-\psi_{ij})^2(1-\chi_{ij})^2a_{ij}^2\big)$. This gives
\begin{align}
	\big|\E\big((1-\chi_{ij})^2a_{ij}^2 \big) - \E\big( tw_{ij}^2 \big)\big| \lesssim tN^{-1-\alpha/2+\alpha\epsilon_b} .
\label{eq:2ndmomentdiff}\end{align} 
Therefore, following the same procedure as in (\ref{eq:J1est1})-(\ref{eq:J1est3}), we can also obtain that for sufficiently large constant $s_4$,
\begin{align*}
	|(J_1)_{ij,1}| \lesssim \frac{\mathbf{1}_{\psi_{ij} = 0} }{N^{2+\epsilon_b/2}}\sum_{\ell=1}^{s_4} \E_\Psi\Big[ |F^{(\ell)}\big(\Im\big[ G_{(ij)}^{\gamma,d_{ij}}\big]_{ab}\big)|\Big]+ N^{-3}\cdot \mathbf{1}_{\psi_{ij} = 0}+  N^{K_5}Q_0\cdot \mathbf{1}_{\psi_{ij}=0 }.
\end{align*}
By combining the estimates of $(J_1)_{ij,k}$'s with $(J_2)_{ij}$'s, we can conclude that (\ref{eq:altercomparisonG}) holds when we choose $\varepsilon \le \min \{ \epsilon_\alpha, \epsilon_b, \epsilon_d \} / (100)$. 
\end{proof}

\subsection{Proof of Theorem \ref{thm:comparison2}}
Since we need to perform the comparison at a random edge, we begin with some preliminary estimates for the derivatives w.r.t.  the matrix entries of the random edge.
\begin{lemma}[\cite{DY2}, Lemma 5]
	Denote $a_{k,\pm}(t)= \Phi_t(\zeta_{k,\pm}(t))$, $1 \le k \le q$. Then $(a_{k,\pm}(t),\zeta_{k,\pm}(t))$ are real solutions of 
	\begin{align*}
		F_t(z,\zeta) = 0,\quad \text{and}\quad \frac{\partial F_t}{\partial \zeta}(z,\zeta) = 0,
	\end{align*}
	where
	\begin{align*}
		F_t(z,\zeta) = 1 + \frac{t(1-c_N)-\sqrt{t^2(1-c_N)^2 + 4\zeta z}}{2\zeta} -c_Ntm_X(\zeta).
	\end{align*}
\label{lem:Fexpansion}\end{lemma}
Using the lemma above, we can derive bounds for the derivatives of the random edge $\lambda_{-,t}$ w.r.t.  the matrix entries $b_{ij}$.
\begin{lemma}Suppose that $\Psi$ is good. If we view $\lambda_{-,t}$ as a function of $\mathsf{B}_{ij}, i \in [M], j \in [N]$. For any $i \in [M]$ and  $j \in [N]$, write $\lambda_{-,t}(x) = \lambda_{-,t}(\mathsf{B}_{ij} = x)$. Then for any integer $k \ge 1$ and for any $b \in [0,\mathsf{B}_{ij}]$, we have
 \begin{align}
	 	\Big|\frac{\partial^k \lambda_{-,t}}{\partial \mathsf{B}_{ij}^k}(b) \Big|\mathbf{1}_{\psi_{ij}=0}  \prec \frac{1}{Nt^{2k+1}}, \qquad \Big| \frac{\partial^k \zeta_{-,t}}{\partial \mathsf{B}_{ij}^k}(b)\Big|  \mathbf{1}_{\psi_{ij}=0} \prec \frac{1}{Nt^{2k+1}}.
	 \label{eq:hpblambdazetaderi}\end{align}
Further, there exists some constants $C_k > 0$ such that the following deterministic bounds hold,
\begin{align}
	\Big|\frac{\partial \lambda_{-,t}}{\partial \mathsf{B}_{ij}}(b)\Big| \mathbf{1}_{\psi_{ij}=0} \le N^{C_1}, \qquad \Big|\frac{\partial^k \lambda_{-,t}}{\partial \mathsf{B}_{ij}^k}(b)\Big| \mathbf{1}_{\psi_{ij}=0} \cdot \Xi(\lambda_{-,t})  \le N^{C_k},
\label{eq:dblambdaderi}\end{align}
where $\Xi(x)$ is a smooth cut off function which equals $0$ when $x < \lambda_{-}^{\mathsf{mp}}/100$ and $1$ when $x > \lambda_{-}^{\mathsf{mp}}/2$ and $|\Xi^{(n)}(x)| = \mathcal{O}(1)$ for all $n \ge 1$.
\label{lem:lambdaderibound}\end{lemma}

\begin{remark} Here we remark that in the second estimate of (\ref{eq:dblambdaderi}), we added a cutoff function, in order to get a deterministic bound for the $\lambda_{-,t}$ derivatives, which is needed when we take expectation $\mathbb{E}_\Psi$. Hence, actually, we should work with $\E_\Psi\big( |N\eta \big(\Im m^\gamma(z_t) - \Im \tilde{m}^0(z_t)\big)\Xi(\lambda_{-,t}) |^{2p} \big)$ instead of $\E_\Psi\big( |N\eta \big(\Im m^\gamma(z_t) - \Im \tilde{m}^0(z_t)\big) |^{2p} \big)$ to make sure that all quantities in the expansions have bounded expectations. Adding such a cutoff factor will not complicates the expansions since again by the chain rule it boils down to the $\lambda_{-,t}$ derivatives. Hence, additional technical inputs are not needed for the comparison of the modified quantity. However, in order to ease the presentation, we will state the reasoning for the original quantity and proceed as if all random factors in the expansion have deterministic upper bound. 
\end{remark}

\begin{proof} Let $\psi_{ij}=0$. 
	To emphasis the dependence with $X$, we first note that $F_t(z,\zeta)$ can be rewritten as,
	\begin{align*}
		F_t(z,\zeta,X) = 1 + \frac{t(1-c_N)-\sqrt{t^2(1-c_N)^2 + 4\zeta z}}{2\zeta} - \frac{c_Nt}{M}\Tr G(X,\zeta).
	\end{align*}
	Using Lemma \ref{lem:Fexpansion}, we have
	\begin{align}
		F_t(\lambda_{-,t}, \zeta_{-,t}, X) = 0,\quad \text{and}\quad \frac{\partial F_t}{\partial \zeta}(\lambda_{-,t},\zeta_{-,t},X) = 0.
	\label{eq:Fexpression1}\end{align}
	Then taking derivative of (\ref{eq:Fexpression1}) gives
	\begin{align*}
		\frac{\partial \lambda_{-,t}}{\partial \mathsf{B}_{ij}} \frac{\partial F_t}{\partial z}(\lambda_{-,t},\zeta_{-,t},X)  + \frac{\partial F_t}{\partial x_{ij}}(\lambda_{-,t}, \zeta_{-,t}, X)) = 0.
	\end{align*}
	Therefore, we may solve the above equation to obtain that 
	\begin{align}
		\frac{\partial \lambda_{-,t}}{\partial \mathsf{B}_{ij}}  = \frac{2  c_N t \sqrt{t^2(1-c_N)^2 + 4\lambda_{-,t}\zeta_{-,t}}}{M}\big[ X^\top (G(X,\zeta_{-,t}))^2 \big]_{ji}.
	\label{eq:lambdafoderi}\end{align}
	Notice that
	\begin{align}
		\big|\big[ X^\top (G(X,\zeta_{-,t}))^2 \big]_{ji}\big| &\overset{(\mathrm{i})}{\le} \big|\big[ X^\top (G(X,\zeta_{-,t}))^2X \big]_{jj}\big|^{1/2}  \cdot \big|\big[ (G(X,\zeta_{-,t}))^2 \big]_{ii}\big|^{1/2} \notag\\
		&= \big|\big[ G(X^\top,\zeta_{-,t})\big]_{jj} + \zeta_{-,t}\big[ (G(X^\top,\zeta_{-,t}))^2\big]_{jj}\big|^{1/2}  \cdot \big|\big[ (G(X,\zeta_{-,t}))^2 \big]_{ii}\big|^{1/2} \notag\\
		&\lesssim \big(\|G(X^\top,\zeta_{-,t})\|^{1/2} + |\zeta_{-,t}| \|G(X^\top,\zeta_{-,t})\|\big) \cdot \| G(X,\zeta_{-,t})\| \overset{(\mathrm{ii})}{\prec} t^{-4},
	\label{eq:XG2zetat}\end{align}
	where in $(\mathrm{i})$ we applied Cauchy-Schwarz inequality, and in $(\mathrm{ii})$ we used Lemma \ref{lem:preEst} (i). 
	Therefore, we can obtain that $\partial_{\mathsf{B}_{ij}} \lambda_{-,t}(b)$.

	Next we view $\Phi_t(\zeta)$ as a function of $X$, and write $\Phi_t(\zeta,X) = \Phi_t(\zeta)$. By Lemma \ref{lem:Phicharacterization}, we have 
	$
		\frac{\partial \Phi_t}{\partial \zeta}(\zeta_{-,t},X) = 0.
	$
	Further taking derivative w.r.t $\mathsf{B}_{ij}$ on this equation gives
	\begin{align}
		\frac{\partial^2 \Phi_t}{\partial \zeta^2}(\zeta_{-,t},X)  \frac{\partial \zeta_{-,t}}{\partial \mathsf{B}_{ij}} + \frac{\partial^2 \Phi_t}{\partial \zeta \partial x_{ij}}(\zeta_{-,t},X)=0	.
	\label{eq:zetaderiexpression}\end{align}
	By direct calculation, we have
	\begin{align*}
		&\frac{\partial^2 \Phi_t}{\partial \zeta \partial x_{ij}}(\zeta_{-,t},X) = \frac{4c_Nt}{M}[X^\top (G(X,\zeta_{-,t}))^2 ]_{ji} - \frac{2c_N^2t^2m_{X}(\zeta_{-,t})}{M} [X^\top (G(X,\zeta_{-,t}))^2 ]_{ji}  \\
		&\quad+ \frac{8c_Nt\zeta_{-,t}(1-c_Ntm_{X}(\zeta_{-,t}))}{M} [X^\top (G(X,\zeta_{-,t}))^3 ]_{ji} - \frac{4c_N^2t^2\zeta_{-,t}m_{X}'(\zeta_{-,t})}{M}[X^\top (G(X,\zeta_{-,t}))^2 ]_{ji}\\
		&\quad+ \frac{4c_N(1-c_N)t^2}{M}[X^\top (G(X,\zeta_{-,t}))^2 ]_{ji}.
	\end{align*}
	A similar argument as in (\ref{eq:XG2zetat}) leads to $\big[X^\top (G(X,\zeta_{-,t}))^3 \big]_{ji} \prec t^{-6}$. This together with the fact that
	 $c_Ntm_{X}(\zeta_{-,t}) \prec t^{1/2}$ and $m_{X}'(\zeta_{-,t}) \sim t^{-1}$ gives
	$
	 	\frac{\partial^2 \Phi_t}{\partial \zeta \partial x_{ij}}(\zeta_{-,t},X) \prec {1}/{(Nt^5)}.
	 $
	 We can also compute that
	 \begin{align}
	 	\frac{\partial^2 \Phi_t}{\partial \zeta^2}(\zeta_{-,t},X) &= -2c_Ntm_X''(\zeta_{-,t})\zeta_{-,t}(1-c_Ntm_X(\zeta_{-,t})) - 4c_Ntm_X'(\zeta_{-,t})(1-c_Ntm_X(\zeta_{-,t}))\notag\\
	&\quad +2\zeta_{-,t}(c_Ntm_X'(\zeta_{-,t}))^2-c_N(1-c_N)t^2m_X''(\zeta_{-,t}).
	 \label{eq:2derofPhi}\end{align}
	 Using Lemma \ref{lem:esthighoderi} with the fact $\zeta_{-,t} - \lambda_M(\mathcal{S}(X)) \sim t^2$ w.h.p., we have w.h.p. that 
	$
	 	\frac{\partial^2 \Phi_t}{\partial \zeta^2}(\zeta_{-,t},X) \sim t^2.
	 $
	Combining the above bounds gives 
	$
	 	\partial_{\mathsf{B}_{ij}}  \zeta_{-,t}\prec {1}/{(Nt^3)}.
	$

	 It is worth noting that for any integer $k \ge 2$, the $\partial_{\mathsf{B}_{ij}}^k\lambda_{-,t}$ can be expressed as a function of $\partial_{\mathsf{B}_{ij}}^{\ell}\lambda_{-,t}$ and $\partial_{\mathsf{B}_{ij}}^{\ell}\zeta_{-,t}$, where $\ell$ ranges from $0$ to $k-1$. Similarly, $\partial_{\mathsf{B}_{ij}}^k\zeta_{-,t}$ is solely dependent on $\partial_{\mathsf{B}_{ij}}^{\ell}\zeta_{-,t}$, where $\ell$ ranges from $0$ to $k-1$. By employing the product rule and adopting a similar argument as used in (\ref{eq:XG2zetat}) to bound the Green function entries, we can observe that the order of $\partial_{\mathsf{B}_{ij}}^{\ell}\lambda_{-,t}$ is determined by the term that includes $\partial_{\mathsf{B}_{ij}}^{k-1}[ X^\top (G(X,\zeta_{-,t}))^2 ]_{ji}$. Similarly, the order of $\partial_{\mathsf{B}_{ij}}^{\ell}\zeta_{-,t}$ is determined by the term that includes $\partial_{\mathsf{B}_{ij}}^{k-1}[ X^\top (G(X,\zeta_{-,t}))^3 ]_{ji}$. This allows us to conclude that for any $k \ge 1$
	 \begin{align*}
	 	\frac{\partial^k \lambda_{-,t}}{\partial \mathsf{B}_{ij}^k} \prec \frac{1}{Nt^{2k+1}}, \qquad \frac{\partial^k \zeta_{-,t}}{\partial \mathsf{B}_{ij}^k} \prec \frac{1}{Nt^{2k+1}}.
	 \end{align*}
	The claim now follows by noting that the above bounds still hold when we replace $\mathsf{B}_{ij}$ in $X$ with some other $b \in [0, \mathsf{B}_{ij}]$. The reason behind this is that the replacement matrix still satisfies the $\eta^*$-regularity condition, ensuring that the corresponding $\zeta_{-,t}$ and $\lambda_{M}$ still satisfy Lemma \ref{lem:preEst} (i).

	Next, we prove a deterministic upper bound for $\partial_{\mathsf{B}_{ij}} \lambda_{-,t} $. For notational simplicity, we will only work on the original matrix $X$, and the argument holds for the replacement matrix $X_{(ij)}(b)$. In view of (\ref{eq:lambdafoderi}), it suffices to obtain deterministic upper bounds for $\lambda_{-,t}$, $\zeta_{-,t}$, and $[ X^\top (G(X,\zeta_{-,t}))^2 ]_{ji}$.  We may first apply Cauchy interlacing theorem to obtain an upper bound for $\zeta_{-,t}$ as follows:
\begin{align}
	\zeta_{-,t} \le \lambda_M(\mathcal{S}(X)) \le \lambda_{M - |\mathcal{D}_r|}(\mathcal{S}(\mathsf{B}^{(\mathcal{D}_r)}))\le N^{2-2\epsilon_b},
\label{eq:zetaminustupb}\end{align}
where in the last step we used the fact that the entries of $\mathcal{S}(\mathsf{B}^{(\mathcal{D}_r)})$ are bounded by $N^{-\epsilon_b}$. From (\ref{eq:fpe}), we have
$
	c_Ntm_X(\zeta_{-,t}) =  {c_Ntm_{t}(\lambda_{-,t})} /{(1+c_Ntm_{t}(\lambda_{-,t}))},
$
which gives the deterministic bound $m_{X}(\zeta_{-,t}) \le (c_Nt)^{-1}$. Using this deterministic bound, we have that there exists some constant $C > 0$ such that
\begin{align*}
	\frac{1}{M} \le   \frac{1}{M}\sum_{i = 1}^M \frac{\lambda_M(\mathcal{S}(X)) - \zeta_{-,t}}{\lambda_i(\mathcal{S}(X)) - \zeta_{-,t}} \le  C t^{-1}(\lambda_M(\mathcal{S}(X)) - \zeta_{-,t}).
\end{align*}
This together with the fact that $\lambda_M(\mathcal{S}(X)) \ge \zeta_{-,t}$ (cf. Lemma \ref{lem:preEst} (i)) gives $\lambda_M(\mathcal{S}(X)) - \zeta_{-,t} \ge C^{-1}t/M$.
Therefore, we are able to obtain deterministic bounds for the high order derivatives $m_X^{(k)}(\zeta_{-,t})$ as well as the spectral norm of $G(X,\zeta_{-,t})$.
 We can also obtain that 
\begin{align*}
	|\lambda_{-,t}| = \big|\big[1-c_Nt m_X(\zeta_{-,t})\big]^2\zeta_{-,t} + (1-c_N)t\big[1-c_Ntm_X(\zeta_{-,t})\big]\big| \lesssim N^{2-2\epsilon_b}.
\end{align*}
For the upper bound of $|[ X^\top (G(X,\zeta_{-,t}))^2 ]_{ji}|$, we have
\begin{align*}
	|[ X^\top (G(X,\zeta_{-,t}))^2 ]_{ji}| &\le \big|[ X^\top (G(X,\zeta_{-,t}))^2X ]_{jj}\big|^{1/2}  \cdot \big|[ (G(X,\zeta_{-,t}))^2 \big]_{ii}|^{1/2} \\
	&\le \| (G(X,\zeta_{-,t}))^2X X^\top \|^{1/2} \cdot \| (G(X,\zeta_{-,t}))^2\|^{1/2} \\
	&\le \big( \| G(X,\zeta_{-,t})\|^{1/2} + |\zeta_{-,t}|  \| G(X,\zeta_{-,t})\|\big) \cdot \| G(X,\zeta_{-,t})\| \lesssim N^{2-2\epsilon_b}M^2t^{-2}.
\end{align*}
Collecting the above bounds proves the first bound in (\ref{eq:dblambdaderi}). 

To prove the second bound in (\ref{eq:dblambdaderi}), it suffices to provide a lower bound for $\frac{\partial^2 \Phi_t}{\partial \zeta^2}(\zeta_{-,t},X)$ (cf. (\ref{eq:zetaderiexpression})). When $\lambda_{-,t} \ge \lambda_-^{\mathsf{mp}}/100$, we have
\begin{align*}
	|c_Nt m_X(\zeta_{-,t})| = \Big|\frac{c_Ntm_{t}(\lambda_{-,t})}{1 + c_Ntm_{t}(\lambda_{-,t})}\Big| \le | c_Ntm_{t}(\lambda_{-,t})| \lesssim \frac{t^{1/2}}{|\lambda_{-t}|} \lesssim t^{1/2}. 
\end{align*}
Therefore, using Cauchy-Schwarz inequality, we have
\begin{align*}
	(c_Ntm_X'(\zeta_{-,t}))^2 \le \frac{c_Ntm_X(\zeta_{-,t})   \cdot c_Ntm''_X(\zeta_{-,t})}{2} \ll c_Ntm''_X(\zeta_{-,t}).
\end{align*}
This implies that
$
	-2c_Ntm_X''(\zeta_{-,t})\zeta_{-,t}(1-c_Ntm_X(\zeta_{-,t}))+2\zeta_{-,t}(c_Ntm_X'(\zeta_{-,t}))^2  < 0.
$
Then using (\ref{eq:2derofPhi}), we may lower bound  $\frac{\partial^2 \Phi_t}{\partial \zeta^2}(\zeta_{-,t},X)$  as follows:
\begin{align*}
	\Big| \frac{\partial^2 \Phi_t}{\partial \zeta^2}(\zeta_{-,t},X) \Big| > c_N(1-c_N)t^2m_X''(\zeta_{-,t})\ge \frac{2c_N(1-c_N)t^2}{M(\lambda_{M}(\mathcal{S}(X)) -\zeta_{-,t})^3}  \ge \frac{2c_N(1-c_N)t^2}{MN^{6-6\epsilon_b}} ,
\end{align*}
where in the last step we used (\ref{eq:zetaminustupb}).
\end{proof}

Next, we start the proof of Theorem \ref{thm:comparison2}.

\begin{proof}[Proof of Theorem \ref{thm:comparison2}]

We begin by collecting some notation to simplify the presentation of the proof. Consider $\tilde{w}_{ij}$ as the $(i,j)$-entry of $\tilde{W}$, and define $\tilde{Y}^\gamma$ analogously to $Y^\gamma$, with the substitution of $W$ by $\tilde{W}$. Recall (\ref{eq:dijeij}) and we write $d_{ij} = d_{ij}(\gamma,w_{ij})$, $e_{ij} = e_{ij}(\gamma,w_{ij})$, $\tilde{d}_{ij}= d_{ij}(0,\tilde{w}_{ij})$, and $\tilde{e}_{ij}= \tilde{e}_{ij}(0,\tilde{w}_{ij})$ in the sequel. To emphasize that $\lambda_{-,t}$ is a function of $X$, we introduce the notation $\lambda^{(ij)}_{-,t}(\beta) = \lambda_{-,t}(X_{(ij)}^{\beta})$. Consequently, we define $z^{(ij)}_t(\beta) = \lambda_{-,t}^{(ij)}(\beta) + E + \mathrm{i}\eta$. For simplicity, we use the shorthand notation $G^{\gamma, \lambda, \beta}_{(ij)}$ as $G^{\gamma, \lambda}_{(ij)}\big(z_t^{(ij)}(\beta)\big)$, and we define $\tilde{G}^{\gamma, \lambda, \beta}_{(ij)}$ analogously, replacing $W$ with $\tilde{W}$. 

We will focus on the estimation of $\frac{\partial \E_\Psi( |N\eta (\Im m^\gamma(z_t) - \Im \tilde{m}^0(z_t))|^{2p} )}{\partial \gamma}$. To this end, let us define $f_{\gamma,(ab),(ij)}(\lambda,\beta) =  \Im [G^{\gamma,\lambda,\beta}_{(ij)} ]_{ab}$, $\tilde{f}_{\gamma,(ab),(ij)}(\lambda,\beta) =  \Im [\tilde{G}^{\gamma,\lambda,\beta}_{(ij)}]_{ab}$, $g_{(ij)}(\lambda,\beta) = \eta\Im [\big(G^{\gamma,\lambda,\beta}_{(ij)} \big)^2Y_{(ij)}^{\gamma,\lambda,\beta}]_{ij}$, and $F_p(\lambda,\tilde{\lambda},\beta) =( \eta\sum_{a}f_{\gamma,(aa),(ij)}(\lambda,\beta) - \eta\sum_{a}\tilde{f}_{0,(aa),(ij)}(\tilde{\lambda}_,\beta))^{p}$. Some elementary calculation gives
\begin{align*}
	\frac{\partial \E_\Psi\big( \big|N\eta \big(\Im m^\gamma(z_t) - \Im \tilde{m}^0(z_t)\big) \big|^{2p} \big) }{\partial \gamma} =-2p\sum_{i,j} \Big( (J_1)_{ij} + (J_2)_{ij}\Big),
\end{align*}
where
\begin{align*}
	&(J_1)_{ij} = \E_\Psi\Big[ g_{(ij)}(d_{ij},\chi_{ij}b_{ij})\mathcal{E}_{ij}F_{2p-1}(d_{ij},\tilde{d}_{ij},\chi_{ij}b_{ij}) \Big]\cdot \mathbf{1}_{\psi_{ij} = 0},\\
	& (J_2)_{ij} = - \frac{\gamma t^{1/2}}{(1-\gamma^2)^{1/2}} \E_\Psi\Big[ w_{ij}g_{(ij)}(e_{ij},c_{ij}) F_{2p-1}(e_{ij},\tilde{e}_{ij},c_{ij})\Big]\cdot \mathbf{1}_{\psi_{ij} = 1},\\
	& \mathcal{E}_{ij} = (1-\chi_{ij})a_{ij} - \gamma t^{1/2}(1-\gamma^2)^{-1/2} w_{ij}.
\end{align*}
For $(J_2)_{ij}$, we may apply Gaussian integration by parts to obtain that
\begin{align*}
	(J_2)_{ij} &=- \frac{\gamma t^{1/2}}{(1-\gamma^2)^{1/2}N}\Big(\E_\Psi\Big[  \partial_{w_{ij}}\big\{ g_{(ij)}(e_{ij},c_{ij})\big\}   F_{2p-1}(e_{ij},\tilde{e}_{ij},c_{ij}) \Big]\\
	&\quad +(2p-1) \E_\Psi\Big[ g_{(ij)}(e_{ij},c_{ij})  \partial_{w_{ij}} \big\{ \eta\sum_{a}f_{\gamma,(aa),(ij)}(e_{ij},c_{ij}) \big\} F_{2p-2}(e_{ij},\tilde{e}_{ij},c_{ij}) \Big]\Big) \cdot \mathbf{1}_{\psi_{ij} = 1}.
\end{align*}
Note by directly calculation, we have
\begin{align}
	\partial_{w_{ij}}\big\{ g_{(ij)}(e_{ij},c_{ij})\big\} &=\Big(\Im \big[ \big(G^{\gamma,e_{ij},c_{ij}}_{(ij)} \big)^2 \big]_{ii} - 2\Im \big(\big[  \big(G^{\gamma,e_{ij},c_{ij}}_{(ij)} \big)^2 \big]_{ii} \big[(Y_{(ij)}^{\gamma,e_{ij}})^{\top} G^{\gamma,e_{ij},c_{ij}}_{(ij)}Y_{(ij)}^{\gamma,e_{ij}} \big]_{jj}\big) \notag\\
		&\quad- 2 \Im \big(\big[  \big(G^{\gamma,e_{ij},c_{ij}}_{(ij)} \big)^2  Y_{(ij)}^{\gamma,e_{ij}}\big]_{ij} \big[ G^{\gamma,e_{ij},c_{ij}}_{(ij)}Y_{(ij)}^{\gamma,e_{ij}} \big]_{ij}\big)\Big) (1-\gamma^2)^{1/2}t^{1/2} \eta ,
	\label{eq:Recurderivative}
	\end{align}
and $\partial_{w_{ij}} \big\{\eta\sum_{a}f_{\gamma,(aa),(ij)}(e_{ij},c_{ij}) \big\} = - 2t^{1/2}(1-\gamma^2)^{1/2}g_{(ij)}(e_{ij},c_{ij})$. Using Wald's identity with the fact that $i \in \mathcal{D}_r$ and $j \in \mathcal{D}_c$ when $\psi_{ij} = 1$, we can obtain that
\begin{align}
		|\big[ \big(G^{\gamma,e_{ij},c_{ij}}_{(ij)} \big)^2 \big]_{ii} | \le \sum_{a} | \big[ G^{\gamma,e_{ij},c_{ij}}_{(ij)}\big]_{ai}|^2 = \frac{\Im \big[ G^{\gamma,e_{ij},c_{ij}}_{(ij)}\big]_{ii}}{\eta} \prec t^{-2}\eta^{-1},
	\label{eq:G2est}\end{align}
and
	\begin{align}
		|\big[(G^{\gamma,e_{ij},c_{ij}}_{(ij)})^2Y_{(ij)}^{\gamma,e_{ij}}\big]_{ij}| 		&\le \sum_{a} | \big[ G^{\gamma,e_{ij},c_{ij}}_{(ij)}Y_{(ij)}^{\gamma,e_{ij}}\big]_{aj}|^2 + \sum_{a}  |\big[ G^{\gamma,e_{ij},c_{ij}}_{(ij)} \big]_{ia}|^2 \notag\\
		&\overset{(\mathrm{i})}{=} \big[(Y_{(ij)}^{\gamma,e_{ij}})^{\top} |G^{\gamma,e_{ij},c_{ij}}_{(ij)}|^2 Y_{(ij)}^{\gamma,e_{ij}}  \big]_{jj} +\eta^{-1}\Im \big[G^{\gamma,e_{ij},c_{ij}}_{(ij)}  \big]_{ii}\notag \\
		&\overset{(\mathrm{ii})}{=} \big[(Y_{(ij)}^{\gamma,e_{ij}})^{\top} Y_{(ij)}^{\gamma,e_{ij}} |\mathcal{G}^{\gamma,e_{ij},c_{ij}}_{(ij)}|^2 \big]_{jj} + \eta^{-1}\Im \big[G^{\gamma,e_{ij},c_{ij}}_{(ij)}  \big]_{ii} \notag\\
		&= \big[\bar{\mathcal{G}}^{\gamma,e_{ij},c_{ij}}_{(ij)}  \big]_{jj} + z\big[|\mathcal{G}^{\gamma,e_{ij},c_{ij}}_{(ij)}|^2   \big]_{jj}+ \eta^{-1}\Im \big[G^{\gamma,e_{ij},c_{ij}}_{(ij)}  \big]_{ii}\overset{(\mathrm{iii})}{\prec} t^{-2}\eta^{-1},
	\label{eq:G2Yest}\end{align}
	where in $(\mathrm{i})$ we applied Wald's identity, in $(\mathrm{ii})$ we used the fact that for any $A \in \mathbb{R}^{M\times N}$, $(AA^\top - z)^{-1}A = A^\top(AA^\top - z)^{-1}$ when $z$ does not lie inside the spectrum of $A$, and in $(\mathrm{iii})$ we estimate $\big[|\mathcal{G}^{\gamma,e_{ij},c_{ij}}_{(ij)}|^2 \big]_{jj}$ in a similar way as done in (\ref{eq:G2est}).
	
	 Combining the above estimates with the fact that $\sum_{i,j}\mathbf{1}_{\psi_{ij} = 1} \le N^{1-\epsilon_\alpha}$, we arrive at 
	\begin{align*}
		|(J_2)_{ij}| &\lesssim \frac{\gamma }{(1-\gamma^2)^{1/2}N^{1-\epsilon_\alpha}} \sum_{k=1}^2\E_\Psi\Big[|\mathcal{O}_{\prec}(N^{-\epsilon_\alpha}t^{-3})|\cdot |F_{2p-k}(e_{ij},\tilde{e}_{ij},c_{ij})| \Big]\cdot \mathbf{1}_{\psi_{ij} = 1}.
	\end{align*}
	We may then apply Young's inequality as the following:
\begin{align}
		\E_\Psi\Big[|\mathcal{O}_{\prec}(N^{-\epsilon_\alpha}t^{-3})|\cdot |F_{2p-1}(e_{ij},\tilde{e}_{ij},c_{ij})| \Big] &\overset{(\ast)}{=} \E_\Psi\Big[|\mathcal{O}_{\prec}(N^{-\epsilon_\alpha}t^{-3})|\cdot \frac{|F_{2p-1}(e_{ij},\tilde{e}_{ij},c_{ij})|}{\log N} \Big] \notag\\
		&\overset{(\ast\ast)}{\lesssim}(\log N)^{\frac{2p}{1-2p} } \E_{\Psi} \Big[ F_{2p}(e_{ij},\tilde{e}_{ij},c_{ij}) \Big] +  N^{-\epsilon_\alpha p},
	\label{eq:young}\end{align}
	where in $(\ast)$ we used the definition of stochastic domination, and in ($\ast\ast)$ we used the fact that $t \gg N^{-\epsilon_\alpha/6}$. Similar argument can be applied to the second term involving $F_{2p-2}$, and therefore,
	\begin{align*}
		|(J_2)_{ij}| \lesssim \frac{\gamma }{(1-\gamma^2)^{1/2}N^{1-\epsilon_\alpha}} \Big(  (\log N)^{\frac{2p}{1-2p} }\E_{\Psi} \Big[ F_{2p}(e_{ij},\tilde{e}_{ij},c_{ij}) \Big] +  N^{-\epsilon_\alpha p/2}\Big)\cdot \mathbf{1}_{\psi_{ij} = 1}.
	\end{align*}  
	
	Next, we consider $(J_1)_{ij}$. 	 Observe that by repeatedly taking derivatives w.r.t.  $d_{ij}$, it can be easily seen that $\partial^k_{d_{ij}}\big\{ g_{(ij)}(d_{ij},\chi_{ij}b_{ij})\big\} / \eta$ can be expressed as a linear combination of the imaginary parts of $\mathfrak{A}(a_1,a_2,a_3)\cdot \mathfrak{B}$, where $\mathfrak{A}(a_1,a_2,a_3) = \big(\big[G^{\gamma,d_{ij},\chi_{ij}b_{ij}}_{(ij)}Y_{(ij)}^{\gamma,d_{ij}}\big]_{ij}\big)^{a_1} \cdot \big(\big[G^{\gamma,d_{ij},\chi_{ij}b_{ij}}_{(ij)}\big]_{ii} \big)^{a_2}\cdot \big(\big[(Y_{(ij)}^{\gamma,d_{ij}})^{\top}G^{\gamma,d_{ij},\chi_{ij}b_{ij}}_{(ij)}Y_{(ij)}^{\gamma,d_{ij}}\big]_{jj}\big)^{a_3}$ for any integer $a_1,a_2,a_3 \ge 0$, and  $\mathfrak{B} \in \{\big[(Y_{(ij)}^{\gamma,d_{ij}})^\top(G^{\gamma,d_{ij},\chi_{ij}b_{ij}}_{(ij)})^2Y_{(ij)}^{\gamma,d_{ij}}\big]_{jj}, \big[(G^{\gamma,d_{ij},\chi_{ij}b_{ij}}_{(ij)})^2Y_{(ij)}^{\gamma,d_{ij}}\big]_{ij},\big[(G^{\gamma,d_{ij},\chi_{ij}b_{ij}}_{(ij)})^2\big]_{ii}\}$. The same holds for  $\partial_{d_{ij}}^k \big\{\sum_{a}f_{\gamma,(aa),(ij)}(d_{ij},\chi_{ij}b_{ij}) \big\}$ since
	$\partial_{d_{ij}} \big\{\eta\sum_{a}f_{\gamma,(aa),(ij)}(d_{ij},\chi_{ij}b_{ij}) \big\} = - 2g_{(ij)}(d_{ij},\chi_{ij}b_{ij})$. This together with a similar argument as (\ref{eq:G2est}) and (\ref{eq:G2Yest}) implies that  
	\begin{align}
		|\mathfrak{A}(a_1,a_2,a_3)|\cdot \mathbf{1}_{\psi_{ij} = 0} \prec t^{-2(a_1 +a_2+a_3)},\quad \text{and}\quad |\mathfrak{B}|\cdot \mathbf{1}_{\psi_{ij} = 0}  \prec t^{-2}\eta^{-1}.
	\label{eq:ABEST}\end{align}      
	The estimation of $(J_1)_{ij}$ relies on a careful analysis of expansion. Here, we introduce $g^{(k_1,k_2)}_{(ij)}$ as the mixed $(k_1,k_2)$-th order derivative of $g_{(ij)}(\lambda,\beta)$ w.r.t.  $\lambda$ and $\beta$, and $F_{2p-1}^{(k_1,k_2,k_3)}$ represents the mixed $(k_1,k_2,k_3)$-th order derivative of $F_{2p-1}(\lambda,\tilde{\lambda},\beta)$ w.r.t.  $\lambda$, $\tilde{\lambda}$, and $\beta$. Applying Taylor expansion on $g_{(ij)}(d_{ij},\chi_{ij}b_{ij})$ on the first variable around $0$, we have for an $s_1$ to be chosen later, there exists $\bar{d}_{ij} \in [0,d_{ij}]$ such that,
	\begin{align*}
		(J_1)_{ij} &= \sum_{k=0}^{s_1} \E_\Psi\Big[ \frac{d_{ij}^k g^{(k,0)}_{(ij)}(0,\chi_{ij}b_{ij}) }{k!}  \mathcal{E}_{ij} F_{2p-1}(d_{ij},\tilde{d}_{ij},\chi_{ij}b_{ij}) \Big]\cdot \mathbf{1}_{\psi_{ij} = 0}  \\
		 & + \E_\Psi\Big[\frac{ d_{ij}^{s_1+1}g^{(s_1+1,0)}_{(ij)}(\bar{d}_{ij},\chi_{ij}b_{ij})}{(s_1+1)!}  \mathcal{E}_{ij} F_{2p-1}(d_{ij},\tilde{d}_{ij},\chi_{ij}b_{ij}) \Big]\cdot \mathbf{1}_{\psi_{ij} = 0} = \sum_{k=0}^{s_1} \frac{1}{k!} (J_1)_{ij,k}+ \mathsf{Rem}_1.
	\end{align*}
	By the entries bound in Proposition \ref{prop:locallawentries}, (\ref{eq:G2est}), (\ref{eq:G2Yest}), and the perturbation argument in (\ref{eq:perturbbound}), we may crudely bound the above remainder term as follows:
	\begin{align*}
		|\mathsf{Rem}_1| &\lesssim \frac{\mathbf{1}_{\psi_{ij} = 0}}{N^{(s_1+2)\epsilon_b}} \E_\Psi\Big[ | g^{(s_1+1,0)}_{(ij)}(\bar{d}_{ij},\chi_{ij}b_{ij})| \cdot|F_{2p-1}(d_{ij},\tilde{d}_{ij},\chi_{ij}b_{ij})| \Big] \lesssim \frac{N^{\epsilon}N^{2p}}{N^{(s_1+2)\epsilon_b}t^{2s_1+4}} \lesssim N^{-p},
	\end{align*}
	where in the second inequality, we used the deterministic bound $|F_{2p-1}(d_{ij},\tilde{d}_{ij},\chi_{ij}b_{ij})| \le N^{2p-1}$, and in the last step, we chose $s_1 > 6p/\epsilon_b$ and used the fact that $t \gg N^{-\epsilon_b/8}$. We may apply similar argument to expand the first two variables of $F_{2p-1}(d_{ij},\tilde{d}_{ij},\chi_{ij}b_{ij})$ in $(J_1)_{ij,k}$ to obtain that
	\begin{align*}
		(J_1)_{ij,k} &= \sum_{\ell = 0}^{s_2}\sum_{m=0}^\ell \E_\Psi\Big[ \frac{d_{ij}^{k+m}\tilde{d}_{ij}^{\ell-m} g^{(k,0)}_{(ij)}(0,\chi_{ij}b_{ij}) }{m!(\ell - m)!}  \mathcal{E}_{ij} F_{2p-1}^{(m,\ell-m,0)} (0,0,\chi_{ij}b_{ij}) \Big]\cdot \mathbf{1}_{\psi_{ij} = 0}+\mathcal{O}(N^{-p}) \\
		&= \sum_{\ell = 0}^{s_2} (J_1)_{ij,k\ell} + \mathcal{O}(N^{-p}) .
	\end{align*}
	where $s_2$ is a large integer satisfying $s_2 > 6p/\epsilon_b$. To estimate $(J_1)_{ij,k\ell}$, we start by introducing the notation $t_{ij} = t^{2} \cdot (1 - \mathbf{1}_{i \in \mathcal{T}_r, j \in \mathcal{T}_c}) + \mathbf{1}_{i \in \mathcal{T}_r, j \in \mathcal{T}_c}$ for presentation simplicity. Note by the chain rule, we have for any integer $\ell \ge 0$ and $m \le \ell$, 
	\begin{align}
		F_{2p-1}^{m,\ell-m,0}(0,0,\chi_{ij}b_{ij})  &= \sum_{k=1}^{\ell\wedge(2p-1) } \mathcal{C}_{k,m}^{\chi_{ij}b_{ij}} F_{2p-1-k}(0,0,\chi_{ij}b_{ij}) + \mathcal{C}_{\ell+1,m}^{\chi_{ij}b_{ij}}\mathbf{1}_{\ell \ge (2p-1)},
	\label{eq:deriexpansion}\end{align}
	where for all $k \in [\ell+1]$, $m \in [\ell]$, $\mathcal{C}_{k,m}^{\chi_{ij}b_{ij}}$ are polynomials of the following terms
	\begin{align*}
		&\big[G^{\gamma,0,\chi_{ij}b_{ij}}_{(ij)}Y_{(ij)}^{\gamma,0}\big]_{ij},  \big[G^{\gamma,0,\chi_{ij}b_{ij}}_{(ij)}\big]_{ii},\big[(Y_{(ij)}^{\gamma,0})^{\top}G^{\gamma,0,\chi_{ij}b_{ij}}_{(ij)}Y_{(ij)}^{\gamma,0}\big]_{jj},\big[(G^{\gamma,0,\chi_{ij}b_{ij}}_{(ij)})^2Y_{(ij)}^{\gamma,0}\big]_{ij}, \big[(G^{\gamma,0,\chi_{ij}b_{ij}}_{(ij)})^2\big]_{ii},\\
		&\big[\tilde{G}^{0,0,\chi_{ij}b_{ij}}_{(ij)}\tilde{Y}_{(ij)}^{0,0}\big]_{ij},  \big[\tilde{G}^{0,0,\chi_{ij}b_{ij}}_{(ij)}\big]_{ii}, \big[(\tilde{Y}_{(ij)}^{0,0})^{\top}\tilde{G}^{0,0,\chi_{ij}b_{ij}}_{(ij)}\tilde{Y}_{(ij)}^{0,0}\big]_{jj},\big[(\tilde{G}^{0,0,\chi_{ij}b_{ij}}_{(ij)})^2\tilde{Y}_{(ij)}^{0,0}\big]_{ij},\big[(\tilde{G}^{0,0,\chi_{ij}b_{ij}}_{(ij)})^2\big]_{ii}.
	\end{align*}
	After carrying out a similar derivation as shown in (\ref{eq:Recurderivative})-(\ref{eq:G2Yest}) and employing the perturbation argument described in (\ref{eq:perturbbound}), it can be easily verified that $\mathcal{C}_{k,m}^{\chi_{ij}b_{ij}}\cdot \mathbf{1}_{\psi_{ij}=0} \prec t_{ij}^{-(\ell+1)}$. 
	
	Plugging (\ref{eq:deriexpansion}) into $(J_1)_{ij,k\ell}$, we have
	\begin{align*}
		(J_1)_{ij,k\ell} 
		& = \sum_{n=1}^{\ell\wedge(2p-1) } \sum_{m=0}^\ell\E_\Psi\Big[ \frac{d_{ij}^{k+m}\tilde{d}_{ij}^{\ell-m} \mathcal{E}_{ij}  }{m!(\ell - m)!} g^{(k,0)}_{(ij)}(0,\chi_{ij}b_{ij}) \mathcal{C}_{n,m}^{\chi_{ij}b_{ij}}F_{2p-1-n} (0,0,\chi_{ij}b_{ij}) \Big]\cdot \mathbf{1}_{\psi_{ij} = 0}\\
		& +\sum_{m=0}^\ell \E_\Psi\Big[ \frac{d_{ij}^{k+m}\tilde{d}_{ij}^{\ell-m}\mathcal{E}_{ij} }{m!(\ell - m)!}g^{(k,0)}_{(ij)}(0,\chi_{ij}b_{ij})   \mathcal{C}_{\ell+1,m}^{\chi_{ij}b_{ij}} \mathbf{1}_{\ell \ge (2p-1)}\Big]\cdot \mathbf{1}_{\psi_{ij} = 0} = (\mathsf{T}_1)_{ij,k\ell} + (\mathsf{T}_2)_{ij,k\ell}.
	\end{align*}
	For $(\mathsf{T}_2)_{ij,k\ell}$, we only need to consider the case when $\ell \ge 2p-1$. Using $g^{(k,0)}_{(ij)}(0,\chi_{ij}b_{ij})  \prec t_{ij}^{-(k+1)}$ with the fact that  $t \gg N^{-\epsilon_b/8}$, we can conclude that $|(\mathsf{T}_2)_{ij,k\ell}| \lesssim N^{-\epsilon_b p}$. Next, we focus on the estimation of $(\mathsf{T}_1)_{ij,k\ell}$. 

When $k+\ell$ is even,   we have by the law of total expectation that,
\begin{align}
		(\mathsf{T}_1)_{ij,k\ell} 
		&=\sum_{n=1}^{\ell\wedge(2p-1) } \sum_{m=0}^\ell\E_\Psi\Big[ \frac{(\gamma a_{ij}+(1-\gamma^2)^{1/2} t^{1/2} w_{ij})^{k+m}(t^{1/2}\tilde{w}_{ij})^{\ell-m}  }{m!(\ell - m)!} g^{(k,0)}_{(ij)}(0,0) \notag \\
		&\qquad \qquad \qquad\qquad\times \Big(a_{ij} - \frac{\gamma t^{1/2} w_{ij}}{(1-\gamma^2)^{1/2}} \Big)\mathcal{C}_{n,m}^{0}F_{2p-1-n} (0,0,0)  \Big]\cdot \mathbb{P}(\chi_{ij} = 0) \cdot \mathbf{1}_{\psi_{ij} = 0}\notag\\
		&\quad-\sum_{n=1}^{\ell\wedge(2p-1) } \sum_{m=0}^\ell\E_\Psi\Big[ \frac{\gamma(b_{ij}+(1-\gamma^2)^{1/2}t^{1/2} w_{ij})^{k+m}(b_{ij} + t^{1/2}\tilde{w}_{ij})^{\ell-m} t^{1/2}w_{ij} }{(1-\gamma^2)^{1/2}m!(\ell - m)!}  \notag\\
		&\qquad \qquad \qquad\qquad\times g^{(k,0)}_{(ij)}(0,b_{ij})\mathcal{C}_{n,m}^{b_{ij}}F_{2p-1-n} (0,0,b_{ij})\Big]\cdot \mathbb{P}(\chi_{ij} =1) \cdot \mathbf{1}_{\psi_{ij} = 0}.
	\label{eq:T1excludeeven}\end{align}
From the above equation, one can easily verify that $(\mathsf{T}_1)_{ij,k\ell} = 0$ when $k + \ell$ is even. Therefore, in the rest of the estimation, we consider the case of $k + \ell$ is odd. In this case, we need to  further expand out $\chi_{ij}b_{ij}$ in $\mathcal{C}_{n,m}^{\chi_{ij}b_{ij}}$, $g^{(k,0)}_{(ij)}(0,\chi_{ij}b_{ij})$ and $F_{2p-1-n} (0,0,\chi_{ij}b_{ij})$.

First note by Taylor expansion, for any $s_3 \ge 0$ there exists $b_{ij}^{(1)} \in [0, \chi_{ij}b_{ij}]$ such that
\begin{align}
		g^{(k,0)}_{(ij)}(0,\chi_{ij}b_{ij}) = \sum_{q=0}^{s_3} \frac{(\chi_{ij}b_{ij})^q}{q!} g^{(k,q)}_{(ij)}(0,0) + \frac{(\chi_{ij}b_{ij})^{s_3+1}}{(s_3+1)!}g^{(k,s_3+1)}_{(ij)}(0,b_{ij}^{(1)}).
	\label{eq:gderi1}\end{align}
By Fa\`{a} di Bruno's formula, for $q \ge 1$, $g^{(k,q)}_{(ij)}(\lambda,\beta)$ can be expressed as
	\begin{align}
		g^{(k,q)}_{(ij)}(\lambda,\beta) = \sum_{(u_1,\cdots,u_q)} \frac{q!}{u_1!u_2!\cdots u_q!} \partial_{z}^{u_1+\cdots+u_q}  g^{(k,0)}_{(ij)}(\lambda,\beta) \cdot \prod_{v = 1}^q \Big( \frac{\partial^v \lambda_{-,t}^{(ij)} }{\partial \beta^v} (\beta)\Big)^{u_v},
	\label{eq:gderi2}\end{align}
	where the sum $\sum_{(u_1,\cdots,u_q)}$ is over all $q$-tuples of nonnegative integers $(u_1,\cdots,u_q)$ satisfying $\sum_{i=1}^q iu_i = q$. We may then use (\ref{eq:hpblambdazetaderi}) in Lemma \ref{lem:lambdaderibound} to bound the derivatives of $\lambda^{(ij)}_{-,t}$ and a Cauchy integral argument to bound the derivatives of $ g^{(k,0)}_{(ij)}$ w.r.t $z$, which gives
	\begin{align}
		g^{(k,q)}_{(ij)}(0,0)\cdot \mathbf{1}_{\psi_{ij}=0} \prec \sum_{(u_1,\cdots,u_q)}  \frac{1}{\eta^{u_1+\cdots+u_q}t_{ij}^{k+1}} \prod_{v = 1}^q  \frac{1}{N^{u_v}t^{(2v+1)u_v}} 
		\prec \frac{1}{N\eta t^{3q} t_{ij}^{k+1}},\quad q \ge 1.
	\label{eq:gkqest}\end{align}
	and the same bound holds for $g^{(k,q)}_{(ij)}(0,b_{ij}^{(1)})$. Therefore, by choosing $s_3 > 6p/\epsilon_b$ together with the facts that $\mathcal{C}_{n,m}^{\chi_{ij}b_{ij}} \prec t_{ij}^{-(k+1)}$, $|F_{2p-1-n} (0,0,\chi_{ij}b_{ij})| \lesssim N^{2p-1-n}$, we can obtain that
	\begin{align*}
		(\mathsf{T}_1)_{ij,k\ell} &=  \sum_{n=1}^{\ell\wedge(2p-1) } \sum_{q=0}^{s_3} \sum_{m=0}^\ell\E_\Psi\Big[ \frac{d_{ij}^{k+m}\tilde{d}_{ij}^{\ell-m} (\chi_{ij}b_{ij})^q \mathcal{E}_{ij} }{q!m!(\ell - m)!}g^{(k,q)}_{(ij)}(0,0)\mathcal{C}_{n,m}^{\chi_{ij}b_{ij}}\\
		&\quad \times  F_{2p-1-n} (0,0,\chi_{ij}b_{ij}) \Big]\cdot \mathbf{1}_{\psi_{ij} = 0} + \mathcal{O}(N^{-\epsilon_bp}) =  \sum_{n=1}^{\ell\wedge(2p-1) } \sum_{q=0}^{s_3} (\mathsf{T}_1)_{ij,k\ell,nq} + \mathcal{O}(N^{-\epsilon_bp}).
	\end{align*}
	For $(\mathsf{T}_1)_{ij,k\ell,nq}$, the term $\mathcal{C}_{n,m}^{\chi_{ij}b_{ij}}$ can be expanded in a similar way as done for $g^{(k,0)}_{(ij)}(0,\chi_{ij}b_{ij})$ in (\ref{eq:gderi1}) and (\ref{eq:gderi2}), we omit the details. This leads to
	\begin{align*}
		(\mathsf{T}_1)_{ij,k\ell,nq} &= \sum_{r=0}^{s_4}\sum_{m=0}^\ell \E_\Psi\Big[ \frac{d_{ij}^{k+m}\tilde{d}_{ij}^{\ell-m} (\chi_{ij}b_{ij})^{q+r} \mathcal{E}_{ij}}{r!q!m!(\ell - m)!}g^{(k,q)}_{(ij)}(0,0) \\
		&\quad \times\mathcal{C}^{(r),0}_{n,m}  F_{2p-1-n} (0,0,\chi_{ij}b_{ij}) \Big]\cdot \mathbf{1}_{\psi_{ij} = 0} + \mathcal{O}(N^{-\epsilon_bp}),
	\end{align*} 	
	where $s_4 > 6p/\epsilon_b$ and 
	\begin{align}
		\mathcal{C}^{(r),0}_{n,m}  = \frac{\partial^r\mathcal{C}_{n,m}^\beta}{\partial \beta^r}\Big|_{\beta = 0},\quad \text{and} \quad \mathcal{C}^{(r),0}_{n,m} \prec \frac{1}{N\eta t^{3r}t_{ij}^{\ell +1}},\quad r \ge 1.
	\label{eq:Crnmest}\end{align}
	Next, we deal with $F_{2p-1-n} (0,0,\chi_{ij}b_{ij})$. For any $s \ge 0$, we can compute that
	\begin{align}
		F^{(0,0,s)}_{2p-1-n}(0,0,0) = \sum_{(u_1,\cdots,u_s)} \frac{s!}{u_1!u_2!\cdots u_s!} \partial_{z}^{u_1+\cdots+u_s} F_{2p-1-n} (0,0,0) \cdot \prod_{\mathsf{w} = 1}^s \Big( \frac{\partial^\mathsf{w} \lambda_{-,t}^{(ij)} }{\partial \beta^\mathsf{w}} (0)\Big)^{u_\mathsf{w}},
	\label{eq:F00s}\end{align}
	and for any integer $\vartheta \ge 0$,
\begin{align}
		\partial_{z}^{\vartheta} F_{2p-1-n} (0,0,0)&=\sum_{\substack{(v_1,\cdots,v_\vartheta)\\v_1+\cdots+v_\vartheta \le 2p-1-n}} \frac{\vartheta!}{u_1!v_2!\cdots v_\vartheta!} F_{2p-1-n-(v_1+\cdots+v_\vartheta)} (0,0,0) \notag\\
		&\quad\times \prod_{\mathsf{w} = 1}^\vartheta \Big( \eta \Im\Tr \big(G_{(ij)}^{\gamma,0,0}\big)^\mathsf{w+1} -    \eta\Im \Tr \big(\tilde{G}_{(ij)}^{0,0,0}\big)^\mathsf{w+1}\Big)^{v_\mathsf{w}}.
	\label{eq:F00s1}\end{align}
Combining the above two expression, and using Lemma \ref{lem:lambdaderibound}, we can estimate the remainder term as done for $\mathsf{Rem}_3$, which gives 
	\begin{align}
		(\mathsf{T}_1)_{ij,k\ell,nq} &= \sum_{r=0}^{s_4}\sum_{s=0}^{s_5}\sum_{m=0}^\ell\E_\Psi\Big[ \frac{d_{ij}^{k+m}\tilde{d}_{ij}^{\ell-m} (\chi_{ij}b_{ij})^{q+r+s}\mathcal{E}_{ij} }{s!r!q!m!(\ell - m)!} \Big]\notag\\
		&\quad \times \E_\Psi \Big[g^{(k,q)}_{(ij)}(0,0) \mathcal{C}^{0,(r)}_{n,m}  F_{2p-1-n}^{(0,0,s)} (0,0,0) \Big]\cdot \mathbf{1}_{\psi_{ij} = 0} + \mathcal{O}(N^{-\epsilon_bp}),
	\label{eq:T1mnqest1}\end{align}
	for some large enough integer $s_5$. Here we also used the independency between the random variables. Then it suffices to estimate $(\mathsf{T}_1)_{ij,k\ell,nq}$ in two different cases, $k +\ell = 1$ and $k + \ell \ge 3$ (recall that we only need to consider the case when $k+\ell$ is odd, cf. (\ref{eq:T1excludeeven})).  
	
	\textbf{Case 1:} $k + \ell \ge 3$.
	From (\ref{eq:T1mnqest1}), using the estimates (\ref{eq:gkqest}) and (\ref{eq:Crnmest}), and the fact that $\E (b_{ij}^2) \lesssim N^{-1}$, $\E((1-\chi_{ij})a_{ij}^2) \asymp t\E(w_{ij}^2) = t/N $ , we have
	\begin{align}
	(\mathsf{T}_1)_{ij,k\ell,nq}
		&= \sum_{r=0}^{s_4}\sum_{s=0}^{s_5}\mathcal{O}\Big(\frac{t}{N^{2+(k+\ell+q+r+s-3)\epsilon_b } } \Big)\notag\\
		&\quad \times \E_\Psi \Big[\mathcal{O}_\prec \Big( \frac{1}{ t^{3(q+r)}t_{ij}^{k+\ell+2}} \Big)  F_{2p-1-n}^{(0,0,s)} (0,0,0) \Big]\cdot \mathbf{1}_{\psi_{ij} = 0} + \mathcal{O}(N^{-\epsilon_bp}).
	\label{eq:T1nqstart}\end{align}
Note that we have already derived the expression of $F_{2p-1-n}^{(0,0,s)} (0,0,0) $ in (\ref{eq:F00s}) and (\ref{eq:F00s1}). Then using the following inequality:
	\begin{align}
		&\big| \eta\Im \Tr \big(G_{(ij)}^{\gamma,0,0}\big)^\mathsf{w+1} -    \eta \Im\Tr \big(\tilde{G}_{(ij)}^{0,0,0}\big)^\mathsf{w+1}\big|^{v_\mathsf{w}} \lesssim 	\big| \eta \Im\Tr \big(G_{(ij)}^{\gamma,0,0}\big)^\mathsf{w+1} \big|^{v_\mathsf{w}} + \big|\eta\Im \Tr \big(\tilde{G}_{(ij)}^{0,0,0}\big)^\mathsf{w+1}\big|^{v_\mathsf{w}} \notag \\
		&\le \eta^{-\mathsf{w}v_\mathsf{w}}\big( \big| \eta\Im \Tr G_{(ij)}^{\gamma,0,0}\big|^{v_\mathsf{w}}+\big|\eta\Im \Tr \tilde{G}_{(ij)}^{0,0,0}\big|^{v_\mathsf{w}} \big)
		 \lesssim \eta^{-\mathsf{w}v_\mathsf{w}}\big(|F_{v_\mathsf{w}}(0,0,0)|+\big|\eta \Im\Tr\tilde{G}_{(ij)}^{0,0,0}\big|^{v_\mathsf{w}} \big) ,
	\label{eq:F00s3}\end{align}
	together with Lemma \ref{lem:lambdaderibound} and the fact that $\eta \Im \Tr \tilde{G}_{(ij)}^{0,0,0}\prec N\eta\sqrt{|E|+\eta} \le N^{1-\varepsilon_1/2}\eta$ (this can be done by bounding $\big(\eta \Im \Tr \tilde{G}_{(ij)}^{0,0,0} - \eta \Im \Tr \tilde{G}_{(ij)}^{0,d_{ij},\chi_{ij}b_{ij}})\cdot \mathbf{1}_{\psi_{ij} = 0}$ through Taylor expansion and then using local law for the Gaussian divisible model (cf. (\ref{eq:averagelocallawGDM})) that $\big(\eta \Im \Tr \tilde{G}_{(ij)}^{0,d_{ij},\chi_{ij}b_{ij}}- N\eta \Im m_{t}(z_t)\big)\cdot \mathbf{1}_{\psi_{ij} = 0} \prec 1$ with $\Im m_{t}(z_t) \prec\sqrt{|E| + \eta}$ (cf. (\ref{eq:Immc}))), we can obtain that
\begin{align*}
		|(\mathsf{T}_1)_{ij,k\ell,nq}| &\le \frac{1}{N^2} \sum_{\mathsf{a}=0}^{2p-1-n} \E_\Psi\Big[\mathcal{O}_\prec \Big( \frac{t}{N^{(k+\ell+\mathsf{a} -3)\epsilon_b}t^{3\mathsf{a}} t_{ij}^{(k+\ell+2)}}\Big)|F_{2p-1-n-\mathsf{a}}(0,0,0)|  \Big]\cdot \mathbf{1}_{\psi_{ij} = 0} + \mathcal{O}(N^{-\epsilon_bp}).
	\end{align*}
Substituting this back into $(\mathsf{T}_1)_{ij,k\ell}$ and considering that $t \gg N^{-\epsilon_b / 100} \vee N^{-\epsilon_d/20}$, a straightforward calculation yields that: if $k + \ell \ge 5$,
\begin{align}
		|(\mathsf{T}_1)_{ij,k\ell}| \le \frac{1}{N^2} \sum_{n=1}^{2p-1 }  \E_{\Psi} \Big[ \mathcal{O}_\prec \Big( \frac{1}{N^{(n+1)\epsilon_b/10}}\Big) |F_{2p-1-n} (0,0,0) |\Big]\cdot \mathbf{1}_{\psi_{ij} = 0}+ \mathcal{O}(N^{-\epsilon_b p}),
	\label{eq:T1est}\end{align}
	and if $k+\ell \ge 3$,
	\begin{align}
		|(\mathsf{T}_1)_{ij,k\ell}| &\le \sum_{n=1}^{\ell\wedge(2p-1) } \sum_{\mathsf{a}=0}^{2p-1-n}  \Big( \frac{\mathbf{1}_{\psi_{ij} = 0} (1- \mathbf{1}_{i\in\mathcal{T}_r, j\in \mathcal{T}_c})}{N^{2-\epsilon_d}} \E_\Psi\Big[\mathcal{O}_\prec \Big( \frac{1}{N^{\mathsf{a}\epsilon_b/10 + \epsilon_d/2}}\Big)|F_{2p-1-n-\mathsf{a}}(0,0,0)|  \Big] \notag\\
		&\quad+ \frac{\mathbf{1}_{\psi_{ij} = 0}\mathbf{1}_{i\in\mathcal{T}_r, j\in \mathcal{T}_c}}{N^{2}} \E_\Psi\Big[\mathcal{O}_\prec \Big( \frac{t}{N^{\mathsf{a}\epsilon_b/10}}\Big)|F_{2p-1-n-\mathsf{a}}(0,0,0)|  \Big] \Big)+ \mathcal{O}(N^{-\epsilon_b p}).
	\label{eq:T1est2}\end{align}
	
	Next, we shall replace  $F_{2p-1-n} (0,0,0)\cdot \mathbf{1}_{\psi_{ij} = 0}$ back by $F_{2p-1-n} (d_{ij},\tilde{d}_{ij},\chi_{ij}b_{ij})\cdot \mathbf{1}_{\psi_{ij} = 0}$. Applying Taylor expansion on the third variable and then using (\ref{eq:F00s})-(\ref{eq:F00s3}), we can obtain that
	\begin{align*}
		|F_{2p-1-n}(0,0,0)| \le \sum_{\mathsf{a}=0}^{2p-1-n}\mathcal{O}_\prec\big(N^{-\epsilon_b\mathsf{a}/10} \big) \cdot  |F_{2p-1-n-\mathsf{a}}(0,0,\chi_{ij}b_{ij}) | + \mathcal{O}_\prec(N^{-\epsilon_b p}).
	\end{align*}
	Therefore, we have that (\ref{eq:T1est}) and (\ref{eq:T1est2}) remain valid, with $(0,0,0)$ replaced by $(0,0,\chi_{ij}b_{ij})$. Using Taylor expansion again, for a large enough integer $s_7$, there exists $d_{1,ij} \in [0, d_{ij}], d_{2,ij} \in [0, \tilde{d}_{ij}]$ such that
	\begin{align*}
		F_{2p-1-n} (0,0,\chi_{ij}b_{ij}) &= \sum_{u=0}^{s_7}\sum_{v = 0}^u \frac{(-d_{ij})^v (-\tilde{d}_{ij})^{u-v}}{v! (u-v)!}  \cdot F_{2p-1-n}^{(v,u-v,0)}(d_{ij}, \tilde{d}_{ij}, \chi_{ij}b_{ij})  \notag\\
		&\quad + \sum_{\ell=0}^{s_7 + 1} \frac{(-d_{ij})^v (-\tilde{d}_{ij})^{s_7 + 1-v}}{v! (s_7 + 1 -v)!}  \cdot  F_{2p-1-n}^{(v,s_7 + 1-v,0)}  (d_{1,ij},d_{2,ij}, \chi_{ij}b_{ij}).
	\end{align*}
	Then we may use (\ref{eq:deriexpansion})(with minor modification that replace $(0,0,\chi_{ij}b_{ij})$ by $(d_{ij},\tilde{d}_{ij},\chi_{ij}b_{ij})$) to transform $F_{2p-1-n}^{(v,u-v,0)}(d_{ij}, \tilde{d}_{ij}, \chi_{ij}b_{ij})$ to $F_{2p-1-r}(d_{ij},\tilde{d}_{ij},\chi_{ij}b_{ij})$ for some $r \ge n$. It can also be easily checked that the resulting coefficients of $F_{2p-1-r}$ can be compensated by bounding $|d_{ij}|, |\tilde{d}_{ij}|$ by $ N^{-\epsilon_b}$ (w.h.p). This finally confirms that (\ref{eq:T1est}) and (\ref{eq:T1est2}) still hold when $(0,0,0)$ are replaced by $(d_{ij},\tilde{d}_{ij},\chi_{ij}b_{ij})$.
	
	Therefore, using straightforward power counting and applying Young's inequality as shown in (\ref{eq:young}), we may conclude that when $k+\ell \ge 3$, there exits some constants $K = K(p) > 0$ and $\delta = \delta(\epsilon_a,\epsilon_b,\epsilon_d) > 0$, such that
	\begin{align}
		|(J_1)_{ij,k\ell} | &\lesssim \frac{ \mathbf{1}_{\psi_{ij} = 0}}{N^2} \Big( (\log N)^{-K}\E_{\Psi} \Big[ F_{2p}(d_{ij},\tilde{d}_{ij},\chi_{ij}b_{ij}) \Big] + N^{-\delta p}\Big) \notag\\
		&\quad+  \frac{\mathbf{1}_{\psi_{ij} = 0} (1- \mathbf{1}_{i\in\mathcal{T}_r, j\in \mathcal{T}_c})}{N^{2-\epsilon_d}} \Big( (\log N)^{-K}\E_{\Psi} \Big[ F_{2p}(d_{ij},\tilde{d}_{ij},\chi_{ij}b_{ij}) \Big] + N^{-\delta p}\Big) .
	\label{eq:J1ijklest}\end{align}
		
	\textbf{Case 2:} $k +\ell = 1$.
	Recall from (\ref{eq:T1mnqest1}) that
	\begin{align*}
		(\mathsf{T}_1)_{ij,k\ell,nq} &= \sum_{r=0}^{s_4}\sum_{s=0}^{s_5}\sum_{m=0}^\ell\E_\Psi\Big[ \frac{d_{ij}^{k+m}\tilde{d}_{ij}^{\ell-m} (\chi_{ij}b_{ij})^{q+r+s}\mathcal{E}_{ij} }{s!r!q!m!(\ell - m)!} \Big]\notag\\
		&\quad \times \E_\Psi \Big[g^{(k,q)}_{(ij)}(0,0) \mathcal{C}^{0,(r)}_{n,m}  F_{2p-1-n}^{(0,0,s)} (0,0,0) \Big]\cdot \mathbf{1}_{\psi_{ij} = 0} + \mathcal{O}(N^{-\epsilon_bp}).
	\end{align*}
	
\textbf{Case 2.1:} $q+r+s$ is odd. In this case, we can directly compute that
	\begin{align*}
		&\sum_{m=0}^\ell\E_\Psi\Big[ \frac{d_{ij}^{k+m}\tilde{d}_{ij}^{\ell-m} (\chi_{ij}b_{ij})^{q+r+s}\mathcal{E}_{ij} }{s!r!q!m!(\ell - m)!}\Big] =\E_{\Psi} \Big[ \frac{\gamma((1-\chi_{ij})a^2_{ij}  - tw^2_{ij})(\chi_{ij}b_{ij})^{q+r+s} }{s!r!q!}\Big]=0.
	\end{align*}
	Thus, we have $(\mathsf{T}_1)_{ij,k\ell,nq} = \mathcal{O}(N^{-\epsilon_bp})$ in this case .

\textbf{Case 2.2:} $q+r+s \ge 0$ is even. Using (\ref{eq:2ndmomentdiff}) and the simple facts that $\chi_{ij}(1-\chi_{ij}) = 0$ and $\E(b_{ij}^2) \lesssim N^{-1}$, we have
\begin{align*}
		\sum_{m=0}^\ell\E_\Psi\Big[ \frac{d_{ij}^{k+m}\tilde{d}_{ij}^{\ell-m} (\chi_{ij}b_{ij})^{q+r+s}\mathcal{E}_{ij} }{s!r!q!m!(\ell - m)!} \Big] &= \frac{-\gamma t^{1/2} }{(1-\gamma^2)^{1/2}}\E_\Psi\Big[ \frac{d_{ij}w_{ij} (\chi_{ij}b_{ij})^{q+r+s} }{s!r!q!}  \Big] \\
		&= \frac{t}{N^2}\Big( \mathcal{O}\Big(\frac{\mathbf{1}_{q+r+s \ge 2} }{N^{(q+r+s - 2)\epsilon_b}} \Big) +\mathcal{O}\Big(\frac{\mathbf{1}_{q+r+s=0}}{N^{\epsilon_b}} \Big) \Big) .
	\end{align*}
	Further using (\ref{eq:gkqest}) and (\ref{eq:Crnmest}), we can obtain that 
	\begin{align*}
		(\mathsf{T}_1)_{ij,k\ell,nq}
		&= \sum_{r=0}^{s_4}\sum_{s=0}^{s_5}\frac{t}{N^2}\Big( \mathcal{O}\Big(\frac{\mathbf{1}_{q+r+s \ge 2} }{N^{(q+r+s - 2)\epsilon_b}} \Big) + \mathcal{O}\Big(\frac{\mathbf{1}_{q+r+s=0}}{N^{\epsilon_b}} \Big) \Big)\\
		&\quad\times \E_\Psi \Big[ \mathcal{O}_\prec \Big( \frac{1}{(N\eta \mathbf{1}_{q+r \ge 1} + \mathbf{1}_{q+r =0} ) t^{3(q+r)}t_{ij}^{3}} \Big)  F_{2p-1-n}^{(0,0,s)} (0,0,0) \Big]\mathbf{1}_{\psi_{ij} = 0} + \mathcal{O}(N^{-\epsilon_bp}).
	\end{align*}
	Observing that the above equation has a similar form to (\ref{eq:T1nqstart}), we may proceed in a similar manner as in Case 1 to estimate $(\mathsf{T}_1)_{ij,k\ell,nq}$. We will omit the repetitive details for brevity. Consequently, we can conclude that, by possibly adjusting the constants, (\ref{eq:J1ijklest}) also holds when $k + \ell = 1$.

Combining Case 1, Case 2, and the estimates for $(J_2)_{ij}$'s, we arrive at
\begin{align*}
	\sum_{i,j} |(I)_{ij}|&\lesssim \frac{ \mathbf{1}_{\psi_{ij} = 1}}{N^{1-\epsilon_\alpha}}\sum_{i,j} \Big(  (\log N)^{\frac{2p}{1-2p} }\E_{\Psi} \Big[ F_{2p}(e_{ij},\tilde{e}_{ij},c_{ij}) \Big] +  N^{-\epsilon_\alpha p/2}\Big)\\
	&\quad +\sum_{i,j}\frac{\mathbf{1}_{\psi_{ij} = 0}}{N^2} \Big( (\log N)^{-K}\E_{\Psi} \Big[ F_{2p}(d_{ij},\tilde{d}_{ij},\chi_{ij}b_{ij}) \Big] +  N^{-\delta p}\Big)  \\
	&\quad +\sum_{i,j} \frac{\mathbf{1}_{\psi_{ij} = 0} (1 - \mathbf{1}_{i \in \mathcal{T}_r, j \in \mathcal{T}_c})}{N^{2-\epsilon_d}}\Big((\log N)^{-K} \E_{\Psi} \Big[ F_{2p}(d_{ij},\tilde{d}_{ij},\chi_{ij}b_{ij}) \Big] +  N^{-\delta p} \Big) \\
	&\lesssim (\log N)^{-(K \wedge \frac{2p}{2p-1})} \E_{\Psi} \Big[ \big|N\eta \big(\Im m^\gamma(z) - \Im \tilde{m}^0(z)\big)\big|^{2p}  \Big] + N^{-\tilde{\delta} p},
\end{align*}
where $\tilde{\delta} = \tilde{\delta}(\epsilon_a,\epsilon_b,\epsilon_d) > 0$. Therefore, for any $0 \le  \gamma \le 1$,
\begin{align}
&\E_\Psi\Big( \big|N\eta \big(\Im m^\gamma(z_t) - \Im \tilde{m}^0(z_t)\big) \big|^{2p} \Big) - \E_\Psi\Big( \big|N\eta \big(\Im m^0(z_t) - \Im \tilde{m}^0(z_t)\big) \big|^{2p} \Big)\notag\\
&= \int_0^\gamma \frac{\partial \E\Big( \big|N\eta \big(\Im m^{\gamma'}(z_t) - \Im \tilde{m}^0(z_t)\big) \big|^{2p} \Big) }{\partial \gamma'} \mathrm{d}\gamma'.
\label{eq:integralest1}\end{align}
Taking supremum over $\gamma$, and using the estimates above, we have
\begin{align}
	&\sup_{0\le  \gamma \le 1} \E_\Psi\Big( \big|N\eta \big(\Im m^\gamma(z_t) - \Im \tilde{m}^0(z_t)\big) \big|^{2p} \Big) - \E_\Psi\Big( \big|N\eta \big(\Im m^0(z_t) - \Im \tilde{m}^0(z_t)\big) \big|^{2p} \Big)\notag \\
	&\lesssim (\log N)^{-(K \wedge \frac{2p}{2p-1})} \sup_{0\le  \gamma \le 1}\E_{\Psi} \Big[ \big|N\eta \big(\Im m^\gamma(z_t) - \Im \tilde{m}^0(z_t)\big)\big|^{2p}  \Big] +  N^{-\tilde{\delta} p}.
\label{eq:integralest2}\end{align}
The claim now follows by rearranging the terms.
 
\end{proof}

\subsection{Proof of Theorem \ref{Green function comparison thm}}\label{sec: green function comparison pf}
The proof of Theorem \ref{Green function comparison thm} is essentially the same as Theorem \ref{thm:comparison2}. We outline the proof here while the detailed proof can be found in Appendix \ref{4271}. 

 Using the same notation as in the proof of Theorem \ref{thm:comparison2} and further defining $h_{\gamma,(ij)}(\lambda,\beta) := \eta_0\sum_{a}f_{\gamma,(aa),(ij)}(\lambda,\beta)$ and $\mathsf{H}_{(ij)}(\lambda,\beta):= F' \big(h_{\gamma,(ij)} (\lambda,\beta )  \big) g_{(ij)}(\lambda,\beta ) $. Observe that 
\begin{align*}
		&\frac{\partial \E_\Psi \big(F(N\eta_0\Im m^\gamma(z_t))\big)}{\partial \gamma} 
		= -2\Big(\sum_{i,j} (I_1)_{ij} - (I_2)_{ij}\Big),
	\end{align*}
	where $(I_1)_{ij} = \E_\Psi \big[ \mathsf{A}_{ij}\mathsf{H}_{(ij)}([Y^\gamma]_{ij},X_{ij})\big]$ and $(I_2)_{ij} = \gamma(1-\gamma^2)^{-1/2} t^{1/2}\E_\Psi \big[ w_{ij}\mathsf{H}_{(ij)}([Y^\gamma]_{ij},X_{ij})\big]$. We estimate them by considering the cases $\psi_{ij} = 1$ and $\psi_{ij} = 0$ separately.
For $(I_2)_{ij}$, in both cases, we can estimate it by Gaussian integration by part, which leads to  
\begin{align*}
	(I_2)_{ij} = \frac{\gamma t^{1/2} }{(1-\gamma^2)^{1/2}N} \Big( \E_\Psi\big[\partial_{w_{ij}}\big\{ \mathsf{H}_{(ij)}(d_{ij},\chi_{ij}b_{ij}) \big\}  \big]\cdot \mathbf{1}_{\psi_{ij} = 0} + \E_\Psi\big[\partial_{w_{ij}}\big\{ \mathsf{H}_{(ij)}(e_{ij},c_{ij}) \big\}  \big]\cdot \mathbf{1}_{\psi_{ij} = 1}  \Big).
\end{align*}
The term involving $\mathbf{1}_{\psi_{ij} = 1}$ can be estimated directly by the fact that $t^{1/2}N^{-1} \cdot \sum_{i,j}\mathbf{1}_{\psi_{ij} = 1} \sim t^{1/2}N^{-1}\cdot N^{1-\epsilon_\alpha} = \mathfrak{o}(1)$. Therefore, by the definition of $d_{ij}$, we have
\begin{align}
	(I_2)_{ij} \approx \frac{\gamma t}{N} \E_\Psi\big[\partial_{d_{ij}}\big\{ \mathsf{H}_{(ij)}(d_{ij},\chi_{ij}b_{ij}) \big\}  \big]\cdot \mathbf{1}_{\psi_{ij} = 0}. 
\label{eq:I2comparison3}\end{align}
For $(I_1)_{ij}$, we only need to consider the case $\psi_{ij}=\chi_{ij} = 0$ since $\mathsf{A}_{ij}\mathbf{1}_{\psi_{ij} = 1 \text{ or } \chi_{ij} = 1} = 0$. Using Taylor expansion and the law of total expectation gives
\begin{align*}
	(I_1)_{ij}  \approx  \sum_{k}\frac{1}{k!} \E_\Psi[a_{ij}d_{ij}^k| \chi_{ij} =0]  \cdot \E_\Psi\big[\partial_{d_{ij}}^k\big\{ \mathsf{H}_{(ij)}(d_{ij},\chi_{ij}b_{ij}) \big\} | \chi_{ij} =0  \big]  \cdot \mathbb{P}(\chi_{ij} = 0)  \cdot \mathbf{1}_{\psi_{ij} = 0}. 
\end{align*}
For even values of $k$, it holds that $\E_\Psi[a_{ij}d_{ij}^k| \chi_{ij} =0] = 0$. In the case where $k \ge 3$, we have $\E_\Psi[a_{ij}d_{ij}^k| \chi_{ij} =0] \sim N^{-2-\varepsilon}$ for some small $\varepsilon > 0$, effectively compensating for the size of the summation $\sum_{i,j}$. Consequently, we arrive at 
\begin{align}
	(I_1)_{ij}  \approx  \E_\Psi[\gamma a_{ij}^2] \mathbb{P}(\chi_{ij} = 0)  \cdot \E_\Psi\big[\partial_{d_{ij}} \big\{ \mathsf{H}_{(ij)}(d_{ij},\chi_{ij}b_{ij}) \big\} | \chi_{ij} =0  \big]   \cdot \mathbf{1}_{\psi_{ij} = 0}.
\label{eq:I1comparison3}\end{align}
In view of (\ref{eq:I2comparison3}) and (\ref{eq:I1comparison3}), we can conclude the proof by leveraging the moment matching (\ref{eq:2ndmomentdiff}) and exploiting the smallness of $|\E_\Psi\big[\partial_{d_{ij}}\big\{ \mathsf{H}_{(ij)}(d_{ij},\chi_{ij}b_{ij}) \big\}  \big] -\E_\Psi\big[\partial_{d_{ij}} \big\{ \mathsf{H}_{(ij)}(d_{ij},\chi_{ij}b_{ij}) \big\} | \chi_{ij} =0  \big]|$.

\subsection*{Acknowledgments}
The authors would like to thank Fan Yang for helpful discussion.


\appendix

\section{Remaining proofs for the Gaussian divisible model}

\subsection{Proof of Lemma \ref{lem: domain of parameter}}\label{2416}

%
%

%

\noindent
Consider
\begin{equation}\label{eq: parameter u}
	z = (\lambda_{-}^{\mathsf{mp}} + E) + \mathrm{i}\eta, \quad
	|E|\le N^{-\vareps_{1}}, \quad
    N^{-2/3-\vareps_{2}} \le \eta \le \vareps_{3}.
\end{equation}
Recall that
\begin{equation*}
	V_{t} = \sqrt{t}W + X,
\end{equation*}
where $t=N\mathbb{E}|\mathsf{A}_{ij}|^2$.

%

\noindent
By the eigenvalue rigidity (the left edge analog of \cite[Theorem 2.13]{DY2}),
\begin{equation*}
	|\lambda_{M}(\mathcal{S}(V_{t}))-\lambda_{-,t}|\prec N^{-2/3}.
\end{equation*}
As an analog of Lemma \ref{lem: left edge rigidity XX^T}, 
\begin{equation*}
	|\lambda_{M}(\mathcal{S}(V_{t}))-\lambda_{-}^{\mathsf{mp}}|\prec  N^{-2\eps_{b}}.
\end{equation*}
Thus,
\begin{equation*}
	| \lambda_{-}^{\mathsf{mp}} - \lambda_{-,t} |\prec N^{-2/3} + N^{-2\eps_{b}} \lesssim N^{-2\vareps_{1}}.
\end{equation*}
We write
\begin{equation*}
	z = \{ \lambda_{-,t} + (\lambda_{-}^{\mathsf{mp}}-\lambda_{-,t}) + E \} + \mathrm{i}\eta
	\eqqcolon (\lambda_{-,t}+E') + \mathrm{i}\eta,
\end{equation*}
where $E'\coloneqq E + (\lambda_{-}^{\mathsf{mp}}-\lambda_{-,t})$. Then, with high probability, there exists $\kappa\in\mathbb{R}$ such that
\begin{equation}\label{eq: parameter z}
	z = (\lambda_{-,t} + \kappa) + \mathrm{i}\eta, \quad
	|\kappa|\le 2N^{-\vareps_{1}}, \quad
    N^{-2/3-\vareps_{2}} \le \eta \le \vareps_{3}.
\end{equation}

\noindent
Then, the desired result directly follows from the lemma below. Define $b_{t}\equiv b_{t}(z) \coloneqq 1 + c_{N} t m_{t}(z)$. Then we have $\zeta_{t}(z) \coloneqq z b_{t}^{2} - t b_{t} (1-c_{N})$.
\begin{lemma}
\label{lem: domain of para z}
Let $z$ as in \eqref{eq: parameter z}.
There exist constants $c,C>0$ such that the following holds:
\begin{itemize}
	\item[(i)] For $|\kappa|+\eta\le ct^{2}(\log{N})^{-2C}$,
	\begin{equation*}
		\lambda_{M}(XX^{\mathsf{T}})-\textnormal{Re}\,\zeta_{t}(z) \ge c t^{2}, \quad
		\textnormal{Im}\,\zeta_{t}(z) \ge c t N^{-2/3-\vareps_{2}}.
	\end{equation*}
	\item[(ii)] For $|\kappa|+\eta\ge ct^{2}(\log{N})^{-2C}$,
	\begin{equation*}
		\textnormal{Im}\,\zeta_{t}(z) \ge ct^{2}(\log{N})^{-C}.
	\end{equation*}
\end{itemize}
\end{lemma} 
\begin{proof}
This lemma is essentially a byprduct of Theorem \ref{thm:srbofmxt} through some elementary calculations. Comparing $\zeta_{t}(\lambda_{-,t})$ and $\zeta_{t}(z)$, it boils down to the size of $m_{t}(\lambda_{-,t}) - m_{t}(z)$. We shall rely on the square root behavior of $\rho_{t}$.

\noindent
Case (1) $|\kappa|\le 2\eta$.
Notice that
\begin{equation*}
	|m_{t}(\lambda_{-,t}) - m_{t}(z)|
	\le \int_{\lambda_{-,t}}^{\lambda_{+,t}} \frac{3\eta}{|\lambda - \lambda_{-,t}||\lambda - z|} \rho_{t}(\lambda)d\lambda.
\end{equation*}
By the square-root behavior of $\rho_{t}$ near the left edge,
\begin{equation*}
	\int_{\lambda_{-,t}}^{\lambda_{-,t}+6\eta} \frac{\eta}{|\lambda - \lambda_{-,t}||\lambda - z|} \rho_{t}(\lambda)d\lambda
	\lesssim \int_{\lambda_{-,t}}^{\lambda_{-,t}+6\eta} \frac{\eta}{\eta\sqrt{\lambda - \lambda_{-,t}}} d\lambda
	\lesssim \sqrt{\eta}.
\end{equation*}
If $\lambda\ge \lambda_{-,t}+6\eta$, we have $\lambda-\lambda_{-,t} - 3\eta \ge (\lambda-\lambda_{-,t})/2$. Thus,
\begin{equation*}
	\int_{\lambda_{-,t}+6\eta}^{\lambda_{+,t}} \frac{\eta}{|\lambda - \lambda_{-,t}||\lambda - z|} \rho_{t}(\lambda)d\lambda \lesssim \int_{\lambda_{-,t}+6\eta}^{\lambda_{+,t}} \frac{\eta}{(\lambda-\lambda_{-,t})^{3/2}} d\lambda
	\lesssim \sqrt{\eta}.
\end{equation*}

\noindent
Case (2) $\kappa > 2\eta$. We need to estimate
\begin{equation*}
	\int_{\lambda_{-,t}}^{\lambda_{+,t}} \frac{\kappa}{|\lambda - \lambda_{-,t}||\lambda - z|} \rho_{t}(\lambda)d\lambda.
\end{equation*}
Due to the square-root decay,
\begin{equation*}
	\int_{\lambda_{-,t}}^{\lambda_{-,t}+\eta} \frac{\kappa}{|\lambda - \lambda_{-,t}||\lambda - z|} \rho_{t}(\lambda)d\lambda \lesssim \int_{\lambda_{-,t}}^{\lambda_{-,t}+\eta} \frac{\kappa}{\kappa\sqrt{\lambda - \lambda_{-,t}}} d\lambda
	\lesssim \sqrt{\eta}.
\end{equation*}

\noindent
We also observe
\begin{equation*}
	\int_{\lambda_{-,t}+\eta}^{\lambda_{-,t}+\kappa-\eta} \frac{\kappa}{|\lambda - \lambda_{-,t}||\lambda - z|} \rho_{t}(\lambda)d\lambda
	\lesssim \int_{\eta}^{\kappa-\eta} \frac{\kappa}{\sqrt{x}(\kappa-x)}dx
	\lesssim \sqrt{\kappa} \log(\kappa/\eta).
\end{equation*}

\noindent
If $\lambda\in[\lambda_{-,t}+\kappa-\eta,\lambda_{-,t}+2\kappa]$, we have $\lambda - \lambda_{-,t} \sim \kappa$, which implies
\begin{equation*}
	\int_{\lambda_{-,t}+\kappa-\eta}^{\lambda_{-,t}+2\kappa} \frac{\kappa}{|\lambda - \lambda_{-,t}||\lambda - z|} \rho_{t}(\lambda)d\lambda
	\lesssim \int_{0}^{\kappa}\frac{\sqrt{\kappa}}{\sqrt{x^{2}+\eta^{2}}}dx
	\lesssim \sqrt{\kappa} \log(\kappa/\eta).
\end{equation*}

\noindent
For $\lambda\in[\lambda_{-,t}+2\kappa, \lambda_{+,t}]$, 
\begin{equation*}
	\int_{\lambda_{-,t}+2\kappa}^{\lambda_{+,t}} \frac{\kappa}{|\lambda - \lambda_{-,t}||\lambda - z|} \rho_{t}(\lambda)d\lambda
	\lesssim \sqrt{\kappa}
\end{equation*}

\noindent
Case (3) $\kappa < -2\eta$. By splitting $[\lambda_{-,t},\lambda_{+,t}]$ into $[\lambda_{-,t},\lambda_{-,t}+|\kappa|]$ and $[\lambda_{-,t}+|\kappa|,\lambda_{+,t}]$, we find that 
\begin{equation*}
	\int_{\lambda_{-,t}}^{\lambda_{+,t}} \frac{\kappa}{|\lambda - \lambda_{-,t}||\lambda - z|} \rho_{t}(\lambda)d\lambda 
	\lesssim \sqrt{|\kappa|}.
\end{equation*}

\noindent
Note $|b_{t}(\lambda_{-,t})|=O(1)=|b_{t}(z)|$ due to the fact that $|m_{t}(u)|\lesssim (t|u|)^{-1/2}$. 
Thus,
for $|\kappa|+\eta \le (\log{N})^{-C} t^{2}$,
\begin{equation*}
	|\zeta_{t}(z)-\zeta_{t}(\lambda_{-,t})| \ll t^{2}.
\end{equation*}
By Lemma \ref{lem: left edge rigidity XX^T} and Lemma \ref{lem:preEst},
\begin{equation*}
	(1-t)\lambda_{-}^{\mathsf{mp}}-\text{Re}\,\zeta_{t}(z) = \big( (1-t)\lambda_{-}^{\mathsf{mp}} - \lambda_{M}(\mathcal{S}(X)) \big) + (\lambda_{M}(\mathcal{S}(X)) - \zeta_{t}(\lambda_{-,t})) + \text{Re}\,[\zeta_{t}(\lambda_{-,t}))-\zeta_{t}(z)]  \sim t^{2}.
\end{equation*}

\noindent
Next, we consider the imaginary part of $\zeta_{t}(z)$.
Setting
\begin{equation*}
	\Phi(\kappa,\eta) =
	\begin{cases}
		\sqrt{\kappa + \eta}, & \kappa \ge 0, \\
		\frac{\eta}{\sqrt{|\kappa| + \eta}}, & \kappa < 0,
	\end{cases}
\end{equation*}
we have $\text{Im}\,\zeta_{t}(z) \sim \eta + t\Phi(\kappa,\eta),$
which gives the desired estimates on the imaginary part of $\zeta_{t}(z)$.\\
\end{proof}


\subsection{Proof of Proposition \ref{prop: resolvent entry size for X=B+C}}\label{pf: resolvent entry size for X=B+C}
We estimate the size of $G_{ij}(X, \zeta)$ only. We can bound $G_{ij}(X^{\top}, \zeta)$ in a similar way.
Define $H\coloneqq X/\sqrt{1-t}$ and denote $\omega\coloneqq\zeta/(1-t)$. It is enough to find a constant $c=c(\eps_a,\eps_{\alpha},\eps_{b})$ such that
\begin{equation*}
	|G_{ij}(H, \omega)-\delta_{ij}\mathsf{m}_{\mathsf{mp}}(\omega)| \prec N^{-c}\mathbf{1}_{i,j \in \mathcal{T}_r} + t^{-2}(1-\mathbf{1}_{i,j \in \mathcal{T}_r}).
\end{equation*}
This can be proved by a minor modication of \cite[Section 6]{PY}.
In light of Lemma \ref{lem: left edge rigidity XX^T}, the following two lemmas are  trivial.
We may use the rigidiy estimate, Lemma \ref{lem: left edge rigidity XX^T}, to get Lemma \ref{lem: crude bound for resolvent entry} below.

\begin{lemma}[Crude bound using the imaginary part]\label{lem:799}
Consider $\omega=E+\mathrm{i}\eta\in\mathbb{C}_{+}$.
If $\eta>C$,
\begin{equation*}
	|G_{ij}(H,\omega)| \le C^{-1}.
\end{equation*}
\end{lemma}

\begin{lemma}[Crude bound on the domain $\mathsf{D}_\zeta$]\label{lem: crude bound for resolvent entry}
Let $\mathsf{D}_\zeta=\mathsf{D}_\zeta(c_{0},C_{0})$ be as in Eq.~\eqref{eq: domain of zeta}. Let $\zeta\in\mathsf{D}_\zeta$. 
Denote $\omega=\zeta/(1-t)$. Then with high probability,
\begin{equation*}
	|G_{ij}(H,\omega)| \lesssim (\log N)^{C_{0}} t^{-2}.
\end{equation*}
\end{lemma}

\noindent
Let us write $H=(h_{ij})$. By Schur complement,
\begin{equation}\label{eq: resol id 0}
	G_{ii}(H,\omega) = -\frac{1}{\omega + \frac{\omega}{N}\sum_{k=1}^{N}G_{kk}((H^{(i)})^{\top},\omega) + Z_{i}}
\end{equation}
where we denote by $H^{(i)}$ the matrix obtained from $H$ by removing $i$-th row and
\begin{equation*}
	Z_{i} \coloneqq \omega\sum_{1\le k,l\le N} h_{ik}h_{il} G_{kl}((H^{(i)})^{\top},\omega) - \frac{\omega}{N}\sum_{k=1}^{N}G_{kk}((H^{(i)})^{\top},\omega).
\end{equation*}
We define $\Lambda_{d}(\omega)$, $\Lambda_{o}(\omega)$ and $\Lambda(\omega)$ by
\begin{equation*}
	\Lambda_{d}(\omega) = \max_{i\in\mathcal{T}_{r}}|G_{ii}(H,\omega)-\mathsf{m}_{\mathsf{mp}}(\omega)|, \;\;
	\Lambda_{o}(\omega) = \max_{\substack{i\neq j \\ i,j\in \mathcal{T}_{r}}}|G_{ij}(H,\omega)|, \;\;
	\Lambda(\omega) = |m_H(\omega)-\mathsf{m}_{\mathsf{mp}}(\omega)|.
\end{equation*}
For $\omega=E+\mathrm{i}\eta$, we define
\begin{equation*}
	\Phi \equiv \Phi(\omega) \coloneqq \sqrt{\frac{\Im \mathsf{m}_{\mathsf{mp}}(\omega)+\Lambda(\omega)}{N\eta}} + t^{-2}N^{-\eps_{\alpha}/2} + t^{-2}N^{-\eps_{b}}.
\end{equation*}
Define the events $\Omega(\omega,K)$, $\mathbf{B}(\omega)$ and $\Gamma(\omega,K)$ for $K>0$ by
\begin{equation*}
	\Omega(\omega,K) \coloneqq \Big\{ \max\Big(\Lambda_{o}(\omega),\max_{i\in\mathcal{T}_{r}}|G_{ii}(H,\omega)-m_H(\omega)|,\max_{i\in\mathcal{T}_{r}}|Z_{i}(\omega)|\Big)\ge K\Phi \Big\},
\end{equation*}
\begin{equation*}
	\mathbf{B}(\omega) \coloneqq \{ \Lambda_{o}(\omega) + \Lambda_{d}(\omega) > (\log N)^{-1} \}, \quad
	\Gamma(\omega,K) \coloneqq \Omega^{c}(\omega,K) \cup \mathbf{B}(\omega).
\end{equation*}
We also introduce the logarithmic factor $\varphi \equiv \varphi_{N} \coloneqq (\log N)^{\log\log N}.$
\begin{lemma}\label{lem:858}
Suppose $\Psi$ is good. Recall $\omega\equiv\omega(\zeta)=\zeta/(1-t)$. There exist a constant $C>0$ such that the event
\begin{equation*}
		\bigcap_{\zeta\in\mathsf{D}_{\zeta}}\Gamma(\omega,\varphi^{C})
\end{equation*}
holds with high probability.
\end{lemma}
\begin{proof}
By a standard lattice argument, it is enough to show that $\Gamma(\omega,\varphi^{C})$ holds with with high probability for any $\omega=\omega(\zeta)$ with $\zeta\in\mathsf{D}_\zeta$.
Fix $\omega=\omega(\zeta)$ with $\zeta\in\mathsf{D}_\zeta$. We define
\begin{align*}
	\Omega_{o}(\omega,K) &\coloneqq \big\{ \Lambda_{o}(\omega)\ge K\Phi(\omega) \big\}, \\
	\Omega_{d}(\omega,K) &\coloneqq \Big\{ \max_{i\in\mathcal{T}_{r}}|G_{ii}(H,\omega)-m_H(\omega)|\ge K\Phi(\omega) \Big\}, \\
	\Omega_{Z}(\omega,K) &\coloneqq \Big\{ \max_{i\in\mathcal{T}_{r}}|Z_{i}| \ge K\Phi(\omega) \Big\}.
\end{align*}
Since $\Omega=\Omega_{o}\cup\Omega_{d}\cup\Omega_{Z}$, it is sufficient to show $\Omega_{o}^{c}\cup\mathbf{B}$, $\Omega_{d}^{c}\cup\mathbf{B}$ and $\Omega_{Z}^{c}\cup\mathbf{B}$ hold with high probability respectively.

\noindent
(1) Consider the event $\Omega_{o}^{c}\cup\mathbf{B}$.
Fix $i\neq j$ with $i,j\in\mathcal{T}_{r}$.
On the event $\mathbf{B}^{c}$, we have $|G_{ii}(H,\zeta)| \sim 1$. Then, by the resolvent identity,
\begin{equation}\label{eq: resol id 1}
	G_{jj}(H^{(i)},\omega) = G_{jj}(H,\omega) - \frac{G_{ji}(H,\omega)G_{ij}(H,\omega)}{G_{ii}(H,\omega)},
\end{equation}
it follows that $G_{jj}(H^{(i)},\omega)\sim 1$ on $\mathbf{B}^{c}$.
Thus, we can get 
\begin{equation*}
	\Lambda_{o}(\omega) \lesssim \max_{\substack{i\neq j\\i,j\in\mathcal{T}_{r}}}\left|\sum_{1\le k,l\le N} h_{ik}h_{jl} G_{kl}((H^{(ij)})^{\top},\omega)\right|,
\end{equation*}
where we denote by $H^{(ij)}$ the matrix obtained from $H$ by removing $i$-th and $j$-th rows.
Since $i,j\in\mathcal{T}_{r}$, applying the large deviation estimate \cite[Corollary 25]{Amol}, the following estimate holds with high probability:
\begin{equation*}
	\bigg|\sum_{1\le k,l\le N} h_{ik}h_{jl} G_{kl}((H^{(ij)})^{\top},\omega)\bigg| \le \varphi^{C}\left(N^{-\eps_{b}}\max_{k,l}|G_{kl}((H^{(ij)})^{\top},\omega)| + \frac{1}{N}\bigg(\sum_{k,l}|G_{kl}((H^{(ij)})^{\top},\omega)|^{2}\bigg)^{1/2} \right).
\end{equation*}
Note that
\begin{equation}\label{eq: resol id 2}
	\sum_{k,l}|G_{kl}((H^{(ij)})^{\top},\omega)|^{2} = \frac{\sum_{k} \Im G_{kk}((H^{(ij)})^{\top},\omega)}{\eta},
\end{equation}
and
\begin{equation}\label{eq: resol id 3}
	\sum_{k}G_{kk}((H^{(ij)})^{\top},\omega) - \sum_{\ell}G_{\ell\ell}(H^{(ij)},\omega) = \frac{O(N)}{\omega}.
\end{equation}
Using \eqref{eq: resol id 1}, \eqref{eq: resol id 2} and \eqref{eq: resol id 3}, together with Lemma \ref{lem: crude bound for resolvent entry},
we conclude that on the event $\mathbf{B}^{c}$, with high probability, for some constant $C>0$ large enough,
\begin{equation*}
	\Lambda_{o}(\omega) \le \varphi^{C} \paren{t^{-2}N^{-\eps_{b}}+\sqrt{ \frac{\Im \mathsf{m}_{\mathsf{mp}}+\Lambda+\Lambda_{o}^{2} + t^{-4}N^{-\eps_{\alpha}} }{N\eta} + \frac{1}{N} }},
\end{equation*}
with high probability for some constant $C>0$ large enough.
The event $\Omega_{o}^{c}\cap\mathbf{B}^{c}$ holds with high probability.

\noindent
(2) We claim that $\Omega_{Z}^{c}\cup\mathbf{B}$ holds with high probability. In fact, the claim directly follows from the large deviation estimate \cite[Corollary 25]{Amol} repeating the same argument we used above; 
on the event $\mathbf{B}^{c}$, for $i\in\mathcal{T}_{r}$, we have $|Z_{i}| \le \varphi^{C} \Phi$ with high probability for some constant $C>0$.

\noindent
(3) We shall prove $\Omega_{d}^{c}\cup\mathbf{B}$ holds with high probability. For $i\in\mathcal{T}_{r}$,
\begin{equation*}
	G_{ii}(H,\omega) - m_H(\omega)
	\le \max_{j\in\mathcal{T}_{r}}|G_{ii}(H,\omega)-G_{jj}(H,\omega)| + \varphi^{C} t^{-2}N^{-\eps_{\alpha}},
\end{equation*}
where we use Lemma \ref{lem: crude bound for resolvent entry} to bound $G_{jj}$ with $j\notin\mathcal{T}_{r}$.
For $i,j\in\mathcal{T}_{r}$ with $i\neq j$, on the event $\mathbf{B}^{c}$, with high probability, we can find that
\begin{align*}
	|G_{ii}(H,\omega)-G_{jj}(H,\omega)|&\leq \bigg|\frac{1}{\omega + \frac{\omega}{N}\sum_{k=1}^{N}G_{kk}((H^{(i)})^{\top},\omega) + Z_{i}} - \frac{1}{\omega + \frac{\omega}{N}\sum_{k=1}^{N}G_{kk}((H^{(j)})^{\top},\omega) + Z_{j}} \bigg| \\
	&\lesssim \max_{i\in\mathcal{T}_{r}}|Z_{i}| + \Lambda_{o}^{2} + t^{-4}N^{-\eps_{\alpha}}
\end{align*}
where we use 
\begin{equation}\label{eq: resol id 4}
	\sum_{k}G_{kk}((H^{(i)})^{\top},\omega) - \sum_{\ell}G_{\ell\ell}(H^{(i)},\omega) = \frac{M-N+1}{\omega}
\end{equation}
and the estimates we have shown above. The desired result follows.
\end{proof}

\begin{corollary}\label{cor:964}
Suppose $\Psi$ is good. Let $C'>0$ be a constant. There exist a constant $C>0$ such that the event $\Omega^{c}(E+\mathrm{i}\eta,\varphi^{C})$ holds with high probability.
\end{corollary}
\begin{proof}
Recall the argument we used in the proof of the previous lemma. Using the large deviation estimate \cite[Corollary 25]{Amol} with Lemma \ref{lem:799}, it is straightforward that $\Omega_{o}^{c}$ and $\Omega_{Z}^{c}$ hold with high probability. For $\Omega_{d}^{c}$, the desired result follows from the consequence of Cauchy's interlacing theorem, that is,
\begin{equation*}
	\frac{1}{N}\sum_{k=1}^{N}G_{kk}((H^{(i)})^{\top},\omega) - \frac{1}{N}\sum_{k=1}^{N}G_{kk}((H^{(j)})^{\top},\omega) \lesssim \frac{1}{N\eta}.
\end{equation*}
\end{proof}

\noindent
Let us introduce the deviance function $D(u(\omega),\omega)$ by setting
\begin{equation*}
	D(u(\omega),\omega) \coloneqq  \paren{\frac{1}{u(\omega)} + c_{N} \omega u(\omega)} - \paren{\frac{1}{\mathsf{m}_{\mathsf{mp}}(\omega)} + c_{N} \omega \mathsf{m}_{\mathsf{mp}}(\omega)}.
\end{equation*}

\begin{lemma}\label{lem:981}
On the event $\Gamma(\omega,\varphi^{C})$,
\begin{equation*}
	|D(m_H(\omega),\omega)|\le O(\varphi^{2C}\Phi^{2}) + \infty\indic_{\mathbf{B}(\omega)}.
\end{equation*}
\end{lemma}
\begin{proof}
Recall that $(\mathsf{m}_{\mathsf{mp}})^{-1}(\omega)=-\omega+(1-c_{N})-\omega c_{N}\mathsf{m}_{\mathsf{mp}}$.
Using \eqref{eq: resol id 0}, \eqref{eq: resol id 1} and \eqref{eq: resol id 4}, on the event $\Omega^{c}\cap\mathbf{B}^{c}$, we have
\begin{equation*}
	G_{ii}^{-1}(H,\omega) = (\mathsf{m}_{\mathsf{mp}})^{-1}(\omega) + \omega c_{N} (\mathsf{m}_{\mathsf{mp}}(\omega)-m_H(\omega)) - Z_{i} + O(\varphi^{2C}\Phi^{2}+t^{-4}N^{-\eps_{\alpha}}+N^{-1}),
\end{equation*}
so it follows that
\begin{equation*}
	m^{-1}_H(\omega) - G_{ii}^{-1}(H,\omega) = D(m_H(\omega),\omega) + Z_{i} + O(\varphi^{2C}\Phi^{2}+t^{-4}N^{-\eps_{\alpha}}+N^{-1}).
\end{equation*}
Averaging over $i\in\mathcal{T}_r$ yields
\begin{align*}
\frac{1}{|\mathcal{T}_r|}\sum_{i \in \mathcal{T}_r} (m^{-1}_H(\omega) - G_{ii}^{-1}(H,\omega)) = D(m_H(\omega),\omega) + \frac{1}{|\mathcal{T}_r|}\sum_{i \in \mathcal{T}_r}Z_{i} + O(\varphi^{2C}\Phi^{2}+t^{-4}N^{-\eps_{\alpha}}+N^{-1}).
\end{align*}
Since $\sum_{i}G_{ii}(H,\omega) - m_{H}(\omega) = 0$ and
\begin{equation*}
	m^{-1}_H(\omega) - G_{ii}^{-1}(H,\omega) = \frac{G_{ii}(H,\omega)-m_{H}(\omega)}{m_{H}^{2}(\omega)} - \frac{\big(G_{ii}(H,\omega)-m_{H}(\omega)\big)^{2}}{m_{H}^{3}(\omega)} + O\Big( \frac{\big(G_{ii}(H,\omega)-m_{H}(\omega)\big)^{3}}{m_{H}^{4}(\omega)} \Big),
\end{equation*}
we obtain that $|D(m_H(\omega),\omega)| \le O(\varphi^{2C}\Phi^{2})$ on the event $\Omega^{c}\cap\mathbf{B}^{c}$.
\end{proof}

\begin{lemma}\label{lem: lem 6.12 PY14} 
Recall $\omega\equiv\omega(\zeta)=\zeta/(1-t)$ and write $\omega=E+\mathrm{i}\eta$. Let $C,C'>0$ be constants. Consider an event $A$ such that
\begin{equation*}
	A \subset \bigcap_{\zeta\in\mathsf{D}_{\zeta}}\Gamma(\omega,\varphi^{C}) \cap \bigcap_{\zeta\in\mathsf{D}_{\zeta},\eta=C'} \mathbf{B}^{c}(\omega).
\end{equation*}
Suppose that in $A$, for $\omega=\omega(\zeta)$ with $\zeta\in\mathsf{D}_{\zeta}$,
\begin{equation*}
	|D(m_H(\omega),\omega)| \le \mathfrak{d}(\omega) + \infty\indic_{B(\omega)},
\end{equation*}
where $\mathfrak{d}:\mathbb{C}\mapsto\mathbb{R}_{+}$ is a continuous function such that $\mathfrak{d}(E+\mathrm{i}\eta)$ is decreasing in $\eta$ and $|\mathfrak{d}(z)|\le(\log{N})^{-8}$.

Then, for all $\omega\equiv\omega(\zeta)$ with $\zeta\in\mathsf{D}_{\zeta}$, we have
\begin{equation}\label{eq:1044}
	|m_H(\omega)-\mathsf{m}_{\mathsf{mp}}(\omega)| \lesssim \log{N} \frac{\mathfrak{d}(\zeta)}{\sqrt{|E-\lambda_{-}^{\mathsf{mp}}|+\eta+\mathfrak{d}(\zeta)}} \quad \text{in $A$, }
\end{equation}
and
\begin{equation}\label{eq:1048}
	A \subset \bigcap_{\zeta\in\mathsf{D}_{\zeta}} \mathbf{B}^{c}(\zeta).
\end{equation}
\end{lemma}
\begin{proof}
We follow the proof of \cite[Lemma 6.12]{PY}.
Denote $\omega=\omega(\zeta)=E+\mathrm{i}\eta$ with $\zeta\in\mathsf{D}_{\zeta}$. For each $E$, we define
\begin{equation*}
	I_{E} \coloneqq \{ \eta: \Lambda_{o}(E+\mathrm{i}\eta')+\Lambda_{d}(E+\mathrm{i}\eta')\le(\log{N})^{-1} \text{ for all $\eta'\ge\eta$ such that $(1-t)\cdot(E+\mathrm{i}\eta')\in\mathsf{D}_\zeta$} \}.
\end{equation*}
Let $m_{1}$ and $m_{2}$ be two solutions of equation $D(m(\omega),\omega)=\mathfrak{d}(\omega)$. On $\mathbf{B}^{c}(\omega)$, by assumption, we have
\begin{equation*}
	|D(m_H(\omega),\omega)|\le\mathfrak{d}(\omega).
\end{equation*}
Then, the estimate \eqref{eq:1044} immediately follows from the argument around \cite[Eq.~(6.45)--Eq.~(6.46)]{PY}. 

Next, we will prove the second statement \eqref{eq:1048}. Due to the case $\eta=C'$, we know $I_{E}\neq\emptyset$ on $A$. Let us argue by contradiction.
Define
\begin{equation*}
	\mathcal{D}_{E} = \{\eta: \omega=E+\mathrm{i}\eta, (1-t)\cdot\omega\in\mathsf{D}_\zeta\}.
\end{equation*}
Assume $I_{E}\neq\mathcal{D}_{E}$. Let $\eta_{0}=\inf I_{E}$. For $\omega_{0}=E+\mathrm{i}\eta_{0}$, we have 
$\Lambda_{o}(\omega_{0})+\Lambda_{d}(\omega_{0}) = (\log{N})^{-1}$. 
It also follows
\begin{multline*}
	\Lambda(\omega_{0}) \le \Big|\frac{1}{N}\sum_{i\in\mathcal{T}_{r}} \big( G_{ii}(H,\omega_{0})-\mathsf{m}_{\mathsf{mp}}(\omega_{0}) \big)\Big| + \Big|\frac{1}{N}\sum_{i\notin\mathcal{T}_{r}} \big(G_{ii}(H,\omega_{0})-\mathsf{m}_{\mathsf{mp}}(\omega_{0}) \big)\Big|  \\ 
	\le (\log{N})^{-1} + \varphi^{C}t^{-2}N^{-\eps_{\alpha}} \lesssim (\log{N})^{-1}.
\end{multline*}
By the first statement we already proved, on the event $A$, we obtain
\begin{equation*}
	\Lambda(\omega_{0}) \lesssim (\log{N})^{-3}.
\end{equation*}
Since $\Lambda_{o}(\omega_{0})+\Lambda_{d}(\omega_{0}) = (\log{N})^{-1}$, we have $A\subset \mathbf{B}^{c}(\omega_{0})$ and thus, by the assumption for $A$,
we conclude that $\Lambda_{o}(\omega_{0})+\Lambda_{d}(\omega_{0})\ll (\log{N})^{-1}$ on the event $A$, which makes a contradiction.

\end{proof}

\begin{proposition}
Recall $\omega\equiv\omega(\zeta)=\zeta/(1-t)$ and write $\omega=E+\mathrm{i}\eta$. There exist a constant $C>0$ such that the following event holds with high probability:
\begin{equation*}
	\bigcap_{\zeta\in\mathsf{D}_{\zeta}} \{ \Lambda_{o}(\omega) + \Lambda_{d}(\omega) \le \varphi^{C}( t^{-2}(N\eta)^{-1/2} + t^{-3}N^{-\eps_{\alpha}/2} + t^{-3}N^{-\eps_{b}}) \}.
\end{equation*}
\end{proposition}
\begin{proof}
Consider the event
\begin{equation*}
	A_{0} = \bigcap_{\zeta\in\mathsf{D}_{\zeta}} \Gamma(\omega,\varphi^{C}).
\end{equation*}
Also we set (for some constant $C'>1$ and $\omega=E+\mathrm{i}\eta$)
\begin{equation*}
	A = A_{0} \cap \bigcap_{\zeta\in\mathsf{D}_{\zeta}, \eta=C'} \mathbf{B}^{c}(\omega).
\end{equation*}
By Lemma \ref{lem:858} and Corollary \ref{cor:964}, the event $A$ holds with high probability.
Using Lemma \ref{lem: crude bound for resolvent entry}, we observe that for $\omega=\omega(\zeta)$ with $\zeta\in\mathsf{D}_{\zeta}$,
$$\Phi(\omega)\lesssim \varphi t^{-1}(N\eta)^{-1/2} + t^{-2}N^{-\eps_{\alpha}/2} + t^{-2}N^{-\eps_{b}}.$$
Let us set
$$\mathfrak{d}(\omega)=\varphi^{C}\big( t^{-1}(N\eta)^{-1/2}+ t^{-2}N^{-\eps_{\alpha}/2} + t^{-2}N^{-\eps_{b}} \big).$$
On the event $A$, for $\omega=\omega(\zeta)$ with $\zeta\in\mathsf{D}_{\zeta}$, by Lemma \ref{lem:981} and Lemma \ref{lem: lem 6.12 PY14},
\begin{equation*}
	\Lambda(\omega) \lesssim \frac{\mathfrak{d}(\omega)}{\sqrt{|E-\lambda_{-}^{\mathsf{mp}}|+\eta}}.
\end{equation*}
Also, by Lemma \ref{lem: lem 6.12 PY14},
\begin{equation*}
	A \subset \bigcap_{\zeta\in\mathsf{D}_{\zeta}}\mathbf{B}^{c}(\omega),
\end{equation*}
which means the event $A$ is contained in $\Omega^{c}(\omega,\varphi^{C})$ for any $\omega=\omega(\zeta)$ with $\zeta\in\mathsf{D}_{\zeta}$.
The bound for $\Lambda_{d}$ is given by $\max_{k\in \mathcal{T}_r}|G_{kk}(H,\omega)-m_H| + \Lambda$.
\end{proof}


\subsection{Proof of Theorem \ref{thm: resolvent entry size V_t}}\label{sec: resolvent estimates V_t}

Recall $b_{t} = 1+c_{N}t m_{t}$ and $\zeta_{t}=\zeta_{t}(z)=zb_{t}^{2}-tb_{t}(1-c_{N})$. We also set
\begin{equation*}
	\underline{m}_{t} = c_{N}m_{t} - \frac{1-c_{N}}{z}, \qquad \underline{\mathsf{m}}_{\mathsf{mp}}^{(t)}(\zeta)=c_{N}\mathsf{m}_{\mathsf{mp}}^{(t)}(\zeta)-\frac{1-c_{N}}{\zeta}. 
\end{equation*}

\noindent
Let us state a left edge analog of \cite[Theorem 2.7]{DY2}.

\begin{theorem}\label{thm: local law Gaussian divisible model}
Suppose that the assumptions in Theorem \ref{thm: resolvent entry size V_t} hold.
Then,
\begin{equation*}
	| G_{ij}(V_{t}, z) - b_{t} G_{ij}(X, \zeta_{t}(z)) | \prec t^{-3}\paren{\sqrt{\frac{\imag\,m_{t}}{N\eta}}+\frac{1}{N\eta}} + \frac{t^{-7/2}}{N^{1/2}},
\end{equation*}
and
\begin{equation*}
	| G_{ij}(V_{t}^{\top}, z) - (1+t\underline{m}_{t}) G_{ij}(X^{\top}, \zeta_{t}(z)) | \prec t^{-3}\paren{\sqrt{\frac{\imag\,m_{t}}{N\eta}}+\frac{1}{N\eta}} + \frac{t^{-7/2}}{N^{1/2}},
\end{equation*}
uniformly in $z\in\mathsf{D}(\varepsilon_1, \varepsilon_2, \varepsilon_3)$. 
In addition,
\begin{equation*}
	| (G(V_{t},z)V_{t})_{ij} - (G(X,\zeta_{t}(z))X)_{ij} | \prec t^{-3}\paren{\sqrt{\frac{\imag\,m_{t}}{N\eta}}+\frac{1}{N\eta}} + \frac{t^{-7/2}}{N^{1/2}},
\end{equation*}
and
\begin{equation*}
	| (V_{t}^{\top}G(V_{t},z))_{ij}- (X^{\top}G(X,\zeta_{t}(z)))_{ij} | \prec t^{-3}\paren{\sqrt{\frac{\imag\,m_{t}}{N\eta}}+\frac{1}{N\eta}} + \frac{t^{-7/2}}{N^{1/2}},
\end{equation*}
uniformly in $z\in\mathsf{D}(\varepsilon_1, \varepsilon_2, \varepsilon_3)$.\\
\end{theorem}
\begin{proof}
	Roughly speaking, the conclusion is a left edge analog of \cite[Theorem 2.7]{DY2}. The proof is nearly the same, and thus we only highlight some differences. We first record the notations from \cite[Section B of Supplement]{DY2}. Due to the rotationally invariant property of Gaussian matrix, we have
	\begin{align}
		V_t = X + \sqrt{t}W \overset{d}{=} O_1\tilde{V}_tO_2^\top, \quad \tilde{V}_t := \tilde{X} + \sqrt{t}W,
	\label{eq:orthinva}\end{align} 
	where $\tilde{X}$ is a diagonal matrix with diagonal entries being $\lambda_i(\mathcal{S}(X))^{1/2}, i \in [M]$. Recall the notations in Lemma \ref{lem:entryperturb}, and we briefly write $\mathcal{R}(z) = \mathcal{R}(\tilde{V}_t,z)$ in this proof. By (\ref{eq:orthinva}), to prove an entrywise local law for $\mathcal{R}({V}_t,z)$, it suffices to prove an anisotropic local law for the resolvent $\mathcal{R}(z)$. We further define the asymptotic limit of $\mathcal{R}(z)$ as 
	\begin{align*}
		\Pi^x(z):=\left[\begin{array}{cc}
		\frac{-\left(1+c_N t m_{t}\right)}{z\left(1+c_N t m_{t}\right)\left(1+t \underline{m}_{t}\right)-\tilde{X} \tilde{X}^{\top}} & \frac{-z^{-1 / 2}}{z\left(1+c_N t m_{t}\right)\left(1+t \underline{m}_{t}\right)-\tilde{X} \tilde{X}^{\top}} \tilde{X} \\
		\tilde{X}^{\top} \frac{-z^{-1 / 2}}{z\left(1+c_N t m_{t}\right)\left(1+t \underline{m}_{t}\right)-\tilde{X} \tilde{X}^{\top}} & \frac{-\left(1+t \underline{m}_{t}\right)}{z\left(1+c_N t m_{t}\right)\left(1+t \underline{m}_{t}\right)-\tilde{X}^{\top} \tilde{X}}
\end{array}\right] .
	\end{align*}
	We define the index sets
	\begin{align*}
		\mathcal{I}_1 := \{1,\cdots, M \}, \quad \mathcal{I}_2 := \{M+1,\cdots,M+N \}, \quad \mathcal{I} := \mathcal{I}_1 \cup \mathcal{I}_2.
	\end{align*}
	In the sequel, we use the Latin letter $i,j \in \mathcal{I}_1$, Greek letters $\mu,\nu \in \mathcal{I}_2$, $\mathfrak{a},\mathfrak{b} \in \mathcal{I}$. For an $\mathcal{I} \times \mathcal{I}$ matrix $A$ and  $i,j \in \mathcal{I}_1$, we define the $2 \times 2$ minor as 
	\begin{align*}
		{A}_{[i j]}:=\left(\begin{array}{cc}
{A}_{i j} & {A}_{i \bar{j}} \\
{A}_{\bar{i} j} & {A}_{\bar{i} \bar{j}}
\end{array}\right),
	\end{align*}
	where $\bar{i} := i + M \in \mathcal{I}_2$. Moreover, for $\mathfrak{a} \in \mathcal{I} \setminus \{i,\bar{i} \}$, we denote
	\begin{align*}
		{A}_{[i] \mathfrak{a}}=\left(\begin{array}{c}
{A}_{i \mathfrak{a}} \\
{A}_{\bar{i} \mathfrak{a}}
\end{array}\right), \quad {A}_{\mathfrak{a}[i]}=\left({A}_{\mathfrak{a} i}, {A}_{\mathfrak{a} \bar{i}} .\right)
	\end{align*}
Let the error parameter $\Psi(z)$ be defined as follows,
\begin{align*}
	\Psi(z) := \sqrt{\frac{\Im m_{t}}{N\eta}}+\frac{1}{N\eta}.
\end{align*}

Instead of proving \cite[Eq.~(B.68) in Supplement]{DY2}, which aims at bounding $u^\top (\Pi^x(z))^{-1}[R(z) - \Pi^x(z)](\Pi^x(z))^{-1} v $ for any deterministic unit vector $u, v \in \mathbb{R}^{M+N}$, we shall prove
	\begin{align}
		|u^\top [\mathcal{R(}z) - \Pi^x(z)] v | \prec t^{-3}\Psi(z) + \frac{t^{-7/2}}{N^{1/2}}.
	\label{eq:isotrop1}\end{align}
	We remark here that in \cite{DY2}, it is assumed that all $\lambda_i(\mathcal{S}(X))$'s are $O(1)$. Under this assumption, adding $(\Pi^x(z))^{-1}$ is harmless. However, in our case, $\lambda_i(\mathcal{S}(X))$ could diverge with $N$. Then, adding the $(\Pi^x(z))^{-1}$ factor which will blow up along with big $\lambda_i(\mathcal{S}(X))$, will complicate the proof of the anisotropic law. On the other hand, (\ref{eq:isotrop1}) is what we need anyway. Hence, we get rid of the $(\Pi^x(z))^{-1}$ and adapt the proof in \cite{DY2} to our estimate (\ref{eq:isotrop1}). Without the $(\Pi^x(z))^{-1}$ factor, the $\mathcal{R}(z)$ and $\Pi^x(z)$ entries are well controlled, and the remaining proof is nearly the same as \cite{DY2}. 
	
We shall first prove an entrywise version of (\ref{eq:isotrop1}): for any $\mathfrak{a}, \mathfrak{b} \in \mathcal{I}$,
	\begin{align}
		| [\mathcal{R}(z) - \Pi^x(z)]_{\mathfrak{a} \mathfrak{b}}| \prec t^{-3}\Psi(z) + \frac{t^{-7/2}}{N^{1/2}}.
	\label{eq:isotrop2}\end{align}
	The derivation of (\ref{eq:isotrop2}) follows the same procedure as the proof of \cite[Eq.~(B.69) in Supplement]{DY2}. This proof primarily relies on Schur complement, the large deviation of quadratic forms of Gaussian vector, and the fact that $\min_{i}|\lambda_i(\mathcal{S}(X)) - \zeta_t(z)| \gtrsim t^2$. 
	
	Then, for general $u,v$, analogous to \cite[Eq.~(B. 72) in Supplement]{DY2}, we have
	\begin{align*}
		|u^\top [\mathcal{R}(z) - \Pi^x(z)] v | &\prec t^{-3}\Psi(z) + \frac{t^{-7/2}}{N^{1/2}} + \Big|\sum_{i \neq j} u_{[i]}^\top \mathcal{R}_{[ij]} u_{[j]}     \Big| \\
		&\quad+ \Big| \sum_{\mu \neq \nu \ge 2M+1} u_{\mu}^\top \mathcal{R}_{\mu \nu} u_{\nu}  \Big| + 2 \Big|\sum_{i \in \mathcal{I}_1, \mu \ge 2M+1 } u_{[i]}^\top \mathcal{R}_{[i]\mu} u_{\mu} \Big|. 
	\end{align*}
	Therefore, it suffices to prove the following high moment bounds, for any $a \in \mathbb{N}$,
	\begin{align*}
		&\E \Big|\sum_{i \neq j} u_{[i]}^\top \mathcal{R}_{[ij]} u_{[j]}     \Big|^{2a} \prec \Big(t^{-3}\Psi(z) + \frac{t^{-7/2}}{N^{1/2}}\Big)^{2a}, \\
		&\E \Big| \sum_{\mu \neq \nu \ge 2M+1} u_{\mu}^\top \mathcal{R}_{\mu \nu} u_{\nu}  \Big|^{2a} \prec \Big(t^{-3}\Psi(z) + \frac{t^{-7/2}}{N^{1/2}}\Big)^{2a}, \\
		&\E  \Big|\sum_{i \in \mathcal{I}_1, \mu \ge 2M+1 } u_{[i]}^\top \mathcal{R}_{[i]\mu} u_{\mu} \Big|^{2a} \prec \Big(t^{-3}\Psi(z) + \frac{t^{-7/2}}{N^{1/2}}\Big)^{2a}.
	\end{align*}
	The above estimates are proven using a polynomialization method outlined in \cite[Section 5]{BELKA}, with input from the entrywise estimates (\ref{eq:isotrop2}) and resolvent expansion (cf. \cite[Lemma B.2 in Supplement]{DY2}). We omit the details.      
\end{proof}
\begin{remark}
Actually, the estimates in Theorem \ref{thm: local law Gaussian divisible model} hold uniformly in $z$ such that
\begin{equation}\label{eq:1184}
	\lambda_{-,t}-\vartheta^{-1}t^{2} \le \Re{z} \le \lambda_{-,t} + \vartheta^{-1},\quad
	\Im{z}\cdot\Big(t+\big(|\Re{z}-\lambda_{-,t}|+\Im{z}\big)^{1/2}\Big)\ge N^{-1+\vartheta},\quad \Im{z}\le \vartheta^{-1},
\end{equation}
for any $\vartheta>0$. We can observe that every $z\in\mathsf{D}(\varepsilon_1, \varepsilon_2, \varepsilon_3)$ satisfies \eqref{eq:1184} if $\eps_{a}, \varepsilon_1, \varepsilon_2$ and $\vartheta$ are sufficiently small. 
Also note that $b_{t}=\mathcal{O}(1)$ and $1+t\underline{m}_{t}=\mathcal{O}(1)$ in the domain $\mathsf{D}(\varepsilon_1, \varepsilon_2, \varepsilon_3)$.
\end{remark}

By Theorem \ref{thm: local law Gaussian divisible model} and Lemma \ref{lem: domain of parameter}, it is enough to analyze $G(X,\zeta)$ and $G(X^{\top},\zeta)$ with $\zeta\in\mathsf{D}_\zeta$ in order to get the desired result. This was be done in Proposition \ref{prop: resolvent entry size for X=B+C}. Together with Proposition \ref{prop: resolvent product X} and Corollary \ref{cor: resolvent entry size for X=B+C, off-diag} below, we complete the proof of Theorem \ref{thm: resolvent entry size V_t}.

\begin{proposition}\label{prop: resolvent product X}
Suppose that the assumptions in Proposition \ref{prop: resolvent entry size for X=B+C} hold.  The following estimates hold with respect to the probability measure $\prob_{\Psi}$.
\begin{itemize}
	\item[(i)] If $i\in\mathcal{T}_{r}$, we have
	\begin{equation*}
		|[G(X, \zeta)X]_{ij}| \prec N^{-\eps_{b}/2}.
	\end{equation*}
	\item[(ii)] If $j\in\mathcal{T}_{c}$, we have
	\begin{equation*}
		|[G(X, \zeta)X]_{ij}| \prec N^{-\eps_{b}/2}.
	\end{equation*}
	\item[(iii)] Otherwise, we have the crude  bound
	\begin{equation*}
		|[G(X, \zeta)X]_{ij}|\leq \lVert G(X, \zeta)X \rVert \lesssim t^{-2}.
	\end{equation*}
\end{itemize}
\end{proposition}
\begin{proof}
Using Proposition \ref{prop: resolvent entry size for X=B+C}, it follows from Proposition \ref{prop: bound the form GX} below.
\end{proof}

With the above bounds, we can further improve the bound of the off-diagonal Green function entries when $i$ or $j$ is typical index.

\begin{corollary}\label{cor: resolvent entry size for X=B+C, off-diag}
Suppose that the assumptions in Proposition \ref{prop: resolvent entry size for X=B+C} hold. The following estimates hold with respect to the probability measure $\prob_{\Psi}$.
\begin{itemize}
	\item[(i)] If $i\neq j$ and $i\in\mathcal{T}_{r}$ (or $j\in\mathcal{T}_{r}$), there exists a constant $\delta=\delta(\eps_a,\eps_{\alpha},\eps_{b})>0$ such that
	\begin{equation*}
		|G_{ij}(X, \zeta)| \prec N^{-\delta}.
	\end{equation*}
	\item[(ii)] If $i\neq j$ and $i\in\mathcal{T}_{c}$ (or $j\in\mathcal{T}_{c}$), there exists  a constant $\delta=\delta(\eps_a,\eps_{\alpha},\eps_{b})>0$ such that
	\begin{equation*}
		|G_{ij}(X^{\top}, \zeta)| \prec N^{-\delta}.
	\end{equation*}
\end{itemize}
\end{corollary}
\begin{proof}[Proof of Corollary \ref{cor: resolvent entry size for X=B+C, off-diag}]
We shall give the proof only for the case $i\neq j$ and $i\in\mathcal{T}_{r}$. The other cases can be proved in the same way.
Assume $i\neq j$ and $i\in\mathcal{T}_{r}$, observe that
\begin{equation*}
	|G_{ij}(X,\zeta)| =  |G_{ii}(X,\zeta)| \cdot \left|\sum_{k,l}x_{ik}G_{kl}((X^{(i)})^{\top},\zeta)x_{jl}\right|,
\end{equation*}
where we denote by $X^{(i)}$ the matrix obtained from $X$ by removing $i$-th row.
Note that
\begin{equation*}
	\sum_{l}G_{kl}((X^{(i)})^{\top},\zeta)x_{jl} = [G((X^{(i)})^{\top},\zeta)(X^{(i)})^{\top}]_{kj}.
\end{equation*}
Since $i\in\mathcal{T}_{r}$, we apply the large deviation estimates in \cite[Corollary 25]{Amol} to bound
\begin{equation*}
	\left|\sum_{k}x_{ik} [G((X^{(i)})^{\top},\zeta)(X^{(i)})^{\top}]_{kj} \right|,
\end{equation*}
where we also use Proposition \ref{prop: bound the form GX} below to get a high probability bound for $\lVert G((X^{(i)})^{\top},\zeta)(X^{(i)})^{\top} \rVert$.
\end{proof}

\begin{proposition}\label{prop: bound the form GX}
Let $\zeta=E+\mathrm{i}\eta\in\mathbb{C}_{+}$.
\begin{itemize}
	\item[(i)] If $i\in\mathcal{T}_{r}$, we have
	\begin{multline*}
		|[G(X,\zeta)X]_{ij}|
	\prec \paren{N^{-\eps_{b}}\max_{k}|G_{kj}((X^{(i)})^{\top},\zeta)| + \paren{\frac{\imag\,G_{jj}((X^{(i)})^{\top},\zeta)}{N\eta}}^{1/2}} \\
	\times \paren{ 1 + |\zeta|\cdot|G_{ii}(X,\zeta)|\cdot \paren{N^{-\eps_{b}}\max_{k,l}|G_{kl}((X^{(i)})^{\top},\zeta)| + \paren{\frac{\sum_{k}\imag\,G_{kk}((X^{(i)})^{\top},\zeta)}{N^{2}\eta}}^{1/2}} },
	\end{multline*}
	where we denote by $X^{(i)}$ the matrix obtained from $X$ by removing $i$-th row.
	\item[(ii)] If $j\in\mathcal{T}_{c}$, we have
	\begin{multline*}
		|[G(X,\zeta)X]_{ij}| \prec \paren{ N^{-\eps_{b}}\max_{k}|G_{ik}(X^{[j]},\zeta)|  + \paren{\frac{\imag\,G_{ii}(X^{[j]},\zeta)}{N\eta}}^{1/2} } \\
	\times \paren{ 1 + |\zeta|\cdot|G_{jj}(X^{\top},\zeta)|\cdot \paren{N^{-\eps_{b}}\max_{k,l}|G_{kl}(X^{[j]},\zeta)| + \paren{\frac{\sum_{k}\imag\,G_{kk}(X^{[j]},\zeta)}{N^{2}\eta}}^{1/2}} },
	\end{multline*}
	where we denote by $X^{[j]}$ the matrix obtained from $X$ by removing $j$-th column.
	\item[(iii)] Let $X = UDV$ be a singular value decomposition of $X$ where
	\begin{equation*}
		\text{diag}(D)=(d_{1},d_{2}, \cdots, d_{p})\equiv\Big(\sqrt{\lambda_{1}(\mathcal{S}(X))},\sqrt{\lambda_{2}(\mathcal{S}(X))},\cdots,\sqrt{\lambda_{M}(\mathcal{S}(X))}\Big).
	\end{equation*}
	(Here we also assume $M<N$ without loss of generality.)
	Then,
	\begin{equation*}
		\lVert G(X,\zeta)X \rVert \le \max_{1\le i\le p} \left| \frac{d_{i}}{d_{i}^{2}-\zeta} \right|.
	\end{equation*}
\end{itemize}
\end{proposition}
\begin{proof}
(i) Assume $i\in\mathcal{T}_{r}$. Note that $G(X,\zeta)X = X G(X^{\top},\zeta)$. Let $x_{(i)}$ be the $i$-th row of $X$. See that
\begin{equation*}
	X^{\top}X - \zeta = (X^{(i)})^{\top}X^{(i)} - \zeta + x_{(i)}^{\top}x_{(i)}.
\end{equation*}
By the Sherman-Morrison formula, 
\begin{equation*}
	G(X^{\top},\zeta) = G((X^{(i)})^{\top},\zeta) - \frac{G((X^{(i)})^{\top},\zeta)x_{(i)}^{\top}x_{(i)}G((X^{(i)})^{\top},\zeta)}{1+x_{(i)}G((X^{(i)})^{\top},\zeta)x_{(i)}^{\top}}.
\end{equation*}
Since $\big(G_{ii}(X,\zeta)\big)^{-1}=-\zeta \big(1+x_{(i)}G((X^{(i)})^{\top},\zeta)x_{(i)}^{\top}\big)$,
\begin{equation*}
	G(X^{\top},\zeta) = G((X^{(i)})^{\top},\zeta) + (\zeta G_{ii}(X,\zeta)) \cdot G((X^{(i)})^{\top},\zeta)x_{(i)}^{\top}x_{(i)}G((X^{(i)})^{\top},\zeta).
\end{equation*}
We write $[X G(X^{\top},\zeta)]_{ij}=x_{(i)}G(X^{\top},\zeta)e_{j}$. Then,
\begin{equation*}
	x_{(i)}G(X^{\top},\zeta)e_{j} = x_{(i)}G((X^{(i)})^{\top},\zeta)e_{j} 
	+ (\zeta G_{ii}(X,\zeta)) \cdot (x_{(i)}G((X^{(i)})^{\top},\zeta)x_{(i)}^{\top}) \cdot (x_{(i)}G((X^{(i)})^{\top},\zeta)e_{j}).
\end{equation*}
Since $i\in\mathcal{T}_{r}$, by the large deviation estimate \cite[Corollary 25]{Amol}, the desired result follows.

\noindent
(ii) Assume $j\in\mathcal{T}_{c}$.
Let $x_{[j]}$ be $j$-th column of $X$. See that
\begin{equation*}
	[G(X,\zeta)X]_{ij} = e_{i}^{\top}G(X,\zeta) x_{[j]}.
\end{equation*}
By the Sherman-Morrison formula,
\begin{equation*}
	G(X,\zeta) = G(X^{[j]},\zeta) + (\zeta G_{jj}(X^{\top},\zeta)) \cdot G(X^{[j]},\zeta)x_{[j]}x_{[j]}^{\top}G(X^{[j]},\zeta),
\end{equation*}
where we denote by $X^{[j]}$ the matrix obtained from $X$ by removing $j$-th column. Then,
\begin{equation*}
	e_{i}^{\top}G(X,\zeta) x_{[j]} = e_{i}^{\top}G(X^{[j]},\zeta)x_{[j]}
	+ (\zeta G_{jj}(X^{\top},\zeta)) \cdot (e_{i}^{\top}G(X^{[j]},\zeta)x_{[j]}) \cdot x_{[j]}^{\top}G(X^{[j]},\zeta)x_{[j]}).
\end{equation*}
Using $j\in\mathcal{T}_{c}$, we get the desired result using the large deviation estimate \cite[Corollary 25]{Amol}.

\noindent
(iii) This is elementary, and thus we omit the details. 
\end{proof}


\subsection{Remark on Theorem \ref{thm:081601}}\label{sec:081601}
Theorem \ref{thm:081601} is a version of \cite[Theorem V.3]{DY} with respect to the left edge. The required modification would be straightforward. Let us summarize the main idea of \cite{DY} as follows.
Let $\mathsf{B}_{i}$ $(i=1,\cdots,M)$ be independent standard Brownian motions.
We fix two time scales:
\begin{equation}\label{eq: two time scale}
	t_{0} = N^{-\frac{1}{3}+\phi_{0}},\quad t_{1}=N^{-\frac{1}{3}+\phi_{1}},
\end{equation}
where $\phi_{0}\in (\frac{1}{3} - \frac{\epsilon_b}{2},\frac{1}{3})$ and $0<\phi_{1}<\frac{\phi_{0}}{100}$.

For time $t\ge 0$, we define the process $\{\lambda_{i}(t):1\le i\le M\}$ as the unique strong solution to the following system of SDEs:
\begin{equation*}
	{\rm d}\lambda_{i} = 2\lambda_{i}^{1/2}\frac{{\rm d}\mathsf{B}_{i}}{\sqrt{N}} + \paren{\frac{1}{N}\sum_{j\neq i}\frac{\lambda_{i}+\lambda_{j}}{\lambda_{i}-\lambda_{j}}}{\rm d}t, \quad
	1\le i\le M,
\end{equation*}
with initial data $\lambda_{i}(0)=\lambda_{i}(\gamma_{w}\mathcal{S}(V_{t_{0}}))$ where $\gamma_{w}$ is chosen to match the edge eigenvalue gaps of $\mathcal{S}(V_{t_{0}})$ with those of Wigner matrices. Recall the convention: $\lambda_{1}\ge \lambda_{2}\ge \cdots\ge\lambda_{M}$.

Note that the process $\{\lambda_{i}(t)\}$ has the same joint distribution as the eigenvalues of the matrix
\begin{equation*}
	\gamma_{w}\mathcal{S}(V_{t_{0}+\frac{t}{\gamma_{w}}}) = (\gamma_{w}^{1/2}X+(\gamma_{w}t_{0}+t)^{1/2}W)(\gamma_{w}^{1/2}X+(\gamma_{w}t_{0}+t)^{1/2}W)^{\top}.
\end{equation*}
Denote by $\rho_{\lambda,t}$ the asymptotic spectral distribution of $\mathcal{S}(V_{t_{0}+\frac{t}{\gamma_{w}}})$ (in terms of the rectangular free convolution actually). Let $E_{\lambda}(t)$ be the left edge of $\rho_{\lambda,t}$.
Now we introduce a deforemd Wishart matrix $\mathcal{U}\mathcal{U}^{\top}$.
Define $\mathcal{U}\coloneqq\Sigma^{1/2}\mathcal{X}$ where $\mathcal{X}$ is a $M\times N$ real Gaussian matrix (mean zero and variance $N^{-1}$) and $\Sigma=\text{diag}(\sigma_{1},\cdots,\sigma_{M})$ is a diagonal population matrix. 
Let $\rho_{\mu,0}$ be the asymptotic spectral distribution of $\mathcal{U}\mathcal{U}^{\top}$ (given by the multiplicative free convolution of the MP law and the ESD of $\Sigma$).
We choose the diagonal population covariance matrix $\Sigma$ such that $\rho_{\mu,0}$ matches $\rho_{\lambda,0}$ near the left edge $E_{\lambda}(0)$ (square-root behavior).
We write $\mu_{i}(0)\coloneqq\mu_{i}(\mathcal{U}\mathcal{U}^{\top})$.
Next, define the process $\{\mu_{i}(t):1\le i\le M\}$ through the rectangular DBM with initial data $\{\mu_{i}(0)\}$.
We can show that the edge eigenvalues of $\{\mu_{i}(t)\}$ are governed by the Tracy-Widom law. We denote by $\rho_{\mu,t}$ the rectangular free convolution of $\rho_{\mu,0}$ with the Marchenko-Pastur (MP) law at time $t$. Let $E_{\mu}(t)$ be the left edge of $\rho_{\mu,t}$. We remark that $E_{\lambda}(0)=E_{\mu}(0)$. Then, in order to get Theorem \ref{thm:081601}, it is enough to show
\begin{align*}
	\big| \big(\lambda_{M}(t_{1})-E_{\lambda}(t_{1})\big) - \big(\mu_{M}(t_{1})-E_{\mu}(t_{1})\big) \big| \prec N^{-2/3-\delta},
\end{align*}
for $\delta>0$ sufficiently small. The proof of the above estimate relies on the local equilibrium mechanism of the rectangle DBM, which does not have any difference between the left edge or the right edge of the spectrum, given $\eta_\ast$-regularities of the initial states. Hence, we omit  the remaining argument, and refer to \cite{DY} for details. 


\subsection{Proof of Lemma \ref{lem:preEst}}\label{3793}

We shall prove Lemma \ref{lem:preEst} in this section.

\begin{proof}[Proof of Lemma \ref{lem:preEst} (i)]
The proof is similar to that in \cite{DY2}, we provide proof here completeness.
The statement $\zeta_{-,t} - \lambda_M(\mathcal{S}(X))  \le 0$ follows directly from Lemma \ref{lem:Phicharacterization}. For the other estimate, by Lemma \ref{lem:Phicharacterization},
we know that $\Phi_t(\zeta_{-,t})$ is the only local extrema of $\Phi_t(\zeta)$ on the interval $(0,\lambda_M(\mathcal{S}(X)) )$. Hence we have $\Phi'_t(\zeta_{-,t})=0$, which gives the equation
\begin{align*} 
(1-c_Nt m_{X}(\zeta_{-,t}))^2 - 2 c_Nt m_{X}'(\zeta_{-,t}) \cdot \zeta_{-,t} \left( 1-c_Nt m_{X}(\zeta_{-,t})\right) - c_N(1-c_N)t^2m_{X}'(\zeta_{-,t})=0.
\end{align*}
Rearranging the terms, we can get 
\begin{align}\label{eq:mx0p}
	 c_Nt m_{X}'(\zeta_{-,t}) = \frac{(1-c_Nt m_{X}(\zeta_{-,t}))^2}{2\zeta_{-,t} \left( 1-c_Nt m_{X}(\zeta_{-,t})\right)+(1-c_N)t}.
\end{align}
By Lemma \ref{existuniq}~(iv) 
and Eq.~\eqref{eq:fpe}, 
we have on $\Omega_\Psi$ that  
\begin{align}\label{eq:estmxt2}
		c_Nt m_{X}(\zeta_{-,t})=\mathcal{O}(t^{1/2}).
\end{align} 
Plugging the above bound back to (\ref{eq:mx0p}), we can get $m'_{X}(\zeta_{-,t}) \sim t^{-1}$. This together with Lemma \ref{lem:esthighoderi} gives $\sqrt{\lambda_M(\mathcal{S}(X)) - \zeta_{-,t}} \sim t$.
\end{proof}


\begin{proof}[Proof of Lemma \ref{lem:preEst} (ii)]
Since $\mathcal{S}(X)$ is $\eta_\ast$-regular in the sense of Definition \ref{def:etastar}, 
the estimates for $|m_X^{(k)}(\zeta)|$ on the event $\Omega_\Psi$ is an immediate consequence of Lemmas \ref{lem:esthighoderi} and Lemma \ref{lem:preEst} (i).

We prove the estimate for $|m_X(z) - \mathsf{m}_{\mathsf{mp}}^{(t)}(z)|$ as follows. Recall that $\beta = (\alpha-2)/24$.
First, we establish the convergence of Stieltjes transform of a truncated matrix model using the result in  \cite{HLS}. To this end, let us define $\bar{X} = (\bar{x}_{ij}) :=  (x_{ij}\mathbf{1}_{x_{ij} < N^{-\beta}})$ and $\bar{t} := 1- N\E |\bar{x}_{ij}|^2$. It is easy to show that $|\bar{t}-t|=\mathfrak{o}(N^{-1})$, and thus we have $|\mathsf{m}_{\mathsf{mp}}^{(t)}(z_1)-\mathsf{m}_{\mathsf{mp}}^{(\bar{t})}(z_1)|\leq (N\eta_1)^{-1}$.  Then 
it follows from \cite[Theorem 2.7]{HLS} that for any $z_1$ such that $|z_1 - \zeta_{-,t}| \le \tau t^2$ and $\eta_1 \equiv \Im z_1 > N^{-1+\delta}$ with $1>\delta > 0$ to be chosen later, 
\begin{align}
		m_{\bar{X}}(z_1) - \mathsf{m}_{\mathsf{mp}}^{(t)}(z_1) \prec \frac{1}{N^{\beta}} + \frac{1}{N\eta_1}, \label{081410}
\end{align}
We remark here that the local law proved in \cite[Theorem 2.7]{HLS} is for deterministic $z$. But it is easy to show that the local law holds uniformly in $z$ in the mentioned domain in \cite[Theorem 2.7]{HLS}, with high probability, by a simple continuity argument. Hence, as long as $z_1$ fall in this domain with high probability, even though $z_1$ might be random, we still have (\ref{081410}). Using the facts $|\lambda_{M}(\mathcal{S}(X))-(1-t)\lambda_{-}^{\mathsf{mp}}|\lesssim N^{-\epsilon_b}$ and $\lambda_M(\mathcal{S}(X)) - \zeta_{-,t} \sim t^2$ with high probability (cf. Lemmas \ref{lem: left edge rigidity XX^T} and \ref{lem:preEst} (i)), we have for $\tau$ small enough, 
\begin{align*}
		|z_1 - (1-t)\lambda_{-}^{\mathsf{mp}}| \ge |\zeta_{-,t} - \lambda_{M}(\mathcal{S}(X))| - |\lambda_{M}(\mathcal{S}(X)) - (1-t)\lambda_{-}^{\mathsf{mp}}| - |z_1 - \zeta_{-,t}| \gtrsim t^2,
\end{align*}
which gives $|(\mathsf{m}_{\mathsf{mp}}^{(t)})'(z_1)| \lesssim t^{-4}$ with high probability. Also, we have $|m'_{\bar{X}}(z_1)|\lesssim t^{-4}$ with high probability,  by the choice of $z_1$, Eq. (\ref{eq: left edge rigidity BB^T}), and Lemma \ref{lem:preEst} (i). Therefore, for any $z_2$ satisfying $\Re z_2 = \Re z_1$ and $\eta_2 = \Im z_2 < N^{-1+\delta}$, we have
\begin{align}
		|m_{\bar{X}}(z_2) - \mathsf{m}_{\mathsf{mp}}^{(t)}(z_2)| {\lesssim}   |m_{\bar{X}}(z_1) - \mathsf{m}_{\mathsf{mp}}^{(t)}(z_1)| + t^{-4} |z_1 - z_2| {\prec}  \frac{1}{N^{\beta}} + \frac{1}{N^{1/2}} + \frac{1}{t^4N^{1/2}} \lesssim \frac{1}{N^{\beta}},\label{081411}
\end{align}
where in the first step we used the fact $|z_i - \zeta_{-,t}| \le  \tau t^2, i = 1,2$, and in the second step we chose $\delta = 1/2$. 
	
Next, we use the rank inequality to compare $m_{\bar{X}}(z)$ with $m_X(z)$. Notice that
\begin{align*}
		m_{\bar{X}}(z) - m_X(z) 
		 \le \frac{2}{N} \mathrm{Rank}(\bar{X} - X)\cdot(\|(\mathcal{S}(\bar{X}) -z)^{-1}  \| +\|(\mathcal{S}(X) -z)^{-1}  \| )\prec \frac{\mathrm{Rank}(\bar{X} - X)}{Nt^2}.
\end{align*}

A similar argument as in the proof of Lemma \ref{lem:goodpsi} 
shows that,
\begin{align*}
		\mathrm{Rank}(\bar{X} - X) \prec N^{1 - (\alpha-2-2\alpha\beta )/4}.
\end{align*} 
Therefore, we can obtain $m_{\bar{X}}(z) - m_X(z) \prec N^{- (\alpha-2-2\alpha\beta )/4 }t^{-2}$. Together with the estimate in (\ref{081411}), we have
\begin{align*}
		m_X(z) - \mathsf{m}_{\mathsf{mp}}^{(t)}(z) \prec \frac{1}{N^{(\alpha-2-2\alpha\beta )/4 }t^2} + \frac{1}{N^{\beta}} .
\end{align*}
The claim now follows by the fact $t \gg N^{(2-\alpha)/16}$ in light of Eq.~\eqref{081401}. 
\end{proof}


\begin{proof}[Proof of Lemma \ref{lem:preEst} (iii)]
	Repeating the proof of \cite[Lemma A.2]{DY2}, we can obtain
	\begin{align*}
		|\bar{\zeta}_{-,t} -  \zeta_{-,t}| \lesssim t^3 |m_X'(\zeta_{-,t}) - (\mathsf{m}_{\mathsf{mp}}^{(t)})'(\zeta_{-,t}) |.
	\end{align*}
	By the Cauchy integral formula, we have
	\begin{align}\label{eq:cauchyintegral}
		|m_X'(\zeta_{-,t}) - (\mathsf{m}_{\mathsf{mp}}^{(t)})'(\zeta_{-,t})| \lesssim \oint_{\omega} \frac{|m_X(a) - \mathsf{m}_{\mathsf{mp}}^{(t)}(a)|}{|a-\zeta_{-,t}|^2} \mathrm{d}a,
	\end{align}
	where $\omega \equiv \{a : |a-\zeta_{-,t}| = \tau t^2\}$ for some small $\tau$. Therefore, we have by Lemma \ref{lem:preEst} (ii), 
	\begin{align*}
		|\bar{\zeta}_{-,t} -  \zeta_{-,t}| \lesssim t \sup_{a \in \omega} |m_X(a) - \mathsf{m}_{\mathsf{mp}}^{(t)}(a)| \prec t N^{-\beta},
	\end{align*}
	proving the claim.
\end{proof}


\subsection{Proof of Proposition \ref{prop:removing off-diagonal}}
\label{3891}

In this section, we shall give the proof of Proposition \ref{prop:removing off-diagonal}. 
\begin{proof}[Proof of Proposition \ref{prop:removing off-diagonal}]
By a minor process argument, we have with probability at least $1 - N^{-D}$ for arbitrary large $D$, there exists constant $C_k > 0$, such that
\begin{align}\label{eq:hpeforlb}
	 &	|\lambda_{M}(\mathcal{S}(\tilde{X}^{(k)}))- \hat{\zeta}_{\mathsf{e}}|= \Big|(1-t)\lambda_{-}^{\mathsf{mp}} - \bar{\zeta}_{-,t} +\lambda_{M}(\mathcal{S}(\tilde{X}^{(k)})) - (1-t)\lambda_{-}^{\mathsf{mp}} + \mathrm{i}N^{-100K} + \bar{\zeta}_{-,t}-\hat{\zeta}_{\mathsf{e}} \Big|\notag\\
		&\ge  \sqrt{c_N}t^2 - |\lambda_{M}(\mathcal{S}(\tilde{X}^{(k)})) - (1-t)\lambda_{-}^{\mathsf{mp}}| - |\bar{\zeta}_{-,t}-\hat{\zeta}_{\mathsf{e}}| - N^{-100K} \ge C_kt^2.
\end{align}
Here in the last step, we used Eq.~\eqref{eq:zetaminusbarzeta} 
and the fact that  $|\lambda_{M}(\mathcal{S}(\tilde{X}^{(k)}))- (1-t)\lambda_-^{\mathsf{mp}}|  \prec N^{-\epsilon_b}$. Therefore, for any $k \in [N]$, we can define the event
	$
		\Omega_k \equiv \{ \lambda_{M}(\mathcal{S}(\tilde{X}^{(k)})) - \bar{\zeta}_{-,t} \ge C_kt^2 \}
	$ with  $\mathbb{P}(\Omega_k) \ge 1 - N^{-D}$ for arbitrary large $D$. 
	
%
	Choosing $\tau \le \min_k C_k/2$. For any $\zeta$ satisfying $| \zeta - \hat{\zeta}_{\mathsf{e}}| \le \tau t^2$, we define 
	\begin{align*}
		&F_k(\zeta) := \log |1 + \tilde{x}_k^\top (G(\tilde{X}^{(k)}, \zeta))\tilde{x}_k|^2,\quad \tilde{F}_k(\zeta) := \log |1 + \tilde{x}_k^\top (G(\tilde{X}^{(k)}, \zeta))_{\mathrm{diag}}\tilde{x}_k|^2.
	\end{align*}
	Since $|\lambda_M(\mathcal{S}(\tilde{X}^{(k)})) - \zeta| = |\lambda_{M}(\mathcal{S}(\tilde{X}^{(k)}))- \hat{\zeta}_{\mathsf{e}}| - |\zeta - \hat{\zeta}_{\mathsf{e}}| \ge  C_kt^2/2 > 0$ on  $\Omega_k$, we can obtain that $\Re (\tilde{x}_k^\top (G(\tilde{X}^{(k)}, \zeta))\tilde{x}_k) \vee  \Re (\tilde{x}_k^\top (G(\tilde{X}^{(k)}, \zeta))_{\mathrm{diag}}\tilde{x}_k) \ge 0$. Hence, the functions $F_k(\zeta)$, $\tilde{F}_k(\zeta)$ are well defined on the event $\Omega_k$.  For any  $\zeta\in \Xi(\tau)$, using Cauchy integral formula with a cutoff of the contour chosen carefully, 
	we can express $Y_k \equiv Y_k(\zeta)$ as
	\begin{align*}
		Y_k = \frac{t}{2\pi \mathrm{i}N^{1- \alpha/4}}(\E_k - \E_{k-1}) \oint_{\omega\cap \gamma} \frac{F_k(z)}{(z-\zeta)^2}\mathrm{d}z + \mathsf{err}_k(\zeta) =: I_k(\zeta) + \mathsf{err}_k(\zeta),
	\end{align*} 
	with the contour $\omega \equiv \{z\in \mathbb{C} : |z - \zeta| = \tau t^2 / 10 \}$ and   $\gamma \equiv \{ z \in \mathbb{C}: |\Im z| \ge N^{-100}  \}$, and $\mathsf{err}_k$ collects all the tiny error terms which will not affect our further analysis.  	
Similarly, we can define $\tilde{I}_k(\zeta)$ and $\tilde{\mathsf{err}}_k(\zeta)$ for $\tilde{Y}_k$ in the same manner as shown above.
	 Therefore,
	\begin{align*}
		\E_{k-1}(Y_kY_k') - \E_{k-1}(\tilde{Y}_k\tilde{Y}_k') = \E_{k-1}((I_k(\zeta)I_k(\zeta') ) - \E_{k-1}((\tilde{I}_k(\zeta)\tilde{I}_k(\zeta') ) + \mathsf{HOT},
	\end{align*}
	where $\mathsf{HOT}$ collects terms containing $\mathsf{err}_k(\zeta)$ or $\tilde{\mathsf{err}}_k(\zeta)$, which are irrelevant in our analysis. 
For the leading term, since $F_k(z)$, $\tilde{F}_k(z)$,$\tilde{F}_k(z)$,$\tilde{F}_k(z')$ are uniformly bounded on $z \in 
\omega\cap \gamma$ and $z' \in \omega' \cap \gamma$, we may commute the conditional expectation and the integral to obtain
\begin{align}
	\E_{k-1}((I_k(\zeta)I_k(\zeta') )-\E_{k-1}((\tilde{I}_k(\zeta)\tilde{I}_k(\zeta') ) = -\frac{t^2}{4\pi^2N^{2- \alpha/2}} \oint_{\omega \cap \gamma} \oint_{\omega' \cap \gamma} \frac{\varphi_k(z,z') -\tilde{\varphi}_k(z,z')  }{(z-\zeta)^2(z'-\zeta')^2}\mathrm{d}z'\mathrm{d}z   ,
\label{eq:CIforphi}\end{align}
where
\begin{align*}
	\varphi_k(z,z') \coloneqq \E_{k-1} \big((\E_k - \E_{k-1}) F_k(z) (\E_k - \E_{k-1}) F_k(z')\big)\notag\\
	\tilde{\varphi}_k(z,z') \coloneqq \E_{k-1} \big((\E_k - \E_{k-1}) \tilde{F}_k(z) (\E_k - \E_{k-1}) \tilde{F}_k(z')\big),
\end{align*} and $\omega' \coloneqq \{z\in \mathbb{C} : |z - \zeta'| =  at^2  \}$  with a small constant $a$. 

In view of (\ref{eq:CIforphi}), it suffices to prove that uniformly on $z\in 
\omega\cap \gamma$ and $z' \in \omega' \cap \gamma$, $\varphi_k - \tilde{\varphi_k} \equiv \varphi_k(z,z') - \tilde{\varphi}_k(z,z') \ll   t^2N^{1-\alpha/2}$. In the sequel, we write $F_k = F_k(z)$, $\tilde{F}_k = \tilde{F}_k(z)$, $F_k' =F_k(z')$, and $\tilde{F}_k' = \tilde{F}_k(z')$
for simplicity. Let 
$$\eta_k =\eta_k(z) := \tilde{x}_k^\top (G(X^{(k)}, z))\tilde{x}_k - \tilde{x}_k^\top (G(\tilde{X}^{(k)}, z))_{\mathrm{diag}}\tilde{x}_k = \sum_{i \neq j} [G(\tilde{X}^{(k)}, z)]_{ij} \tilde{x}_{ik}\tilde{x}_{jk},$$
and
$$	\varepsilon_k = \varepsilon_k(z) := F_k - \tilde{F}_k = \log |1 + \eta_k (1+ \tilde{x}_k^\top (G(\tilde{X}^{(k)}, z))_{\mathrm{diag}}\tilde{x}_k)^{-1} |^2.$$
We also write $\eta_k' \equiv \eta_k(z')$ and $\varepsilon_k' \equiv \varepsilon_k(z')$. Using the following elementary identity,
\begin{align*}
	\E_{k-1}\big( (\E_k - \E_{k-1})(A)(\E_k - \E_{k-1})(B) \big) = \E_{k-1}\big( \E_k(A) \E_k(B) \big) - \E_{k-1}(A) \E_{k-1}(B),
\end{align*}
 we may rewrite $\varphi_k$ and $\tilde{\varphi}_k$ as 
\begin{align*}
	\varphi_k = \E_{k-1}\big( \E_k(F_k) \E_k(F_k') \big) - \E_{k-1}(F_k) \E_{k-1}(F_k'), 
\notag\\
\tilde{\varphi}_k = \E_{k-1}\big( \E_k(\tilde{F}_k) \E_k(\tilde{F}_k') \big) - \E_{k-1}(\tilde{F}_k) \E_{k-1}(\tilde{F}_k').
\end{align*}
Therefore, let $E_{\tilde{x}_k}$ denote the expectation with respect to the randomness of $k$-th column of $\tilde{X}$, we have by the definitions of $\varepsilon_k$, $\varepsilon_k'$,
\begin{align*}
	\varphi_k  - \tilde{\varphi}_k &= \E_{\tilde{x}_k}\big( \E_k(\tilde{F}_k)\E_k(\varepsilon_k') \big) +\E_{\tilde{x}_k}\big( \E_k(\tilde{F}_k')\E_k(\varepsilon_k) \big) + \E_{\tilde{x}_k}\big( \E_k({\varepsilon}_k)\E_k(\varepsilon_k') \big) \\
	&\quad -\E_k \E_{\tilde{x}_k}(\tilde{F}_k)\E_k \E_{\tilde{x}_k}({\varepsilon}_k')-\E_k \E_{\tilde{x}_k}(\tilde{F}_k')\E_k \E_{x_k}({\varepsilon}_k)-\E_k \E_{\tilde{x}_k}({\varepsilon}_k')\E_k \E_{\tilde{x}_k}({\varepsilon}_k) \\
	&\equiv T_1 + T_2 + T_3 + T_4 + T_5 + T_6.
\end{align*}
Before bounding $T_i$'s, $1 \le i \le 6$, we introduce some shorthand notation for simplicity. Let
\begin{align*}
	J_k = J_k(z) := \frac{1}{1 + \tilde{x}_k^\top G(\tilde{X}^{(k)}, z)\tilde{x}_k},\quad J_{k,\mathrm{diag}} = J_{k,\mathrm{diag}}(z) := \frac{1}{1 + \tilde{x}_k^\top (G(\tilde{X}^{(k)}, z))_{\mathrm{diag}}\tilde{x}_k},
\end{align*} 
and $J_k' = J_k(z')$, $J_{k,\mathrm{diag}}' = J_{k,\mathrm{diag}}(z')$. Further set
\begin{align*}
	J_{k,\Tr} \coloneqq \frac{1}{1 + \frac{\sigma_N^2}{N}\Tr  G(\tilde{X}^{(k)}, z)},\qquad \mathcal{E} \coloneqq \tilde{x}_k^\top (G(\tilde{X}^{(k)}, z))_{\mathrm{diag}}\tilde{x}_k - \frac{\sigma_N^2}{N}\Tr  G(\tilde{X}^{(k)}, z). 
\end{align*}
 This gives $J_{k,\mathrm{diag}} = J_{k,\Tr} - \mathcal{E} J_{k,\Tr}J_{k,\mathrm{diag}}$. We may now establish an upper bound for $\E_{\tilde{x}_k}(\varepsilon_k)$ as follows:
\begin{align*}
	\E_{\tilde{x}_k}(\varepsilon_k) &= \E_{\tilde{x}_k} \log |1 +\eta_k J_{k,\mathrm{diag}}|^2 \overset{(\mathrm{i})} \le \log \E_{\tilde{x}_k}  |1 +\eta_k J_{k,\mathrm{diag}}|^2\\
	&= \log \E_{\tilde{x}_k}(1+2\Re (\eta_kJ_{k,\Tr}-\eta_kEJ_{k,\Tr}J_{k,\mathrm{diag}}) + |\eta_k J_{k,\mathrm{diag}}|^2 )\\
	&\overset{(\mathrm{ii})}{\le} \log \big(1 + \mathcal{O}(\E_{\tilde{x}_k}(|\eta_k||\mathcal{E}|)) + \mathcal{O}(\E_{\tilde{x}_k}(|\eta_k|^2)) \big),
\end{align*} 
where, in $(\mathrm{i})$, Jensen's inequality is applied, and in $(\mathrm{ii})$, we used the fact that $J_{k,\Tr}$ and $J_{k,\mathrm{diag}}$ are uniformly bounded for $\zeta \in \Xi$ on the event $\Omega_k$. Similarly, using the identity $|1 +\eta_k J_{k,\mathrm{diag}}||1 - \eta_k J_k| = 1$, we have
\begin{align*}
	\E_{\tilde{x}_k}(-\varepsilon_k) &= \E_{\tilde{x}_k} \log |1 -\eta_k J_{k}|^2 = \E_{\tilde{x}_k} \log |1 -\eta_k J_{k,\Tr} -\eta_k(\eta_k+\mathcal{E})J_{k,\mathrm{diag}}J_k |^2 \\
	&\le \log\big(1 + \mathcal{O}(\E_{\tilde{x}_k}(|\eta_k||\mathcal{E}|)) + \mathcal{O}(\E_{\tilde{x}_k}(|\eta_k|^2)) \big).
\end{align*}
By the Cauchy-Schwarz inequality and Lemma \ref{lem:ldforheavytail} and Lemma \ref{lem:Exketak},
\begin{align*}
	\E_{\tilde{x}_k}(|\eta_k||\mathcal{E}|) \le \sqrt{\E_{\tilde{x}_k}(|\eta_k|^2) \cdot \E_{\tilde{x}_k}(|\mathcal{E}|^2)  } \lesssim N^{-1/2}t^{-2}N^{\vartheta(2-\alpha/2)-1/2}\|G(\tilde{X}^{(k)}, z)  \|
\end{align*}
Since $\vartheta = 1/4+1/\alpha+\epsilon_{\vartheta} > 1/4+1/\alpha$, and recall that $\|G(\tilde{X}^{(k)}, z)  \| \le |\lambda_1(\mathcal{S}(\tilde{X}^{(k)})) - z |^{-1} \lesssim t^{-2}$ on $\Omega_k$, the above bound can be further simplified as
\begin{align*}
	\E_{\tilde{x}_k}(|\eta_k||\mathcal{E}|) \lesssim N^{1-\alpha/2} \cdot  N^{2/\alpha + 3\alpha/8+ \epsilon_{\vartheta}(4-\alpha)/2 - 2}t^{-4}.
\end{align*}
By the facts $\epsilon_{\vartheta} < (3\alpha - 5)/(4\alpha)$ and $t \gg N^{(\alpha - 4) / 48}$, it can be verified that $\E_{\tilde{x}_k}(|\eta_k||\mathcal{E}|) \ll t^2N^{1-\alpha/2}$. Therefore,  we can conclude that  $|\E_{\tilde{x}_k} (\varepsilon_k) | \ll t^2N^{1-\alpha/2}$. This shows $|T_6| \ll t^2N^{1-\alpha/2}$. Together with the crude bound 
$
	\tilde{F}_k \le \log |1 + N^{2\vartheta}\|G(\tilde{X}^{(k)}, z)  \| |^2 \lesssim \log N
$, we have $|T_4|, |T_5| \ll t^2N^{1-\alpha/2}.$ 

For $|T_3|$, by Cauchy-Schwarz inequality, it suffices to give a bound on $\E_{\tilde{x}_k}\big( |\E_k(\varepsilon_k)|^2 \big)$. By Jensen's inequality, 
\begin{align*}
	\E_{\tilde{x}_k}\big( |\E_k(\varepsilon_k)|^2 \big) \le \E_k  \E_{\tilde{x}_k}(|\varepsilon_k|^2).
\end{align*}
Using again the identity $|1 +\eta_k J_{k,\mathrm{diag}}||1 - \eta_k J_k| = 1$,
\begin{align*}
	&|\log|1 + \eta_k J_{k,\mathrm{diag}}|^2| = \mathbf{1}_{\{|1 + \eta_k J_{k,\mathrm{diag}}| \geq  1 \}} \log|1 + \eta_k J_{k,\mathrm{diag}}|^2 + \mathbf{1}_{\{|1 - \eta_k J_{k}| > 1 \}} \log|1 - \eta_k J_{k}|^2 \\
	& = \mathbf{1}_{\{|1 + \eta_k J_{k,\mathrm{diag}}| > 1 \}} \log(1 +2\Re (\eta_k J_{k,\mathrm{diag}}) +  |\eta_k J_{k,\mathrm{diag}}|^2) + \mathbf{1}_{\{|1 - \eta_k J_{k}| > 1 \}} \log(1 - 2\Re (\eta_k J_{k})  + |\eta_k J_{k}|^2) \\
	& \le \mathbf{1}_{\{|1 + \eta_k J_{k,\mathrm{diag}}| > 1 \}} \big( 2\Re (\eta_k J_{k,\mathrm{diag}}) +  |\eta_k J_{k,\mathrm{diag}}|^2 \big) + \mathbf{1}_{\{|1 - \eta_k J_{k}| > 1 \}} \big(- 2\Re (\eta_k J_{k})  + |\eta_k J_{k}|^2 \big).
\end{align*}
Therefore, with the fact that $|\eta_k J_{k,\mathrm{diag}}| \le N^C$ for some $C > 0$, 
\begin{align*}
	\E_{\tilde{x}_k}|\log|1 + \eta_k J_{k,\mathrm{diag}}|^2|^2 \lesssim \log N \cdot  \E_{\tilde{x}_k} \log|1 + \eta_k J_{k,\mathrm{diag}}|^2 \lesssim \log N\cdot \E_{\tilde{x}_k}(|\eta_k|^2) \lesssim N^{-1}t^{-5},
\end{align*}
which gives $|T_3| \ll t^2N^{1-\alpha/2}$ by the fact $t \gg N^{-2/7+\alpha/14}$. 

To evaluate $|T_2|$, we start by expressing it as follows:
\begin{align*}
T_2 &= \E_{\tilde{x}_k}\big( \E_k(\varepsilon_k) \E_k (\log |1 + N^{-1}\sigma_N^2\Tr G(\tilde{X}^{(k)}, z') + \mathcal{E}|^2) \big) \\
&= \E_{\tilde{x}_k}\big( \E_k(\varepsilon_k) \E_k (\log |1 + N^{-1}\sigma_N^2\Tr G(\tilde{X}^{(k)}, z')|^2) \big) + \E_{\tilde{x}_k}\big( \E_k(\varepsilon_k) \E_k (\log |1 + \mathcal{E}J_{k,\Tr}|^2) \big).
\end{align*}
 First, we use the fact that $\log |1 + N^{-1}\sigma_N^2\Tr G(\tilde{X}^{(k)}, z)|^2$ is independent of $\tilde{x}_k$ and that $\E_{\tilde{x}_k}(\varepsilon_k) = 0$ to obtain the inequality
$$ T_2 \lesssim \E_{\tilde{x}_k}\big( |\E_k(\varepsilon_k)| \cdot \E_k |\mathcal{E}| \big). $$
Next, we apply the Cauchy-Schwarz inequality to obtain
$$ T_2 \le \sqrt{\E_{\tilde{x}_k} \big(|\E_k(\varepsilon_k)|^2\big) \E_{\tilde{x}_k} \big( |\E_k( |\mathcal{E}|)|^2\big) }. $$
Finally, by Jensen's inequality, we have
$$ T_2 \le \sqrt{\E_k\E_{\tilde{x}_k} \big(|\varepsilon_k|^2\big) \E_k\E_{\tilde{x}_k} \big( |\mathcal{E}|^2\big) } \le N^{1-\alpha/2} \cdot  N^{2/\alpha + 3\alpha/8+ \epsilon_{\vartheta}(4-\alpha)/2 - 2}t^{-9/2}. $$
The bound $|T_2| \ll t^2N^{1-\alpha/2}$ follows by the facts $\epsilon_{\vartheta} < (3\alpha - 5)/(4\alpha)$ and $t \gg N^{(\alpha - 4) / 56}$. The same bound holds for $|T_1|$. Therefore, we can obtain that for any $z \in 
\omega\cap \gamma$ and $z' \in \omega' \cap \gamma$, $|\varphi_k - \tilde{\varphi}_k| \ll t^2N^{1-\alpha/2}$, which conludes the proof.
\end{proof}

\begin{lemma}[\cite{BM}, Lemma 4.1]\label{lem:ldforheavytail}
Let $a \equiv (a_1,\cdots, a_N)^\top$ be a column vector whose entries are i.i.d.~centered and satisfy $({ii})$ and $({iii})$ in Lemma \ref{lem:truncationprop}. 
Then for deterministic matrix $G$, the random variables
\begin{align*}
	X \equiv \sum_{i\neq j} G_{ij}a_ia_j, \quad E \equiv \sum_{i} G_{ii}a_i^2 - \frac{1}{N}\Tr G
\end{align*}
satisfy
\begin{align*}
	\E |X|^2 \le 2N^{-1}\|G \|^2, \quad \E |E|^2 \le 10C(\| G\|^2 + 1)N^{\vartheta(4-\alpha)-1}.
\end{align*}
\end{lemma}

\noindent
The following lemma is a directly consequence of Lemma \ref{lem:ldforheavytail}.

\begin{lemma}\label{lem:Exketak}
 Fix $C > 0$. For  any $\zeta \in \{ \xi \in \mathbb{C} : | \xi - \bar{\zeta}_{-,t}| \le C t^2 \}$, we have there exist constant $\tau =  \tau(C)$ such that $\E_{\tilde{x}_k}(|\eta_k|^2) \le \tau^{-2} N^{-1}t^{-4}$ on the event $\Omega_k = \{ \lambda_{1}(\mathcal{S}(\tilde{X}^{(k)}))- \bar{\zeta}_{-,t} \ge \tau t^2 \}$.
\end{lemma}


\section{Remaining proofs for the general model}

\subsection{Proof of Lemma \ref{lem:monotonlema}}
\label{4103}

We need the following lemma on the monotonicity of the Green function to the linearization of $\mathcal{S}(Y^\gamma)$.

\begin{lemma}[\cite{bauerschmidt2017local}, Lemma 2.1]\label{lem:continuitylemma}
	For deterministic matrix $A \in \mathbb{R}^{M\times N}$, let $\mathcal{L}(A)$ be defined as in Eq.~\eqref{081611} 
	Further define $\Gamma(z) \coloneqq \max_{i, j \in [M+N]}[(\mathcal{L}(A)-z)^{-1}]_{ij}\vee 1$.
	We have for any $L > 1$ and $z \in \mathbb{C}^+$, we have $\Gamma(E+\mathrm{i}\eta/L) \le L\Gamma(E+\mathrm{i}\eta)$.
\end{lemma}

Recall that for any $\delta > 0$, $z = E + \mathrm{i}\eta \in \mathsf{D}$,
\begin{align*}
&\mathfrak{P}_{0}(\delta,z,\Psi) =\mathbb{P}_{\Psi}\Big( \sup_{\substack{a , b \in [M] \\ 0 \le \gamma \le 1}} |z^{1/2}\mathfrak{X}_{ab}[G^{\gamma}(z)]_{ab}| >  N^{\delta}\Big),\\
	&\mathfrak{P}_{1}(\delta,z,\Psi) = \mathbb{P}_{\Psi}\Big( \sup_{\substack{u , v \in [N] \\ 0 \le \gamma \le 1}} |z^{1/2}\mathfrak{Y}_{uv}[\mathcal{G}^{\gamma}(z)]_{uv}| > N^{\delta}\Big), \\
	&\mathfrak{P}_{2}(\delta,z,\Psi) = \mathbb{P}_{\Psi}\Big( \sup_{\substack{a \in [M], u \in [N] \\ 0 \le \gamma \le 1}} |\mathfrak{Z}_{au}[G^{\gamma}(z)Y^\gamma ]_{au}|  > N^{\delta}\Big).
	\end{align*}

\noindent
Now let us give the proof of Lemma \ref{lem:monotonlema}.

\begin{proof}[Proof of Lemma \ref{lem:monotonlema}]
Let $p$ be any sufficiently large (but fixed) integer, and $F_p(x) := |x|^{2p}+ 1$. It can be easily verified that there exists a constant $C_p$, only depends on $p$ such that $|F^{(a)}_p(x)| \le C_p F_p(x)$, for all $x \in \mathbb{R}$ and $a \in \mathbb{Z}^+$. Recall Theorem \ref{thm:comparison1}, and we will focus on the case when $(\#_1,\#_2,\#_3) =(\mathfrak{X}_{ab}\Im [G^\gamma(z)]_{ab}, \;\mathfrak{X}_{ab}\Im [G^0(z)]_{ab}, \;\mathfrak{I}_{0,ab}   )$ therein.
	Applying Theorem \ref{thm:comparison1} with $F(x) = F_p(x)$, we have for any $a, b \in [M]$, there exists constant $C_1 > 0$ such that,
	\begin{align*}
		&\E_\Psi \big(F_p( \mathfrak{X}_{ab}\Im [G^\gamma(z)]_{ab})  \big)  - \E_\Psi \big(F_p(\mathfrak{X}_{ab}\Im [G^0(z)]_{ab})  \big)  < C_1N^{-\omega}(\mathfrak{I}_{p,0} + 1) + C_1Q_0N^{C_1}, 
	\end{align*}
	where
	$\mathfrak{I}_{p,0} \equiv \sup_{{i , j \in [M], 0 \le \gamma \le 1}} \E_\Psi \big( \big|  F_p( \mathfrak{X}_{ij}\Im [G^{\gamma}(z)]_{ij}) \big|\big).
	$
	Taking supremum over $a,b \in [M]$ and $0 \le \gamma \le 1$ yields
	\begin{align*}
		&(1 - C_1N^{-\omega}) \mathfrak{I}_{p,0} \le \max_{i,j \in [M]}\E_\Psi \big(F_p(\mathfrak{X}_{ij}\Im [G^0(z)]_{ij})  \big)  + C_1N^{-\omega} + 3C_1 N^{C_1}\max_{\substack{k \in [0:2]}}\mathfrak{P}_{k}(\varepsilon,z,\Psi).
	\end{align*}	
	Applying Lemma \ref{lem:continuitylemma} on $\mathcal{R}(Y^\gamma, z) = z^{-1/2}(\mathcal{L}(Y^\gamma) - z^{1/2})^{-1}$ with $z^{1/2} = \tilde{E} + \mathrm{i}\tilde{\eta}$, we have,
	\begin{align*}
		\max_{i,j\in[M+N]} |z^{1/2}[\mathcal{R}(Y^\gamma, z)]_{ij}| \vee 1 \le L \Big(\max_{i,j\in[M+N]} |(z')^{1/2}[\mathcal{R}(Y^\gamma, z')]_{ij}|\vee 1 \Big),
	\end{align*} 
	for any $L > 0$ and $z' \in \mathbb{C}^+$ satisfies $(z')^{1/2} = \tilde{E} +  \mathrm{i}L\tilde{\eta}$. Let $L \equiv N^{\varepsilon/6}$ and thus $(z')^{1/2} \equiv \tilde{E} +  \mathrm{i}N^{\varepsilon/6}\tilde{\eta}$, to obtain
	\begin{align*}
		&\max_{i,j\in[M]} |z^{1/2} [G^\gamma(z)]_{ij}| \vee 1,\; \max_{i,j\in[N]} |z^{1/2} [\mathcal{G}^\gamma(z)]_{ij}| \vee 1,\;  \max_{i \in [M],j\in[N]} | [G^\gamma(z)Y^\gamma]_{ij}|\le \mathfrak{S},
	\end{align*}
	where
	\begin{align*}
		\mathfrak{S}\equiv N^{\varepsilon/6} \Big(\max_{i,j\in[M]} |(z')^{1/2} [G^\gamma(z')]_{ij}| \vee \max_{i,j\in[N]} |(z')^{1/2} [\mathcal{G}^\gamma(z')]_{ij}|\vee \max_{i\in[M],j\in [N]} | [G^\gamma(z')Y^\gamma]_{ij}| \vee 1\Big).
	\end{align*}
	This implies that
	\begin{align}
				\max_{\substack{k \in [0:2] }}\mathfrak{P}_{k}(\varepsilon,z,\Psi) \le \max_{\substack{k \in [0:2] }}\mathfrak{P}_{k}(\varepsilon/2,z',\Psi).
	\label{eq:continuity}\end{align}	
	
	 For any  $z_0  =  E_0 + \mathrm{i}\eta_0 \in \mathsf{D}(\varepsilon_1,\varepsilon_2,\varepsilon_3)$, we have $\E_\Psi \big(F_p(\mathfrak{X}_{ij} \Im [G^0(z_0)]_{ij})  \big) \lesssim N$ (cf. Theorem \ref{thm: resolvent entry size V_t}). Then there exists some large constant $C_2 > 0$ such that
	\begin{align*}
		\mathfrak{I}_{p,0} \le C_2N + C_2 N^{C_2} \max_{\substack{k \in [0:2] }}\mathfrak{P}_{k}(\varepsilon,z_0,\Psi).
	\end{align*}
Using (\ref{eq:continuity}) by setting $z \equiv z_0$, we have for $z_1 = E_1 + \mathrm{i}\eta_1$ where $(E_1,\eta_1)$ are defined through (\ref{eq:E2eta2}),
	\begin{align*}
		\mathfrak{I}_{p,0}\le C_2N + C_2 N^{C_2} \max_{\substack{k \in [0:2] }}\mathfrak{P}_{k}(\varepsilon/2,z_1,\Psi).
	\end{align*}
	For any $a,b \in [M]$, and $ 0\le \gamma \le 1$, applying Markov's inequality with the fact that $p\delta > D + 100$, we have that there exists some large constant $C_3 > 0$ such that 
\begin{align*}
	\mathbb{P}_{\Psi}\Big( |z_0^{1/2}\mathfrak{X}_{ab} [\Im G^{\gamma}(z_0)]_{ab}| >  N^{\delta}\Big) &\le \frac{  |z_0|^{p/2} \E_\Psi \big( \big|  F_p(\mathfrak{X}_{ab}\Im [G^{\gamma}(z_0)]_{ab}) \big|\big)}{N^{p\delta}} \le  \frac{|z_1|^{p/2}\mathfrak{I}_{p,0}}{N^{p\delta}} \\
	&\le C_3N^{-D-90} + C_3N^{C_2}\max_{\substack{k \in [0:2] }}\mathfrak{P}_{k}(\varepsilon/2,z_1,\Psi),
\end{align*}
where in the last step we used the fact that $|z_0|$ is bounded. Similar bound holds when $\Im$ is replaced by $\Re$,  we omit the details.
	Now we may apply union bounds on $i , j \in [M]$ and an $\epsilon$-net argument on $\gamma$ with the following deterministic bounds
	\begin{align*}
		\bigg|\frac{\partial [G^{\gamma}(z)]_{ab} }{\partial \gamma}\bigg| \lesssim \frac{\|A \| + \gamma\| t^{1/2}W \|}{\eta^2}, 
	\end{align*}
	$\eta > N^{-1}$, $\|A \|\le N^{1/2}$ and $\mathbb{P}(\|t^{1/2}W \| > 2) < N^{-D}$, to obtain that
	\begin{align*}
		\mathfrak{P}_{0}(\delta,z_0,\Psi) = &\mathbb{P}_{\Psi}\Big( \sup_{\substack{a,b \in [M] \\ 0 \le \gamma \le 1}} |z_0^{1/2}\mathfrak{X}_{ab}[G^{\gamma}(z_0)]_{ab}| >  N^{\delta}\Big) \\
		&\le C_4N^{-D-50} + C_4N^{C_4}\max_{\substack{k \in [0:2] }}\mathfrak{P}_{k}(\varepsilon/2,z_1,\Psi),
	\end{align*}
	for some large constant $C_4 > 0$. Repeating the above procedure for all  $\mathfrak{P}_{k}(\delta,\eta,\Psi), k = 1,2$ proves the claim. 
\end{proof}

\subsection{Proof of Corollary \ref{cor.081801}}
\label{4216}
We prove this corollary using a similar argument as in [Section 4, \cite{PY}] or [Section 4, \cite{HLY}]. The key inputs are the rigidity estimate in Theorem \ref{rigidity} and the Green function comparison in Theorem \ref{Green function comparison thm}.
\begin{proof}[Proof of Corollary \ref{cor.081801}]
	Let us first define for any $E$,
	\begin{align*}
		\mathcal{N}(E) := \big|\{ i: \lambda_{i}(\mathcal{S}(Y)) \le \lambda_{-,t}+ E \} \big|.
	\end{align*}
	For any $\epsilon > 0$, we take $\ell = N^{-2/3-\epsilon/3}$ and $\eta = N^{-2/3 - \epsilon}$. Recall from  Theorem \ref{rigidity} that $\lambda_M(\mathcal{S}(Y)) \ge   \lambda_{-,t} - N^{-2/3+\epsilon}$ holds  with high probability. We further define 
	\begin{align*}
		&{\chi}_E(x) := \mathbf{1}_{[-N^{-2/3+\epsilon}, E]}(x - \lambda_{-,t}),\\
		&\theta_\eta(x) := \frac{\eta}{\pi (x^2 + \eta^2)} = \frac{1}{\pi} \im \frac{1}{x-\mathrm{i}\eta}. 
	\end{align*}
	Then following the same arguments as in [Lemma 2.7, \cite{AY2011}], we can obtain that for $|E| \le N^{-2/3+\epsilon}$, the following holds with high probability:
	\begin{align*}
		\Tr (\chi_{E-\ell} \ast \theta_{\eta})(\mathcal{S}(Y)) - N^{-\epsilon/9} \le \mathcal{N}(E)  \le \Tr (\chi_{E+\ell} \ast \theta_{\eta})(\mathcal{S}(Y)) + N^{-\epsilon/9}.
	\end{align*}
	Let $K(x) : \mathbb{R} \to [0,1]$ be a smooth monotonic increasing function such that 
	\begin{align*}
		K(x) = 1 \quad \text{if} \quad x \ge 2/3, \quad  K(x) = 0 \quad \text{if} \quad x \le 1/3.
	\end{align*} 
	Therefore, we have with high probability that
	\begin{align*}
		K( \Tr (\chi_{E-\ell} \ast \theta_{\eta})(\mathcal{S}(Y))) + \mathcal{O}(N^{-\epsilon/9}) &\le K(\mathcal{N}(E) ) = \mathbf{1}_{\mathcal{N}(E)  \ge 1} \\
		& \le K( \Tr (\chi_{E+\ell} \ast \theta_{\eta})(\mathcal{S}(Y))) + \mathcal{O}(N^{-\epsilon/9}).
	\end{align*}
	Taking expectation on the above inequality, we have for $|s| \le  N^{\epsilon}/2$ that
	\begin{align}
		&\E \Bigg[K \bigg(\Im \bigg[ \frac{N}{\pi}\int_{-N^{-2/3+\epsilon}}^{sN^{-2/3}-\ell} m^1(\lambda_{-,t}+y+\mathrm{i}\eta) \bigg]\mathrm{d}y  \bigg)  \Bigg] + \mathcal{O}(N^{-\epsilon/9})\notag\\
		&\le \mathbb{P} \Big(N^{2/3}(\lambda_M(\mathcal{S}(Y)) -\lambda_{-,t} ) \le s \Big) = \E \Big[  \mathbf{1}_{\mathcal{N}(sN^{-2/3})  \ge 1}  \Big]\notag\\
		&\le \E \Bigg[K \bigg(\Im \bigg[ \frac{N}{\pi}\int_{-N^{-2/3+\epsilon}}^{sN^{-2/3}+\ell} m^1(\lambda_{-,t}+y+\mathrm{i}\eta) \bigg]\mathrm{d}y  \bigg)  \Bigg] + \mathcal{O}(N^{-\epsilon/9}).
	\label{eq:probality1}\end{align}
	Similarly, repeating the above arguments with $\mathcal{S}(Y)$ replaced by $\mathcal{S}(V_t)$, we can also have
	\begin{align}
		&\E \Bigg[K \bigg(\Im \bigg[ \frac{N}{\pi}\int_{-N^{-2/3+\epsilon}}^{sN^{-2/3}-\ell} m^0(\lambda_{-,t}+y+\mathrm{i}\eta) \bigg]\mathrm{d}y  \bigg)  \Bigg] + \mathcal{O}(N^{-\epsilon/9})\notag\\
		&\le \mathbb{P} \Big(N^{2/3}(\lambda_M(\mathcal{S}(V_t)) -\lambda_{-,t} ) \le s \Big)\notag\\
		&\le \E \Bigg[K \bigg(\Im \bigg[ \frac{N}{\pi}\int_{-N^{-2/3+\epsilon}}^{sN^{-2/3}+\ell} m^0(\lambda_{-,t}+y+\mathrm{i}\eta) \bigg]\mathrm{d}y  \bigg)  \Bigg] + \mathcal{O}(N^{-\epsilon/9}).
	\label{eq:probality2}\end{align}
	Note that the conditional expectation $\E_\Psi$ in (\ref{eq:GFcomparison 2}) can be replaced by $\E$ using the law of total expectation together with the fact that $\Omega_\Psi$ holds with high probability. Therefore, we can combine (\ref{eq:probality1}) and (\ref{eq:probality2}) with (\ref{eq:GFcomparison 2}) to obtain that
	\begin{align*}
	\mathbb{P} \Big(N^{2/3}(\lambda_M(\mathcal{S}(V_t)) -\lambda_{-,t} ) \le s - 2\ell N^{-2/3} \Big)+ \mathcal{O}(N^{-\epsilon/9})	\le \mathbb{P} \Big(N^{2/3}(\lambda_M(\mathcal{S}(Y)) -\lambda_{-,t} ) \le s \Big)\\
		\le\mathbb{P} \Big(N^{2/3}(\lambda_M(\mathcal{S}(V_t)) -\lambda_{-,t} ) \le s + 2\ell N^{-2/3} \Big)+ \mathcal{O}(N^{-\epsilon/9}).
	\end{align*}  
	Now (\ref{eq:convergeindistribution1}) follows by the fact that $\ell N^{-2/3} \ll 1$. For (\ref{eq:convergeindistribution2}), we first note by Theorem \ref{thm: lambda -,t asymp} that
	\begin{align}
		|\lambda_{-,t} - \lambda_{\mathsf{shift}}| \le N^{-2/3+\epsilon}
	\label{eq:lambdainprobalityest}\end{align}
	holds in probability. This together with Theorem \ref{rigidity} implies that
	\begin{align*}
		|\lambda_{M}(\mathcal{S}(Y)) - \lambda_{\mathsf{shift}}| \le N^{-2/3+\epsilon}
	\end{align*}
	also holds in probability. Then we may proceed similar to the proof of (\ref{eq:convergeindistribution1}), but with all high probability estimates replaced by in probability estimates. It's worth noting that during the derivation of (\ref{eq:probality1}) and (\ref{eq:probality2}), the error term $\mathcal{O}(N^{-\epsilon/9})$ will become $\mathfrak{o}(1)$ because we lack an polynomial bound for the failure probability of (\ref{eq:lambdainprobalityest}). Finally, we can conclude the proof of (\ref{eq:convergeindistribution2}) by using Theorem \ref{Green function comparison thm}.
\end{proof}

\subsection{Proof of Theorem \ref{Green function comparison thm}}
\label{4271}

\begin{proof}
	To ease presentation, we show the proof of the following comparison instead: for any $|E| \le N^{-2/3+\epsilon}$,
	\begin{align}
	\Big|\mathbb{E}_\Psi \Big(F(N\eta_0\Im m^1(\lambda_{-,t}+E+\mathrm{i}\eta_0))\Big)-\mathbb{E}_\Psi \Big(F(N\eta_0\Im m^0(\lambda_{-,t}+E+\mathrm{i}\eta_0))\Big)\Big|\leq CN^{-\delta_1}. \label{eq:GFcomparison 1}
	\end{align}
	The proof of (\ref{eq:GFcomparison 2}) is similar, and thus we omit it. Using the same notation as in the proof of Theorem \ref{thm:comparison2} and further defining $h_{\gamma,(ij)}(\lambda,\beta) \equiv \eta_0\sum_{a}f_{\gamma,(aa),(ij)}(\lambda,,\beta)$, we have
	\begin{align*}
		&\frac{\partial \E_\Psi \big(F(N\eta_0\Im m^\gamma(z_t))\big)}{\partial \gamma} 
		= -2\Big(\sum_{i,j} (I_1)_{ij} - (I_2)_{ij}\Big),
	\end{align*}
	with
	\begin{align*}
		&(I_1)_{ij} \equiv \E_\Psi\bigg[A_{ij}F'\Big(h_{\gamma,(ij)}\big([Y^\gamma]_{ij},X_{ij}\big)  \Big) g_{(ij)}\big([Y^\gamma]_{ij},X_{ij}\big) \bigg],\\
		&(I_2)_{ij} \equiv \frac{\gamma t^{1/2} }{(1-\gamma^2)^{1/2}}\E_\Psi\bigg[w_{ij} F'\Big(h_{\gamma,(ij)}\big([Y^\gamma]_{ij},X_{ij}\big)  \Big) g_{(ij)}\big([Y^\gamma]_{ij},X_{ij}\big)  \bigg].
	\end{align*}
	We first consider the estimation for $(I_1)_{ij}$. Notice that $(I_1)_{ij}$ can be further decomposed as
	\begin{align*}
		(I_1)_{ij} = (I_1)_{ij}\cdot \mathbf{1}_{\psi_{ij} = 0} + (I_1)_{ij}\cdot \mathbf{1}_{\psi_{ij} = 1} = (I_1)_{ij}\cdot \mathbf{1}_{\psi_{ij} = 0},
	\end{align*}
	where in the last step we used the fact that $A_{ij}\cdot \mathbf{1}_{\psi_{ij} = 1} = 0$. Therefore, we only need to consider the case when $\psi_{ij} = 0$, and $(I_1)_{ij}$ can be rewritten as
	\begin{align*}
		(I_1)_{ij} = \E_\Psi\bigg[(1-\chi_{ij})a_{ij}F'\Big(h_{\gamma,(ij)}(d_{ij},\chi_{ij}b_{ij})  \Big) g_{(ij)}(d_{ij},\chi_{ij}b_{ij}) \bigg]\cdot \mathbf{1}_{\psi_{ij} = 0}.
	\end{align*} 
	By Taylor expansion, for an $s_1 > 0$ to be chosen later, there exists $\tilde{d}_{ij} \in [0,d_{ij}]$ such that,
	\begin{align*}
		(I_1)_{ij} &= \sum_{k_1=0}^{s_1}\frac{1}{k_1!} \E_\Psi\bigg[(1-\chi_{ij})a_{ij}d_{ij}^{k_1} g_{(ij)}^{(k_1,0)}(0,\chi_{ij}b_{ij}) F'\big(h_{\gamma,(ij)}(d_{ij},\chi_{ij}b_{ij})  \big)\bigg]\cdot \mathbf{1}_{\psi_{ij} = 0} \\
		&\quad+ \frac{1}{(s_1+1)!} \E_\Psi\bigg[(1-\chi_{ij})a_{ij}d_{ij}^{s_1+1} g_{(ij)}^{(s_1+1,0)}(\tilde{d}_{ij},\chi_{ij}b_{ij}) F'\big(h_{\gamma,(ij)}(d_{ij},\chi_{ij}b_{ij})  \big)\bigg]\cdot \mathbf{1}_{\psi_{ij} = 0} \\
		&\equiv \sum_{k_1=0}^{s_1}(I_1)_{ij,k_1} + \mathsf{Rem}_1.
	\end{align*}
	Using (\ref{eq:Recurderivative})-(\ref{eq:G2Yest}), ,the perturbation argument as in (\ref{eq:perturbbound}), and the fact that $\Im m^\gamma(z_t) \prec 1$, we have for any (small)$\epsilon > 0$ and (large)$D > 0$,
	\begin{align*}
		\mathbb{P}_\Psi\bigg(\Omega_{\epsilon,1}:= \Big\{ \big| g_{(ij)}^{(s_1+ 1,0)}(\tilde{d}_{ij},\chi_{ij}b_{ij}) F'\big(h_{\gamma,(ij)}(d_{ij},\chi_{ij}b_{ij})  \big) \big|\cdot \mathbf{1}_{\psi_{ij} = 0} < t^{-s_1-2}N^{\epsilon} \Big\}  \bigg) \ge 1 - N^{-D}.
	\end{align*}
	Further, by the Gaussianity of $w_{ij}$, we have
	\begin{align*}
		\mathbb{P}_\Psi\bigg(\Omega_{\epsilon,2}:=\Big\{ \max_{i \in [M],j \in [N]} |t^{1/2}w_{ij}| < N^{-1/2+\epsilon} \Big\}  \bigg) \ge 1 - N^{-D}.
	\end{align*}
	Let $\Omega_{\epsilon} := \Omega_{\epsilon,1}\cap \Omega_{\epsilon,2}$. Then
	\begin{align}
		|\mathsf{Rem}_1| &\lesssim \E_\Psi\bigg[|(1-\chi_{ij})a_{ij}d_{ij}^{s_1+ 1}| \cdot \big| g_{(ij)}^{(s_1+ 1,0)}(\tilde{d}_{ij},\chi_{ij}b_{ij}) F'\big(h_{\gamma,(ij)}(d_{ij},\chi_{ij}b_{ij})  \big)\big|\cdot \mathbf{1}_{\Omega_\epsilon} \bigg]\cdot \mathbf{1}_{\psi_{ij} = 0}\notag \\
		&\quad +\E_\Psi\bigg[|(1-\chi_{ij})a_{ij}d_{ij}^{s_1+1}| \cdot \big| g_{(ij)}^{(s_1+ 1,0)}(\tilde{d}_{ij},\chi_{ij}b_{ij}) F'\big(h_{\gamma,(ij)}(d_{ij},\chi_{ij}b_{ij})  \big)\big|\cdot \mathbf{1}_{\Omega^c_\epsilon} \bigg]\cdot \mathbf{1}_{\psi_{ij} = 0} \notag\\
		&\overset{(\mathrm{i})}{\lesssim} \E_\Psi\bigg[|(1-\chi_{ij})a_{ij}d_{ij}^{s_1+1}| \cdot \big| g_{(ij)}^{(s_1+ 1,0)}(\tilde{d}_{ij},\chi_{ij}b_{ij}) F'\big(h_{\gamma,(ij)}(d_{ij},\chi_{ij}b_{ij})  \big)\big|\cdot \mathbf{1}_{\Omega_\epsilon} \bigg]\cdot \mathbf{1}_{\psi_{ij} = 0} \notag\\
		&\quad+ N^{-D+ C_1+2(s_1+3)}\notag\\
		&\overset{(\mathrm{ii})}{\lesssim} \frac{N^{\epsilon}}{N^{1/2+\epsilon_b(s_1+1)}t^{s_1+2}} ,
	\label{eq:estimateRem1}\end{align}
	where in $(\mathrm{i})$ we used the deterministic bound $\big| g_{(ij)}^{(s_1+ 1,0)}(\tilde{d}_{ij},\chi_{ij}b_{ij}) F'\big(h_{\gamma,(ij)}(d_{ij},\chi_{ij}b_{ij})  \big)\big| \le N^{C_1 + 2(s_1 + 3)}$ when $\eta \ge N^{-2}$, and $(\mathrm{ii})$ is a consequence of the definition of $\Omega_\epsilon$. Choosing $s_1$ sufficiently large, i.e., $s_1 > 4/\epsilon_b$, and $t \gg N^{-\epsilon_b/2}$ we can obtain
	\begin{align*}
		|\mathsf{Rem}_1| \lesssim N^{-5/2}.
	\end{align*}
	For $(I_1)_{ij,k_1}$, we need to further expand $F'\big( h_{\gamma,(ij)}(d_{ij})  \big)$ as follows:
	\begin{align*}
		F'\big( h_{\gamma,(ij)}(d_{ij},\chi_{ij}b_{ij})  \big) = \sum_{k=0}^{s_2} \frac{d_{ij}^k}{k!} \frac{\partial^k F'}{\partial d_{ij}^k}\big( h_{\gamma,(ij)}(0,\chi_{ij}b_{ij})  \big)  + \frac{d_{ij}^{s_2+1}}{(s_2+1)!} \frac{\partial^k F'}{\partial d_{ij}^k}\big( h_{\gamma,(ij)}(\hat{d}_{ij},\chi_{ij}b_{ij})  \big),
	\end{align*}
	where $s_2$ is a positive integer to be chosen later, and $\hat{d}_{ij} \in [0,d_{ij}]$. Then $(I_1)_{ij,k_1}$ can be rewritten as,
	\begin{align*}
		&(I_1)_{ij,k_1} =  \sum_{k_2 = 0}^{s_2} \frac{1}{k_1!k_2!}\E_\Psi\bigg[(1-\chi_{ij})a_{ij}d_{ij}^{k_1+k_2} g_{(ij)}^{(k_1,0)}(0,\chi_{ij}b_{ij}) \frac{\partial^{k_2} F'}{\partial d_{ij}^{k_2}}\big( h_{\gamma,(ij)}(0,\chi_{ij}b_{ij})  \big)\bigg]\cdot \mathbf{1}_{\psi_{ij} = 0} \\
		&+ \frac{1}{k_1!(s_2+1)!}\E_\Psi\bigg[(1-\chi_{ij})a_{ij}d_{ij}^{k_1+s_2+1} g_{(ij)}^{(k_1,0)}(0,\chi_{ij}b_{ij}) \frac{\partial^{s_2+1} F'}{\partial d_{ij}^{s_2+1}}\big( h_{\gamma,(ij)}(\hat{d}_{ij},\chi_{ij}b_{ij})  \big)\bigg]\cdot \mathbf{1}_{\psi_{ij} = 0}\\
		&\equiv  \sum_{k_2 = 0}^{s_2}(I_1)_{ij,k_1k_2} + \mathsf{Rem}_2.
	\end{align*} 
	By Fa\`{a} di Bruno's formula, we have for any integer $n > 0$,
	\begin{align}
		\frac{\partial^n F'}{\partial d_{ij}^n}\big( h_{\gamma,(ij)}(d_{ij},\chi_{ij}b_{ij})  \big) &= \sum_{(m_1,\cdots,m_n)} \frac{n!}{m_1!m_2!\cdot m_n!}\cdot F^{(m_1+\cdots+m_n+1)}\big( h_{\gamma,(ij)}(d_{ij},\chi_{ij}b_{ij})  \big)\notag\\
		&\qquad\qquad\qquad\times \prod_{\ell=1}^n\bigg(\frac{h^{(\ell)}_{\gamma,(ij)}(d_{ij},\chi_{ij}b_{ij})}{\ell!} \bigg)^{m_\ell}
	\label{eq:FdBfomula}\end{align}
	Considering  (\ref{eq:FdBfomula}), (\ref{eq:Recurderivative})-(\ref{eq:G2Yest}), and using the perturbation argument as described in (\ref{eq:perturbbound}), we arrive at the following result:
		\begin{align}
		\frac{\partial^{s_2+1} F'}{\partial d_{ij}^{s_2+1}}\big( h_{\gamma,(ij)}(\hat{d}_{ij},\chi_{ij}b_{ij})  \big) \prec \prod_{\ell=1}^nt^{-(\ell+ 1)m_\ell} \le t^{-2n}.
	\label{eq:Fpbound}\end{align}
	Moreover, taking into account the fact that $g_{(ij)}^{(k_1)}(0) \prec t^{-(k_1+1)}$, we can deduce that:
	\begin{align*}
		|\mathsf{Rem}_2| \lesssim \frac{N^\epsilon}{N^{1/2+\epsilon_b(k_1+s_2+1)}t^{k_1+2(s_2+1)}} \lesssim N^{-5/2},
	\end{align*}
	where, for the final step, we have chosen $s_2 \ge 4/\epsilon_b$ and $t \gg N^{-\epsilon_b/4}$. 
	Next, we estimate $(I_1)_{ij,k_1k_2}$ in different cases. 

	\textbf{Case 1:} $k_1+k_2$ is even. By the law of total expectation, 
	\begin{align}
		&(I_1)_{ij,k_1k_2}\notag\\
		 &= \frac{\mathbf{1}_{\psi_{ij} = 0}}{k_1!k_2!}\sum_{n=0}^1 \E_\Psi\bigg[(1-\chi_{ij})a_{ij}d_{ij}^{k_1+k_2} g_{(ij)}^{(k_1,0)}(0,\chi_{ij}b_{ij}) \frac{\partial^{k_2} F'}{\partial d_{ij}^{k_2}}\big( h_{\gamma,(ij)}(0,\chi_{ij}b_{ij})  \big)\bigg| \chi_{ij} = n \bigg] \mathbb{P}(\chi_{ij} = n) \notag\\
		&= \frac{\mathbf{1}_{\psi_{ij} = 0} }{k_1!k_2!}\E_\Psi\bigg[a_{ij}d_{ij}^{k_1+k_2} g_{(ij)}^{(k_1,0)}(0,0) \frac{\partial^{k_2} F'}{\partial d_{ij}^{k_2}}\big( h_{\gamma,(ij)}(0,0)  \big)\bigg| \chi_{ij} = 0 \bigg]\mathbb{P}(\chi_{ij} = 0) \notag\\
		&=\frac{ \mathbf{1}_{\psi_{ij} = 0}= 0 }{k_1!k_2!} \E_\Psi\bigg[a_{ij}d_{ij}^{k_1+k_2}  \bigg| \chi_{ij} = 0\bigg]  \E_\Psi\bigg[g_{(ij)}^{(k_1,0)}(0,0) \frac{\partial^{k_2} F'}{\partial d_{ij}^{k_2}}\big( h_{\gamma,(ij)}(0,0)  \big)\bigg]\mathbb{P}(\chi_{ij} = 0),
	\label{eq:condtiononChi}\end{align}
	where the last step follows from the symmetry condition.

	\textbf{Case 2:} $k_1+k_2$ is odd and $k_1 + k_2 \ge 5$. Similar to (\ref{eq:condtiononChi}), we have
	\begin{align*}
		|(I_1)_{ij,k_1k_2}| 
		&\lesssim\bigg|\E_\Psi\bigg[a_{ij}d_{ij}^{k_1+k_2}  \bigg| \chi_{ij} = 0\bigg]\bigg|  \E_\Psi\bigg[|g_{(ij)}^{(k_1,0)}(0,0)| \bigg|\frac{\partial^{k_2} F'}{\partial d_{ij}^{k_2}}\big( h_{\gamma,(ij)}(0,0)  \big)\bigg| \bigg| \chi_{ij} = 0 \bigg]\mathbb{P}(\chi_{ij} = 0) \mathbf{1}_{\psi_{ij} = 0} \\
		&\lesssim \frac{1}{N^{2+2\epsilon_a + (k_1+k_2 - 3)\epsilon_b}}\E_\Psi\bigg[|g_{(ij)}^{(k_1,0)}(0,\chi_{ij}b_{ij})| \bigg|\frac{\partial^{k_2} F'}{\partial d_{ij}^{k_2}}\big( h_{\gamma,(ij)}(0,\chi_{ij}b_{ij})  \big)\bigg| \bigg]\mathbf{1}_{\psi_{ij} = 0,\chi_{ij} = 0}. 
	\end{align*}
	We may again obtain the bound $|g_{(ij)}^{(k_1)}(0,\chi_{ij}b_{ij})|\cdot \mathbf{1}_{\psi_{ij} = 0,\chi_{ij} = 0} \prec t^{-(k_1+1)}$ by  (\ref{eq:Recurderivative})-(\ref{eq:G2Yest}), and the perturbation argument as described in (\ref{eq:perturbbound}). Using (i)equation (\ref{eq:FdBfomula}) with $d_{ij}$ replaced by $0$, and (ii)the following rank inequality,
	\begin{align}
		&|h_{\gamma,(ij)}(0,\chi_{ij}b_{ij}) - h_{\gamma,(ij)}(d_{ij},\chi_{ij}b_{ij})|\cdot\mathbf{1}_{\psi_{ij} = 0,\chi_{ij} = 0} \le 2\eta_0 \big( \| G^{\gamma,d_{ij}}_{(ij)}(z_t) \| + \| G^{\gamma,0}_{(ij)}(z_t) \|\big)\mathbf{1}_{\psi_{ij} = 0,\chi_{ij} = 0}\le 2,
	\label{eq:rankinequalityforh}\end{align}
	with the fact that $h_{\gamma,(ij)}(d_{ij},\chi_{ij}b_{ij})\cdot\mathbf{1}_{\psi_{ij} = 0,\chi_{ij} = 0} \prec 1$, we can obtain that
	\begin{align}
		\bigg|\frac{\partial^{k_2} F'}{\partial d_{ij}^{k_2}}\big( h_{\gamma,(ij)}(0,\chi_{ij}b_{ij})  \big)\bigg|\cdot\mathbf{1}_{\psi_{ij} = 0,\chi_{ij} = 0} \prec  t^{-2k_2}.
	\label{eq:FFdijbound}\end{align}
	Combining the above estimates and choosing $t \gg N^{-\epsilon_b/8}$, we arrive at
	\begin{align*}
		|(I_1)_{ij,k_1k_2}| \lesssim \frac{N^{\epsilon}}{N^{2+2\epsilon_a + (k_1+k_2 - 3)\epsilon_b} t^{k_1+1+2k_2}} \lesssim \frac{1}{N^{2+2\epsilon_a}}.
	\end{align*}

	\textbf{Case 3:} $k_1+k_2 = 3$. The estimation in this case is similar to Case 2 above, but we need to use the bound $g_{(ij)}^{(k_1,0)}(0,\chi_{ij}b_{ij}) \prec 1$ when $i \in \mathcal{T}_r$ and $j \in \mathcal{T}_c$. Recall that $|\mathcal{D}_r| \vee |\mathcal{D}_c| \le N^{1-\epsilon_d}$. Then we have
	\begin{align*}
		|(I_1)_{ij,k_1k_2}| 
		&\lesssim \frac{N^\epsilon}{N^{2+2\epsilon_a }} \cdot \mathbf{1}_{\psi_{ij} = 0}\cdot \mathbf{1}_{i \in \mathcal{T}_r, j \in \mathcal{T}_c} + \frac{1}{N^{2-\epsilon_d}} \cdot\frac{N^\epsilon}{N^{2\epsilon_a +\epsilon_d}t^{k_1+1+2k_2}} \cdot \mathbf{1}_{\psi_{ij} = 0}\cdot (1 - \mathbf{1}_{i \in \mathcal{T}_r, j \in \mathcal{T}_c}) \\
		&\lesssim \frac{1}{N^{2+\epsilon_a }} \cdot \mathbf{1}_{\psi_{ij} = 0}\cdot \mathbf{1}_{i \in \mathcal{T}_r, j \in \mathcal{T}_c} + \frac{1}{N^{2-\epsilon_d+\epsilon_a}} \cdot \mathbf{1}_{\psi_{ij} = 0}\cdot (1 - \mathbf{1}_{i \in \mathcal{T}_r, j \in \mathcal{T}_c}),
	\end{align*}
	where in the last step, we used the fact $t \gg N^{-\epsilon_d/8}$.

	\textbf{Case 4:} $k_1+k_2 = 1$. In this case, using (\ref{eq:condtiononChi}) we may compute that
	\begin{align*}
		(I_1)_{ij,k_1k_2} = \E_{\Psi}\big[\gamma a_{ij}^2 \big]\cdot \E_\Psi\bigg[g_{(ij)}^{(k_1,0)}(0,0) \frac{\partial^{k_2} F'}{\partial d_{ij}^{k_2}}\big( h_{\gamma,(ij)}(0,0)  \big)\bigg| \chi_{ij} = 0 \bigg]\cdot \mathbb{P}(\chi_{ij} = 0)  \cdot \mathbf{1}_{\psi_{ij} = 0}.
	\end{align*}
	We note that there will be corresponding terms in $(I_2)_{ij}$, and these terms will cancel out with the ones described above.
	
	Combining the estimates in the above cases, we can obtain that there exists some constant $\delta_1 = \delta_1(\epsilon_a)$ such that 
	\begin{align}
		\sum_{i,j}(I_1)_{ij} &= \sum_{i,j}\sum_{\substack{k_1, k_2 \ge 0, \\k_1 + k_2 = 1}} \E_{\Psi}\big[\gamma a_{ij}^2 \big]\cdot \E_\Psi\bigg[g_{(ij)}^{(k_1,0)}(0,0) \frac{\partial^{k_2} F'}{\partial d_{ij}^{k_2}}\big( h_{\gamma,(ij)}(0,0)  \big)\bigg]\notag\\
		&\qquad\qquad\qquad\qquad\qquad\qquad\times \mathbb{P}(\chi_{ij} = 0)  \cdot \mathbf{1}_{\psi_{ij} = 0} + \mathcal{O}(N^{-\delta_1}).
	\label{eq:compREI1}\end{align}
	
	Next, we consider the estimation for $(I_2)_{ij}$. When $\psi_{ij} = 1$, we can apply Gaussian integration by parts to obtain that
	\begin{align*}
		|(I_2)_{ij}\cdot\mathbf{1}_{\psi_{ij} =1}| \lesssim \frac{t^{1/2}}{N}\E_\Psi\bigg[ \Big|\partial_{w_{ij}} \big\{ g_{(ij)}(e_{ij},c_{ij})  F'\big( h_{\gamma,(ij)}(e_{ij},c_{ij})  \big) \big\} \Big|\bigg]		\cdot\mathbf{1}_{\psi_{ij} =1} \lesssim \frac{N^{\epsilon}}{Nt}\cdot\mathbf{1}_{\psi_{ij} =1},
	\end{align*} 
	where the last step follows from (\ref{eq:Recurderivative})-(\ref{eq:G2Yest}). 
	The estimation for $(I_2)_{ij} \cdot \mathbf{1}_{\psi_{ij} = 0}$ is similar to those of $(I_1)_{ij}$, we omit repetitive details. In summary, with the independence between $z_t$ and $w_{ij}$, we have by possibly adjusting $\delta_1$, 
	\begin{align}
		&\sum_{i,j}(I_2)_{ij} =  \sum_{i,j}(I_2)_{ij}\cdot \mathbf{1}_{\psi_{ij} = 0}+ \sum_{i,j}(I_2)_{ij}\cdot \mathbf{1}_{\psi_{ij} = 1} \notag\\
		 &=\sum_{i,j}\sum_{\substack{k_1, k_2 \ge 0, \\k_1 + k_2 = 1}} \E_{\Psi}\big[\gamma tw_{ij}^2 \big]\E_\Psi\bigg[g_{(ij)}^{(k_1,0)}(0,\chi_{ij}b_{ij}) \frac{\partial^{k_2} F'}{\partial d_{ij}^{k_2}}\big( h_{\gamma,(ij)}(0,\chi_{ij}b_{ij})  \big)\bigg] \cdot \mathbf{1}_{\psi_{ij} = 0} + \mathcal{O}(N^{-\delta_1}).
	\label{eq:compREI2}\end{align}
	Note by (\ref{eq:2ndmomentdiff}) and the choices of $\epsilon_a$ and $\epsilon_b$, we have
	\begin{align*}
		\E_{\Psi}\big[\gamma a_{ij}^2 \big] \mathbb{P}(\chi_{ij} = 0) - \E_{\Psi}\big[\gamma tw_{ij}^2 \big] = \mathcal{O}\bigg( \frac{t}{N^{2+2\epsilon_b}} \bigg).
	\end{align*}
	This together with the $t$ dependent bounds for $g_{(ij)}^{(k_1,0)}$ and $\partial^{k_2}F'/(\partial d_{ij}^{k_2})$ implies that it suffices to bound the following quantity:
	\begin{align*}
		\mathsf{G} := \bigg( \E_\Psi\bigg[g_{(ij)}^{(k_1,0)}(0,\chi_{ij}b_{ij}) \frac{\partial^{k_2} F'}{\partial d_{ij}^{k_2}}\big( h_{\gamma,(ij)}(0,\chi_{ij}b_{ij})  \big)\bigg] -\E_\Psi\bigg[g_{(ij)}^{(k_1,0)}(0,0) \frac{\partial^{k_2} F'}{\partial d_{ij}^{k_2}}\big( h_{\gamma,(ij)}(0,0)  \big) \bigg]\bigg)\cdot \mathbf{1}_{\psi_{ij} = 0}
	\end{align*}
		To provide a more precise distinction between (\ref{eq:compREI1}) and (\ref{eq:compREI2}), we let $$\mathsf{F}_{k_1,k_2}(z^{(ij)}_t(\beta)) := g_{(ij)}^{(k_1)}(0,\beta) \frac{\partial^{k_2} F'}{\partial d_{ij}^{k_2}}\big( h_{\gamma,(ij)}(0,\beta)  \big).$$ Therefore, 
	\begin{align*}
		\mathsf{G} =\bigg(\E_\Psi\bigg[\mathsf{F}_{k_1,k_2}\big(z_t(\chi_{ij}b_{ij})\big)\bigg] -\E_\Psi\bigg[ \mathsf{F}_{k_1,k_2}\big(z_t(0)\big) \bigg] \bigg)\cdot \mathbf{1}_{\psi_{ij} = 0}. 
	\end{align*}
	We may apply Taylor expansion to obtain that
	\begin{align}
		&\bigg(\E_\Psi\bigg[\mathsf{F}_{k_1,k_2}\big(z_t(\chi_{ij}b_{ij})\big)\bigg] -\E_\Psi\bigg[ \mathsf{F}_{k_1,k_2}\big(z_t(0)\big) \bigg] \bigg)\cdot \mathbf{1}_{\psi_{ij} = 0}\notag\\
		&= \E_\Psi\bigg[\chi_{ij}^2b_{ij}^2 \mathsf{F}'_{k_1,k_2}\big(z_t(b)\big)\cdot \frac{\partial^2 \lambda_{-,t}}{\partial B_{ij}^2}(b)  \bigg]\cdot \mathbf{1}_{\psi_{ij} = 0}+ \E_\Psi\bigg[\chi_{ij}^2b_{ij}^2 \mathsf{F}''_{k_1,k_2}\big(z_t(b)\big)\cdot \bigg(\frac{\partial \lambda_{-,t}}{\partial B_{ij}}(b)\bigg)^2  \bigg]\cdot \mathbf{1}_{\psi_{ij} = 0},
	\label{eq:FchibF0}\end{align}
	with $b \in [0,B_{ij}]$. Here the first oder term disappeared due to symmetry. 
	To bound the above terms we need to first verify that $z_t(b)$ still lies inside $\mathsf{D}$ (w.h.p). This can be done by noting that for the replacement matrix $X_{(ij)}(b)$ which replace the $B_{ij}$ by $b$ in $X$ still satisfies the $\eta^*$- regularity. Therefore by Weyl's inequality,
	\begin{align}
		|\lambda_{-,t}(\chi_{ij}b_{ij}) - \lambda_{-,t}(b)| &\prec   |\lambda_{-,t}(\chi_{ij}b_{ij})  - \lambda_M(\mathcal{S}(X))| +  |\lambda_M(\mathcal{S}(X)) - \lambda_M(\mathcal{S}(X_{(ij)}(b)))| \notag\\ 
		&+ |\lambda_M(\mathcal{S}(X_{(ij)}(b))) -\lambda_{-,t}(b)|  \prec N^{-2/3} +  N^{-\epsilon_b} + N^{-2/3} \prec  N^{-\epsilon_b}.
	\label{eq:lambdab}\end{align}	
	Applying the perturbation argument as in (\ref{eq:perturbbound}) to relate $g_{(ij)}^{(k_1)}(0,b)$ back to $g_{(ij)}^{(k_1)}(d_{ij},b)$, and then using (\ref{eq:lambdab}) to verify that $z^{(ij)}_t(b) \in \mathsf{D}$, we can see that the bound $g_{(ij)}^{(k_1)}(0,b) \prec t^{-(k_1+1)}$ still holds. Similarly, we can also obtain $h^{(k_2)}_{\gamma,(ij)}(0, b) \prec t^{-k_2}$ for $k_2 \ge 1$. For the case when $k_2 = 0$, we may use (\ref{eq:rankinequalityforh}) and the fact that $N\eta_0\Im m^\gamma(z^{(ij)}_t(b)) \prec 1$ to conclude that $h_{\gamma,(ij)}(0, b) \prec 1$. Combining the above bounds with a Cauchy integral argument, we have
	\begin{align*}
		\mathsf{F}'_{k_1,k_2}\big(z_t(b)\big) \prec \frac{1}{\eta_0t^2},\quad  \mathsf{F}''_{k_1,k_2}\big(z_t(b)\big) \prec \frac{1}{\eta_0^2t^2}.
	\end{align*}
	Further using Lemma \ref{lem:lambdaderibound}, we have for arbitrary (small)$\epsilon > 0$ and (large)$D > 0$, 
	\begin{align*}
		\mathbb{P} \bigg( \Omega := \Big\{ \Big|\mathsf{F}'_{k_1,k_2}\big(z_t(b)\big)\cdot \frac{\partial^2 \lambda_{-,t}}{\partial B_{ij}^2}(b)\Big| \le \frac{N^{\epsilon}}{N\eta_0t^7}  \Big\} \bigcap \Big\{ \Big|\mathsf{F}''_{k_1,k_2}\big(z_t(b)\big)\cdot \bigg(\frac{\partial \lambda_{-,t}}{\partial B_{ij}}(b)\bigg)^2\Big| \le \frac{N^{\epsilon}}{N^2\eta_0^2t^8}  \Big\} \bigg) \ge 1 -N^{-D}.
	\end{align*}
	Since 
	\begin{align*}
	&\chi_{ij}^2b_{ij}^2 \Big( \mathsf{F}'_{k_1,k_2}\big(z_t(b)\big)\cdot \frac{\partial^2 \lambda_{-,t}}{\partial B_{ij}^2}(b)+\mathsf{F}''_{k_1,k_2}\big(z_t(b)\big)\cdot \bigg(\frac{\partial \lambda_{-,t}}{\partial B_{ij}}(b)\bigg)^2 \Big)\cdot \mathbf{1}_{\psi_{ij} = 0}\\
		&=\Big(\mathsf{F}_{k_1,k_2}\big(z_t(\chi_{ij}b_{ij})\big) - \mathsf{F}_{k_1,k_2}\big(z_t(0)\big)  \Big)\cdot \mathbf{1}_{\psi_{ij} = 0} - \Big(\chi_{ij}b_{ij}F'_{k_1,k_2}(z_t(0))\cdot \frac{\partial \lambda_{-,t}}{\partial B_{ij}}(0) \Big)\cdot \mathbf{1}_{\psi_{ij} = 0}, 
	\end{align*} 
	the deterministic upper bound for the left hand side of the above equation follows from (\ref{eq:dblambdaderi}) in Lemma \ref{lem:lambdaderibound} and the fact that $\Im z_t \ge N^{-1}$. Then we may follow the steps as in (\ref{eq:estimateRem1}) to obtain that
	\begin{align*}
		\E_\Psi\bigg[\chi_{ij}^2b_{ij}^2 \Big( \mathsf{F}'_{k_1,k_2}\big(z_t(b)\big)\cdot \frac{\partial^2 \lambda_{-,t}}{\partial B_{ij}^2}(b)+\mathsf{F}''_{k_1,k_2}\big(z_t(b)\big)\cdot \bigg(\frac{\partial \lambda_{-,t}}{\partial B_{ij}}(b)\bigg)^2 \Big)\bigg] \cdot \mathbf{1}_{\psi_{ij} = 0} \lesssim \frac{N^{\epsilon}}{N^2\eta_0t^7}
	\end{align*}
	Therefore, with the fact that $\E_{\Psi}[\gamma a_{ij}^2]\mathbb{P}(\chi_{ij} = 0) \sim t\E_{\Psi}[\gamma w_{ij}^2] = \gamma	t/N$, we have by possibly adjusting $\delta_1$,
	\begin{align*}
		\Big|\sum_{i,j}(I_1)_{ij} - (I_2)_{ij}\Big| =\sum_{i,j} \frac{\gamma t}{N}\sum_{\substack{k_1, k_2 \ge 0, \\k_1 + k_2 = 1}} \Big|\E_\Psi\Big[\mathsf{F}_{k_1,k_2}\big(z_t(\chi_{ij}b_{ij})\big)\Big] -\E_\Psi\Big[ \mathsf{F}_{k_1,k_2}\big(z_t(0)\big) \Big] \Big|\mathbf{1}_{\psi_{ij} = 0}+ \mathcal{O}(N^{-\delta_1}) = \mathcal{O}(N^{-\delta_1}).
	\end{align*}
	This together with the arguments as in (\ref{eq:integralest1})-(\ref{eq:integralest2}) completes the proof of (\ref{eq:GFcomparison 1}). 
The proof for the case $\alpha = 8/3$ closely parallels, and is in fact simpler, primarily due to the absence of randomness in $\lambda_{\mathsf{shift}}$. Thus we omit the details. This concludes the proof.
\end{proof}


\bibliographystyle{plain}
\normalem

\end{document}